\documentclass[a4paper,reqno,12pt]{amsart} 
\usepackage{amsmath,amssymb}

\usepackage{longtable,array}
\usepackage{booktabs}
\newlength\FirstColumn
\newlength\Spacer
\usepackage{amsmath} 

\usepackage{amssymb}      
\usepackage{mathdots}  

\usepackage{amsthm}      

\usepackage{tikz,amsmath}
\usetikzlibrary{arrows}

\usepackage{epsfig}      

\usepackage{graphicx}
\usepackage[normalem]{ulem}
\usepackage[all,line,
]{xy} 
\CompileMatrices 

\newcommand{\id}{\operatorname{id}} 
\newcommand{\Aut}{\operatorname{Aut}}

\newcommand{\Span}{\operatorname{Span}}
\newcommand{\Tr}{\operatorname{Tr}}

\newcommand{\Int}{\operatorname{Int}}

\newcommand{\Br}{\operatorname{Br}}
\newcommand{\Wan}{\operatorname{Wan}}
\newcommand{\Ray}{\operatorname{Ray}}

\newcommand{\Rray}{\operatorname{Rray}}

   \theoremstyle{plain}
   \newtheorem{thm}{Theorem}[section]
   \newtheorem{prop}[thm]{Proposition}
   \newtheorem{lemma}[thm]{Lemma}  
   \newtheorem{cor}[thm]{Corollary}
   \theoremstyle{definition}
   
   \newtheorem{defn}[thm]{Definition}
   \newtheorem{example}[thm]{Example}
   \theoremstyle{remark}
   \newtheorem{obs}[thm]{Observation}
   \newtheorem{remark}[thm]{Remark}

\newtheorem{?}[thm]{Question}

\usepackage{mdframed,xcolor}

\definecolor{mybgcolor}{gray}{0.8}
\definecolor{myframecolor}{rgb}{.647,.129,.149}

\mdfdefinestyle{mystyle}{
  usetwoside=false,
  skipabove=0.6em plus 0.8em minus 0.2em,
  skipbelow=0.6em plus 0.8em minus 0.2em,
  innerleftmargin=.25em,
  innerrightmargin=0.25em,
  innertopmargin=0.25em,
  innerbottommargin=0.25em,
  leftmargin=-.75em,
  rightmargin=-0em,
  topline=false,
  rightline=false,
  bottomline=false,
  leftline=false,
  backgroundcolor=mybgcolor,
  splittopskip=0.75em,
  splitbottomskip=0.25em,
  innerleftmargin=0.5em,
  leftline=true,
  linecolor=myframecolor,
  linewidth=0.25em,
}

\newmdenv[style=mystyle]{important}

\usepackage{lipsum}

   \numberwithin{equation}{section}

        \date{\today}

\title[KMS states, conformal measures and ends in digraphs]{KMS states, conformal measures and ends in digraphs}  
\author{Klaus Thomsen}

\date{\today}

\email{matkt@math.au.dk}
\address{Institut for Matematik, Aarhus University, Ny Munkegade, 8000 Aarhus C, Denmark}

\begin{document}

\maketitle

\begin{abstract} The paper develops a series of tools for the study of KMS-weights on graph $C^*$-algebras and KMS-states on their corners. The approach adopts methods and ideas from graph theory, random walks and dynamical systems.  
\end{abstract}

\tableofcontents

\section{Introduction}\label{introsec}

Recent work has shown that the structure of KMS-weights on a simple graph $C^*$-algebra can exhibit a remarkable richness and complexity already when the one-parameter group  is the so-called gauge action which arises naturally from the construction of the algebra. The fact that the structure of the algebra is so relatively transparent, and the action so elementary means that the structure of its KMS-weights is amenable to a much more detailed investigation than what is possible with the examples of Bratteli, Elliott, Herman and Kishimoto, \cite{BEH},\cite{BEK}, which are the only known examples of one-parameter actions on simple $C^*$-algebras with a similar complexity in the KMS-structure. In this paper we therefore continue the exploration of the structure of KMS-weights and KMS-states for actions on simple graph $C^*$-algebras and their corners which was started in \cite{Th1} and \cite{Th3}. The overall goal here is to develop tools which can be applied to all (simple) graph $C^*$-algebras equipped with a generalized gauge action and to enlarge the class of directed graphs for which the structure of the KMS-weights can be completely determined. In addition we develop methods to construct examples in a controlled way.
Recall that a one-parameter group of automorphisms on a $C^*$-algebra is used as a model in quantum statistical mechanics, \cite{BR}, where the one-parameter group represents the time evolution of the system and the KMS states represent the equilibrium states. So not only as a tool to illustrate the general results we obtain, but also to make the theory potentially useful as a source of models in quantum statistical mechanics it is important to have ways to obtain a desired structure of KMS states. For these reasons a good part of the paper is devoted to the development of such methods.

 In this work three things are achieved:
\begin{itemize}
\item A class of directed graphs, which we call \emph{meager}, is presented and the structure of KMS-weights for the generalized gauge actions on the $C^*$-algebra on a meager graph is fully explored. 
\item It is shown how the structure of KMS-weights for a generalized gauge action on the $C^*$-algebra of a row-finite strongly connected digraph can be determined from similar actions on primitive AF-algebras. In fact, for many graphs including Cayley graphs of finitely generated groups, it suffices to consider generalized gauge actions on simple AF-algebras. 
\item Methods are developed to construct examples, and in particular we present an example of a simple unital $C^*$-algebra with a periodic one-parameter action such that there are $\beta$-KMS states for all $\beta $ in an infinite half-line and the simplices of $\beta$-KMS states are not affinely homeomorphic for any pair of different $\beta$'s in the interval.     
\end{itemize}

The class of meager digraphs greatly extends both the class of digraphs obtained by considering trees as digraphs in the natural way and the class of graphs with countably many exits which were studied in \cite{Th3}. The structure of KMS-weights for generalized gauge actions on the $C^*$-algebras of meager graphs are more complex, but the methods which we develop for their study are natural extensions of methods from \cite{Th3}.

The results alluded to in the second item show that the overall problem we consider can be reduced, at least in principle and for row-finite graphs, to the same problem for primitive or simple Bratteli diagrams. In this setting the problem has been considered by Jean Renault in \cite{R1} and \cite{R2} as a special case of what he calls the Radon-Nikodym problem. We develop here the methods by which KMS-weights for generalized gauge actions on the $C^*$-algebra of a row-finite graph can be decomposed into weights that arise from generalized gauge actions on primitive or in many cases even simple AF-algebras.

The example mentioned in the last item exhibit a behaviour which has not been observed before. The example we construct is the gauge action on a corner of the $C^*$-algebra of a strongly connected row-finite graph. The possible inverse temperatures realized by such an action is always an infinite half-line, if not the empty set, and hence not at all as exotic as the sets of possible inverse temperatures which can occur in the examples of Bratteli, Elliott and Herman, \cite{BEH}. The virtue of our example is that it exhibits a variation with $\beta$ of the simplices which is as wild as at all conceivable. Furthermore, it is possible to make a fusion of the example with the examples from \cite{BEH} by using tensor products of $C^*$-algebras and one-parameter actions. The result are simple unital $C^*$-algebras with a continuous one-parameter action for which the set of inverse temperatures can be an arbitrary closed subset $F$ of positive real numbers, and such that the simplices of $\beta$-KMS states exhibit the same extreme variation within $F$ as the examples constructed here. Our examples can also be given other guises in which it may have better chances of being recognized by mathematicians not involved with operator algebras or mathematical physics. Indeed, the adjacency matrices of the graphs we construct in Section \ref{gluing} show that for all $h> 1$ there is an irreducible row-finite infinite matrix $A$ with entries from $\mathbb N \cup \{0\}$ such that there are no non-zero solutions $\psi \in [0,\infty[^{\mathbb N}$ to the equation
\begin{equation}\label{03-03-18}
A\psi = \lambda \psi
\end{equation}
when $\lambda < h$ while for $\lambda \geq h$ the solutions that are normalized such that $\psi_1 = 1$ form a Bauer-simplex $S_{\lambda}$ with the property that $S_{\lambda}$ is only affinely homeomorphic to $S_{\lambda'}$ when $\lambda = \lambda'$. Formulated in this way the example supplements a fundamental result by Pruitt, specifically the corollary on page 1799 in \cite{Pr}.

The endeavour we undertake requires that we use methods and ideas from graph theory and from the theory of Markov chains; more specifically, from random walks. The problem we consider is rather closely related to the problem of finding the harmonic functions of a random walk as will be explained in detail in Section \ref{randomw}. In some cases results or methods from random walks can be used directly to obtain all the information we are looking for as we show by example, but in general it is necessary to modify and supplement the methods used for random walks. For example, the end space of an undirected graph is well-known and used in connection with random walks, but this is not the case for digraphs. There is a notion of an end space for digraphs considered by Zuther in \cite{Zu}, several actually, but Zuther does not put a topology on the end spaces he considers and it is crucial for our purposes that the end space of a digraph comes equipped with a natural and well-behaved topology. We therefore describe the end space which Zuther has introduced in a way that brings this topology to light. When a Bratteli diagram is considered as a digraph in the obvious way, its end space is homeomorphic to the primitive ideal space of the AF-algebra it defines, equipped with the Fell topology. Another example concerns the Martin boundary which is used to represent the harmonic functions of a countable state Markov chain as integrals. While this devise is available to some extend in the setting we consider, it does not give the best setting for a study of the conformal measures that define the KMS-weights and states since they are measures on the path space of the graph and not on any boundary. The Martin kernels that define the Martin boundary are crucial also here, but the Martin boundary itself is redundant. In fact, we show in an appendix that there is a universally measurable subset of rays in the graph which can serve as the supporting space of an integral representation both of the conformal measures and of the harmonic vectors that determine the KMS-weights we are interested in, even in cases where the harmonic vectors can have zeroes and where the Martin boundary is usually not defined.

The paper is organized as follows. 
\begin{itemize}
\item Section \ref{recap} describes the background results and mark out the field for the following sections. In particular, it is explained why the focus is on graphs with wandering rays and no sinks. 
\item Section \ref{sec1} establishes the bijective correspondence between harmonic vectors and conformal measures which is fundamental to the project, and the relation to Markov chains and random walks is explained in detail. 
\item Section \ref{extcon1} combines Choquet- and disintegration theory with the methods used to prove the convergence to the boundary theorem for countable state Markov chains in order to obtain a pivotal description of the extremal conformal measures.
\item Section \ref{tools} is devoted to the development of tools for the construction of strongly connected graphs out of graphs that are not strongly connected, in order to control the behaviour of conformal measures and hence the KMS-weights. 
\item Section \ref{endspace} contains the construction of the end space of a digraph. It is compared to the end space of an undirected graph and used to disintegrate conformal measures.
\item Section \ref{meager} introduces the meager digraphs, describes the structure of conformal measures on their path spaces and give examples. 
\item Section \ref{bratdiag} shows how to get from general row-finite graphs to primitive and simple Bratteli diagrams in the study of conformal measures. The ends of digraphs play a fundamental role for this. 
\item Section \ref{gluing} is devoted to the construction of strongly connected row-finite graphs such that there is a corner in the associated $C^*$-algebra for which the gauge action exhibits the wildest possible variation of KMS-simplices. 
\end{itemize}

We end the paper with three appendices. In the first we justify the unproven assertions made in Section \ref{recap} and in the second we obtain the version of the Poisson-Martin integral representation alluded to above. In the third and final appendix we study the tensor product of the one-parameter actions constructed in Section \ref{gluing} with the examples from \cite{BEH} in order to obtain the fusion mentioned above.


\bigskip

\emph{Acknowledgement} I am grateful to Johannes Christensen for discussions and help to minimize the number of mistakes, and to George Elliott for an e-mail exchange regarding the examples in \cite{BEH} and Appendix \ref{OlaGeorge}. The work was supported by the DFF-Research Project 2 `Automorphisms and Invariants of Operator Algebras', no. 7014-00145B.

\section{Recap on KMS weights and generalized gauge actions on graph $C^*$-algebras}\label{recap}

Let $A$ be a $C^*$-algebra and $A_{+}$ the convex cone of positive
elements in $A$. A \emph{weight} on $A$ is map $\psi : A_{+} \to
[0,\infty]$ with the properties that $\psi(a+b) = \psi(a) + \psi(b)$
and $\psi(\lambda a) = \lambda \psi(a)$ for all $a, b \in A_{+}$ and all
$\lambda \in \mathbb R, \ \lambda > 0$. By definition $\psi$ is
\emph{densely defined} when $\left\{ a\in A_{+} : \ \psi(a) <
  \infty\right\}$ is dense in $A_{+}$ and \emph{lower semi-continuous}
when $\left\{ a \in A_{+} : \ \psi(a) \leq \alpha \right\}$ is closed
for all $\alpha \geq 0$. We refer to \cite{KV} and the references
therein for more information on weights. As in \cite{KV} we say that a weight is \emph{proper}
when it is non-zero, densely defined and lower semi-continuous. A \emph{ray of weights} is a set of the form $\left\{ \lambda \phi: \ \lambda > 0\right\}$ for some proper weight $\phi$.

Let $\psi$ be a proper weight on $A$. Set $\mathcal N_{\psi} = \left\{ a \in A: \ \psi(a^*a) < \infty
\right\}$ and note that 
\begin{equation*}\label{f3}
\mathcal N_{\psi}^*\mathcal N_{\psi} = \Span \left\{ a^*b : \ a,b \in
  \mathcal N_{\psi} \right\}
\end{equation*} 
is a dense
$*$-subalgebra of $A$, and that there is a unique well-defined linear
map $\mathcal N_{\psi}^*\mathcal N_{\psi} \to \mathbb C$ which
extends $\psi : \mathcal N_{\psi}^*\mathcal N_{\psi} \cap A_+ \to
[0,\infty)$. We denote also this densely defined linear map by $\psi$.

Let $\gamma : \mathbb R \to \Aut A$ be a point-wise
norm-continuous one-parameter group of automorphisms on
$A$. Let $\beta \in \mathbb R$. Following \cite{C} we say that a proper weight
$\psi$ on $A$ is a \emph{$\beta$-KMS
  weight} for $\gamma$ when
\begin{enumerate}
\item[i)] $\psi \circ \gamma_t = \psi$ for all $t \in \mathbb R$, and
\item[ii)] for every pair $a,b \in \mathcal N_{\psi} \cap \mathcal
  N_{\psi}^*$ there is a continuous and bounded function $H$ defined on
  the closed strip $D_{\beta}$ in $\mathbb C$ consisting of the numbers $z \in \mathbb C$
  whose imaginary part lies between $0$ and $\beta$, and is
  holomorphic in the interior of the strip and satisfies that
$$
H(t) = \psi(a\gamma_t(b)), \ H(t+i\beta) = \psi(\gamma_t(b)a)
$$
for all $t \in \mathbb R$. 
\end{enumerate}   
A $\beta$-KMS weight $\psi$ with the property that 
$$
\sup \left\{ \psi(a) : \ 0 \leq a \leq 1 \right\} = 1
$$
will be called a \emph{$\beta$-KMS state}. This is consistent with the
standard
definition of KMS states, \cite{BR}, except when $\beta = 0$ in which
case our definition requires that a $0$-KMS state, which is also a trace state, is $\gamma$-invariant. We are particularly interested in the extremal $\beta$-KMS weights since a general $\beta$-KMS weight is the average, in a certain sense, of extremal $\beta$-KMS weights. Here a $\beta$-KMS weight $\varphi$ is \emph{extremal} when it only dominates multiplies of itself, i.e. when every $\beta$-KMS weight $\psi$ such that $\psi \leq \varphi$ has the form $\psi = \lambda \varphi$ for some $\lambda > 0$.

Let $\Gamma$ be a countable directed graph with vertex set $\Gamma_V$ and arrow set
$\Gamma_{Ar}$. \label{GammaV}\label{GammaAr}For an arrow $a \in \Gamma_{Ar}$ we denote by $s(a) \in \Gamma_V$ its source and by
$r(a) \in \Gamma_V$ its range. An \emph{infinite path} in $\Gamma$ is an element
$p = (p_i)_{i=1}^{\infty}  \in \left(\Gamma_{Ar}\right)^{\mathbb N}$ such that $r(p_i) = s(p_{i+1})$ for all
$i$.  A finite path $\mu = a_1a_2 \cdots a_n = (a_i)_{i=1}^n \in \left(\Gamma_{Ar}\right)^n$ is
defined similarly. The number of arrows in $\mu$ is its \emph{length}
and we denote it by $|\mu|$. A vertex $v \in \Gamma_V$ will be considered as
a finite path of length $0$.

We let $P(\Gamma)$\label{PGamma} denote the (possibly empty) set of
infinite paths in $\Gamma$ and $P_f(\Gamma)$\label{PfGamma} the set of finite paths in $\Gamma$. The set $P(\Gamma)$ is a complete metric space when the metric is given by
\begin{equation*}\label{25-02-18}
d(p,q) = \sum_{i=1}^{\infty} 2^{-i}\delta(p_i,q_i),
\end{equation*}
where $\delta(a,a) = 0$ and $\delta(a,b) = 1$ when $a \neq b$. We
extend the source map to $P(\Gamma)$ such that
$s(p) = s(p_1)$ when $p = \left(p_i\right)_{i=1}^{\infty} \in P(\Gamma)$, and the
range and source maps to $P_f(\Gamma)$ such that $s(\mu) = s(a_1)$ and $r(\mu)
= r(a_{|\mu|})$ when $|\mu|\geq 1$, and $s(v) = r(v) = v$ when $v\in \Gamma_V$. A vertex $v$ which does not emit any arrow is
a \emph{sink}, while a vertex $v$ which emits infinitely many arrows (that is, $\# s^{-1}(v) = \infty$,) will be called an \emph{infinite
  emitter}. The set of sinks and infinite emitters in  $\Gamma$ is denoted by $V_{\infty}$. We shall sometimes assume that $\Gamma$ is \emph{row-finite} in the sense that all vertexes only emit a finite number of arrows, i.e. $\# s^{-1}(v) < \infty$ for all $v \in \Gamma_V$. In that case $P(\Gamma)$ is locally compact.

The $C^*$-algebra $C^*(\Gamma)$ \label{C*g}of the graph $\Gamma$ was introduced in this generality in \cite{BHRS} as the universal
$C^*$-algebra generated by a collection $S_a, a \in \Gamma_{Ar}$, of partial
isometries and a collection $P_v, v \in \Gamma_V$, of mutually orthogonal projections subject
to the conditions that
\begin{enumerate}
\item[1)] $S^*_aS_a = P_{r(a)}, \ \forall a \in \Gamma_{Ar}$,
\item[2)] $S_aS_a^* \leq P_{s(a),} \ \forall a \in \Gamma_{Ar}$,
\item[3)] $\sum_{a \in F} S_aS_a^* \leq P_v$ for all $v \in
  \Gamma_V$ when $F\subseteq s^{-1}(v)$ is a finite subset, and
\item[4)] $P_v = \sum_{a  \in s^{-1}(v)} S_aS_a^*, \ \forall v \in \Gamma_V
  \backslash V_{\infty}$.
\end{enumerate} 
For a finite path $\mu = (a_i)_{i=1}^{|\mu|} \in P_f(\Gamma)$ of positive length we set
$$
S_{\mu} = S_{a_1}S_{a_2}S_{a_3} \cdots S_{a_{|\mu|}} \ ,
$$
while $S_{\mu} = P_v$ when $\mu$ is the vertex $v$. The elements $S_{\mu}S_{\nu}^*, \mu,\nu \in P_f(\Gamma)$, span a dense $*$-subalgebra in $C^*(\Gamma)$.

A function $F :  \Gamma_{Ar} \to \mathbb R$ will be called a \emph{potential} on $\Gamma$ in the following. Using it we can define a continuous one-parameter group $\alpha^F = \left(\alpha^F_t\right)_{t \in \mathbb R}$  on $C^*(\Gamma)$ such that \label{alphaF}
$$
\alpha^F_t(S_a) = e^{i F(a) t} S_a
$$
for all $a \in \Gamma_{Ar}$ and 
$$
\alpha^F_t(P_v) = P_v
$$
for all $v \in \Gamma_V$. Such an action is called a \emph{generalized gauge action}; the \emph{gauge action} itself being the one-parameter group $\alpha = \left(\alpha_t\right)_{t \in \mathbb R}$, corresponding to the constant function $F =1$.

 Since we are mostly interested in the case where $C^*(\Gamma)$ is a simple $C^*$-algebra and since it simplifies several things, we will assume this here. In the course of the following development, however, it will be essential that we consider much more general digraphs even when our main interest is in graphs for which the $C^*$-algebra is simple. We describe here the necessary and sufficient conditions which $\Gamma$ must satisfy for $C^*(\Gamma)$ to be simple. These conditions were identified by Szymanski in \cite{Sz}. A \emph{loop} in $\Gamma$ is a finite path $\mu \in P_f(\Gamma)$ of positive length such that $r(\mu) = s(\mu)$. We will say that a loop $\mu$ \emph{has an exit} then $\# s^{-1}(v) \geq 2$ for at least one vertex $v$ in $\mu$. A subset $H \subseteq \Gamma_V$ is \emph{hereditary} when $a \in
\Gamma_{Ar}, \ s(a) \in H \Rightarrow r(a) \in H$, and \emph{saturated} when 
$$
v \in \Gamma_V\backslash V_{\infty}, \  
r(s^{-1}(v))  \subseteq H \ \Rightarrow \ v\in H.
$$ 
In the following we say that $\Gamma$ is \emph{cofinal} when the only
non-empty subset of $\Gamma_V$ which is both hereditary and
saturated is $\Gamma_V$ itself.

\begin{thm}\label{Sz} (Theorem 12 in \cite{Sz}.) $C^*(\Gamma)$ is simple if and only if $\Gamma$ is cofinal and every loop in $\Gamma$ has an exit.
\end{thm}

In particular, $C^*(\Gamma)$ can be simple when there are infinite emitters, and even when there is a sink. But there can not be more than one sink by Corollary 4.2 in \cite{Th3}. We note that when $\Gamma$
is \emph{strongly connected}, meaning that for every pair of vertexes $v,w$ there is
a finite path $\mu$ such that $s(\mu ) = v$ and $r(\mu) =w$, then $\Gamma_V$
contains no proper non-empty hereditary subset and $\Gamma$ is therefore also cofinal. Hence $C^*(\Gamma)$ is simple when $\Gamma$ is strongly connected, except when $\Gamma$ only consists of a single loop.


For $\beta \in \mathbb R$, let $A(\beta) = \left(A(\beta)_{v,w}\right)_{v,w \in \Gamma_V}$ be the matrix over $\Gamma_V$ defined such that\label{Abeta}
$$
A(\beta)_{v,w} \ \ = \sum_{ a \in s^{-1}(v) \cap r^{-1}(w)} e^{-\beta F(a)} \ ,
$$
where we sum over all arrows $a$ going from $v$ to $w$. The entries of $A(\beta)$ are non-negative and we can unambiguously define the higher powers $A(\beta)^n$ of $A(\beta)$ in the usual way, although the entries of $A(\beta)^n$ may be infinite when $\Gamma$ is not row-finite. By definition $A(\beta)^0$ is the 'identity matrix', i.e. $A(\beta)^0_{v,w} = 0$ when $v \neq w$ while $A(\beta)^0_{v,v} = 1$ for all $v \in \Gamma_V$. When $F = 1$, where the action $\alpha^F$ is the gauge action $\alpha$, we have that
$$
A(\beta) = e^{-\beta}A(\Gamma) \ ,
$$
where $A(\Gamma)$\label{AGamma} denotes the adjacency matrix of $\Gamma$, i.e.
$$
A(\Gamma)_{v,w} = \# r^{-1}(w) \cap s^{-1}(v) \ . 
$$ 

As in \cite{Th2} we say that a non-zero non-negative vector $\psi = \left(\psi_v\right)_{v \in \Gamma_V}$ is \emph{almost $A(\beta)$-harmonic} when 
\begin{itemize}
\item $\sum_{w \in \Gamma_V} A(\beta)_{v,w} \psi_w \leq \psi_v, \ \forall v \in \Gamma_V$, and
\item $\sum_{w \in \Gamma_V} A(\beta)_{v,w} \psi_w = \psi_v, \ \forall v \in \Gamma_V \backslash V_{\infty}$.
\end{itemize} 
The almost $A(\beta)$-harmonic vectors $\psi$ for which 
\begin{itemize}
\item $\sum_{w \in \Gamma_V} A(\beta)_{v,w} \psi_w = \psi_v, \ \forall v \in \Gamma_V,$
\end{itemize}
will be called \emph{$A(\beta)$-harmonic}. In particular, when $\Gamma$ is row-finite without sinks an almost $A(\beta)$-harmonic vector is automatically $A(\beta)$-harmonic. An almost $A(\beta)$-harmonic vector which is not $A(\beta)$-harmonic will be said to be a \emph{proper almost $A(\beta)$-harmonic vector}.

\begin{thm}\label{intro1} Assume that $C^*(\Gamma)$ is simple and let $\alpha^F$ be a generalized gauge action on $C^*(\Gamma)$. 
\begin{itemize}
\item There is a bijection between the set of $\beta$-KMS weights for $\alpha^F$ and the set of almost $A(\beta)$-harmonic vectors.
\item The $\beta$-KMS weight $W_{\psi}$ corresponding to an almost $A(\beta)$-harmonic vector $\psi = \left(\psi_v\right)_{v \in \Gamma_V}$ satisfies that $S_{\mu}^* \in \mathcal N_{W_{\psi}}$ and
$$
W_{\psi}\left(S_{\mu}S_{\nu}^*\right) = \begin{cases} 0, & \ \mu \neq \nu \\  e^{-\beta F( \mu)} \psi_{r(\mu)}, & \ \mu = \nu  \end{cases}
$$
when $\mu, \nu \in P_f(\Gamma)$.
\end{itemize}
\end{thm}
\begin{proof} Under the present assumption all $\beta$-KMS weights are gauge-invariant by Proposition 5.6 in \cite{CT2}. Therefore the statement follows from Theorem 2.7 in \cite{Th3}.
\end{proof} 

The proof of Theorem \ref{intro1} from \cite{Th1}, \cite{Th3} and \cite{CT2} used a description of $C^*(\Gamma)$ as the $C^*$-algebra of a locally compact \'etale groupoid, but this point of view will not be needed here.

It follows from Theorem 2.4 in \cite{Th3} that there is a bijection between the rays of $\beta$-KMS weights for $\alpha^F$ and the $\beta$-KMS states for the restriction of $\alpha^F$ to any corner $pC^*(\Gamma)p$ defined by a projection $p$ in the fixed point algebra of $\alpha^F$ which is full in $C^*(\Gamma)$. The bijection sends a $\beta$-KMS weight $\varphi$ on $C^*(\Gamma)$ to the state 
$$
\varphi(p)^{-1}\varphi|_{pC^*(\Gamma)p}
$$ 
on $pC^*(\Gamma)p$. Therefore Theorem \ref{intro1} implies the following.

\begin{thm}\label{intro2} Assume that $C^*(\Gamma)$ is simple and let $\alpha^F$ be a generalized gauge action on $C^*(\Gamma)$. Fix a vertex $v_0 \in \Gamma_V$ and let $\gamma^F$ denote the one-parameter group of automorphisms on $P_{v_0}C^*(\Gamma)P_{v_0}$ defined by restricting $\alpha^F_t$ to $P_{v_0}C^*(\Gamma)P_{v_0}$. 
\begin{itemize}
\item There is an affine bijection between the set of $\beta$-KMS states for $\gamma^F$ and the set of almost $A(\beta)$-harmonic vectors $\psi$ that are $v_0$-normalized in the sense that $\psi_{v_0} = 1$.
\item The $\beta$-KMS state $w_{\psi}$ corresponding to a $v_0$-normalized almost $A(\beta)$-harmonic vector $\psi = \left(\psi_v\right)_{v \in \Gamma_V}$ satisfies that
$$
w_{\psi}\left(S_{\mu}S_{\nu}^*\right) = \begin{cases} 0, & \ \mu \neq \nu \\  e^{-\beta F( \mu)} \psi_{r(\mu)}, & \ \mu = \nu  \end{cases}
$$
when $\mu, \nu \in P_f(\Gamma) \cap s^{-1}(v_0)$.
\end{itemize}
\end{thm}

Because of these theorems the study of the $\beta$-KMS weights for $\alpha^F$, or the $\beta$-KMS states for its restriction to a corner, begins with a search for almost $A(\beta)$-harmonic vectors. To explain what is already known consider first the case where $A(\beta)$ is \emph{recurrent} in the sense that 
\begin{equation}\label{27-02-18a}
\sum_{n=0}^{\infty} A(\beta)^n_{v,v} = \infty
\end{equation}
for some $v \in \Gamma_V$. When this happens there are no proper almost $A(\beta)$-harmonic vectors, and by Proposition 4.9 and Theorem 4.14 in \cite{Th3} there exists an $A(\beta)$-harmonic vector if and only if
\begin{equation}\label{20-11-17}
\limsup_n \left(A(\beta)^n_{v,v}\right)^{\frac{1}{n}} = 1 \ ,
\end{equation}
and it is then unique up to multiplication by scalars. So at least in principle we know in the recurrent case when there are any almost $A(\beta)$-harmonic vectors and how many. Consider then the case when $A(\beta)$ is \emph{transient}, i.e. not recurrent. Since $\Gamma$ is cofinal this means that
\begin{equation}\label{transient0}
\sum_{n=0}^{\infty} A(\beta)^n_{v,w} < \infty
\end{equation}
for all $v,w \in \Gamma_V$ by Lemma 4.1 in \cite{Th2}. The situation is \emph{much} more complicated in this case and to describe what is known, 
recall from \cite{Th3} that an infinite path $p \in P(\Gamma)$ is \emph{wandering} when the sequence $\{s(p_i)\}$ of vertexes goes to infinity in the sense that for all finite subsets $F \subseteq \Gamma_V$ there is an $N \in \mathbb N$ such that $s(p_i) \notin F$ when $i \geq N$. In the following we denote by $\Wan(\Gamma)$\label{WanG} the set of wandering paths in $\Gamma$. We can then divide the cofinal directed graphs into five types, called A,B,C,D and E, according to the following diagram.
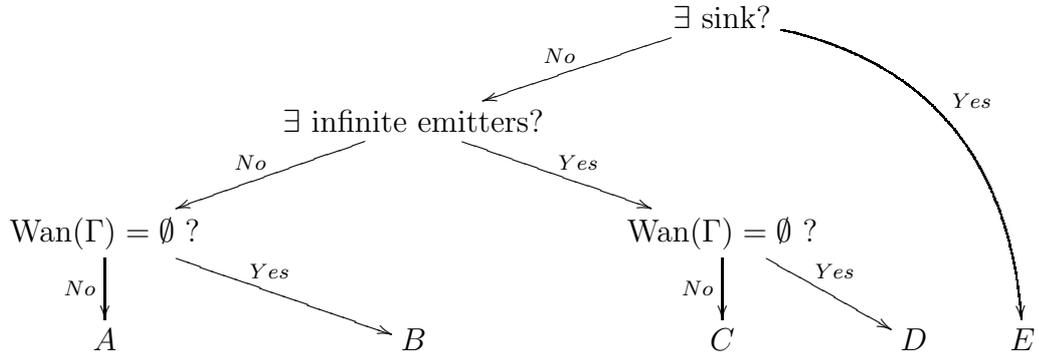
\begin{equation*}\label{2-11-17j}
\begin{xymatrix}{
&& \text{$\exists$ sink?}  \ar[dl]_{No} \ar@/^3pc/[dddrr]^{Yes} && \\
&\text{$\exists$ infinite emitters?} \ar[dl]_{No} \ar[dr]^{Yes} &&&\\
\Wan(\Gamma) = \emptyset \ ?\ar[d]_{No}\ar[dr]^{Yes}&&\Wan(\Gamma) = \emptyset \ ?\ar[d]_{No}\ar[dr]^{Yes} &&\\
A&B& C& D&E}
\end{xymatrix}
 \end{equation*}
 We consider the five types of graphs one by one. The following conclusions are valid when $\Gamma$ is cofinal and $A(\beta)$ is transient.
\begin{itemize}
\item[Type A:] There are no proper almost $A(\beta)$-harmonic vectors, but there are always $A(\beta)$-harmonic vectors.
\item[Type B:]  There are no almost $A(\beta)$-harmonic vectors at all.
\item[Type C:] There is a bijective correspondence between extremal rays of proper almost $A(\beta)$-harmonic vectors and the infinite emitters in $\Gamma_V$. 
\item[Type D:] There are no $A(\beta)$-harmonic vectors and there is a bijective correspondence between extremal rays of proper almost $A(\beta)$-harmonic vectors and the infinite emitters in $\Gamma_V$.
\item[Type E:] There are no $A(\beta)$-harmonic vectors and a unique ray of proper almost $A(\beta)$-harmonic vectors for all $\beta$.
\end{itemize}

Notice the lack of information regarding $A(\beta)$-harmonic vectors for graphs of type C. The graphs with countably many exits which were studied in \cite{Th3} show that they can exist in abundance in this case, but there are also examples where they do not exist.

 A more detailed explanation of how to derive the presented conclusions can be found in Appendix \ref{App1}. The main reason they are presented here is that they explain the focus in the following pages. Note namely that concerning the recurrent case and the extremal proper almost $A(\beta)$-harmonic vectors we know in all cases if they exist and how many. Moreover, an explicit formula for the essentially unique $A(\beta)$-harmonic vector which exists in the recurrent case when \eqref{20-11-17} holds was given by Vere-Jones, \cite{V}. The extremal proper almost $A(\beta)$-harmonic vectors can also be described explicitly by use of Lemma 2.10 and Theorem 3.4 in \cite{Th3}; they are all in the ray of a vector $\varphi$ of the form
 $$
 \varphi_v = \sum_{n=0}^{\infty} A(\beta)^n_{v,u} 
 $$
 for some vertex $u$ which is either a sink or an infinite emitter.

Summing up, we see that what remains are the questions about the structure of the $A(\beta)$-harmonic vectors in the transient case, and only when $\Wan(\Gamma)$ is not empty. Graphs for which $\Wan(\Gamma)$ is not empty are naturally subdivided into graphs for which $NW_{\Gamma}$ is empty and those for which it is not. In the latter case we may in fact assume that $\Gamma$ is strongly connected by Proposition 4.9 in \cite{Th3}, but even if one is only interested in the case where $\Gamma$ is strongly connected it is highly advantageous to work in a generality which includes more general digraphs, and in particular Bratteli diagrams. In fact, the results we obtain show that the problem of determining the $A(\beta)$-harmonic vectors for a strongly connected row-finite graph can be boiled down to the same problem for primitive Bratteli diagrams. Hence it appears that also to handle graphs $\Gamma$ with  $NW_{\Gamma}$ non-empty, it will be necessary to cope with cases where $NW_{\Gamma} = \emptyset$.

Based on results and methods developed in \cite{GV}, \cite{CT1} and \cite{Th3} it is not hard to determine what the set of possible inverse temperatures must be when $C^*(\Gamma)$ is simple. Slightly surprising perhaps is it that the qualitative features of this set depends very little on $F$. The most significant input is the set $NW_{\Gamma}$\label{NWGamma} of non-wandering vertexes in $\Gamma$ which was introduced in \cite{Th1}. Recall that by definition a vertex $v$ is in  $NW_{\Gamma}$ when $A(\Gamma)^n_{v,v} \neq 0$ for some $n \geq 1$. Set\label{betaF}
\begin{equation}\label{16-03-18}
\beta(F) = \left\{\beta \in \mathbb R: \ \text{There is a $\beta$-KMS weight for}
  \ \alpha^F \right\} \ .
\end{equation}

\begin{prop}\label{17-03-18d} Assume $C^*(\Gamma)$ is simple and consider an arbitrary potential function $F: \Gamma_{Ar} \to \mathbb R$. 

\begin{itemize} 
\item When $NW_{\Gamma}$ is infinite the set $\beta(F)$ is either empty or it is an infinite interval of the form $(-\infty, - \beta_0]$ or $[\beta_0,\infty)$ for some $\beta_0 > 0$.
\item When $NW_{\Gamma}$ is finite and non-empty, and there are no infinite emitters, $\beta(F)$ is either empty or it consists of exactly one non-zero real number.
\item When $NW_{\Gamma}$ is finite, non-empty, and there are infinite emitters, $\beta(F)$ is either empty or it is an infinite interval of the form $(-\infty, - \beta_0]$ or $[\beta_0,\infty)$ for some $\beta_0 > 0$.
\item $\beta(F) = \mathbb R$ when $NW_{\Gamma} = \emptyset$.
\end{itemize}
\end{prop}
\begin{proof} See Appendix \ref{App1}.
\end{proof}

We remark that Proposition \ref{17-03-18d} gives no information about for which values of $\beta$ the matrix $A(\beta)$ is recurrent or transient. In the setting of random walks this is a well-known and often a difficult question, and we shall not contribute any general results on it here, but we note that when the set $\beta(F)$ of Proposition \ref{17-03-18d} is an infinite interval, the recurrent case occurs only when $\beta$ is equal to $\beta_0$, the endpoint of the interval and sometimes not even then. The values of $\beta$ for which $A(\beta)$ is transient will be referred to as \emph{the transient range}.

\section{Harmonic vectors and conformal measures}\label{sec1}
In this section $\Gamma$ is an arbitrary countable digraph and $F : \Gamma_{Ar} \to \mathbb R$ an arbitrary potential function. We extend $F$ to a map $F : P_f(\Gamma) \to \mathbb R$ such that $F(v) = 0$ when $v \in \Gamma_V$ and 
$$
F(\mu) = \sum_{i=1}^n F(p_i)
$$
when $\mu = (p_i)_{i=1}^n  \in \left(\Gamma_{Ar}\right)^n$. Associated to the finite path $\mu$ is the  cylinder set \label{Zmu}
$$
Z(\mu) = \left\{(x_i)_{i=1}^{\infty} \in P(\Gamma) : \  
x_j = p_j, \ j = 1,2, \cdots, n \right\} 
$$
which is an open and closed set in $P(\Gamma)$. In particular, when $\mu$ has length $0$ and hence is just a vertex $v$,
$$
Z(v) = \left\{ p \in P(\Gamma) :  \  s(p) =v \right\} \ .
$$
With a slight abuse of terminology we say that a Borel measure $m$ on $P(\Gamma)$ is \emph{regular} when $m\left(Z(v)\right)< \infty$ for all $v \in \Gamma_V$. The \emph{shift map} $\sigma : P(\Gamma) \to P(\Gamma)$ is defined such that
$$
\sigma(p)_i = p_{i+1} \ .
$$
Note that $\sigma$ is continuous, and injective on $Z(a)$ for each $a \in \Gamma_{Ar}$.

\begin{defn}\label{confmes} A non-zero regular Borel measure $m$ on $P(\Gamma)$ is \emph{$e^{\beta F}$-conformal} when
\begin{equation}\label{aug1}
m\left(\sigma(B\cap Z(a))\right) = e^{\beta F(a)} m( B \cap
Z(a))
\end{equation} 
for every edge $a \in \Gamma_{Ar}$ and every Borel subset $B$ of $P(\Gamma)$. 
\end{defn}

The notion of a conformal measure in dynamical systems goes back to work by Sullivan, and in a generality which covers the cases we consider it was coined by Denker and Urbanski in \cite{DU}.

Given a vertex $v \in \Gamma_V$, a finite path $\mu \in P_f(\Gamma)$ with $r(\mu) = v$ and a subset $A \subseteq Z(v)$ we let
$Z(\mu)A$
denote the set $Z(\mu) \cap \sigma^{-|\mu|}(A)$. The next lemma will be used often and sometimes without reference in the following. The proof is left to the reader.

\begin{lemma}\label{nov2} Let $m$ be an $e^{\beta F}$-conformal Borel measure on $P(\Gamma)$. Consider a vertex $v \in \Gamma_V$, a Borel subset $B \subseteq Z(v)$ and a finite path $\mu$ in $\Gamma$ such that $r(\mu) =v$. Then
$$
m\left(Z(\mu) B\right) = e^{-\beta F(\mu)} m(B) \  .
$$
\end{lemma}

We shall need the following fact in several places below.

\begin{lemma}\label{03-03-18} Let $\Gamma$ be a digraph and $\mathcal C$ a collection of subsets of $P(\Gamma)$ such that
\begin{itemize}
\item $Z(\mu) \in \mathcal C$ for all $\mu \in P_f(\Gamma)$,
\item $A,B \in \mathcal C, \ B \subseteq A \ \Rightarrow \ A \backslash B \in \mathcal C$ and
\item $\bigcup_n A_n \in \mathcal C$ when $A_n \in \mathcal C$ is a sequence such that $A_n \subseteq A_{n+1}$ for all $n$. 
\end{itemize} 
Then $\mathcal C$ contains all Borel subsets of $P(\Gamma)$.
\end{lemma}
\begin{proof} This follows from Theorem 1.6.1 in \cite{Co} since the cylinder sets constitute a $\pi$-system which generates the Borel $\sigma$-algebra.
\end{proof}

The following observation is fundamental for the following. For Bratteli diagrams it was pointed out by Renault in Proposition 3.3 of \cite{R1}.

 \begin{prop}\label{nov1} There is a bijection $m \mapsto \psi$ between the set of $e^{\beta F}$-conformal measures $m$ on $P(\Gamma)$ and the $A(\beta)$-harmonic vectors $\psi$ given by
 $$
 \psi_v = m(Z(v)), \ v \in \Gamma_V .
 $$
 \end{prop}
 \begin{proof} When $m$ is an $e^{\beta F}$-conformal Borel measure on $P(\Gamma)$, the calculation
\begin{equation*}
\begin{split}
&\sum_{w \in \Gamma_V} A(\beta)_{v,w}m(Z(w)) = \sum_{w \in \Gamma_V} \sum_{a \in s^{-1}(v) \cap r^{-1}(w)} e^{-\beta F(a)} m(Z(w))\\
&=  \sum_{w \in \Gamma_V} \sum_{a \in s^{-1}(v) \cap r^{-1}(w)} e^{-\beta F(a)} m(\sigma(Z(a)) \\
&= \sum_{w \in \Gamma_V} \sum_{a \in s^{-1}(v) \cap r^{-1}(w)} m(Z(a)) = m(Z(v)) 
\end{split}
\end{equation*}
shows that the corresponding vector $\psi$ is $A(\beta)$-harmonic. To show that the map is injective consider two regular $e^{\beta F}$-conformal measures $m$ and $m'$ such that $m(Z(v)) = m'(Z(v))$ for all $v \in \Gamma_V$. It follows from Lemma \ref{nov2} that $m(Z(\mu)) = m'(Z(\mu))$ for all $\mu \in P_f(\Gamma)$ and then from Lemma \ref{03-03-18} that $m = m'$. That the map $m \mapsto \psi$ is surjective follows from standard constructions in measure theory. See e.g. Lemma 3.7 in \cite{Th3}.
\end{proof}

In the following we denote the set of $e^{\beta F}$-conformal measure on $P(\Gamma)$ by $M_{\beta F}(\Gamma)$\label{MbetaF} and the set of $A(\beta)$-harmonic vectors by $H_{\beta F}(\Gamma)$.\label{HbetaF} The content of Proposition \ref{nov1} is an explicit description of a bijection
\begin{equation}\label{06-02-18e}
M_{\beta F}(\Gamma) \ \simeq \ H_{\beta F}(\Gamma) \ .
\end{equation}
The $e^{\beta F}$-conformal measure which corresponds to the $A(\beta)$-harmonic vector $\psi \in  H_{\beta F}(\Gamma) $ will in the following be denoted by $m_{\psi}$. \label{mpsi}It is determined by the condition that
$$
m_{\psi}(Z(\mu)) = e^{-\beta F(\mu)}\psi_{r(\mu)}
$$
when $\mu \in P_f(\Gamma)$.

We will consider $H_{\beta F}(\Gamma)$ as a topological space in the topology inherited from the inclusion $ H_{\beta F}(\Gamma) \subseteq \mathbb R^{\Gamma_V}$, where the latter is equipped with the product topology. This is a second countable Hausdorff topology and we give $M_{\beta F}(\Gamma)$ the topology making \eqref{06-02-18e} a homeomorphism. This is the weakest topology making the map
$$
M_{\beta F}(\Gamma) \ni m \ \mapsto \ m(Z(\mu))
$$
continuous for all $\mu \in P_f(\Gamma)$.

It follows from Proposition \ref{nov1} that in the setting of Theorem \ref{intro1} the $\beta$-KMS weights that correspond to $A(\beta)$-harmonic vectors are in bijective correspondence with the $e^{\beta F}$-conformal measures. Let $C_b(P(\Gamma))$ be the Banach space of continuous bounded functions on $P(\Gamma)$. The $\beta$-KMS weight $\psi_m$ on $C^*(\Gamma)$ given by a $e^{\beta F}$-conformal measure $m$ is defined such that
$$
\psi_m(a) = \int_{P(\Gamma)} Q(a) \ \mathrm{d}m \ ,
$$ 
where $Q : C^*(\Gamma) \to C_b\left(P(\Gamma)\right)$ is the  unique norm-continuous linear map such that
$$
Q\left(S_{\mu}S_{\nu}^*\right) = \begin{cases} 0, & \ \mu \neq \nu, \\ 1_{Z(\mu)}, &  \ \mu = \nu \ . \end{cases} 
$$ 
 $P(\Gamma)$ is a locally compact Hausdorff space when $\Gamma$ is row-finite and has no sinks, and in this case $Q$ takes values in the space $C_0(P(\Gamma))$ of continuous functions that vanish at infinity, and $C_0(P(\Gamma))$ can then be identified with an abelian $C^*$-subalgebra of $C^*(\Gamma)$.

\subsection{On the relation to random walks}\label{randomw} There is a canonical way to obtain  a countable digraph with a potential from a random walk. To describe it recall that a random walk on a countable set is a pair $(N,p)$ where $N$ is the countable set (the state space) and $p : N \times N \to [0,1]$ is a matrix which is \emph{stochastic} in the sense that
\begin{equation*}\label{10-01-18} 
\sum_{y \in N} p(x,y) = 1 \ 
\end{equation*}
for all $x \in N$. The value $p(x,y)$ is the transition probability from state $x$ to state $y$. We refer to the books \cite{Wo1} and \cite{Wo2} for the theory of random walks. The digraph $\Gamma$ of such a random walk has as vertex set $\Gamma_V = N$ and an arrow from $n  \in N$ to $m \in N$ if and only if $p(n,m) > 0$. Then $\Gamma_{Ar} \subseteq N \times N$ and the range and source maps $s,r : \Gamma_{Ar} \to \Gamma_V$ are the projections to the first and second coordinate, respectively. We define a potential $F_p : \Gamma_{Ar} \to \mathbb R$ such that
$$
F_p(a) = - \log p\left(s(a),r(a)\right) \ .
$$  
Note that $F_p$ is non-negative and that 
$$
p(x,y) = e^{-F_p(a)} \ ,
$$
when $a = (x,y) \in \Gamma_{Ar}$. In this way a random walk on a countable set gives rise to a digraph $\Gamma$ and hence a $C^*$-algebra $C^*(\Gamma)$ which comes equipped with the one-parameter automorphism group $\alpha^{F_p}$ determined by the transition probabilities $p(x,y)$ of the walk. We note that
\begin{itemize}
\item $C^*(\Gamma)$ is simple when $(N,p)$ is irreducible, unless $N$ is a finite set and $p$ is a cyclic permutation matrix. 
\item $(N,p)$ has finite range if and only if $\Gamma$ is row-finite.
\item A vector $\psi : N \to [0,\infty)$ is a harmonic function for $(N,p)$ if and only if it is $A(1)$-harmonic.
\end{itemize} 

There is also a construction which goes the other way. In the setting of Section \ref{recap}, assume that we are given an $A(\beta)$-harmonic vector $\psi$. If $\psi_v$ is strictly positive for all $v \in \Gamma_V$, as will automatically be the case when $\Gamma$ is cofinal by Lemma 2.5 in \cite{Th2}, we can set
\begin{equation}\label{doobs}
p(v,w) = A(\beta)_{v,w}\psi_v^{-1}\psi_w \ .
\end{equation}
This is a version of what is known as \emph{Doob's $h$-transform}. Note that $p$ is a stochastic matrix and the map $h \mapsto h\psi$ is an affine homeomorphism from the set of harmonic functions $h$ for the random walk $(\Gamma_V,p)$ and the set of $A(\beta)$-harmonic vectors. In this way we can obtain a lot of information on $A(\beta)$-harmonic vectors, and hence $\beta$-KMS weights for the one-parameter group $\alpha^F$ on $C^*(\Gamma)$ from results concerning random walks. In the present work we focus on the $e^{\beta F}$-conformal measures rather than the $A(\beta)$-harmonic vectors. In the language of Markov chain theory this means that we focus on what is called the $h$-processes going with $p$, cf. e.g. \cite{Wo1}. Indeed, if $m$ is an $e^{\beta F}$-conformal measure and $\varphi$ the corresponding $A(\beta)$-harmonic vector, the vector $ h = \varphi\psi^{-1}$ is a harmonic function for $p$ and the measure on the trajectory space $P(\Gamma)$ defined by the $h$-process with transition matrix
$$
p_h(x,y) = p(x,y)h(y)h(x)^{-1} 
$$
is absolutely continuous with respect to $m$ with a Radon-Nikodym derivative defined by $\varphi$.   

The relation to random walks was used in Example 5.3 of \cite{Th1} where results of Ney and Spitzer on random walks in $\mathbb Z^n$ were used in this way, and in much the same spirit we shall extend work by Sawyer in Example \ref{27-02-18d} below in connection with Pascal's triangle and strongly connected graphs obtained from Pascal's triangle. This gives only a glimpse of the examples which can be completely explored using results from random walks. In many cases, however, the results from the theory of random walks lets us down because of the assumptions made. In particular, in many cases the results are proved under the assumption that the random walk is uniformly irreducible which is impossible to verify for \eqref{doobs} regardless of the assumptions we make on $A(\beta)$ without some information on $\psi$. In fact, it seems that the only type of digraphs for which the random walk theory can determine the set of $A(\beta)$-harmonic vectors for us whenever $A(\beta)$ is transient and regardless of which potential we consider, is when the graph is obtained from a tree. See Example \ref{trees}. In order to handle graphs like the graphs with at most countable many exits which was studied in \cite{Th3} and the meager graphs we introduce here it is necessary to push and develop the methods used in connection with random walks and more general Markov chains. On the way we shall supply new methods which can be used in the study of random walks. See Remark \ref{27-02-18e}, Theorem \ref{09-02-18Markov} and Theorem \ref{09-02-18Markov(trans)} for applications to random walks.

\section{Extremal conformal measures}\label{extcon1}

An $A(\beta)$-harmonic vector $\psi$ is \emph{minimal} when every $A(\beta)$-harmonic vector $\phi$ with the property that $\phi_v \leq \psi_v$ for all $v \in \Gamma_V$ is a multiple of $\psi$, i.e. $\phi = \lambda \psi$ for some $\lambda > 0$. The minimal $A(\beta)$-harmonic vectors are those whose associated $\beta$-KMS weight, cf. Theorem \ref{intro1}, is extremal. We aim to show that the minimality condition (or, alternatively, the extremality of the associated $\beta$-KMS weight) is equivalent to the ergodicity of the corresponding $e^{\beta F}$-conformal measure. Recall that a Borel measure $m$ on $P(\Gamma)$ is \emph{ergodic} with respect to the shift $\sigma$ when every Borel subset $B \subseteq P(\Gamma)$ which is totally shift invariant in the sense that $\sigma^{-1}(B) = B$, is either a null-set or a co-null set for $m$, i.e. either $m(B) = 0$ or $m(P(\Gamma) \backslash B)= 0$. We shall work under the assumption that there is a vertex $v_0$ in $\Gamma$ \emph{from where all vertexes can be reached}. In general this means that for all $v \in \Gamma_V$ there is a finite path $\mu \in P_f(\Gamma)$ such that $s(\mu) = v_0$ and $r(\mu) = v$. We can formulate this condition by requiring
\begin{equation}\label{05-10-17a}
\sum_{n=0}^{\infty} A(\Gamma)^n_{v_0,v} > 0
\end{equation}
for all $v \in \Gamma_V$. We will later also assume transience of $A(\beta)$, but in the present section it suffices to assume that
\begin{equation}\label{05-10-17b}
A(\beta)^n_{v_0,v} < \infty
\end{equation}
for all $n \in \mathbb N$ and all $v \in \Gamma_V$. For each vertex $v \in \Gamma_V$ we can choose $k \in \mathbb N$ such that $A(\beta)^k_{v_0,v} > 0$ and we set\label{bv}
$$
b_v = \left( A(\beta)^k_{v_0,v}\right)^{-1}.
$$ 
A vector $\psi : \Gamma_V \to [0,\infty)$ is \emph{normalized}, or \emph{$v_0$-normalized}, when $\psi_{v_0} =1$ and a regular measure $m$ on $P(\Gamma)$ is \emph{normalized}, or \emph{$v_0$-normalized}, when $m(Z(v_0))= 1$. When $\psi$ is a normalized almost $A(\beta)$-harmonic vector the estimate 
\begin{equation*}\label{estimate}
A(\beta)^k_{v_0,v}\psi_v \leq \sum_{w \in \Gamma_V} A(\beta)^k_{v_0,w} \psi_w \leq \psi_{v_0} = 1
\end{equation*}
shows that $\psi_v \leq b_v$. An application of Fatous lemma shows that the set $\Delta$\label{delta} of normalized almost $A(\beta)$-harmonic vectors is a closed convex subset of the compact product space
$$
 \prod_{v \in \Gamma_V} \left[ 0, b_v \right] \ ,
 $$
 and hence a compact convex set. Let $\partial \Delta$\label{partdelta} be the set of extreme points in $\Delta$; a Borel subset by Choquet theory. In fact, $\partial \Delta$ is a $G_{\delta}$-set by Theorem 4.1.11 in \cite{BR}. Let $H^{v_0}_{\beta F}(\Gamma)$\label{Hv0F} be the set of normalized $A(\beta)$-harmonic vectors and\label{partHv0F}
$$
\partial H^{v_0}_{\beta F}(\Gamma) = \left\{ \psi \in \partial \Delta : \ \sum_{w \in \Gamma_V} A(\beta)_{v,w} \psi_w = \psi_v  \ \forall v \in \Gamma_V \right\} ;
$$
the set of normalized minimal $A(\beta)$-harmonic vectors. Note that $H^{v_0}_{\beta F}(\Gamma)$ and $\partial H^{v_0}_{\beta F}(\Gamma)$ are both Borel subsets of $\Delta$, and that $H^{v_0}_{\beta F}(\Gamma)$ is closed in $\Delta$ and hence a compact convex set when $\Gamma$ is row-finite.

 An $e^{\beta F}$-conformal measure $m$ is \emph{extremal} when all $e^{\beta F}$-conformal measures $m'$ such that $m' \leq m$ are scalar multiples of $m$, i.e. $m' = \lambda m$ for some $\lambda > 0$. For normalized $e^{\beta F}$-conformal measures this is the same as being extremal in the convex set of normalized $e^{\beta F}$-conformal measures. Clearly, an $A(\beta)$-harmonic vector is minimal if and only if the corresponding $e^{\beta F}$-conformal measure is extremal. Thus the following is an immediate consequence of Proposition \ref{nov1}.

\begin{lemma}\label{Choquet} Let $\Gamma$ be a countable directed graph and $v_0 \in \Gamma_V$ a vertex such that \eqref{05-10-17a} and \eqref{05-10-17b} hold. For each $x \in \partial H_{\beta F}^{v_0}( \Gamma)$ there is a unique extremal $v_0$-normalized $e^{\beta F}$-conformal measure $m_x$ on $P(\Gamma)$ such that
$$
m_x(Z(v)) = x_v
$$
for all $v \in \Gamma_V$.
\end{lemma} 

We denote in the following by $M^{v_0}_{\beta F}(\Gamma)$\label{Mv0F} the set of normalized $e^{\beta F}$-conformal measures on $P(\Gamma)$. The set of extreme points in the convex set $M^{v_0}_{\beta F}(\Gamma)$ is the set of extremal normalized $e^{\beta F}$-conformal measures on $P(\Gamma)$ and we denote this set by $\partial M^{v_0}_{\beta F}(\Gamma)$.\label{partMv0F} Note that the map $x \mapsto m_x(Z(v))$ is continuous on $\partial H^{v_0}_{\beta F}(\Gamma)$ for all $v \in \Gamma_V$. It follows from an application of Lemma \ref{03-03-18} that $x \mapsto m_x(B)$ is a Borel map on $\partial  H^{v_0}_{\beta F}(\Gamma)$ for all Borel subsets $B \subseteq P(\Gamma)$. Hence every Borel probability measure $\nu$ on $\partial  H^{v_0}_{\beta F}(\Gamma)$ gives rise to a measure 
$$
\int_{\partial   H^{v_0}_{\beta F}(\Gamma)} m_x \ \mathrm{d} \nu(x)
$$
on $P(\Gamma)$ defined such that
$$
B \mapsto \int_{\partial  H^{v_0}_{\beta F}(\Gamma)} m_x(B) \ \mathrm{d} \nu(x) \ .
$$
Note that $\int_{\partial  H^{v_0}_{\beta F}(\Gamma)} m_x \ \mathrm{d} \nu(x)$ is normalized and $e^{\beta F}$-conformal since each $m_x$ is.

\begin{prop}\label{rep} Let $\Gamma$ be a countable directed graph and $v_0 \in \Gamma_V$ a vertex such that \eqref{05-10-17a} and \eqref{05-10-17b} hold. Let $m \in M^{v_0}_{\beta F}(\Gamma)$. There is a Borel probability measure $\nu$ on $\partial   H^{v_0}_{\beta F}(\Gamma)$ such that 
\begin{equation}\label{eq111}
m = \int_{\partial  H^{v_0}_{\beta F}(\Gamma)} m_x \ \mathrm{d} \nu(x) \ .
\end{equation}
\end{prop}
\begin{proof} It follows from Proposition \ref{nov1} and Choquet theory that there is a Borel probability measure $\nu$ on $\partial \Delta$ such that 
$$
m(Z(v)) = \int_{\partial \Delta} x_v \ \mathrm{d} \nu(x)
$$ 
for all $v \in \Gamma_V$, cf. e.g. Proposition 4.1.3 and Theorem 4.1.11 in \cite{BR}. Since 
\begin{equation*}
\begin{split}
&\int_{\partial \Delta} \left(x_v - \sum_{w \in \Gamma_V} A(\beta)_{v,w}x_w \right) \ \mathrm{d} \nu(x) \\
&= m(Z(v)) - \ \sum_{w \in \Gamma_V} A(\beta)_{v,w} m(Z(w)) = 0
\end{split}
\end{equation*}
for all $v \in \Gamma_V$, it follows that $\nu$ is concentrated on $\partial  H^{v_0}_{\beta F}(\Gamma)$. The equality  (\ref{eq111}) follows from Proposition \ref{nov1} since 
$$
\int_{\partial  H^{v_0}_{\beta F}(\Gamma)} m_x(Z(v)) \ \mathrm{d} \nu(x) = \int_{\partial  H^{v_0}_{\beta F}(\Gamma)} x_v \ \mathrm{d} \nu(x) = m(Z(v)) \ .
$$
\end{proof}

\begin{cor}\label{14-01-18d} $M^{v_0}_{\beta F}(\Gamma) \neq \emptyset \ \Rightarrow  \ \partial M^{v_0}_{\beta F} \neq \emptyset$.
 \end{cor}

When $\Gamma$ is row-finite the set $M^{v_0}_{\beta F}(\Gamma)$ is compact, and then Corollary \ref{14-01-18d} follows from the Krein-Milman theorem.

Although we shall not need the fact, it should be noted that $\Delta$ is a Choquet simplex when $\Gamma$ is cofinal and maybe in general, cf. Theorem 4.6 in \cite{Th2}. This implies that the representing measure $\nu$ is unique in Proposition \ref{rep}. Thus the convex set of normalized $e^{\beta F}$-conformal measures on $P(\Gamma)$ is in affine bijection with the set of Borel probability measures on $\partial  H^{v_0}_{\beta F}(\Gamma)$.

\subsection{Conformal measures and Martin kernels}\label{sec3}

In this section we extend the work of Sawyer, \cite{Sa}, to the more general setting we consider. We shall work in a generality where $A(\beta)$-harmonic vectors can take the value zero and where there is no straightforward way to obtain the results we need by direct appeal to the theory of random walks. 

It is a standing assumption in this section that there is a vertex $v_0$ such that
\begin{equation}\label{v0trans}
 0 < \sum_{n=0}^{\infty} A(\beta)^n_{v_0,v} < \infty
\end{equation}
 for all $v\in \Gamma_V$.

\begin{lemma}\label{ooo} Let $m$ be a $e^{\beta F}$-conformal measure on $P(\Gamma)$ and let $M \subseteq \Gamma_V$ be a set of vertexes. Then
$$
\varphi_v = m\left(\left\{p \in Z(v): \ s(p_k) \in M \ \text{for infinitely many} \ k \right\}\right) 
$$ 
defines an $A(\beta)$-harmonic vector $\varphi$, unless it is identically zero.
\end{lemma}
\begin{proof} It follows from Lemma \ref{nov2} that 
\begin{equation*}
\begin{split}
&A(\beta)_{v,w} \varphi_w = \sum_{a \in s^{-1}(v) \cap r^{-1}(w)} e^{-\beta F(a)}\varphi_w \\
&=  \sum_{a\in s^{-1}(v) \cap r^{-1}(w)} m\left(\left\{p \in Z(a): \ s(p_k) \in M \ \text{for infinitely many} \ k \right\}\right)  \ .
\end{split}
\end{equation*}
Hence
\begin{equation*}
\begin{split}
&\sum_{w\in \Gamma_V} A(\beta)_{v,w} \varphi_w \\
&= \sum_{a\in s^{-1}(v)} m\left(\left\{p \in Z(a): \ s(p_k) \in M \ \text{for infinitely many} \ k \right\}\right) \\
&=  m\left(\left\{p \in Z(v): \ s(p_k) \in M \ \text{for infinitely many} \ k \right\}\right)  = \varphi_v \ .
\end{split}
\end{equation*}
\end{proof}

\begin{lemma}\label{oxx}  Let $m$ be an extremal $e^{\beta F}$-conformal measure on $P(\Gamma)$, and let $M \subseteq \Gamma_V$ be a set of vertexes. Then either
$$
 m\left(\left\{p \in Z(v): \ s(p_k) \in M \ \text{for infinitely many} \ k \right\}\right) = 0
 $$
 for all $v \in \Gamma_V$, or
 $$
 m\left(\left\{p \in Z(v): \ s(p_k) \in M \ \text{for infinitely many} \ k \right\}\right) = m(Z(v))
 $$
 for all $v \in \Gamma_V$.
 \end{lemma}
 \begin{proof} By extremality and Lemma \ref{ooo} there is a $\lambda \in [0,1]$ such that
 \begin{equation*}\label{lordag1}
  m\left(\left\{p \in Z(v): \ s(p_k) \in M \ \text{for infinitely many} \ k \right\}\right) = \lambda m(Z(v))
 \end{equation*}
 for all $v \in \Gamma_V$. Set 
 $E = \left\{p \in P(\Gamma): \ s(p_k) \in M \ \text{for infinitely many} \ k \right\}$. Fix $v \in \Gamma_V$ and consider a finite path $\mu $ with $s(\mu) = v$. Then $Z(\mu) \cap E = Z(\mu) (Z(r(\mu)) \cap E)$, and hence Lemma \ref{nov2} implies that
 \begin{equation}\label{hvd}
\begin{split}
& m\left(Z(\mu) \cap E\right)= e^{-\beta F(\mu)} m\left(Z(r(\mu))\cap E\right)  \\
&= e^{-\beta F(\mu)} \lambda m (Z(r(\mu))) = \lambda m(Z(\mu)) \ .
\end{split}
\end{equation} 
Let $\epsilon > 0$ and consider a vertex $v \in \Gamma_V$. Since $m$ is a finite measure on $Z(v)$ and the Borel $\sigma$-algebra is generated by the cylinder sets it follows that there is a finite collection of mutually disjoint cylinders $Z(\mu_i)$, all with $s(\mu_i) =v$, such that 
$$
 m\left((Z(v) \cap E) \backslash \left(\bigcup_i Z(\mu_i)\right)\right) + m\left( \left(\bigcup_i Z(\mu_i)\right) \backslash (Z(v) \cap E)\right) \leq \epsilon \ ,
$$
cf. e.g. \cite{KT}, p. 84. Applying (\ref{hvd}) to each $\mu_i$ we find that
\begin{equation*}\label{traet2}
\begin{split} 
&  m\left(Z(v) \cap E\right) - \epsilon \leq m\left(Z(v) \cap E \cap \left(\bigcup_i Z(\mu_i)\right)\right)  \\ 
& =\sum_i  m\left(Z(\mu_i) \cap E\right) =\lambda\sum_i m\left( Z(\mu_i) \right) \\
&\leq   \lambda m(Z(v) \cap E) +  \epsilon \lambda \  .
\end{split}
\end{equation*} 
Since $\epsilon >0$ is arbitrary we conclude that $ m\left(Z(v) \cap E\right)= \lambda m\left(Z(v) \cap E\right)$. And $v\in \Gamma_V$ was also arbitrary so it follows that $\lambda = 1$ unless $m(E)=0$ in which case $\lambda =0$.
 \end{proof}

 For each $v \in \Gamma_V$ we choose $k \in \mathbb N$ such that $  A(\beta)^k_{v_0, v} \neq 0$. Then 
 \begin{equation}\label{estimate2}
  A(\beta)^k_{v_0, v}\sum_{n=0}^{\infty} A(\beta)^n_{v,w} \leq \sum_{n=0}^{\infty} A(\beta)^{n+k}_{v_0,w} \leq \sum_{n=0}^{\infty} A(\beta)^n_{v_0,w}
\end{equation}
 for all $v,w \in \Gamma_V$. In particular, $\sum_{n=0}^{\infty} A(\beta)^n_{v,w} < \infty$ for all $v,w \in \Gamma_V$ and we can define the \emph{Martin kernel} $K_{\beta} : \Gamma_V \times \Gamma_V \to [0,\infty)$\label{Kbeta} such that
$$
K_{\beta}(v,w) = \frac{\sum_{n=0}^{\infty} A(\beta)^n_{v,w}}{\sum_{n=0}^{\infty} A(\beta)^n_{v_0,w}}  \ .
$$
It follows from (\ref{estimate2}) that $K_{\beta}(v,w) \leq  b_v$ where $b_v =1 /A(\beta)^k_{v_0, v}$. We define
$$
K_{\beta} : \Gamma_V \to \prod_{v \in \Gamma_V} \left[0, b_v\right]
$$
such that 
$K_{\beta}(w) = \left(K_{\beta}(v,w)\right)_{v \in \Gamma_V}$. Let $\overline{K_{\beta}(\Gamma_V)}$ be the closure of $K_{\beta}(\Gamma_V)$ in $\prod_{v \in \Gamma_V} \left[0, b_v\right]$ and set \label{partKbeta}
$$
\partial K_{\beta} = \overline{K_{\beta}(\Gamma_V)} \backslash K_{\beta}(\Gamma_V) \ .
$$ 
It is straightforward to check that $\partial K_{\beta}  \subseteq \Delta$; the compact convex set of normalized almost $A(\beta)$-harmonic vectors. Since 
$$
\sum_{w \in \Gamma_V} A(\beta)^{n+1}_{v,w} \xi_w \leq \sum_{w \in \Gamma_V} A(\beta)^n_{v,w} \xi_w 
$$
for all $n,v$ when $\xi \in \overline{K_{\beta}(\Gamma_V)}$, we obtain a Borel function 
$q_v : \overline{K_{\beta}(\Gamma_V)} \to \left[ 0,b_v\right]$
for every $v \in \Gamma_V$ defined such that 
$$
q_v(\xi) = \lim_{n \to \infty} \sum_{w \in \Gamma_V} A(\beta)^n_{v,w} \xi_w \ .
$$

\begin{lemma}\label{septxx} Let $m \in M^{v_0}_{\beta F}(\Gamma)$. Let $M \subseteq \Gamma_V$ be a set of vertexes containing $v_0$. There is a Borel probability measure $\mu_M$ on $\overline{K_{\beta}(M)}$ such that 
\begin{equation}\label{septok2}
\begin{split}
&m\left(\left\{ p \in Z(v) : \ s(p_k) \in M \ \text{for infinitely many}  \ k \right\}\right) \\
&= \int_{\overline{K_{\beta}(M)}} q_v(\xi) \ \mathrm{d}\mu_M(\xi) \ 
\end{split}
\end{equation}
for all $v \in \Gamma_V$.
\end{lemma}
\begin{proof} Consider a finite subset $F \subseteq M$ containing $v_0$. Using Lemma \ref{nov2} we find that
\begin{equation}\label{nice1}
\begin{split}
&\sum_{w \in \Gamma_V} A(\beta)^n_{v,w} m\left(\left\{p \in Z(w): \ s(p_k) \in F \ \text{for some} \ k \right\}\right) \\
&=  m\left(\left\{p \in Z(v): \ s(p_k) \in F \ \text{for some} \ k \geq n+1 \right\}\right) \\
& \leq  m\left(\left\{p \in Z(v): \ s(p_k) \in F \ \text{for some} \ k \right\}\right). 
\end{split}
\end{equation}
The case $n=1$ shows that the vector $\varphi$ given by
$$
\varphi_v =  m\left(\left\{p \in Z(v): \ s(p_k) \in F \ \text{for some} \ k \right\}\right)
$$ 
is $A(\beta)$-superharmonic, meaning that
$$
\sum_{w \in \Gamma_V} A(\beta)_{v,w}\varphi_w \leq \varphi_v
$$
for all $v \in \Gamma_V$. Note that
$$
\left\{p \in Z(v): \ s(p_k) \in F \ \text{for some} \ k \geq n \right\}  \ \subseteq \ \bigcup_{\mu \in L} Z(\mu),
$$
where 
$$
L = \left\{ \mu \in P_f(\Gamma) : \ s(\mu) = v, \ r(\mu) \in F, \ |\mu| \geq n-1 \right\} \  .
$$
Hence
$$
 m\left(\left\{p \in Z(v): \ s(p_k) \in F \ \text{for some} \ k \geq n \right\}\right) \leq \sum_{w \in F} \sum_{k \geq n-1} A(\beta)^k_{v,w}m(Z(w)) \ ,
 $$
 which implies that
\begin{equation}\label{nodiss}
\lim_{n \to \infty}  m\left(\left\{p \in Z(v): \ s(p_k) \in F \ \text{for some} \ k \geq n \right\}\right)  = 0.
\end{equation}
Using this in (\ref{nice1}) we find that 
\begin{equation}\label{nodiss3}
\lim_{n \to \infty} \sum_{w \in \Gamma_V}A(\beta)_{v,w}^n \varphi_w = 0.
\end{equation}
Set
$$
k_F (w) = \varphi_w - \sum_{u \in \Gamma_V}A(\beta)_{w,u}\varphi_u.
$$ 
Since
$$
\sum_{u \in \Gamma_V}A(\beta)_{w,u}\varphi_u =  m\left(\left\{p \in Z(w) : \ s(p_k) \in F \ \text{for some} \  k \geq 2 \right\} \right), 
$$
we see that
$$
k_F(w) =  m\left(\left\{p \in Z(w) : \ s(p_1) \in F, \ s(p_k) \notin F \ \forall k \geq 2 \right\} \right).
$$
In particular, $k_F$ is supported in $F$. Furthermore,
\begin{equation}\label{jul16}
\begin{split}
&\sum_{w \in F} \sum_{n=0}^{\infty} A(\beta)^n_{v,w} k_F(w)
 = \sum_{w \in \Gamma_V} \sum_{n=0}^{\infty} A(\beta)^n_{v,w} k_F(w)\\
 & = \lim_{N \to \infty} \sum_{n=0}^{N} \sum_{w \in \Gamma_V} A(\beta)^n_{v,w} k_F(w) \\
& = \lim_{N \to \infty} \sum_{n=0}^{N} \left( \sum_{w \in \Gamma_V}A(\beta)^n_{v,w}\varphi_w  -  \sum_{u \in \Gamma_V}A(\beta)^{n+1}_{v,u}\varphi_u\right)\\
&= \lim_{N \to \infty} \left(\varphi_v - \sum_{u \in \Gamma_V} A(\beta)^{N+1}_{v,u}\varphi_u\right) = \varphi_v \ ,
\end{split}
\end{equation}
where we used (\ref{nodiss3}) in the last step. It follows from (\ref{jul16}) that
\begin{equation}\label{13-03-18}
\varphi_v = \sum_{w \in F}  K_{\beta}(v,w) \sum_{n=0}^{\infty} A(\beta)^n_{v_0,w} k_F(w) \ .
\end{equation}
Note that 
\begin{equation*}
\begin{split}
&\sum_{w \in F}  \sum_{n=0}^{\infty} A(\beta)^n_{v_0,w} k_F(w) \\
&= \sum_{w \in F}  \sum_{n=0}^{\infty} A(\beta)^n_{v_0,w} m\left(\left\{p \in Z(w) :  \ s(p_k) \notin F \ \forall k \geq 2 \right\} \right) \\
& = m\left( \left\{p \in Z(v_0) : \ s(p_k) \notin F \ \text{for all large enough}  \ k \right\}\right) \\
&= m(Z(v_0)) = 1 \ ,
\end{split}
\end{equation*}
where we have used (\ref{nodiss}) and that $v_0 \in F$. Hence \eqref{13-03-18} can be written
\begin{equation}\label{nice29}
\varphi_v = \int_{\Gamma_V} K_{\beta}(v,w) \ \mathrm{d}\nu(w) \ ,
\end{equation}
where $\nu$ is the probability measure on $\Gamma_V$ given by
$$\nu =   \sum_{w \in F} \sum_{n=0}^{\infty} A(\beta)^n_{v_0,w} k_F(w) \delta_w \ .
$$ 
We obtain in this way the push-forward measure $\mu_F =  \nu \circ K_{\beta}^{-1}$ on $\overline{K_{\beta}(\Gamma_V)}$ which is concentrated on $K_{\beta}(F)$. Then \eqref{nice29} implies that
$$
\varphi_v = \int_{\overline{K_{\beta}(\Gamma_V)}} \xi_v \ \mathrm{d}\mu_F(\xi) \ .
$$
Now choose a sequence $v_0 \in F_1 \subseteq F_2\subseteq F_3 \subseteq \cdots$ of finite subsets of $M$ such that $\bigcup_i F_i = M$. For each $i$ we have then a Borel probability measure $\mu_i$ concentrated on $\overline{K_{\beta}(M)}$ such that
$$
m\left( \left\{p \in Z(v) : \ s(p_k) \in F_i \ \text{for some} \ k \right\}\right) = \int_{\overline{K_{\beta}(M)}} \xi_v \ \mathrm{d}\mu_i(\xi) 
$$
for all $v \in \Gamma_V$. Let $\mu_M$ be a weak*-condensation point of the sequence $\{\mu_i\}$ in the set of Borel probability measures on $\overline{K_{\beta}(M)}$. Then 
$$
m\left( \left\{p \in Z(v) : \ s(p_k) \in M \ \text{for some} \ k \right\}\right) = \int_{\overline{K_{\beta}(M)}} \xi_v \ \mathrm{d}\mu_M(\xi) 
$$
for all $v \in \Gamma_V$. Furthermore, we find by using Lemma \ref{nov2} that
\begin{equation*}
\begin{split}
&m\left( \left\{p \in Z(v) : \ s(p_k) \in M \ \text{for some} \ k  \geq n+1\right\}\right) \\
&= \sum_{w \in \Gamma_V} A(\beta)^n_{v,w}m\left( \left\{p \in Z(w) : \ s(p_k) \in M \ \text{for some} \ k \right\}\right)  \\
&= \int_{\overline{K_{\beta}(M)}} \sum_{w \in \Gamma_V} A(\beta)^n_{v,w}\xi_w \ \mathrm{d}\mu_M(\xi) \ .
\end{split}
\end{equation*}
Letting $n$ tend to infinity we obtain (\ref{septok2}).

\end{proof}

\begin{lemma}\label{morgen3'} Let $m \in \partial M^{v_0}_{\beta F}(\Gamma)$. Let $M \subseteq \Gamma_V$ be a set of vertexes containing $v_0$ such that
$$
 m\left(\left\{p \in Z(v): \ s(p_k) \in M\ \text{for infinitely many} \ k \right\}\right) = m(Z(v))
 $$
 for all $v \in \Gamma_V$. There is an element $\xi \in \partial K_{\beta} \cap \overline{K_{\beta}(M)}$ such that
$$
m(Z(v))  = \xi_v
$$
for all $v\in \Gamma_V$.
\end{lemma}
\begin{proof} By Lemma \ref{septxx} there is a Borel probability measure $\nu$ on $\overline{K_{\beta}(M)}$ such that
 \begin{equation*}\label{morgen2}
m(Z(v)) = \int_{ \overline{K_{\beta}(M)}} q_v(\xi) \ \mathrm{d}\nu(\xi)
\end{equation*}
for all $v \in \Gamma_V$. Note that $q_{v_0}(\xi) \leq \xi_{v_0} = 1$ and that 
\begin{equation}\label{slao}
\int_{\overline{K_{\beta}(M)}} q_{v_0}(\xi) \ \mathrm{d}\nu(\xi) = m(Z(v_0)) = 1 \ .
\end{equation}
It follows from this that $q_{v_0}(\xi) = 1$ for $\nu$-almost all $\xi$. Furthermore,
\begin{equation*}\label{yes}
\begin{split}
&q_{v_0}\left(K_{\beta}(w)\right) = \left(\sum_{n=0}^{\infty} A(\beta)^n_{v_0,w}\right)^{-1} \lim_{k \to \infty} \sum_{v \in \Gamma_V} A(\beta)^k_{v_0,v} \sum_{n=0}^{\infty} A(\beta)^n_{v,w}\\
& = \left(\sum_{n=0}^{\infty} A(\beta)^n_{v_0,w}\right)^{-1} \lim_{k \to \infty} \sum_{n \geq k} A(\beta)^n_{v_0,w} = 0
\end{split}
\end{equation*}
for all $w \in \Gamma_V$, which shows that $q_{v_0}$ annihilates $K_{\beta}(\Gamma_V)$. Hence (\ref{slao}) also implies that $\nu$ is concentrated on $\partial K_{\beta} \cap \overline{K_{\beta}(M)}$. Set $\xi'_v = \xi_v - q_v(\xi)$. Then $\xi'$ is $A(\beta)$-superharmonic and $\xi'_{v_0} = 0$ for $\nu$-almost all $\xi \in \partial K_{\beta} \cap \overline{K_{\beta}(M)}$. For each $v \in \Gamma_V$ there is a $k \in \mathbb N$ such that $A(\beta)^k_{v_0,v} \neq 0$. Since
$$
0 \ \leq A(\beta)^k_{v_0,v}\xi'_v \ \leq \ \sum_{w \in \Gamma_V}A(\beta)^k_{v_0,w}\xi'_w \ \leq \ \xi'_{v_0} \ ,
$$
it follows that for $\nu$-almost all $\xi$ the equality $\xi'_v = 0$ holds for all $v \in \Gamma_V$. Hence
\begin{equation*}\label{better}
m(Z(v)) = \int_{\partial K_{\beta}\cap \overline{K_{\beta}(M)}} \xi_v \ \mathrm{d}\nu(\xi) \ .
\end{equation*}
Since $\left(m(Z(v))\right)_{v\in \Gamma_V}$ is $A(\beta)$-harmonic we find that
$$
\int_{\partial K_{\beta}\cap \overline{K_{\beta}(M)}} \xi_v \ \mathrm{d}\nu(\xi) = \int_{\partial K_{\beta}\cap \overline{K_{\beta}(M)}} \sum_{w\in \Gamma_V} A(\beta)_{v,w}\xi_w \ \mathrm{d}\nu(\xi)
$$
for all $v \in \Gamma_V$ and it follows from this that $\nu$-almost all elements of $\partial K_{\beta}\cap \overline{K_{\beta}(M)}$ are $A(\beta)$-harmonic. Let $B$ be a Borel subset of $\partial K_{\beta}\cap \overline{K_{\beta}(M)}$ and assume that $0 < \nu(B)< 1$. We can then define two normalized $A(\beta)$-harmonic vectors $\varphi$ and $\varphi'$ such that
$$
\varphi_v = \frac{1}{\nu(B)} \int_B \xi_v \ \mathrm{d}\nu(\xi) 
$$
and
$$
\varphi'_v = \frac{1}{1-\nu(B)} \int_{(\partial K_{\beta}\cap \overline{K_{\beta}(M)})\backslash B} \ \xi_v \ \mathrm{d}\nu(\xi) \ .
$$
Since $m(Z(v)) = \nu(B) \varphi_v + (1-\nu(B))\varphi'_v$ the assumption that $m$ is extremal together with Proposition \ref{nov1} implies that
$$
\nu(B) m(Z(v)) =  \int_B \xi_v \ \mathrm{d}\nu(\xi)
$$
for all $v \in \Gamma_V$. This identity is obvious when $\nu(B) \in \{0,1\}$ and hence it holds for all Borel subsets $B$ of $\partial K_{\beta}\cap \overline{K_{\beta}(M)}$. It follows from this that the set 
$$
\left\{ \xi \in \partial K_{\beta}\cap \overline{K_{\beta}(M)} : \ \xi_v = m(Z(v))  \ \forall v \in \Gamma_V\right\}
$$
has $\nu$-measure 1; in particular, it is not empty.
\end{proof}

\begin{thm}\label{OKok} 
Let $m$ be an extremal normalized $e^{\beta F}$-conformal measure. For $m$-almost all $p \in P(\Gamma)$,
\begin{equation}\label{just}
\lim_{k \to \infty} K_{\beta}(v,s(p_k)) = m(Z(v))
\end{equation}
for all $v \in \Gamma_V$.
\end{thm}
\begin{proof} By Lemma \ref{morgen3'} applied with $M = \Gamma_V$ there is an element $\xi \in \partial K_{\beta}$ such that 
$m(Z(v)) = \xi_v$
for all $v \in \Gamma_V$. Fix $v' \in \Gamma_V$ and let $\epsilon > 0$. Set
$$
M = \left\{w \in \Gamma_V : \ \left|K_{\beta}(v',w) - \xi_{v'} \right| \geq \epsilon \right\} \  .
$$
Assume for a contradiction that
$$
 m\left(\left\{p \in Z(v): \ s(p_k) \in M\ \text{for infinitely many} \ k \right\}\right) \neq 0
$$
for some $v \in \Gamma_V$. Then 
$$
m\left(\left\{p \in Z(v): \ s(p_k) \in M\ \text{for infinitely many} \ k \right\}\right) = m(Z(v))
$$
for all $v$ by Lemma \ref{oxx} and it follows from Lemma \ref{morgen3'} that there is a $\xi' \in \partial K_{\beta} \cap \overline{K_{\beta}(M)}$ such that $m(Z(v)) = \xi'_v$ for all $v$. But $\left|\xi'_{v'} - \xi_{v'}\right| \geq \epsilon$ since $\xi' \in \overline{K_{\beta}(M)}$, which gives us the desired contradiction. Thus 
$$
 m\left(\left\{p \in Z(v): \ s(p_k) \in M\ \text{for infinitely many} \ k \right\}\right) = 0
$$
for all $v\in \Gamma_V$. Since $\epsilon >0$ and $v' \in \Gamma_V$ were arbitrary it follows that \eqref{just} holds for $m$-almost all $p$. 
\end{proof}

Let $X_{\beta}$\label{Xbeta} be the set of elements $p \in P(\Gamma)$ with the property that the limit $\lim_{k \to \infty} K_{\beta}(v,s(p_k))$ exists for all $v \in \Gamma_V$ and the resulting vector 
 $$
 \psi_v = \lim_{k \to \infty} K_{\beta}(v,s(p_k)), \ \ v \in \Gamma_V,
 $$
 is a minimal $A(\beta)$-harmonic vector. This is a Borel subset of $P(\Gamma)$. Indeed, the set
 \begin{equation*}\label{Bset}
\mathcal B = \left\{ p \in P(\Gamma) : \ \lim_{k \to \infty} K_{\beta}(v,s(p_k)) \ \text{exists for all} \ v \in \Gamma_V \right\} 
\end{equation*}
 is Borel and we can define a Borel map
 $ K : \mathcal B \to \prod_{v \in \Gamma_V} \left[0,b_v\right]$ such that 
 $$
 K(p)_v =\lim_{k \to \infty} K_{\beta}(v,s(p_k)) \ .
 $$
 Then 
 $$
 X_{\beta} = K^{-1}\left(\partial H^{v_0}_{\beta F}(\Gamma)\right) \subseteq \mathcal B
 $$
  is a Borel set since $\partial H^{v_0}_{\beta F}(\Gamma)  \subseteq \prod_{v \in \Gamma_V} \left[0,b_v\right]$ is.

\begin{cor}\label{OKok2} 
Let $m \in M^{v_0}_{\beta F}(\Gamma)$. Then $m$ is concentrated on $X_{\beta}$ and
\begin{equation}\label{septeq26x}
\begin{split}
& \lim_{k \to \infty} \int_{Z(v_0)}K_{\beta}(v,s(p_k)) \ \mathrm{d}m(p) = \int_{Z(v_0)} \lim_{k \to \infty} K_{\beta}(v,s(p_k)) \ \mathrm{d}m(p) \\
& \\
&= m(Z(v))
\end{split}
\end{equation}
for all $v \in \Gamma_V$.
\end{cor}
\begin{proof} By Proposition \ref{rep} there is a Borel probability measure $\nu$ on $\partial H^{v_0}_{\beta F}(\Gamma)$ such that 
$$
m = \int_{\partial H^{v_0}_{\beta F}(\Gamma)} m_x \ \mathrm{d} \nu(x) \ .
$$ 
Each $m_x$ is an extremal $e^{\beta F}$-conformal measure and hence concentrated on $X_{\beta}$ by Theorem \ref{OKok}. It follows that $m$ is also concentrated on $X_{\beta}$. The first identity in \eqref{septeq26x} follows from Lebesgues theorem on dominated convergence. For the second we use Theorem \ref{OKok} for each $m_x$:
\begin{equation*}
\begin{split}
&  \int_{Z(v_0)} \lim_{k \to \infty}K_{\beta}(v,s(p_k)) \ \mathrm{d}m(p) \\
& =   \int_{\partial H^{v_0}_{\beta F}(\Gamma)} \int_{Z(v_0)}\lim_{k \to \infty} K_{\beta}(v,s(p_k)) \ \mathrm{d}m_x (p) \ \mathrm{d}\nu(x)\\
& =  \int_{\partial H^{v_0}_{\beta F}(\Gamma)}  m_x(Z(v))\ \mathrm{d}\nu(x) = m(Z(v)) \ .
\end{split}
\end{equation*}
\end{proof}

\subsection{Disintegration of conformal measures}\label{sec2} 

In this section $\Gamma$ is an arbitrary countable directed graph.

\begin{lemma}\label{dis} Let $X$ be a second countable topological Hausdorff space and $T : P(\Gamma) \to X$ a Borel map. Let $m$ be a regular Borel measure on $P(\Gamma)$ and $\mu$ a $\sigma$-finite Borel measure on $X$ such that $m \circ T^{-1}$ is absolutely continuous with respect to $\mu$ (i.e. $\mu(B) = 0 \ \Rightarrow \ m\left(T^{-1}(B)\right) = 0$). Then $m$ has a $(T,\mu)$-disintegration, i.e. there are regular Borel measures $m_x, x \in X$, on $P(\Gamma)$ such that
\begin{enumerate}
\item[a)] $m_x$ is concentrated on $T^{-1}(x)$ for $\mu$-almost all $x\in X$ ,
\end{enumerate}
and for every Borel function $f: P(\Gamma) \to [0,\infty]$ ,
\begin{enumerate}
\item[b)] the function $X \ni x \mapsto \int_{T^{-1}(x)} f(y) \ \mathrm{d} m_x(y)$ is Borel, and
\item[c)] $$
\int_{P(\Gamma)} f \ \mathrm{d} m = \int_{X} \left( \int_{T^{-1}(x)} f(y) \ \mathrm{d} m_x(y)\right) \ \mathrm{d}\mu(x) \ .
$$
\end{enumerate}
The disintegration is unique in the following sense: If $m'_x, x \in X$, is another collection of regular Borel measures on $P(\Gamma)$ for which a),b) and c) hold, then $m'_x = m_x$ for $\mu$-almost all $x \in X$.

\end{lemma}
\begin{proof} This follows from Theorem 1 in \cite{CP}. We only have to observe that $m$ is $\sigma$-finite, finite on compact sets and inner regular since this is part of the assumptions in \cite{CP}. The first two properties follow from the assumed finiteness of $m\left(Z(v)\right)$ since $Z(v), v \in \Gamma_V$, is an open cover of $P(\Gamma)$. Since $Z(v)$ is a Polish space in the relative topology inherited from $P(\Gamma)$ the measure $A \mapsto m(A \cap Z(v))$ is inner regular for all $v\in V$ by Proposition 8.1.10 and Proposition 7.2.6 in \cite{Co}. It follows easily from this that $m$ is also inner regular.
 \end{proof}

\begin{thm}\label{disTHM2} Let $m$ be an $e^{\beta F}$-conformal measure on $P(\Gamma)$. Let $X$ be a second countable topological Hausdorff space and $T : P(\Gamma) \to X$ a Borel map. Assume that $T (\sigma(p)) = T(p)$ for all $p\in P(\Gamma)$. 
\begin{enumerate}
\item[i)] There is a Borel probability measure on $X$ with the same null-sets as the pushforward measure ${m \circ T^{-1}}$.
\item[ii)] For every $\sigma$-finite Borel measure $\nu$ on $X$ with ${m\circ T^{-1}}$ absolutely continuous with respect to $\nu$ there is a $(T,\nu)$-disintegration $m_x, x \in X$, of $m$ such that $m_x$ is an $e^{\beta F}$-conformal measure for $\nu$-almost all $x \in X$.
\end{enumerate}
\end{thm}
\begin{proof} i) Set $M = \left\{ v \in \Gamma_V : \ m(Z(v)) \neq 0 \right\}$. Note that $M \neq  \emptyset$ because $m \neq 0$. Choose positive real numbers $\lambda_v$ such that $\sum_{v\in M} \lambda_v = 1$ and define a Borel probability measure $m'$ on $P(\Gamma)$ such that
$$
m'(B) = \sum_{v\in M} \lambda_v \frac{m (Z(v) \cap B)}{m(Z(v))}  \ .
$$
Note that $m$ has the same null-sets as $m'$. It follows that the pushforward measure $m\circ T^{-1}$ has the same null-sets as the Borel probability measure $m' \circ T^{-1}$.

ii) By Lemma \ref{dis} there is a $(T,\nu)$-disintegration $m_x, x \in X$, of $m$, and we must show that $m_x$ is $e^{\beta F}$-conformal for $\nu$-almost all $x$. Consider an arrow $a\in \Gamma_{Ar}$. Define Borel measures $m^a$ and $n^a$ on $P(\Gamma)$ such that 
$$
n^a(B) = m(\sigma(B \cap Z(a)))= e^{\beta F(a)} m(B \cap Z(a))
$$
and
$$
m^a(B) =  m(B \cap Z(a)) \  ,
$$
respectively. For each $x \in X$ define Borel measures $m^a_x$ and $n^a_x$ on $P(\Gamma)$ by
$$
n^a_x(B) = m_x(\sigma(B \cap Z(a)))
$$
and
$$
m^a_x(B) =  m_x(B \cap Z(a)) \ ,
$$
respectively. When $f$ is a positive Borel function on $P(\Gamma)$ we have that
\begin{equation*}
\begin{split}
&\int_{P(\Gamma)}  f \ \mathrm{d}m^a  = \int_{P(\Gamma)} 1_{Z(a)} f \ \mathrm{d}m  \\
&=\int_{X}\left( \int_{T^{-1}(x)} 1_{Z(a)} f \ \mathrm{d} m_x\right) \ \mathrm{d}\nu(x) = \int_{X}\left( \int_{T^{-1}(x)}  f \ \mathrm{d} m^a_x\right) \ \mathrm{d}\nu(x)  \ .
\end{split}
\end{equation*}
Since $n^a = e^{\beta F(a)} m^a$ it follows that
\begin{equation}\label{aug9}
\begin{split}
&\int_{P(\Gamma)}  f \ \mathrm{d}n^a  =  \int_X \left( \int_{T^{-1}(x)} f \ e^{\beta F(a)} \mathrm{d} m^a_x \right) \ \mathrm{d}\nu(x) \  .
\end{split}
\end{equation}
The transformation theorem for integrals shows that
$$
\int_{T^{-1}(x)} f \ \mathrm{d} n^a_x = \int_{T^{-1}(x)}  g \ \mathrm{d} m_x
$$
where $g : P(\Gamma) \to [0,\infty)$ is the Borel function 
$$
g(p) = \begin{cases} 0, & \ p \notin Z(r(a)) \\ f  \left(\sigma^{-1}(p)\right), & \ p \in Z(r(a)) \ , \end{cases}
$$
and $\sigma^{-1} : Z(r(a)) \to Z(a)$ is the inverse of $\sigma : Z(a) \to Z(r(a))$. It follows that the map $x \mapsto \int_{T^{-1}(x)} f \ \mathrm{d} n^a_x $ is Borel. Let $B \subseteq P(\Gamma)$ be a Borel set. Then
$$
\int_X \left( \int_{T^{-1}(x)} 1_B \ \mathrm{d} n_x^a \right) \ \mathrm{d} \nu(x) = \int_X m_x\left(\sigma \left(B \cap Z(a) \cap T^{-1}(x)\right)\right) \ \mathrm{d} \nu(x) \ .
$$
Since $T \circ \sigma = T$ by assumption, it follows that
$$
\sigma \left(B \cap Z(a) \cap T^{-1}(x)\right) = \sigma \left(B \cap Z(a)\right)\cap T^{-1}(x)
$$ 
and hence
\begin{equation*}
\begin{split}
&\int_X \left( \int_{T^{-1}(x)} 1_B \ \mathrm{d} n_x^a \right) \ \mathrm{d} \nu(x) = \int_X m_x\left(\sigma \left(B \cap Z(a) \right)\right) \ \mathrm{d} \nu(x) \\
& = m\left(\sigma(B \cap Z(a))\right) = e^{\beta F(a)} m(B \cap Z(a)) = n^a(B) \ .
\end{split}
\end{equation*}
Since $B$ was arbitrary it follows that
\begin{equation}\label{mx}  
\int_X \left( \int_{T^{-1}(x)}  \ f \ \mathrm{d} n_x^a \right) \ \mathrm{d} \nu(x) = \int_{P(\Gamma)} f \ \mathrm{d} n^a \  .
\end{equation}
Note that $n^a_x$ is concentrated on $T^{-1}(x)$ since $\sigma^{-1}(T^{-1}(x)) \subseteq T^{-1}(x)$ by assumption. Hence $n^a_x$ and $e^{\beta F(a)} m^a_x$ are both regular Borel measures concentrated on $T^{-1}(x)$ and as (\ref{aug9}) and (\ref{mx}) show, they both define $(T,\nu)$-disintegrations of $n^a$. It follows therefore from the essential uniqueness of the disintegration, cf. Lemma \ref{dis}, that $n^a_x = e^{\beta F(a)} m^a_x$ for $\nu$-almost every $x$. That is, for $\nu$-almost all $x$ we have that
$$
m_x\left(\sigma(B \cap Z(a))\right) = e^{\beta F(a)} m_x\left(B \cap Z(a)\right) 
$$
for all Borel sets $B$ in $P(\Gamma)$. It follows that $m_x$ is $e^{\beta F}$-conformal for $\nu$-almost all $x$. 
\end{proof}

In the following we write
$$
m \ = \int_X m_x \ \mathrm{d}\nu(x) \ ,
$$
in the setting of ii) in Theorem \ref{disTHM2}.

\begin{lemma}\label{30-09-17} Let $m$ be an $e^{\beta F}$-conformal measure on $P(\Gamma)$. Let $X$ be a second countable topological Hausdorff space and $T : P(\Gamma) \to X$ a Borel map. Assume that $T (\sigma(p)) = T(p)$ for all $p\in P(\Gamma)$ and that $m$ is ergodic for the shift $\sigma$. There is a point $x \in X$ such that $m$ is concentrated on $T^{-1}(x)$.
\end{lemma}
\begin{proof} By i) of Theorem \ref{disTHM2} there is a Borel probability measure $\nu$ on $X$ with the same null sets as $m \circ T^{-1}$. Let $U_i, i = 1,2,3, \cdots$, be a base for the topology of $X$. Set $U_0 = X$. For each $n \in \mathbb N$, let $P^n$ be the partition of $X$ consisting of the sets $V_0,V_1,\cdots , V_n$, where $V_n = U_n$ and
$$
V_j = U_j  \backslash \bigcup_{i=j+1}^{n} U_i, \ j = 0,1,2,3, \cdots, n-1 \ .
$$
 For each $A \in P^n$ the set $T^{-1}(A)$ is totally shift invariant, i.e. $\sigma^{-1}\left(T^{-1}(A)\right) = T^{-1}(A)$, and as $m$ is ergodic for the shift by assumption, it follows that there is one and only one element $A_n \in P^n$ for which $m(T^{-1}(A_n))$ is not zero. Since $\nu$ and $m\circ T^{-1}$ have the same null sets, it follows that $A_n$ is also the unique element of $P^n$ with $\nu(A_n) = 1$. Note that 
$\nu\left( \bigcap_{i=1}^m A_i \right) = 1$
for all $m$, which implies that
$$
\nu \left( \bigcap_{i=1}^{\infty} A_i\right) = 1 \ .
$$
The intersection $ \bigcap_{i=1}^{\infty} A_i$ can not contain more than one element because $X$ is Hausdorff.
It follows that $\nu$ is concentrated on a single element $x \in X$, i.e. $\nu\left( X \backslash \{x\}\right) =0$, and hence $m\left( P(\Gamma) \backslash T^{-1}(x)\right) = 0$.
\end{proof}

\subsection{Extremal conformal measures and Martin kernels}

In this section $\Gamma$ is an arbitrary countable directed graph and $F$ any potential on $\Gamma$.

\begin{lemma}\label{15-05-18} $\Wan(\Gamma)$ is a Borel subset of $P(\Gamma)$.
\end{lemma}
\begin{proof} For every finite subset $H \subseteq \Gamma_V$ the set 
$$
A_H = \left\{ (x_i)_{i=1}^{\infty} : \ s(x_i) \notin H \ \forall i \right\}
$$ 
is closed in $P(\Gamma)$, and $\Wan(\Gamma)$ is a Borel set because
$$
\Wan(\Gamma) = \bigcap_{H} \bigcup_{n=1}^{\infty} \sigma^{-n}\left(A_H\right) \ .
$$
\end{proof}
\begin{lemma}\label{12-11-17x} {}
Assume that
$$
\sum_{n=0}^{\infty} A(\beta)^n_{v,w} \ < \ \infty
$$
for all $v,w \in \Gamma_V$. Every $e^{\beta F}$-conformal measure $m$ on $P(\Gamma)$ is concentrated on $\Wan(\Gamma)$, i.e. $m\left( P(\Gamma) \backslash \Wan(\Gamma)\right) = 0$.  
\end{lemma}
\begin{proof}
This follows from Theorem 4.10 of \cite{Th3} when $\Gamma$ is strongly connected. We check here that the relevant part of the argument also works in the present generality. Given two vertexes $v,w \in \Gamma_V$, set
$$
M(w,v) = \left\{ (p_i)_{i=1}^{\infty} \in Z(w): \  \ s(p_i) = v \ \text{for infinitely many $j$} \right\} \ .
$$
For $N \in \mathbb N$ and $j \geq N$, let
$$
M^j = \left\{ (p_i)_{i=1}^{\infty} \in Z(w): \ s(p_j) = v \right\} \ .
$$
Then $M(w,v) \subseteq \cup_{j \geq N} M^j$ and hence
$$
m(M(w,v)) \leq \sum_{j \geq N} m\left(M^j\right) \ .
$$
It follows from Lemma \ref{nov2} that 
$$
  m\left(M^j\right) = A(\beta)^{j-1}_{w,v} m(Z(v)) \ ,
  $$ 
  and hence 
  \begin{equation}\label{06-02-18g}
  m(M(w,v)) \leq m(Z(v)) \sum_{j \geq N}  A(\beta)^{j-1}_{w,v} \ .
  \end{equation}
  By letting $N$ go to infinity in \eqref{06-02-18g} we find that $m(M(w,v)) = 0$. Since 
  $$
  P(\Gamma) \backslash \Wan(\Gamma) = \bigcup_{w,v} M(w,v) \ ,
  $$ 
  it follows that $m$ is concentrated on $\Wan(\Gamma)$.
\end{proof}

A \emph{ray} in $\Gamma$ is an infinite path $p = (p_i)_{i=1}^{\infty} \in P(\Gamma)$ whose vertexes are distinct, i.e. $i \neq j \Rightarrow s(p_i) \neq s(p_j)$. The set $\Ray(\Gamma)$\label{Ray} of rays in $\Gamma$ is a closed subset of $P(\Gamma)$ and $\Ray(\Gamma) \subseteq \Wan(\Gamma)$. We shall need a retraction map \label{R}\label{R}
\begin{equation}\label{21-02-18b}
R : \Wan(\Gamma) \to \Ray(\Gamma)
\end{equation}
which we construct as follows.
 Let $p = (p_i)_{i=1}^{\infty} \in \Wan(\Gamma)$. When $p \notin \Ray(\Gamma)$, set
$$
j_p = \min \left\{i: \ \exists i' >i, \ s(p_i) = s(p_{i'}) \right\}\ ,
$$
and let
$$
j'_p = \min \left\{i: \ i > j_p, \ s(p_i) = s\left(p_{j_p}\right) \right\} \ .
$$
Then define $R_0(p) \in \Wan(\Gamma)$ such that 
$$
R_0(p)_i = \begin{cases} p_i, & \ i \leq j_p-1 \\ p_{i + j'_p- j_p}, & i \geq j_p \ . \end{cases}
$$
When $p \in \Ray(\Gamma)$ we set $R_0(p) = p$. The limit
\begin{equation*}\label{19-02-18d}
R(p) = \lim_{n \to \infty} R_0^n(p)
\end{equation*}
exists in $P(\Gamma)$ and $R(p) \in \Ray(\Gamma)$ for all $p \in \Wan(\Gamma)$. It will be shown in Lemma \ref{09-11-17a} that $R$ is a Borel map, but we shall not need this fact before then.

\begin{thm}\label{MAIN2} Assume that there is a vertex $v_0$ such that \eqref{v0trans} holds. Let $m$ be an $e^{\beta F}$-conformal measure on $P(\Gamma)$. The following are equivalent:
\begin{enumerate}
\item[a)] There is a ray $p \in \Ray(\Gamma)$ such that $m$ is concentrated on 
\begin{equation*}\label{04-12-17a}
\left\{q \in \Wan(\Gamma) : \ \lim_{k \to \infty} K_{\beta}(v,s(q_k)) = \lim_{k \to \infty} K_{\beta}(v,s(p_k)) \ \forall \ v \in \Gamma_V \right\} \ .
\end{equation*}
\bigskip
\item[b)] $m$ is concentrated on
\begin{equation*}\label{sett}
\left\{ q \in P(\Gamma) : \ \lim_{k \to \infty} K_{\beta}(v, s(q_k)) = \frac{m(Z(v))}{m(Z(v_0))} \ \forall v \in \Gamma_V \right\} \ .
\end{equation*}
\bigskip
\item[c)] $m$ is extremal among the $e^{\beta F}$-conformal measures on $P(\Gamma)$ .
\bigskip
\item[d)] $m$ is ergodic for the shift on $P(\Gamma)$ .
\end{enumerate}
\end{thm}
\begin{proof} a) $\Rightarrow$ b):  It follows from Corollary \ref{OKok2} that 
$$
m(Z(v)) = \int_{Z(v_0)} \lim_{k \to \infty} K_{\beta}(v,s(q_k)) \ \mathrm{d} m(q) \ .
$$
Hence a) implies that $\frac{m(Z(v))}{m(Z(v_0))} = \lim_{k \to \infty} K_{\beta}(v,s(p_k))$, and $m$ is therefore concentrated on the set from b).

b) $\Rightarrow$ a): Combining b) with Lemma \ref{12-11-17x} it follows that $m$ is concentrated on the set
$$
\left\{ q \in \Wan(\Gamma) : \ \lim_{k \to \infty} K_{\beta}(v, s(q_k)) = \frac{m(Z(v))}{m(Z(v_0))} \ \forall v \in \Gamma_V \right\} \ .
$$
In particular, there \emph{is} a wandering path $q$ in this set. Let $p = R(q)\in \Ray(\Gamma)$ and note that $\lim_{k \to \infty} K_{\beta}(v, s(q_k)) = \lim_{k \to \infty} K_{\beta}(v, s(p_k))$. It follows that $m$ is concentrated on the set in a).

b) $\Rightarrow$ c): Let $m'$ be an $e^{\beta F}$-conformal measure such that $m' \leq m$. Since $m$ is concentrated on the set in b) the same holds for $m'$ and it follows from Corollary \ref{OKok2} that
$$
m'(Z(v)) = \int_{Z(v_0)} \lim_{k \to \infty} K_{\beta}(v,s(p_k)) \ \mathrm{d}m'(p) = \frac{m'(Z(v_0))}{m(Z(v_0))} m(Z(v))
$$
for all $v \in \Gamma_V$. Hence $m' = \frac{m'(Z(v_0))}{m(Z(v_0))}m$ by Proposition \ref{nov1}.

c) $\Rightarrow$ d) : Let $A\subseteq P(\Gamma)$ be a Borel subset such that $\sigma^{-1}(A) = A$. If $m(A) \neq 0$ and $m(P(\Gamma)\backslash A) \neq 0$, the measure 
$$
m' (B) = m(B \cap A)
$$
will be an $e^{\beta F}$-conformal measure such that $m' \leq m$, and no $\lambda \geq 0$ will satisfy that $m' = \lambda m$.

d) $\Rightarrow$ b) : We may assume that $m$ is $v_0$-normalized. Add the point $\clubsuit$ to $\partial H^{v_0}_{\beta F}(\Gamma)$ and define $T : P(\Gamma) \to \partial H^{v_0}_{\beta F}(\Gamma) \sqcup \{\clubsuit\}$ such that
$$
T(p) = \begin{cases} \left(\lim_{k \to \infty} K_{\beta}(v,s(p_k))\right)_{v \in \Gamma_V}, & \ p \in X_{\beta}, \\ \clubsuit , & \ p \in P(\Gamma) \backslash X_{\beta}\  . \end{cases}
$$ 
It follows from Lemma \ref{30-09-17} that there is a point $x \in \partial H^{v_0}_{\beta F}(\Gamma) \sqcup \{\clubsuit\}$ such that $m$ is concentrated on $T^{-1}(x)$. It follows from Corollary \ref{OKok2} that $x \in \partial H^{v_0}_{\beta F}(\Gamma)$ and that $m(Z(v)) = x_v$ for all $v \in \Gamma_V$. Hence $T^{-1}(x)$ is a subset of the set from b). 
\end{proof}

\begin{remark}\label{01-10-17k} In a weak moment one may think that an $e^{\beta F}$-conformal measure $m$ for which there is a ray $p \in \Ray(\Gamma)$ such that 
$$
\lim_{k \to \infty} K_{\beta}(v,s(p_k)) = m(Z(v))
$$
for all $v \in \Gamma_V$ must be extremal. After we have developed some general methods for the construction of examples we will show in Example \ref{18-02-18} that this not the case, even when $\Gamma$ is row-finite and strongly connected. Thus the stronger conditions in a) and b) of Theorem \ref{MAIN2} involving the support of $m$ are necessary. It should perhaps be pointed out that the examples in the literature of random walks for which the minimal Martin boundary is not the whole Martin boundary do not automatically yield such an example because the Martin boundary is defined from a much larger set of functions on $\Gamma_V$ than those arising from the limits $\lim_{k \to \infty} K_{\beta}(v,s(p_k))$.  
\end{remark}

When $\Gamma$ is strongly connected, or just cofinal, the results in this section can be obtained from the convergence to the boundary results for random walks on $\Gamma$, using the translation described in Section \ref{randomw}. We shall not need any substitute for the Martin boundary in this work, but refer to Appendix \ref{appB} for a version of the Poisson-Martin integral representation of harmonic vectors which is valid in the generality of this section and which is much closer in spirit to the present project.   
 
\section{Tools for constructions}\label{tools}

\subsection{Turning a vertex into a source}\label{source}

Let $\Gamma$ be an arbitrary countable directed graph with potential $F : \Gamma_{Ar} \to \mathbb R$. Assume that $\beta$ is a real number such that $A(\beta)$ is transient in the sense that
\begin{equation}\label{07-02-18}
\sum_{n=0}^{\infty} A(\beta)^n_{v,w}  < \infty
\end{equation}
for all $v,w \in \Gamma_V$. When $v,w \in \Gamma_V$ a \emph{simple path from $v$ to $w$} is a path of positive length $\nu =a_1a_2\cdots a_n \in P_f(\Gamma)$ such that $s(\nu) = v, \ r(\nu) = w$ and $r(a_i) \neq w, \ 1 \leq i \leq n-1$. We set \label{RGammaF}
$$
R^{\Gamma}_{\beta F}(v,w) = \sum_{\nu } e^{-\beta  F(\nu)} 
$$
where the sum is over the simple paths $\nu$ from $v$ to $w$. Take a vertex $v_0 \in \Gamma_V$ and let $\Gamma^0$ be the digraph obtained from $\Gamma$ by deleting the arrows going into $v_0$; i.e.
$$
\Gamma^0_V  = \Gamma_V
$$
and
$$
\Gamma^0_{Ar} = \Gamma_{Ar} \backslash r^{-1}(v_0)  \ .
$$
Let $B(\beta)$ be the matrix over $\Gamma^0_V$ obtained  by restricting the potential $F$ to $\Gamma^0_{Ar}$; i.e. $B(\beta)_{v,w} = A(\beta)_{v,w}$ when $w \neq v_0$ and $B(\beta)_{v,v_0} =0$ for all $v \in \Gamma_V$.
\begin{lemma}\label{21-09-17a} There is an affine homeomorphism $\varphi \mapsto \psi$ from the $v_0$-normalized $A(\beta)$-harmonic vectors onto the $v_0$-normalized $B(\beta)$-harmonic vectors given by the formula
\begin{equation}\label{21-09-17b} 
\psi_v = \left( 1 - R^{\Gamma}_{\beta F}(v_0,v_0)\right)^{-1}\left(\varphi_v - R^{\Gamma}_{\beta F}(v,v_0)\right) \ .
\end{equation}
\end{lemma}
\begin{proof} Observe first of all that $R^{\Gamma}_{\beta F}(v,v_0) \leq \varphi_v$ for all $ v \in \Gamma_V$. Indeed, when $m_{\varphi}$ is the $e^{\beta F}$-harmonic measure on $P(\Gamma)$ corresponding to $\varphi$, we have that
\begin{equation*}
\begin{split}
& \varphi_v = m_{\varphi}(Z(v)) \geq m_{\varphi}\left(Z(v) \cap \left(\bigcup_{k =1}^{\infty} \sigma^{-k}\left(Z(v_0)\right)\right) \right)= R^{\Gamma}_{\beta F}(v,v_0) \ .
 \end{split}
 \end{equation*}
 Furthermore, it follows from Lemma 4.11 in \cite{Th3} that $ R^{\Gamma}_{\beta F}(v_0,v_0) < 1$ because we are in the transient case where $\sum_{n=0}^{\infty} A(\beta)^n_{v_0,v_0} < \infty$. Therefore the formula \eqref{21-09-17b} does define a $v_0$-normalized element $\psi$ of $[0,\infty)^{\Gamma_V}$. Set $\lambda =1 - R^{\Gamma}_{\beta F}(v_0,v_0)$. We check
 \begin{equation*}
  \begin{split}
  & \sum_{w \in \Gamma_V} B(\beta)_{v,w}\psi_w = \lambda^{-1}\left(  \sum_{w \in \Gamma_V} B(\beta)_{v,w}\varphi_w -  \sum_{w \in \Gamma_V} B(\beta)_{v,w} R^{\Gamma}_{\beta F}(w,v_0)\right) \\
  & =\lambda^{-1}\left(  \sum_{w \in \Gamma_V} A( \beta)_{v,w}\varphi_w - A(\beta)_{v,v_0} - \sum_{w \in \Gamma_V} B(\beta)_{v,w} R^{\Gamma}_{\beta F}(w,v_0)\right) \\
& = \lambda^{-1}\left(  \varphi_v- A(\beta)_{v,v_0} - \sum_{w \in \Gamma_V} B(\beta)_{v,w} R^{\Gamma}_{\beta F}(w,v_0)\right) \\
& =  \lambda^{-1}\left(  \varphi_v -  A(\beta)_{v,v_0} -  \left( R^{\Gamma}_{\beta F}(v,v_0) - A(\beta)_{v,v_0}\right) \right) = \psi_v \ .\\
 \end{split}
 \end{equation*}
 Assume then that we are given a $v_0$-normalized $B(\beta)$-harmonic vector $\psi$. Set
 $$
 \varphi_v = \lambda \psi_v + R^{\Gamma}_{\beta F}(v,v_0) \ .
 $$
 We check
 \begin{equation*}
  \begin{split}
  & \sum_{w \in \Gamma_V} A(\beta)_{v,w}\varphi_w = \lambda \sum_{w \in \Gamma_V} A(\beta)_{v,w}\psi_w +\sum_{w \in \Gamma_V} A(\beta)_{v,w}R^{\Gamma}_{\beta F}(w,v_0)\\
  &= \lambda \left(\sum_{w \in \Gamma_V} B(\beta)_{v,w}\psi_w + A(\beta)_{v,v_0}\right) +\sum_{w \in \Gamma_V}A(\beta)_{v,w}R^{\Gamma}_{\beta F}(w,v_0)\\
  & = \lambda \psi_v + \lambda A(\beta)_{v,v_0} +A(\beta)_{v,v_0} R^{\Gamma}_{\beta F}(v_0,v_0) + \sum_{w \in \Gamma_V}B(\beta)_{v,w}  R^{\Gamma}_{\beta F}(w,v_0)\\
  &  = \lambda \psi_v + A(\beta)_{v,v_0}  + \sum_{w \in \Gamma_V} B(\beta)_{v,w}  R^{\Gamma}_{\beta F}(w,v_0)\\
  & = \lambda \psi_v +  R^{\Gamma}_{\beta F}(v,v_0) = \varphi_v \ .
 \end{split}
 \end{equation*}
\end{proof}

Note that the bijection of Lemma \ref{21-09-17a} is an affine homeomorphism
\begin{equation*}\label{07-02-18b}
H^{v_0}_{\beta F}(\Gamma) \ \simeq \ H^{v_0}_{\beta F}(\Gamma^0) \ .
\end{equation*}
The digraph $\Gamma^0$ in Lemma \ref{21-09-17a} may contain sinks, and since a $B(\beta)$-harmonic vector is zero at a sink, we want to get rid of sinks. A \emph{dead end} in $\Gamma^0$ is a vertex $v \in \Gamma^0_V$ with the property that there is an $n \in \mathbb N$ such that $A(\Gamma^0)^n_{v,w} = 0$ for all $w \in \Gamma_V$. In particular, a sink is a dead end. Let $D \subseteq \Gamma^0_V$ be the set of dead ends in $\Gamma^0$. Let $\Gamma^{v_0}$ be the digraph obtained from $\Gamma^0$ by  deleting all dead ends and all arrows going in or out of a dead end; i.e. 
$$
\Gamma^{v_0}_V = \Gamma^0_V \backslash D \ ,
$$
and
$$
\Gamma^{v_0}_{Ar} = \Gamma^0_{Ar} \backslash \left( r^{-1}\left(D\right) \cup s^{-1}\left(D\right) \right) \ .
$$
Set
$$
B'(\beta)_{v,w} \ \ = \ \sum_{ a \in s^{-1}(v) \cap r^{-1}(w) \cap \Gamma^{v_0}_{Ar}} e^{-\beta F(a)} 
$$
for $v,w \in \Gamma^{v_0}_V$. Note that $B'(\beta)$ is transient since we assume that $A(\beta)$ is.

\begin{lemma}\label{22-09-17} The restriction map $\varphi \mapsto \varphi|_{\Gamma^{v_0}_V}$ is an affine homeomorphism from $H_{\beta F}^{v_0}(\Gamma^0)$ onto $H_{\beta F}^{v_0}(\Gamma^{v_0})$.  
\end{lemma} 
\begin{proof} This is a direct check, using that a $B(\beta)$-harmonic vector is zero at dead ends.
\end{proof}

We will say that the graph $\Gamma^{v_0}$ is obtained from $\Gamma$ by \emph{turning $v_0$ into a source}.

\begin{prop}\label{07-02-18d} Let $\Gamma$ be a countable directed graph with a potential $F : \Gamma_{Ar} \to \mathbb R$ such that \eqref{07-02-18} holds for all $v,w \in \Gamma_V$. Let $v_0$ be a vertex in $\Gamma$ and $\Gamma^{v_0}$ the digraph obtained by turning $v_0$ into a source. 

\begin{itemize}
\item There is an affine homeomorphism $H^{v_0}_{\beta F}(\Gamma) \simeq H^{v_0}_{\beta F}\left( \Gamma^{v_0}\right)$ sending the vector $\varphi \in H^{v_0}_{\beta F}(\Gamma)$ to the vector $\psi \in H^{v_0}_{\beta F}\left( \Gamma^{v_0}\right)$ given by
\begin{equation}\label{07-02-18e}
\psi_v = \left( 1 - R^{\Gamma}_{\beta F}(v_0,v_0)\right)^{-1}\left(\varphi_v - R^{\Gamma}_{\beta F}(v,v_0)\right) \ .
\end{equation}
\item The affine homeomorphism $M^{v_0}_{\beta F}(\Gamma) \simeq M^{v_0}_{\beta F}\left( \Gamma^{v_0}\right)$ corresponding to \eqref{07-02-18e} sends $m \in M^{v_0}_{\beta F}(\Gamma)$ to the measure $m' \in  M^{v_0}_{\beta F}\left( \Gamma^{v_0}\right)$ given by
$$
m'(B) = \left(1-  R^{\Gamma}_{\beta F}(v_0,v_0)\right)^{-1} m(B)
$$
for Borel sets $B \subseteq P(\Gamma^{v_0})$.
\end{itemize}
\end{prop}
\begin{proof} The first item follows by combining  Lemma \ref{21-09-17a} and Lemma \ref{22-09-17}. To prove the second let $m = m_{\varphi}$ where $\varphi \in H^{v_0}_{\beta F}(\Gamma)$. We must show that
\begin{equation*}\label{21-02-18}
m_{\psi}(B) = \left(1-  R^{\Gamma}_{\beta F}(v_0,v_0)\right)^{-1} m_{\varphi}(B)
\end{equation*}
for Borel sets $B \subseteq P(\Gamma^{v_0})$ when $\psi \in H^{v_0}_{\beta F}(\Gamma^{v_0})$ is the vector \eqref{07-02-18e}. Let $\mu \in P_f(\Gamma^{v_0})$. It follows from Lemma \ref{nov2} that
\begin{equation}\label{30-01-18a}
m_{\varphi}\left( Z(\mu) \cap P(\Gamma^{v_0})\right) = e^{-\beta F(\mu)} m_{\varphi}\left( Z(v) \cap P(\Gamma^{v_0})\right)
\end{equation}
and
\begin{equation}\label{30-01-18b}
m_{\psi}\left( Z(\mu) \cap P(\Gamma^{v_0})\right) = e^{-\beta F(\mu)} m_{\psi}\left( Z(v)\cap P(\Gamma^{v_0})\right) \ 
\end{equation}
where $v = r(\mu)$. Note that there is a partition
$$
Z(v) \cap \Wan(\Gamma) = \left(Z(v) \cap \Wan(\Gamma^{v_0})\right) \sqcup \bigcup_{\nu}  Z(\nu) \left[ Z(v_0) \cap \Wan(\Gamma)\right] 
$$
where we take the union over all simple paths $\nu \in P_f(\Gamma)$ such that $ s(\nu) = v, \ r(\nu) = v_0$. Using Lemma \ref{12-11-17x} we find therefore that
\begin{equation*}\label{07-02-18f}
\begin{split}
& m_{\varphi}(Z(v)) = m_{\varphi}(Z(v) \cap \Wan(\Gamma))
\\
& =  m_{\varphi}\left( Z(v) \cap \Wan(\Gamma^{v_0})\right) + \sum_{\nu} e^{-\beta F(\nu)} \varphi_{v_0} \\
&=  m_{\varphi}\left( Z(v) \cap P(\Gamma^{v_0})\right) +  R^{\Gamma}_{\beta F}(v,v_0) \ .
\end{split}
\end{equation*}
Combined with \eqref{07-02-18e} this yields 
\begin{equation*}
\begin{split}
&m_{\varphi}\left(Z(v) \cap P(\Gamma^{v_0})\right) = m_{\varphi}(Z(v)) - R^{\Gamma}_{\beta F}(v,v_0) \\
&= \left( 1 - R^{\Gamma}_{\beta F}(v_0,v_0)\right) \psi_v  = \left( 1 - R^{\Gamma}_{\beta F}(v_0,v_0)\right)m_{\psi}(Z(v)\cap P(\Gamma^{v_0})) \  .
\end{split}
\end{equation*}
Inserted into \eqref{30-01-18a} and \eqref{30-01-18b} it follows that the measure on $P(\Gamma^{v_0})$ given by
$$
B \mapsto \left( 1 - R^{\Gamma}_{\beta F}(v_0,v_0)\right)^{-1}m_{\varphi} \left( B\right)
$$ 
agrees with $m_{\psi}$ on $Z(\mu) \cap P(\Gamma^{v_0})$. By Lemma \ref{03-03-18} they agree on all Borel sets $B \subseteq P(\Gamma^{v_0})$. 
\end{proof}

\subsection{Adding return paths}\label{adding} 

In this section we consider an operation on digraphs which can be considered as the opposite of the one introduced in the previous section. Let $\Gamma$ be a strongly connected digraph and $A(\Gamma)$ its adjacency matrix. Following Ruette, \cite{Ru}, we call $h(\Gamma)$\label{hGamma} the \emph{Gurevich entropy} of $\Gamma$, where 
$$
h(\Gamma) = \log \left(\limsup_{n} \left(A(\Gamma)^n_{v,v}\right)^{\frac{1}{n}} \right),
$$
for some vertex $v \in \Gamma_V$. This quantity, which is independent of $v$, can be any number in $[0,\infty]$. Still following Ruette we say that $\Gamma$ is \emph{recurrent} when
$$
\sum_{n=0}^{\infty} A(\Gamma)^n_{v,v} e^{-n h(\Gamma)} 
$$
is infinite and \emph{transient} when it is finite.

Let $\Gamma^0$ be a countable directed graph and let $v_0, v \in \Gamma^0_V$ be vertexes. Consider a new digraph $\Gamma^1$ where 
$$
\Gamma^1_V = \Gamma^0_V \sqcup \{w_1,w_2 , \cdots, w_n\}
$$
and 
$$
\Gamma^1_{Ar} = \Gamma^0_{Ar} \sqcup \{a_0,a_1,a_2, \cdots , a_n\} \ .
$$
For an arrow in $\Gamma^0_{Ar} \subseteq \Gamma^1_{Ar}$ its range and source vertex in $\Gamma^1_V$ are the same as they are in $\Gamma^0$ and for the new arrows their range and source vertexes are given as follows:
\begin{itemize}
\item $s(a_0) = v, \ r(a_0) = w_1$,
\item $s(a_i) = w_{i}, \ r(a_i) = w_{i+1}, \ i = 1,2, \cdots, n-1$, and
\item  $s(a_n) = w_n, \ r(a_n) = v_0$.
\end{itemize}
Thus the digraph $\Gamma^1$ is obtained from $\Gamma^0$ by adding to $\Gamma$ the path $a_0a_1 \cdots a_n$ from $v$ to $v_0$, which we call a \emph{return path}. This operation can be repeated finitely many times, or even countably many times, still resulting in a countable directed graph $\Gamma$. In the applications we have in mind the terminal vertex of the attached return paths will always be the same vertex $v_0$, but the initial vertexes $v$ will vary. However, they will always come from the original digraph $\Gamma^0$. When this is the case we will say that $\Gamma$ is \emph{obtained from $\Gamma^0$ by adding return paths to $v_0$}.

When $v,w$ are vertexes in a digraph $\Gamma$ we denote by $l^n_{v,w}(\Gamma)$\label{lnvwgamma} the number of simple paths of length $n$ from $v$ to $w$.

\begin{lemma}\label{02-03-18} Let $\Gamma^0$ be a countable directed graph and $v_0 \in \Gamma^0_V$ a vertex such that 
\begin{itemize}
\item[a)] $\limsup_n \left(A(\Gamma^0)^n_{v_0,v}\right)^{\frac{1}{n}} \  < \ \infty$ for all $v \in \Gamma_V$, 
\item[b)]  $v_0$ receives no arrows in $\Gamma^0$ lands in $v_0$, i.e. $r^{-1}(v_0) = \emptyset$,
\item[c)] all vertexes in $\Gamma^0_V$ can be reached from $v_0$, and
\item[d)] $P(\Gamma^0)$ contains a ray $p  = (p_i)_{i=1}^{\infty}\in \Ray(\Gamma^0)$ with $s(p) = v_0$ such that $r^{-1}(r(p_i)) = \{p_i\}$ for all $i$.
\end{itemize} 
For every positive number $h  > 0$ such that
\begin{equation}\label{02-03-18e}
\limsup_n \left(A(\Gamma^0)^n_{v_0,v}\right)^{\frac{1}{n}} \  < \  e^h 
\end{equation}
for all $v \in \Gamma_V$,
there is a digraph $\Gamma$ obtained from $\Gamma^0$ by adding return paths to $v_0$ such that 
\begin{itemize}
\item $\Gamma$ is strongly connected,
\item $\Gamma$ is recurrent,
\item $h(\Gamma) = h$, and
\item there is exactly one arrow $e_0 \in \Gamma_{Ar}$ such that $s(e_0) = r(e_0) = v_0$.
\end{itemize}
Furthermore, $\Gamma$ is row-finite if $\Gamma^0$ is, and the digraph $\Gamma'$ which is obtained from $\Gamma$ by removing the arrow $e_0$ is transient and $h(\Gamma') = h$. 
\end{lemma}
\begin{proof} Let $v_1,v_2, \cdots$ be a numbering of the vertexes in $\Gamma^0_V \backslash \{v_0\}$. It follows from \eqref{02-03-18e} that the sums
\begin{equation}\label{25-06-18a}
\alpha_i = \sum_{n=1}^{\infty} A(\Gamma^0)^n_{v_0,v_i} e^{-nh}
\end{equation}
are all finite. 
We can therefore choose $m_i \in \mathbb N$ such that
$$
\sum_{i=1}^{\infty} \alpha_i e^{-m_i h} \ <  \ 1 - e^{-h} \ .
$$
For each $i = 1,2,3, \cdots$, we add a return path of length $m_i$ from $v_i$ to $v_0$ and we let $\Gamma^1$ denote the resulting graph. Then  
\begin{equation*}\label{19-9-17b}
\begin{split}
& \sum_{n=1}^{\infty} l^n_{v_0,v_0}\left(\Gamma^1\right) e^{-nh} = \sum_{i=1}^{\infty} \sum_{n=1}^{\infty}A(\Gamma^0)^n_{v_0,{v_i}} e^{-(n+m_i)h} \\
 & =  \sum_{i=1}^{\infty} \alpha_i e^{-m_i h} < 1 - e^{-h} \ .
 \end{split}
 \end{equation*}
 Note that $l^1_{v_0,v_0}\left(\Gamma^1\right) = 0$. By Lemma 6.6 in \cite{Th3} there is a sequence $\{b_n\}_{n=1}^{\infty}$ of non-negative integers such that
 \begin{itemize}
 \item $b_1  =1$, 
 \item $ b_n \geq  l^n_{v_0,v_0}\left(\Gamma^1\right)$ for all $n$,
 \item $\limsup_n \left(b_n\right)^{\frac{1}{n}} = e^{h}$ and
 \item $\sum_{n=1}^{\infty} b_n e^{-nh} = 1$.
 \end{itemize}
 Let $v_0 \to w_1 \to w_2 \to w_3 \to \cdots$ be the ray $p$ from assumption d). For each $n$ there is exactly one path in $\Gamma^0$ from $v_0$ to $w_n$ and it has length $n$. Add to $\Gamma^1$ an arrow $e_0$ (or a return path of length $1$) from $v_0$ to itself, and for all $n \geq 2$ add $b_n - l^n_{v_0,v_0}\left(\Gamma^1\right)$ arrows (or return paths of length $1$) from $w_{n-1}$ to $v_0$. Denote the resulting digraph by $\Gamma$. Then $\Gamma$ is strongly connected because all the vertexes can be reached from $v_0$ and all vertexes can reach $v_0$, and $\Gamma$ is obtained from $\Gamma^0$ by adding return paths to $v_0$. Furthermore, $\Gamma$ will be row-finite if $\Gamma^0$ is. Since
$$
 \sum_{n=1}^{\infty} l^n_{v_0,v_0}\left(\Gamma\right) e^{-nh}  = \sum_{n=1}^{\infty} b_n e^{-nh} = 1,
 $$
it follows from Lemma 4.15 in \cite{Th3}, applied with $F =1$ and $\beta = h$, that $\Gamma$ is recurrent and $h(\Gamma) = h$. By Propostion 6.3 in \cite{Th3} the graph $\Gamma'$ obtained from $\Gamma$ by removing $e_0$ is transient, but it still has Gurevich entropy $h$. 

\end{proof}

The constructions in the above proof will always make $\# r^{-1}(v_0)$ infinite and $\Gamma$ and $\Gamma'$ will therefore not be locally finite. It seems almost certain, however, that the method of proof can be refined to yield locally finite graphs when $\Gamma^0$ has this property. For the present purposes row-finiteness is all we care about.

\begin{cor}\label{attaching3}
Let $\Gamma^0$ be a countable directed graph and $v_0 \in \Gamma^0_V$ a vertex such that
\begin{itemize}
\item[a)] there is at least one, but only finitely many paths in $\Gamma^0$ from $v_0$ to any other vertex; in symbols,
$$1 \ \leq \ \# \left\{ \mu \in P_f\left(\Gamma^0\right) : \ s(\mu) = v_0, \ r( \mu) = v \right\} \ < \ \infty
$$ 
for all $v \in \Gamma^0_V \backslash \{v_0\}$,
\item[b)]  no arrow in $\Gamma^0$ lands in $v_0$, i.e. $r^{-1}(v_0) = \emptyset$, and
\item[c)]  $P(\Gamma^0)$ contains a ray $p  = (p_i)_{i=1}^{\infty}\in \Ray(\Gamma^0)$ with $s(p) = v_0$ such that $r^{-1}(r(p_i)) = \{p_i\}$ for all $i$.
\end{itemize} 
For every positive number $h  > 0$ there is a digraph $\Gamma$ obtained from $\Gamma^0$ by adding return paths to $v_0$ such that 
\begin{itemize}
\item $\Gamma$ is strongly connected,
\item $\Gamma$ is recurrent,
\item $h(\Gamma) = h$ and
\item there is exactly one arrow $e_0 \in \Gamma_{Ar}$ such that $s(e_0) = r(e_0) = v_0$.
\end{itemize}
Furthermore, $\Gamma$ is row-finite if $\Gamma^0$ is, and the digraph $\Gamma'$ which is obtained from $\Gamma$ by removing the arrow $e_0$ is transient and $h(\Gamma') = h$.
\end{cor}
\begin{proof} This follows from Lemma \ref{02-03-18}.
 \end{proof}

The conditions d) in Lemma \ref{02-03-18} and c) in Corollary \ref{attaching3} are quite restrictive, and since they are only needed to arrange that the graph $\Gamma'$ is transient it is worthwhile to note the following versions which are easier to prove and can be used much more frequently.

\begin{lemma}\label{25-06-18} Let $\Gamma^0$ be a countable directed graph and $v_0 \in \Gamma^0_V$ a vertex such that 
\begin{itemize}
\item[a)] $\limsup_n \left(A(\Gamma^0)^n_{v_0,v}\right)^{\frac{1}{n}} \  < \ \infty$ for all $v \in \Gamma_V$, 
\item[b)]  no arrow in $\Gamma^0$ lands in $v_0$, i.e. $r^{-1}(v_0) = \emptyset$, and
\item[c)] all vertexes in $\Gamma^0_V$ can be reached from $v_0$. 
\end{itemize} 
For every positive number $h  > 0$ such that
\begin{equation*}
\limsup_n \left(A(\Gamma^0)^n_{v_0,v}\right)^{\frac{1}{n}} \  < \  e^h 
\end{equation*}
for all $v \in \Gamma_V$,
there is a digraph $\Gamma$ obtained from $\Gamma^0$ by adding return paths to $v_0$ such that 
\begin{itemize}
\item $\Gamma$ is strongly connected,
\item $\Gamma$ is recurrent and
\item $h(\Gamma) = h$.
\end{itemize}
Furthermore, $\Gamma$ is row-finite if $\Gamma^0$ is.
\end{lemma}
\begin{proof} As in the proof of Lemma \ref{02-03-18} we label the vertexes of $\Gamma^0 \backslash \{v_0\}$ by $v_1,v_2, \cdots$, and consider the numbers $\alpha_i$ in \eqref{25-06-18a}. We choose natural numbers $b_i , m_i \in \mathbb N$ such that
$$
\sum_{i=1}^k 2^{-i} < \sum_{i=1}^k \alpha_i b_ie^{-m_i h} < \sum_{i=1}^{k+1} 2^{-i} 
$$
for $k = 1,2,3, \cdots$. For each $i$ we add $b_i$ return paths of length $m_i$ from $v_i$ to $v_0$ and we let $\Gamma$ denote the resulting graph. Then
\begin{equation*}
\begin{split}
& \sum_{n=1}^{\infty} l^n_{v_0,v_0}\left(\Gamma\right) e^{-nh} = \sum_{i=1}^{\infty} \sum_{n=1}^{\infty}A(\Gamma^0)^n_{v_0,{v_i}} b_ie^{-(n+m_i)h} \\
 & =  \sum_{i=1}^{\infty} \alpha_i b_ie^{-m_i h} = 1 \ .
 \end{split}
 \end{equation*}
It follows from Lemma 4.15 in \cite{Th3} that $\Gamma$ is recurrent and that $h(\Gamma) = h$. 
\end{proof}

\begin{cor}\label{25-06-18d}
Let $\Gamma^0$ be a countable directed graph and $v_0 \in \Gamma^0_V$ a vertex such that
\begin{itemize}
\item[a)] there is at least one, but only finitely many paths in $\Gamma^0$ from $v_0$ to any other vertex; in symbols,
$$1 \ \leq \ \# \left\{ \mu \in P_f\left(\Gamma^0\right) : \ s(\mu) = v_0, \ r( \mu) = v \right\} \ < \ \infty
$$ 
for all $v \in \Gamma^0_V \backslash \{v_0\}$,
\item[b)]  $v_0$ receives no arrows in $\Gamma^0$, i.e. $r^{-1}(v_0) = \emptyset$. 
\end{itemize} 
For every positive number $h  > 0$ there is a digraph $\Gamma$ obtained from $\Gamma^0$ by adding return paths to $v_0$ such that 
\begin{itemize}
\item $\Gamma$ is strongly connected,
\item $\Gamma$ is recurrent and
\item $h(\Gamma) = h$. 
\end{itemize}
Furthermore, $\Gamma$ is row-finite if $\Gamma^0$ is.
\end{cor}

Let $\Gamma^0$ be a digraph and $v_0 \in \Gamma^0_V$ a vertex such that the conditions a)-c) in Lemma \ref{25-06-18} hold. If $\Gamma$ is a graph obtained from $\Gamma^0$ by adding return paths to $v_0$, and we then subsequently turn $v_0$ into a source in $\Gamma$ by the procedure introduced in Section \ref{source}, the resulting graph is $\Gamma^0$ with its dead ends removed, if any. Therefore the results of the previous section allows one to obtain complete information about the harmonic vectors and conformal measures of $\Gamma$ from information about $\Gamma^0$.

\begin{prop}\label{07-02-18k}
In the setting of Lemma \ref{02-03-18}, Corollary \ref{attaching3}, Lemma \ref{25-06-18} or Corollary \ref{25-06-18d}, let $F : \Gamma^0_{Ar} \to \mathbb R$ be a potential and $F : \Gamma_{Ar} \to \mathbb R$ an extension of $F$. Set 
$$
A(\beta)_{v,w} \ = \sum_{a \in r^{-1}(w)\cap s^{-1}(v) \cap \Gamma_{Ar}} e^{-\beta F(a)}
$$
and assume that $A(\beta)$ is transient, i.e. that $\sum_{n=0}^{\infty} A(\beta)^n_{v_0,v_0} < \infty$. There are affine homeomorphisms 
 $$
 H^{v_0}_{\beta F}(\Gamma^0) \simeq H^{v_0}_{\beta F}(\Gamma) \simeq H^{v_0}_{\beta F}(\Gamma') \ .
 $$ 
 The corresponding affine homeomorphisms $M^{v_0}_{\beta F}(\Gamma) \to M^{v_0}_{\beta F}(\Gamma^0)$ and  $M^{v_0}_{\beta F}(\Gamma') \to M^{v_0}_{\beta F}(\Gamma^0)$ send $m \in M^{v_0}_{\beta F}(\Gamma)$ to the measure on $P(\Gamma^0)$ given by
$$
B \mapsto ( 1 - R_{\beta F}^{\Gamma}(v_0,v_0))^{-1}m\left(B\right) 
$$
and $m'\in M^{v_0}_{\beta F}(\Gamma')$ to the measure on $P(\Gamma^0)$ given by
$$
B \mapsto ( 1 - R_{\beta F}^{\Gamma'}(v_0,v_0))^{-1}m\left(B \right) .
$$
\end{prop}

\begin{proof} As explained above this follows from Proposition \ref{07-02-18d}.
\end{proof}

Lemma \ref{02-03-18}, Corollary \ref{attaching3} and Proposition \ref{07-02-18k} will be our main tools for the construction of strongly connected graphs with various properties. Here we first use it to present the examples mentioned in Remark \ref{01-10-17k}.

\begin{example}\label{18-02-18} Consider the following digraph.

\begin{equation*}\label{02-10-17aX}
\begin{xymatrix}{  &&&\Gamma^0 &&&&&&&&& &\\
 &  v_1^- \ar[ddr]_(.5){d_1} \ar[r]   & v_2^- \ar[r]  & v_3^-  \ar[ddr]_(.5){d_2}  \ar[r]  & v_4^- \ar[r]   & v_5^-  \ar[ddr]_(.5){d_3} \ar[r]  & v^-_6 \ar[r] & \hdots  &  &     &  & &  &\\
 &&&&&&&&&&&& &\\
 v_0 \ar[ruu]  \ar[rdd]  \ar[rr]  & & v_1 \ar[rr]&   &v_2 \ar[rr]&     & v_3\ar[rr]&   & \hdots  &  &   &  &  &  \\
     &&&&&&&&&&&& & \\ 
 & v^+_1 \ar[uur]^(.5){d_1} \ar[r]   & v^+_2 \ar[r]  & v^+_3  \ar[uur]^(.5){d_2}  \ar[r]  & v^+_4 \ar[r]   &  v^5_+  \ar[uur]^(.5){d_3} \ar[r]  & v^+_6 \ar[r] & \hdots   &   & &   &  & \\ 
 }
\end{xymatrix}   
\end{equation*}

\bigskip

The labels $d_i$ on some of the arrows are natural numbers indicating the multiplicity of the arrow; e.g. 
\begin{equation*}\label{19-9-17aX}
\begin{xymatrix}{
 v^+_5 \ar[rr]^(.5){d_3}  & & v_3 &  \\
 }
 \end{xymatrix}   
\end{equation*}
means $d_3$ arrows from $v^+_5$ to $v_3$. Arrows without labels have multiplicity one. The digraph $\Gamma^0$ has the properties a), b) and c) required in Corollary \ref{attaching3} and we obtain therefore, for any $h > 0$, two strongly connected digraphs $\Gamma$ and $\Gamma'$ with the properties specified in Corollary \ref{attaching3}. 
The potential function $F$ we consider is the constant function $F =1$, corresponding to the gauge actions on $C^*(\Gamma)$ and $C^*(\Gamma')$. Note that $\Gamma$ and $\Gamma'$ have exactly three exits in the sense of \cite{Th3}, represented by the rays
$$
p^- : v_1^- \to v_2^- \to v_3^- \to v_4^-  \to \cdots,
$$
$$
p^0: \ v_1 \to v_2 \to v_3 \to v_4 \to \cdots, 
$$
and
$$
p^+: \  v_1^+ \to v_2^+ \to v_3^+\to v_4^+ \to \cdots, 
$$
respectively. The graphs $\Gamma$ and $\Gamma'$ are covered by the results obtained in \cite{Th3}; in particular by Theorem 5.7 in \cite{Th3}. The two exit paths $p^{\pm}$ are both bare and hence $\beta$-summable by Lemma 7.3 in \cite{Th3} for all $\beta > h$ in $\Gamma$ and for $\beta \geq h$ in $\Gamma'$. Consider such a $\beta$. The exit path $p^0$ is $\beta$-summable if and only the sum
$$
S = \sum_{i=1}^{\infty} d_i e^{-i \beta}
$$
is finite. By using Theorem 5.7 in \cite{Th3} and the relation between KMS-weights and conformal measures established above it follows that $\partial M^{v_0}_{\beta}(\Gamma)$ and $\partial M^{v_0}_{\beta}(\Gamma')$ both contain three elements when $S < \infty$, and only two when $S = \infty$. It is straightforward, but maybe tedious to verify that the limits 
$$
\psi_v = \lim_{k \to \infty} K_{\beta}(v, v_k),
$$
$$
\psi^{\pm}_v = \lim_{k \to \infty} K_{\beta}\left(v,v^{\pm}_k\right)
$$
exist and that
\begin{equation}\label{07-02-18l}
\psi_v = \frac{1}{2}\psi^+_v + \frac{1}{2}\psi^-_v
\end{equation}
 for all $v \in \Gamma_V = \Gamma'_V$ when $S = \infty$. Hence the $e^{\beta}$-conformal measure $m_{\psi}$ satisfies that
$$
m_{\psi}(Z(v)) = \lim_{k \to \infty} K_{\beta}(v, s(p^0_k))
$$
for all $v \in \Gamma_V$, but as \eqref{07-02-18l} shows it is not extremal when $S = \infty$. In fact, $m_{\psi}$ is then concentrated on the union
\begin{equation*}
\begin{split}
&\left\{ p \in P(\Gamma): \ \lim_{k \to \infty} K_{\beta}(v,s(p_k)) = \psi^+_v \ \forall v \in \Gamma_V \right\} \\
& \ \ \ \ \ \ \ \ \ \ \ \cup \left\{ p \in P(\Gamma): \ \lim_{k \to \infty} K_{\beta}(v,s(p_k)) = \psi^-_v \ \forall v \in \Gamma_V \right\}
\end{split}
\end{equation*}
which is disjoint from $\left\{ p \in P(\Gamma): \ \lim_{k \to \infty} K_{\beta}(v,s(p_k)) = \psi_v \ \forall v \in \Gamma_V \right\}$.


\end{example}

\section{The end space} \label{endspace}

In this section we retain the general setting from Section \ref{sec2}, where $\Gamma$ is an arbitrary countable directed graph. Let $\mu =(a_i)_{i=1}^{|\mu|}$ be a finite path of positive length in $\Gamma$ and $M \subseteq \Gamma_V$ a set of vertexes. We write $\mu \cap M$ for the set of vertexes from $M$ occurring in $\mu$, i.e.
$$
\mu \cap M = M \cap \left(\bigcup_{i=1}^{|\mu|} \{s(a_i), r(a_i)\}\right) \ .
$$
Given two wandering paths $p,q\in \Wan(\Gamma)$, we write\label{ptoq} 
$$
p \to q
$$ 
when the following holds: For any finite subset $F \subseteq \Gamma_V$ and $N \in \mathbb N$ there are $n_1 \geq N, \ n_2 \geq N$ and a finite path $\mu \in P_f(\Gamma)$ such that $s(\mu) =s(p_{n_1})$, $r(\mu) = s(q_{n_2})$ and $\mu \cap F  = \emptyset$. We write $p \sim q$\label{psimq} when $p \to q$ and $q \to p$. Then $\sim$ is an equivalence relation in $\Wan(\Gamma)$. The collection
$$
\mathcal E(\Gamma) = \Wan(\Gamma) /\sim
$$
of $\sim$-equivalence classes constitute what we shall call \emph{the end space}. 

\begin{remark}\label{onrays} The end space $\mathcal E(\Gamma)$ can also be defined using rays only. In fact, the existence of the retraction \eqref{21-02-18b} shows that\label{Egamma}
$$
\mathcal E(\Gamma) = \Ray(\Gamma) /\sim \ .
$$

\end{remark}

To give $\mathcal E(\Gamma)$ a topology, consider a finite non-empty subset $F \subseteq \Gamma_V$, an element $v \in F$ and a vertex $w \in \Gamma_V$. When $ w \notin F$, let $L_F(v,w)$\label{LFvw} denote the set of finite paths $\mu = (a_i)_{i=1}^{|\mu|}$ in $\Gamma$ such that $s(\mu) = v$, $r(\mu) =w$ and $r(a_i) \notin F, \ i \geq 1$. We write $v \overset{F}{\to} w$ when $L_F(v,w) \neq \emptyset$. Given a wandering path $p \in \Wan(\Gamma)$ we write $v \overset{F}{\to} p$ when the following holds: For all $N \in \mathbb N$ there is a $n \geq N$ such that $L_F(v,s(p_n)) \neq \emptyset$, i.e. $v \overset{F}{\to} s(p_n)$. We set\label{[F]p}
$$
[F]_p = \left\{v \in F: \ v \overset{F}{\to} p \right\} \  .
$$
Of course, $[F]_p$ may be empty. When $ I$ and $F\neq \emptyset$ are finite sets of vertexes with $I \subseteq F$, we set\label{UFI}
\begin{equation*}
U_{F;I}   = \left\{ p \in \Wan(\Gamma): \ [F]_p = I \right\} \ .
\end{equation*}
For $p \in \Wan(\Gamma)$ we write $\mathcal E(p)$\label{Ep} for the element in $\mathcal E(\Gamma)$ which contains $p$. We aim to prove the following 

\begin{prop}\label{Endspace2}  The subsets 
\begin{equation}\label{aug17}
\left\{ \mathcal E(p)  : \ p \in U_{F;I} \right\},
\end{equation}
where $I \subseteq F$ range over all finite subsets of $\Gamma_V$ with $F \neq \emptyset$, form a base for a topology of the end space $\mathcal E(\Gamma)$ which is second countable, totally disconnected and metrizable. 
\end{prop}

Totally disconnected means that for any pair of distinct elements there is an open and closed subset which contains one and not the other. We will prove Proposition \ref{Endspace2} by realizing $\mathcal E(\Gamma)$ as a subset of a totally disconnected  compact metric space.

\begin{lemma}\label{nov3} Let $p,q\in \Wan(\Gamma)$ and let $F,F'$ be two finite non-empty subsets of $\Gamma_V$ such that $F' \subseteq F$. Then
$$
[F]_p = [F]_q \ \Rightarrow \ [F']_p = [F']_q \ .
$$
\end{lemma}
\begin{proof} Assume $[F]_p = [F]_q$. By symmetry it suffices to show that $[F']_p \subseteq [F']_q$. We may therefore assume that $[F']_p \neq \emptyset$. Consider an element $v\in [F']_p$ and let $N \in \mathbb N$. It follows that there is an $i \in \mathbb N$ such that $s(p_j) \notin F$ when $j \geq i$ and a finite path $\mu \in P_f(\Gamma)$ such that $s(\mu) = v$, $r(\mu) = s(p_i)$ and $s(\mu)$ is the only vertex in $\mu$ from $F'$. Then $\mu$ is a concatenation $\mu = \mu' \mu''$ where $s(\mu'') \in F$ is the only vertex in $\mu''$ which is in $F$. Thus $s(\mu'') \in [F]_p$ and since $[F]_p = [F]_q$ by assumption there is a $n \in \mathbb N, \ n\geq N $ such that $s(q_n) \notin F$, and a finite path $\nu \in P_f(\Gamma)$ such that $s(\nu) = s( \mu'')$, $r(\nu) = s(q_n)$ and $s(\nu)$ is the only vertex in $\nu$ from $F$. The concatenation $\mu'\nu$ is then a finite path in $\Gamma$ demonstrating that $v \overset{F'}{\to} q$; i.e. $v \in [F']_q$.

\end{proof}

 Set\label{XF}
 $$
 X_F = \left\{ [F]_p : \ p \in \Wan(\Gamma) \right\} \ ;
 $$
 a collection of subsets of $F$. We consider $X_F$ as a finite compact Hausdorff space in the discrete topology. It follows from Lemma \ref{nov3} that when $F' \subseteq F$ there is a map 
 $$
\pi_{F,F'} : X_F \to X_{F'}
$$
defined such that
$$
\pi_{F,F'}\left([F]_p\right) = [F']_p \ .
$$
Let $\mathcal F$\label{matcF} denote the collection of finite non-empty subsets of $\Gamma_V$. Then $\mathcal F$ is a countable set which is directed by inclusion and we consider the inverse limit 
\begin{equation}\label{Xgamma}
X_{\Gamma} = \varprojlim_{F\in \mathcal F} X_F 
\end{equation}
with $\pi_{F,F'}$ as bonding maps. In more detail $X_{\Gamma}$ is the subset of the product space 
$$
\prod_{F \in \mathcal F} X_F
$$
consisting of the elements $(x_F)_{F \in \mathcal F}$ with the property that $\pi_{F,F'}(x_F) = x_{F'}$ when $F' \subseteq F$. Let $\pi_F : X_{\Gamma} \to X_F$ be the canonical projection. We consider $X_{\Gamma}$ as a topological space in the projective (or initial) topology. Recall that this is the topology with a base consisting of the sets $\pi_F^{-1}(A)$, where $A$ ranges over all subsets of $X_F$ and $F$ over all elements of $ \mathcal F$. It is well-known that since $\mathcal F$ is countable and each $X_F$ is a finite set, $X_{\Gamma}$ is a compact totally disconnected metrizable space in this topology.

We define a map $\chi' : \Wan(\Gamma) \to X_{\Gamma}$ such that
\begin{equation}\label{chi-map0}
\pi_F\left(\chi'(p)\right) = [F]_p 
\end{equation}
when $F \in \mathcal F$.

\begin{lemma}\label{aug14(3)} Let $p,q \in \Wan(\Gamma)$. Then $p \sim q \ \Leftrightarrow \ \chi'(p) = \chi'(q)$.
\end{lemma}
\begin{proof}
Assume $\chi'(p) = \chi'(q)$ and let $N \in \mathbb N$ and a finite subset $F \subseteq \Gamma_V$ be given. Since $p$ is wandering there is an $n \geq N$ such that $s(p_i) \notin F \ \forall i \geq n$. Set $F' = F \cup s(p_n)$. There is a $k\geq n$ such that $s(p_{k}) = s(p_n)$ and $s(p_i) \notin F', \ i > k$. It follows that $s(p_k)  \in [F']_p$. Since $[F']_p= \pi_{F'}\left(\chi'(p)\right) = \pi_{F'}\left(\chi'(q)\right) =  [F']_q$, it follows that there is an $l \geq N$ such that $L_{F'}\left(s(p_k),s(q_l)\right) \neq \emptyset$. Any element $\mu \in L_{F'}\left(s(p_k),s(q_l)\right)$ will satisfy that $\mu \cap F = \emptyset$, showing that $p \to q$. It follows by symmetry that $p \sim q$.

Conversely, assume that $p \sim q$. Let $F\subseteq \Gamma_V$ be a finite subset and consider an element $v \in [F]_p$. Let $N \in \mathbb N$ be given. Choose $n \in \mathbb N$ such that $\{s(p_i),s(q_i)\} \cap F = \emptyset \ \forall  i\geq n \geq N$. Since $v \in [F]_p$ there is a path $\mu \in L_F\left(v,s(p_{j_1})\right)$ for some $j_1 \geq n$. Since $p \sim q$ there are $k_1,k_2 \geq \max j_1+1$ and a finite path $\nu$ such that $\nu \cap F = \emptyset$, $s(\nu) = s(p_{k_1}), \ r(\nu) = s(q_{k_2})$. The concatenation 
$$
\nu ' = \mu p_{j_1}p_{j_1+1}p_{j_1+2} \cdots p_{k_1-1}\nu
$$
 is an element of $L_F(v,s(q_{k_2}))$. It follows that $v \in [F]_q$, showing that $[F]_p \subseteq [F]_q$. By symmetry $[F]_p = [F]_q$. Since $F$ is arbitrary, it follows that $\chi'(p) = \chi'(q)$.
\end{proof}

\emph{Proof of Proposition \ref{Endspace2}:} By Lemma \ref{aug14(3)} the map defined by (\ref{chi-map0}) falls to an injective map
\begin{equation}\label{chi-map}
\chi : \mathcal E(\Gamma) \to X_{\Gamma}
\end{equation}
such that 
\begin{equation}\label{20-11-17a}
\begin{xymatrix}{
 \Wan(\Gamma) \ar[r]^-{\chi'} \ar[d]_-{\mathcal E} & X_{\Gamma} \\
\mathcal E(\Gamma) \ar[ur]_-{\chi} & }
\end{xymatrix}   
\end{equation}
commutes. Since $X_{\Gamma}$ is a totally disconnected compact metric space, the relative topology of $\chi\left(\mathcal E(\Gamma)\right)$ inherited from $X_{\Gamma}$ defines via $\chi$ a second countable metrizable totally disconnected topology on $\mathcal E(\Gamma)$. What remains is therefore only to show that the sets (\ref{aug17}) is a base for this topology. Consider therefore finite sets $I \subseteq F$ in $\Gamma_V$, $F \neq \emptyset$.
The assertion follows then from the definition of the projective limit topology and the observation that
$$
\chi\left(\left\{ \mathcal E(p) : \ p \in U_{F;I} \right\}\right) = \chi\left(\mathcal E(\Gamma)\right) \cap \left\{ x \in X_{\Gamma}: \ \pi_F(x) = I \right\} \ .
$$
\qed

\begin{remark}
When the digraph is row-finite the end space will often be compact. By Theorem \ref{28-09-17=} this is for example the case for undirected graphs, when they are considered as digraphs in the appropriate way. But the end space of a strongly connected row-finite digraph is generally not compact. To show this by example, consider the digraph (\ref{4-5-17}). Its end space is homeomorphic to $\mathbb N$ equipped with the discrete topology.

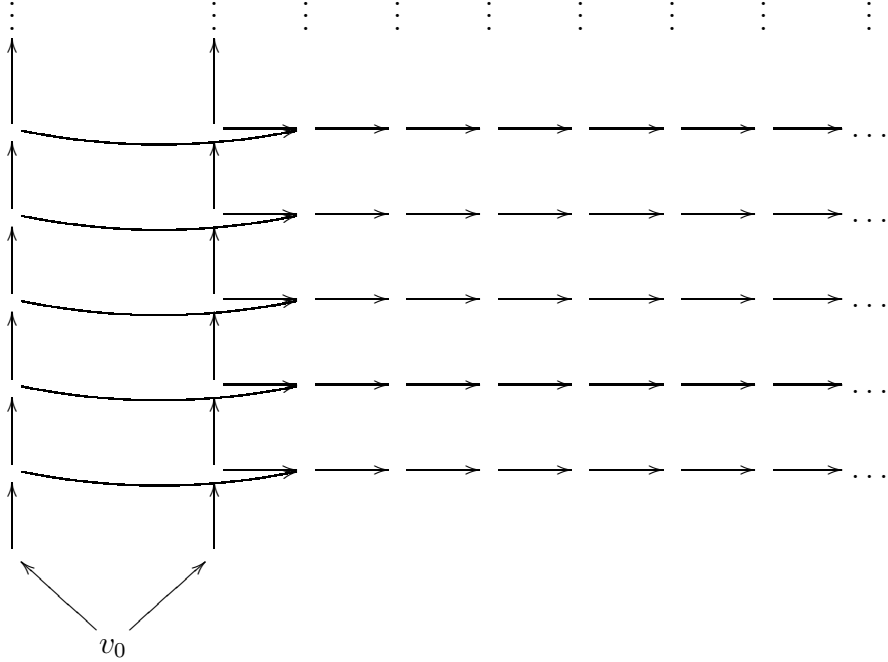
\begin{equation}\label{4-5-17}
\begin{xymatrix}{
\vdots &  &\vdots &\vdots&\vdots&\vdots&\vdots&\vdots &\vdots&\vdots\\
\ar[u] \ar@/_/[rrr] &   & \ar[u] \ar[r]  &  \ar[r]  & \ar[r] & \ar[r] & \ar[r]  & \ar[r]  & \ar[r] &  \hdots\\ 
\ar[u] \ar@/_/[rrr] &   & \ar[u] \ar[r]  &  \ar[r]  & \ar[r] & \ar[r] & \ar[r]  & \ar[r]  & \ar[r] &  \hdots\\ 
\ar[u] \ar@/_/[rrr] &   & \ar[u] \ar[r]  &  \ar[r]  & \ar[r] &  \ar[r] & \ar[r]  & \ar[r]  & \ar[r] & \hdots\\ 
\ar[u] \ar@/_/[rrr] &   & \ar[u] \ar[r] &\ar[r]  &  \ar[r] & \ar[r]  & \ar[r]  & \ar[r]  & \ar[r] & \hdots\\ 
\ar[u] \ar@/_/[rrr] &   & \ar[u] \ar[r]  &   \ar[r]  & \ar[r] & \ar[r] & \ar[r]  & \ar[r] & \ar[r] &  \hdots\\  
\ar[u] &  & \ar[u]  &&&&&&&& \\
 & \ar[ur] v_0 \ar[ul]   &      &     &   &     &   &   &   &   }
\end{xymatrix}   
\end{equation}

\smallskip

It is not difficult to modify the digraph above, for example by adding return paths as described in Section \ref{adding}, to obtain a strongly connected digraph, still row-finite, without altering the end space. 

\end{remark}

Although not compact, the topology we have introduced on the end space is still nice enough to act as target space for disintegration of Borel measures on $P(\Gamma)$, cf. Theorem \ref{disTHM2}.

\begin{lemma}\label{aug13xa} The maps $\chi' : \Wan(\Gamma) \to X_{\Gamma}$ and $\mathcal E: \Wan(\Gamma) \to \mathcal E(\Gamma)$ are both Borel maps.
\end{lemma}
\begin{proof} It follows from the definition of the topology of $\mathcal E( \Gamma)$ that a subset $W \subseteq \mathcal E(\Gamma)$ is open iff $W = \chi^{-1}(U)$ for some open subset $U \subseteq X_{\Gamma}$. It follows from the commutative diagram \eqref{20-11-17a} that $\mathcal E^{-1}(W) = {\chi'}^{-1}(U)$ and it is enough for us to show that ${\chi'}^{-1}(U)$ is a Borel subset of $\Wan(\Gamma)$ when $U \subseteq X_{\Gamma}$ is open. By definition of $\chi'$ and the topology of $X_{\Gamma}$ it suffices for this to show that $U_{F;I}$ is a Borel subset of $\Wan(\Gamma)$ when $I \subseteq F$ are finite sets in $\Gamma_V$ and $F \neq \emptyset$. For $v \in F$ and $m \in \mathbb N$, let $A(v,m)$ be the set of infinite paths $p \in \Wan(\Gamma)$ with the property that $L_F(v,s(p_m)) \neq \emptyset$. Then $A(v,m)$ is open in $\Wan(\Gamma)$ and 
$$
A(v) = \bigcap_{k \in \mathbb N}  \bigcup_{m \geq k} A(v,m)
$$
is a Borel set. Since
$$
U_{F;I} = \left\{p \in \Wan(\Gamma) : \ [F]_p = I \right\}  = \bigcap_{v \in I} A(v) \backslash \left( \bigcup_{w \in F \backslash I} A(w) \right) 
$$
when $I \neq \emptyset$, and $U_{F;\emptyset} = \Wan(\Gamma)\backslash  \left( \bigcup_{w \in F} A(w) \right)$, it follows that $U_{F;I}$   
is Borel.
\end{proof}

\begin{remark}\label{02-03-18g}
The fact will not be needed, but it may be worthwhile to point out that the end space $\mathcal E(\Gamma)$ with its Borel $\sigma$-algebra is isomorphic to a Borel subset of $X_{\Gamma}$, and hence is a standard Borel space. See Remark \ref{01-03-18a}.
\end{remark}

\begin{remark}\label{02-03-18h} The constructions introduced in Section \ref{tools} preserve end spaces. To see this, recall that two elements $p,q \in P(\Gamma)$ are \emph{tail equivalent} when there are $n,m \in \mathbb N$ such that $s\left(p_{n+i}\right) = s\left(q_{m+i}\right)$ for all $i \in \mathbb N$. If $\Gamma$ is obtained from $\Gamma^0$ by turning a vertex into a source or by adding return paths, then $\Wan(\Gamma) \subseteq \Wan(\Gamma^0)$ or $\Wan(\Gamma^0) \subseteq \Wan(\Gamma)$, and in both cases,
$$
\Wan(\Gamma^0)/\text{Tail} \ = \ \Wan(\Gamma)/\text{Tail} \ .
$$
The bijection $\mathcal E(\Gamma^0) \simeq \mathcal E(\Gamma)$ arising from this observation is a homeomorphism. 
\end{remark}

\subsection{The end space of a row-finite almost undirected graph}\label{7.2}

The notion of an endspace is wellknown for undirected graphs; as a set it was introduced by Halin in \cite{Ha}, and it was given a topology by Jung in \cite{Ju}. In the present section we compare the endspace of a digraph we have introduced to the endspace of Halin and Jung. 

Let $\Gamma$ be an infinite row-finite directed graph. Define a relation $\smile$ on $\Gamma_V$ such that $v \smile w$ means that there is an arrow $a \in \Gamma_{Ar}$ with $s(a) = v$ and $r(a) = w$ \emph{or} an arrow $b$ with $s(b) = w$ and $r(b) = v$. The equivalence relation on $\Gamma_V$ generated by $\smile$ will be denoted by $\asymp$. We denote by $\mathcal C(\Gamma)$\label{CGamma} the equivalence classes of $\asymp$;\label{asymp} in symbols
$$
\mathcal C(\Gamma) = \Gamma_V/\asymp \ .
$$ 
Then $\mathcal C(\Gamma)$ is the set of connected components of $\Gamma$ when the digraph is considered as a one-dimensional topological space with the arrows replaced by undirected lines.

For any subset $H \subseteq \Gamma_V$, set\label{int}
$$
\Int H = \left\{ v\in H: \ r\left(s^{-1}(v)\right) \subseteq H \right\} \ ;
$$
the 'interior' of $H$. In the following we write $F \sqsubset H$ to mean that $F \subseteq \Int H$. 

Now assume that there is a vertex $v_0 \in \Gamma_V$ from where all other vertexes in $\Gamma$ can be reached. Since $\Gamma$ is row-finite we can choose a sequence 
\begin{equation}\label{10-11-17}
\{v_0\}  \sqsubset  D_1 \sqsubset D_2 \sqsubset D_3 \sqsubset \cdots
\end{equation}
 of finite subsets of $\Gamma_V$ such that $\bigcup_n D_n = \Gamma_V$. For each $n$ we let $\Gamma^{(n)}$ be the digraph which is obtained from $\Gamma$ by removing the vertexes in $D_n$ and the arrows going in or out of vertexes in $D_n$. Thus
$$
\Gamma^{(n)}_V = \Gamma_V \backslash D_n
$$
and
$$
\Gamma^{(n)}_{Ar} = \Gamma_{Ar} \backslash \left(s^{-1}(D_n) \cup r^{-1}(D_n) \right) \ .
$$
Set $\Gamma^{(n)} = \Gamma$ when $n< 0$. For each $n$ and $v \in \Gamma^{(n)}_V$ we denote the $\asymp$-class of $v$ in $\mathcal C(\Gamma^{(n)})$ by $[v]_n$. There is a map $\iota_{n} : \mathcal C\left(\Gamma^{(n+1)}\right) \to \mathcal C\left(\Gamma^{(n)}\right)$ defined such that $\iota_{n}\left([v]_{n+1}\right) = [v]_n$ for all $v \in \Gamma^{(n+1)}_V$. We denote by $\mathcal C^{\infty}(\Gamma)$\label{CinfG} the inverse limit set of the sequence

\begin{equation}\label{28-09-17f}
\begin{xymatrix}{
{\mathcal C\left(\Gamma^{(1)}\right)} & \ar[l]_-{\iota_1} {\mathcal C\left(\Gamma^{(2)}\right)}  & \ar[l]_-{\iota_2} {\mathcal C\left(\Gamma^{(3)}\right)}  & \ar[l]_-{\iota_3} {\mathcal C\left(\Gamma^{(4)}\right)}  & \ar[l]_-{\iota_4}  &  \hdots &  \\ 
}
\end{xymatrix}
\end{equation}
Since every vertex in $\Gamma$ can be reached from $v_0$ it follows that
$$
\mathcal C(\Gamma^{(n)}) = \left\{ [v]_n : \ v \in r(D_n)\cap \Gamma^{(n)}_V \right\} \ ,
$$
showing, in particular, that $\mathcal C(\Gamma^{(n)})$ is a finite set for all $n$. Hence $\mathcal C^{\infty}(\Gamma)$ is a compact totally disconnected metric space in the topology inherited from the infinite product space
$$
\prod_{n=1}^{\infty} \mathcal C\left(\Gamma^{(n)}\right) \ .
$$
We remark that $\mathcal C^{\infty}(\Gamma)$ is the end space of the undirected graph obtained from $\Gamma$ by forgetting the orientation of the arrows in $\Gamma$, see e.g. \cite{Wo1}.

Let $p \in \Wan(\Gamma)$. Choose $n \in \mathbb N$ such that $s(p) \in D_n$. There is then a sequence $i_n < i_{n+1} < i_{n+2} < \cdots $ in $\mathbb N$ such that 
$s(p_{i_j}) \in D_j$ and $s(p_k) \notin D_j, \ k > i_j$. 
Note that $r(p_{i_j}) \in \Gamma^{(j)}_V$ for all $j \geq n$ and that the sequence $\left[r\left(p_{i_n}\right)\right]_{n} \in \mathcal C\left(\Gamma^{(n)}\right)$,  
defines an element 
$$
\rho(p) \in \mathcal C^{\infty}(\Gamma) \ .
$$
It is easy to see that if two wandering paths $p$ and $q$ define the same element in $\mathcal E(\Gamma)$, i.e. if $p \sim q$ in $\Wan(\Gamma)$, then $\rho(p) = \rho(q)$. We get therefore a map
\begin{equation}\label{28-09-17b}
\rho: \ \mathcal E(\Gamma) \to \mathcal C^{\infty}(\Gamma)
\end{equation}
defined such that 
$\rho\left(\mathcal E(p)\right) = \rho(p)$. In general, the map \eqref{28-09-17b}  
is neither injective nor surjective. See Example \ref{28-09-17}.

\begin{example}\label{28-09-17} Consider the following  digraph $\Gamma_1$:

\begin{xymatrix}{
\Gamma_1 &\ar[r] & \ar[r] \ar[d]& \ar[r] \ar[d] & \ar[r]  \ar[d]& \ar[r]  \ar[d] & \cdots\\
&v_0 \ar[u] \ar[d] & \ar[l] & \ar[l]  & \ar[l]  & \ar[l]  &\ar[l] \cdots \\ 
& \ar[r] & \ar[r] \ar[u] & \ar[r]  \ar[u] & \ar[r]  \ar[u] & \ar[r]  \ar[u] & \cdots  }
\end{xymatrix} 

\smallskip
$\Gamma_1$ is a strongly connected digraph with two ends while the corresponding unoriented graph has one; in particular, the map \eqref{28-09-17b} is not injective when $\Gamma = \Gamma_1$. To show that \eqref{28-09-17b} is also not surjective in general, consider the following digraph $\Gamma_2$:

\bigskip

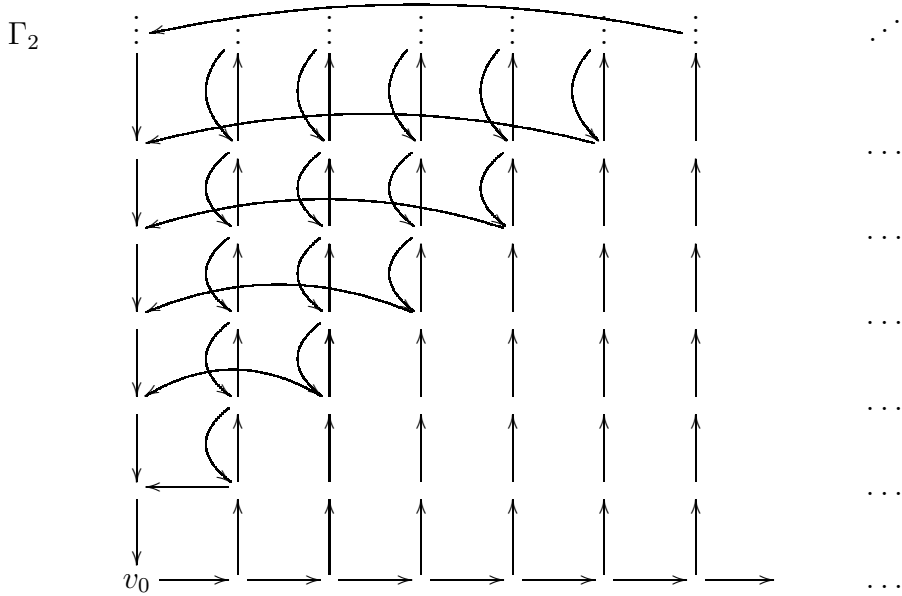
\begin{xymatrix}{
\Gamma_2 &\ar[d] \vdots & \vdots \ar@/_1pc/[d] &\vdots\ar@/_1pc/[d] &\vdots \ar@/_1pc/[d]&\vdots\ar@/_1pc/[d] &\vdots \ar@/_1pc/[d]&\vdots \ar@/_1pc/[llllll]  & &  \iddots \\
&\ar[d]   &\ar[u]  \ar@/_1pc/[d]  &\ar[u] \ar@/_1pc/[d]  &\ar[u] \ar@/_1pc/[d]  &\ar[u]  \ar@/_1pc/[d] &\ar[u] \ar@/_1pc/[lllll] &\ar[u] &  & \hdots \\ 
&\ar[d]   &\ar[u]  \ar@/_1pc/[d]  &\ar[u] \ar@/_1pc/[d]  &\ar[u] \ar@/_1pc/[d]  &\ar[u]  \ar@/_1pc/[llll] &\ar[u]   &\ar[u]   & & \hdots\\
&\ar[d]   &\ar[u]  \ar@/_1pc/[d] &\ar[u] \ar@/_1pc/[d]  &\ar[u] \ar@/_1pc/[lll]  &\ar[u]   &\ar[u]   &\ar[u]  &  &  \hdots\\
&\ar[d]  &\ar[u] \ar@/_1pc/[d]   &\ar[u] \ar[u]\ar@/_1pc/[ll] &\ar[u]   &\ar[u]   &\ar[u]  &\ar[u]   &   & \hdots \\
&\ar[d]   &\ar[u] \ar[l]    &\ar[u]   &\ar[u]   &\ar[u]   &\ar[u]   &\ar[u]   &   &  \hdots\\
 &\ar[r] v_0  &\ar[u] \ar[r]  &\ar[u] \ar[r]  &\ar[u] \ar[r]  &\ar[u] \ar[r]  &\ar[u] \ar[r]  &\ar[u] \ar[r]  & &  \hdots\\
 }
\end{xymatrix} 
\bigskip
$\Gamma_2$ is a strongly connected row-finite digraph for which the map \eqref{28-09-17b} is not surjective. Indeed, the incoming ray to the extreme left in the graph represents an element of $\mathcal C^{\infty}(\Gamma)$ which does not come from $\mathcal E(\Gamma)$.
\end{example}

We introduce now a class of directed graphs for which the map \eqref{28-09-17b} is a homeomorphism.

\begin{defn}\label{almost} A directed graph $\Gamma$ is \emph{almost undirected} when there is a natural number $N \in \mathbb N$ such that for all $a \in \Gamma_{Ar}$ there is a finite path $\mu \in P_f(\Gamma)$ of length $ |\mu| \leq N$ such that $s(\mu) = r(a)$ and $r(\mu) = s(a)$. 
\end{defn}

\begin{remark}\label{cayley} The Cayley graph $\Gamma(G,S)$ defined from a finitely generated group $G$ with a finite generating set $S$ (i.e. every element of $G$ is the product of elements from $S$) is always almost undirected.
\end{remark}

\begin{thm}\label{28-09-17=} Let $\Gamma$ be a strongly connected almost undirected row-finite digraph. The map $\rho$ from \eqref{28-09-17b} and the map $\chi$ from \eqref{chi-map} are both homeomorphisms. In particular, $\mathcal E(\Gamma)$ is a compact totally disconnected metric space.
\end{thm}
\begin{proof} We prove first $\rho$ is a bijection. Let $N$ be the number from Definition \ref{almost}. Consider an arrow $a \in \Gamma^{(n)}_{Ar}, \ n > N$. By assumption there is a path $\mu$ in $\Gamma$ of length $|\mu| \leq N$ such that $s(\mu) = r(a)$ and $r(\mu) = s(a)$. Consider a vertex $u$  in $\mu$. The appropriate portion of $\mu$ provides then a path $\nu$ of length $\leq N$ such that $s(\nu) = u$ and $r(\nu)  = s(a)\in \Gamma^{(n)}_V$. Since $D_i \sqsubset D_{i+1}$ for all $i$, it follows that $u \in \Gamma^{(n-N)}_V$. Therefore the following observation holds.
\begin{obs}\label{30-09-17obs}
When $v,w \in \Gamma^{(n)}_V, \ n > N$, and $[v]_n = [w]_n$, there is a path $\mu$ in $\Gamma^{(n-N)}$ such that $s(\mu) = v$ and $r(\mu) = w$. 
\end{obs}
Thus, given an element $x = \left(x_i\right)_{i=1}^{\infty} \in \mathcal C^{\infty}(\Gamma)$ and vertexes $v_i \in \Gamma^{(i)}_V$ such that $[v_i]_i = x_i$, we can construct by concatenation a path $p \in \Wan(\Gamma)$ such that $s(p) = v_0$ and for some sequence $i_{1} < i_{2} < i_{3} < \cdots$ in $\mathbb N$ we have that
$$
s(p_{i_j}) = v_j
$$
and $s(p_i) \in \Gamma_V^{(j-N)}$ for all $i \geq i_j$. It follows that
$$
\rho(p)_{j-N} = \left[s(p_{i_j})\right]_{j-N} = [v_j]_{j-N} = \left[v_{j-N}\right]_{j-N} = x_{j-N}
$$
for all $j \geq N+1$. Thus $\rho(p) =  x$, and we have shown that $\rho$ is surjective. Assume $p,q \in \Wan(\Gamma)$ and that $\rho(p) = \rho(q)$ in $\mathcal C^{\infty}(\Gamma)$. To show that $\mathcal E(p) = \mathcal E(q)$ we may assume that $s(p) = s(q) = v_0$. Consider a finite subset $F \subseteq \Gamma_V$ and a natural number $K$. We can then choose $j \in \mathbb N$ such that $F \subseteq D_{j-N}$. Since $\rho(p)_j = \rho(q)_j$ in $\mathcal C(\Gamma^{(j)})$ there are natural numbers $n,m$ such that $s(p_n) \in D_j, s(q_m) \in D_j, \ [r(p_n)]_j = [r(q_m)]_j$ in $\mathcal C\left(\Gamma^{(j)}\right)$ and $s(p_i) \notin D_j, \ i > n$, $s(q_i) \notin D_j, \ i > m$. Choose $k > \max \{K,n\}$, $l > \max \{K,m\}$. Since $[s(p_k)]_j = [s(q_l)]_j$, it follows from the observation above that there is a path $\mu$ in $\Gamma$ from $s(p_k)$ to $s(q_l)$ such that $\mu \cap F \subseteq \mu \cap D_{j-N} = \emptyset$. This shows that $p \to q$ and it follows by symmetry that $q \to p$, i.e. $\mathcal E(p) = \mathcal E(q)$, and we have shown $\rho$ to be a bijection.

To show that it is also a homeomorphism, note first that the space $X_{\Gamma}$ from \eqref{Xgamma} is the inverse limit of the sequence
\begin{equation}\label{28-09-17c}
\begin{xymatrix}{
X_{D_1} & \ar[l]_{\pi_{D_2,D_1}} X_{D_2}  &\ar[l]_{\pi_{D_3,D_2}} X_{D_3}    &  \ar[l]_{\pi_{D_4,D_3}} X_{D_4} &  \ar[l]_{\pi_{D_5,D_4}} X_{D_5}  & \ar[l]_{\pi_{D_6,D_5}}  \hdots &  \\ 
}
\end{xymatrix} 
\end{equation}
since the sequence $D_1 \subseteq D_2 \subseteq D_3 \subseteq \cdots $ is cofinal in the directed set $\mathcal F$ defining $X_{\Gamma}$. Define a map $s_n : X_{D_n} \to \mathcal C(\Gamma^{(n-1)})$ such that
$$
s_n\left( [D_n]_p\right) = [v]_{n-1} \ ,
$$
for any choice of $v \in [D_n]_p$. The $s_n$-maps are compatible with the bonding maps in \eqref{28-09-17c} and \eqref{28-09-17f}, in the sense that
\begin{equation*}
\begin{xymatrix}{
X_{D_{n-1}}\ar[d]_{s_{n-1}} & \ar[l]_{\pi_{D_{n},D_{n-1}}} X_{D_n}  \ar[d]^{s_n}\\ 
\mathcal C\left(\Gamma^{(n-2)}\right) & \ar[l]_{\iota_{n-2}} \mathcal C\left(\Gamma^{(n-1)}\right)
}
\end{xymatrix} 
\end{equation*}
commutes for all $n > 3$, and they induce therefore a continuous map $s: X_{\Gamma} \to \mathcal C^{\infty}(\Gamma)$. It is straightforward to check that the diagram
\begin{equation*}
\begin{xymatrix}{
\mathcal E(\Gamma) \ar[r]^{\chi} \ar[d]_{\rho} &  X_{\Gamma} \ar[dl]^s \\
\mathcal C^{\infty}(\Gamma) & 
}
\end{xymatrix} 
\end{equation*}
commutes. In particular, $s$ is surjective since $\rho$ is. By definition the topology of $\mathcal E(\Gamma)$ is the topology it gets by considering it as a subset of $X_{\Gamma}$ under the injective map $\chi$. Therefore, to complete the proof it suffices now to show that $s$ is injective. For this note that if $s_n([D_n]_p) = s_n([D_n]_q)$ for some $n > N+1$, an application of Observation \ref{30-09-17obs} shows that $[D_{n-N-1}]_p = [D_{n-N-1}]_q$.  It follows that $s$ is injective.
\end{proof}

\subsection{Disintegration over the end space}

Let $\Gamma$ be a countable directed graph. We return now to the setting in Section \ref{sec3}. In particular, $\Gamma$ is a countable directed graph, $F : \Gamma_{Ar} \to \mathbb R$ is a potential function and there is a vertex $v_0$ such that \eqref{v0trans} holds.

\begin{lemma}\label{02-12-17c}   (Disintegration over the endspace.) Let $m$ be an $e^{\beta F}$-conformal measure on $P(\Gamma)$. Define a Borel probability measure $m^{\mathcal E}$ on $\mathcal E(\Gamma)$ such that
 $$
m^{\mathcal E}(B) = \frac{m \left( Z(v_0) \cap \mathcal E^{-1}(B)\right)}{m(Z(v_0))} \  .
$$
There is an $\left(\mathcal E, m^{\mathcal E}\right)$-disintegration, 
$$
m = \int_{\mathcal E(\Gamma)} m_E \ \mathrm{d} m^{\mathcal E}(E) \ ,
$$
of $m$ such that $m_E$ is an $e^{\beta F}$-conformal measure concentrated on $E$ for $m^{\mathcal E}$-almost every $E \in \mathcal E(\Gamma)$.
\end{lemma}
\begin{proof} We claim that $m \circ \mathcal E^{-1}$ is absolutely continuous with respect to $m^{\mathcal E}$, and to show it we let $B \subseteq \mathcal E(\Gamma)$ be a Borel subset such that $m^{\mathcal E}(B) = 0$. Let $v \in \Gamma_V$ and let $\mu$ be a path in $\Gamma$ such that $s(\mu) = v_0$ and $r(\mu) = v$. Then 
$$
m(\mathcal E^{-1}(B) \cap Z(v)) = e^{\beta F(\mu)} m\left(Z(\mu)\left(\mathcal E^{-1}(B) \cap Z(v)\right)\right)
$$
 by Lemma \ref{nov2} and since $Z(\mu)\left(\mathcal E^{-1}(B) \cap Z(v)\right) \subseteq \mathcal E^{-1}(B) \cap Z(v_0)$, it follows that $m(\mathcal E^{-1}(B) \cap Z(v)) = 0$. Hence 
$$
m\left(\mathcal E^{-1}(B)\right) = \sum_{v \in \Gamma_V} m(\mathcal E^{-1}(B) \cap Z(v))  = 0 \ , 
$$
proving the claim. Add a point $\clubsuit$ to the end space and consider the Borel map $T : P(\Gamma) \to \mathcal E(\Gamma) \sqcup \{\clubsuit\}$ defined such that
$$
T(p) = \begin{cases} \mathcal E(p), & \ p \in \Wan(\Gamma) \\ \clubsuit, & \ p \in P(\Gamma) \backslash \Wan(\Gamma). \end{cases}
$$ 
Since $m \circ T^{-1}(\clubsuit) = 0$ by Lemma \ref{12-11-17x}, the desired disintegration arise from an application of Theorem \ref{disTHM2}. 
\end{proof}

\begin{prop}\label{aug23} Let $m$ be an extremal $e^{\beta F}$-conformal measure on $P(\Gamma)$. It follows that there is an end $E \in \mathcal E(\Gamma)$ such that $m$ is concentrated on $E$, i.e. $m\left( P(\Gamma)\backslash E\right) = 0$.
\end{prop}  
\begin{proof} Since $m$ is ergodic for the shift by Theorem \ref{MAIN2} the conclusion follows from an application of Lemma \ref{30-09-17} to the map $T$ from the proof of Lemma \ref{02-12-17c}.
\end{proof}

It follows from Proposition \ref{aug23} that there is a map
\begin{equation}\label{14-01-18map}
\Lambda : \ \partial M^{v_0}_{\beta F}(\Gamma) \ \to \ \mathcal E(\Gamma)
\end{equation}
which sends $m \in \partial M^{v_0}_{\beta F}(\Gamma) $ to the unique end $E \in \mathcal E(\Gamma)$ for which $m(E) \neq 0$. Much of the following work is centered around understanding the properties of this map.  As will become apparent in the following the map is neither surjective nor injective in general. For example, we exhibit in Example \ref{27-02-18d} a strongly connected row-finite graph with an end which supports uncountably many different extremal $e^{\beta F}$-conformal measures and at the same time has infinitely many ends which all contains a ray $p$ such that the limits $\psi_v = \lim_{k \to \infty} K_{\beta}(v,s(p_k))$ exist for all vertexes $v$ and define a minimal $A(\beta)$-harmonic vector $\psi$, but nonetheless do not support any extremal $e^{\beta F}$-conformal measure. In Example \ref{22-01-18} we exhibit a cofinal row-finite graph for which the map is injective for some values of $\beta$ and not for others. Strongly connected examples exhibiting the same phenomenon can be obtained by use of the constructions in Section \ref{adding}. When $\Gamma$ is strongly connected, row-finite and almost undirected, the map \eqref{14-01-18map} is surjective, as we show in Theorem \ref{13-12-17a}, but still far from injective in general.

\section{Meager digraphs}\label{meager}

\subsection{Asymptotic weights}

The standing assumption in this section is that $\Gamma$ is a countable directed graph and $F : \Gamma_{Ar} \to \mathbb R$ is a potential such that \eqref{v0trans} holds for some fixed $v_0$.

   Let $w_0,w_1, \cdots, \ w_k$ be distinct vertexes in $\Gamma$. For $j = 0,1, \cdots, k-1$, let $L_j$ be the (possibly empty) set of paths $\mu = a_1a_2\cdots a_{|\mu|} \in P_f(\Gamma)$ such that $s(\mu) = w_j, \ r(\mu) = w_{j+1}$ and $r(a_i) \notin \{w_0,\cdots, w_j\}  \ \forall i$. Assuming that $L_j \neq \emptyset$ for all $j$ we set
 $$
 \mathbb W(w_0,w_1,\cdots, w_k) =  \prod_{j=0}^{k-1} \left( \sum_{\nu \in L_j} e^{-\beta F(\nu)} \right) \ .
 $$ 
When $\mu = a_1a_2 \cdots a_n$ is a finite path in $\Gamma$ such that $s(a_1),s(a_2), \cdots, s(a_n)$ and $r(a_n)$ are distinct, set\label{Wmu}
$$
 \mathbb W(\mu) = \mathbb W\left(s(a_1),s(a_2), \cdots, s(a_n),r(a_n)\right) \ .
 $$
 When $\mu$ is of length $0$, i.e. just a vertex, we define $\mathbb W(\mu) = 1$.
 Let $y = \left(y_i\right)_{i=1}^{\infty} \in \Ray(\Gamma)$ be a ray. When $i < j$ we denote by $y[i,j[$ the path $y_iy_{i+1}y_{i+2} \cdots y_{j-1}$.\label{yij}
 
\begin{lemma}\label{28-02-18} Let $y = (y_i)_{i=1}^{\infty} \in \Ray(\Gamma)$. Then 
\begin{equation*}
\frac{\sum_{n=0}^{\infty} A(\beta)^n_{v,s(y_k)}}{\mathbb W(y[1,k[)}\ \leq \ \frac{\sum_{n=0}^{\infty} A(\beta)^n_{v,s(y_{k+1})}}{\mathbb W(y[1,k+1[)} \ 
\end{equation*} 
for all $v \in \Gamma_V, \ k \in \mathbb N$.
\end{lemma}
\begin{proof} Let $L_k(y)$\label{Lky} be the set of finite paths $\mu = a_1a_2\cdots a_{|\mu|} \in P_f(\Gamma)$ such that $s(\mu) = s(y_k), \ r(\mu) = s\left(y_{k+1}\right)$ and $r(a_i) \notin \{s(y_1),\cdots, s(y_k)\}  \ \forall i$. Then
$$
\sum_{n=0}^{\infty} A(\beta)^n_{v,s(y_k)} \left( \sum_{\mu \in L_k(y)} e^{-\beta F(\mu)}\right) \ \leq \ \sum_{n=0}^{\infty} A(\beta)^n_{v,s(y_{k+1})}
$$
and hence
\begin{equation*}
\begin{split}
& \frac{\sum_{n=0}^{\infty} A(\beta)^n_{v,s(y_k)}}{\mathbb W(y[1,k[)} =   \frac{\sum_{n=0}^{\infty} A(\beta)^n_{v,s(y_k)} \left( \sum_{\mu \in L_k(y)} e^{-\beta F(\mu)}\right)}{\mathbb W(y[1,k[) \left( \sum_{\mu \in L_k(y)} e^{-\beta F(\mu)}\right)}\\
& \\ 
& = \frac{\sum_{n=0}^{\infty} A(\beta)^n_{v,s(y_k)} \left( \sum_{\mu \in L_k(y)} e^{-\beta F(\mu)}\right)}{\mathbb W(y[1,k+1[)} \ \leq \ \frac{\sum_{n=0}^{\infty} A(\beta)^n_{v,s(y_{k+1})}}{\mathbb W(y[1,k+1[)} \ .
\end{split}
\end{equation*}
\end{proof}

It follows Lemma \ref{28-02-18} that the limits\label{Vbetavy}
$$
\mathbb V_{\beta}(v,y) = \lim_{k \to \infty}\frac{\sum_{n=0}^{\infty} A(\beta)^n_{v,s(y_k)} }{\mathbb W(y[1,k[)} 
$$
exist in $[0,\infty]$ for all $v \in \Gamma_V$. Let $y = \left(y_i\right)_{i=1}^{\infty} \in \Ray(\Gamma)$ be a ray. For each $k$, set\label{Bky}
\begin{equation*}
\begin{split}
&B_k(y) = \\
&\left\{(x_i)_{i=1}^{\infty} \in P(\Gamma) : \ s(x_1) = s(y_k), \ s(x_i) \notin \{s(y_1), \cdots, s(y_k)\}, \ i \geq 2 \right\} ,
\end{split}
\end{equation*}
and\label{BBky}
$$
\mathbb B_k(y) = B_k(y) \cap \bigcap_{j \geq k+1} \left(\bigcup_{i=0}^{\infty} \sigma^{-i}\left(B_j(y)\right) \right) .
$$

\begin{lemma}\label{08-11-17a}
 Let $m$ be an $e^{\beta F}$-conformal measure. Then
\begin{equation}\label{AB0bb}
m\left(\mathbb B_1(y) \right)  =  m\left(\mathbb B_k(y)\right)   \mathbb W(y[1,k[) \ .
\end{equation}
for all $k = 2,3, \cdots $.
\end{lemma}
\begin{proof} Let $L_k(y)$ be the same set of paths as in the proof of Lemma \ref{28-02-18}. Let $M_k$ be the finite paths $\mu$ in $\Gamma$ that are concatenations
$\mu = \delta_1 \delta_2 \cdots \delta_{k-1}$
with $\delta_j \in L_j(y)$. Then  
$$
\mathbb B_1(y) = \bigcup_{\mu \in M_k} Z(\mu)\mathbb B_k(y) 
$$
and hence \eqref{AB0bb} follows from Lemma \ref{nov2}.
\end{proof}

Set\label{BB(y)}
$$
\mathbb B(y) = \bigcup_{n,k \geq 0} \sigma^{-n}\left( \sigma^k(\mathbb B_1(y))\right) \ .
$$

\begin{lemma}\label{08-11-17b}  Let $m$ be an $e^{\beta F}$-conformal measure. Then
$$
m\left(Z(v) \cap \mathbb B(y)\right) =\mathbb V_{\beta}(v,y) m\left(\mathbb B_1(y)\right)  
$$
for all $v \in \Gamma_V$.
\end{lemma}
\begin{proof} Note that
$$
\bigcup_{n,k \geq 0} \sigma^{-n}\left( \sigma^k(\mathbb B_1(y)) \right)  = \bigcup_{n,k \geq 0} \sigma^{-n}\left( \mathbb B_k(y) \right) 
$$
and
$$
\bigcup_{n \geq 0} \sigma^{-n}\left( \mathbb B_k(y) \right) \subseteq \bigcup_{n \geq 0} \sigma^{-n}\left( \mathbb B_{k+1}(y)\right) .
$$
Hence
$$
m\left(Z(v) \cap \mathbb B(y)\right) = \lim_{k \to \infty} m\left( Z(v) \cap \bigcup_{n \geq 0} \sigma^{-n}\left( \mathbb B_k(y) \right)\right) .
$$
Since
$$
 m\left( Z(v) \cap \bigcup_{n \geq 0} \sigma^{-n}\left( \mathbb B_k(y) \right)\right)  = \sum_{n=0}^{\infty} A(\beta)^n_{v,  s(y_{k})} m\left(\mathbb B _k(y)\right) 
 $$
by Lemma \ref{nov2}, the conclusion follows from Lemma \ref{08-11-17a}.
\end{proof}

 \begin{lemma}\label{24-11-17a} Let $y \in \Ray(\Gamma)$. Then $\mathbb V_{\beta}(v,y)  < \infty$ for all $v \in \Gamma_V$ if and only if $\mathbb V_{\beta}(v_0,y) < \infty$.

 \end{lemma}
 \begin{proof}  Assume that $\mathbb V_{\beta}(v_0,y) < \infty$ and consider a vertex $v \in \Gamma_V$. By assumption there is an $l\in \mathbb N$ such that $A(\beta)^l_{v_0,v} > 0$. Since
 $$
A(\beta)^l_{v_0,v} \frac{\sum_{n=0}^{\infty} A(\beta)^n_{v,s(y_k)} }{\mathbb W(y[1,k[)} \leq  \frac{\sum_{n=0}^{\infty} A(\beta)^n_{v_0,s(y_k)} }{\mathbb W(y[1,k[)}
 $$
 for all $k$, it follows that $\mathbb V_{\beta}(v,y) \leq \left(A(\beta)^l_{v_0,v}\right)^{-1} \mathbb V_{\beta}(v_0,y) < \infty$. 
 
 \end{proof}

 
\begin{thm}\label{24-11-17}  Let $\Gamma$ be a countable directed graph and $F : \Gamma_{Ar} \to \mathbb R$ a map such that \eqref{v0trans} holds for some vertex $v_0$. Let $y \in \Ray(\Gamma)$. There is a non-zero $e^{\beta F}$-conformal measure concentrated on $\mathbb B(y)$ if and only if $\mathbb V_{\beta}(v,y)  < \infty$ for all $v \in \Gamma_V$. When it exists, it is unique up to scalar multiplication and it is extremal. The corresponding $v_0$-normalized $A(\beta)$-harmonic vector is proportional to $\left( \mathbb V_{\beta}(v,y) \right)_{v \in \Gamma_V}$.
\end{thm}
\begin{proof} Assume first that there is a non-zero $e^{\beta F}$-conformal measure concentrated on $\mathbb B(y)$. Then $m(\mathbb B_1(y)) > 0$ and $\mathbb V_{\beta}(v,y) < \infty$ for all $v \in \Gamma_V$ by Lemma \ref{08-11-17b}. 
Assume then that $\mathbb V_{\beta}(v,y) < \infty$ for all $v \in \Gamma_V$. By using the monotonicity of the sequences defining $\mathbb V_{\beta}(v,y)$ we find that
\begin{equation*}
\begin{split}
&\sum_{w \in \Gamma_V} A(\beta)_{v,w}\mathbb V_{\beta}(w,y) = \sum_{w \in \Gamma_V} A(\beta)_{v,w} \lim_{k \to \infty} \frac{\sum_{n=0}^{\infty} A(\beta)^n_{w, s(y_k)} }{\mathbb W(y[1,k[)}  \\
& \\
& = \lim_{k \to \infty} \sum_{w \in \Gamma_V} \sum_{n=0}^{\infty}\frac{ A(\beta)_{v,w} A(\beta)^n_{w, s(y_k)} }{\mathbb W(y[1,k[)} = \lim_{k \to \infty} \sum_{n=0}^{\infty}\sum_{w \in \Gamma_V} \frac{ A(\beta)_{v,w} A(\beta)^n_{w, s(y_k)} }{\mathbb W(y[1,k[)} \\
& \\
&=  \lim_{k \to \infty} \frac{\sum_{n=0}^{\infty}  A(\beta)^{n+1}_{v, s(y_k)} }{\mathbb W(y[1,k[)}   =   \lim_{k \to \infty} \frac{\sum_{n=0}^{\infty}  A(\beta)^{n}_{v, s(y_k)} }{\mathbb W(y[1,k[)} = \mathbb V_{\beta}(v,y) \ ,
\end{split}
\end{equation*} 
i.e. the vector $\left(\mathbb V_{\beta}(v,y)\right)_{v \in \Gamma_V}$ is $A(\beta)$-harmonic. By Proposition \ref{nov1} there is an $e^{\beta F}$-conformal measure $m'$ such that 
$$
m'(Z(v)) = \mathbb V_{\beta}(v,y)
$$
for all $v \in \Gamma_V$. To show that $m'$ does not annihilate $\mathbb B(y)$ we adopt the notation from the proof of Lemma \ref{08-11-17a}. Then
 $$
 \bigcup_{\mu \in M_{k+1}} Z(\mu) \subseteq \bigcup_{\mu \in M_{k}} Z(\mu)
 $$
 and
 $$ 
\mathbb B_1(y) =  \bigcap_k \left(  \bigcup_{\mu \in M_{k}} Z(\mu) \right) ,
$$
and hence
$$
m'\left(\mathbb B_1(y)\right) = \lim_{k \to \infty} m'\left( \bigcup_{\mu \in M_{k}} Z(\mu)\right) \ .
$$
It follows from Lemma \ref{nov2} that
\begin{equation*}
\begin{split}
&m'\left( \bigcup_{\mu \in M_{k}} Z(\mu)\right) = \mathbb W(y[1,k[) m'\left(Z(s(y_k))\right) = \mathbb W(y[1,k[) \mathbb V_{\beta}\left(s(y_k),y\right) \\
& =  \lim_{l \to \infty} \mathbb W(y[1,k[) \frac{\sum_{n=0}^{\infty} A(\beta)^n_{s(y_k),s(y_l)}  }{\mathbb W(y[1,l[)}  \ .
\end{split}
\end{equation*}
Since 
$$
\mathbb W(y[1,k[)\mathbb W(y[k,l[) \geq \mathbb W(y[1,l[)
$$
and
$$ 
 \sum_{n=0}^{\infty} A(\beta)^n_{s(y_k),s(y_l)} \geq \mathbb W\left(y[k,l[\right) \ ,
$$ 
it follows that
\begin{equation*}
\begin{split}
&m'(\mathbb B_1(y)) = \lim_{k \to \infty} \lim_{l \to \infty} \mathbb W(y[1,k[) \frac{\sum_{n=0}^{\infty} A(\beta)^n_{s(y_k),s(y_l)} }{\mathbb W(y[1,l[)}  \ \geq \ 1 \ .\\
 \end{split}
 \end{equation*}
 Since $\mathbb B(y)$ is totally shift-invariant the measure $m$ defined such that
 $$
 m(B) = m'(B \cap \mathbb B(y))
 $$
 is $e^{\beta F}$-conformal since $m'$ is, and it is non-zero by the above estimate; in fact, $m(\mathbb B_1(y)) \geq 1$. 
 
To prove the essential uniqueness of $m$, let $n$ be another $e^{\beta F}$-conformal measure concentrated on $\mathbb B(y)$. Then $n(\mathbb B_1(y)) > 0$ and it follows from Lemma \ref{08-11-17b} that 
$$
 n\left(\mathbb B_1(y)\right)^{-1}n(Z(v)) = \mathbb V_{\beta}(v,y) = m\left(\mathbb B_1(y)\right)^{-1}m(Z(v))
 $$ 
 for all $v \in \Gamma_V$. By Proposition \ref{nov1} this implies that $m\left(\mathbb B_1(y)\right)^{-1}m = n\left(\mathbb B_1(y)\right)^{-1}n$. Finally, the extremality of $m$ follows straightforwardly from its essential uniqueness.
\end{proof}

In the following, when $\mathbb V_{\beta}(v,y) < \infty$ for all $v \in \Gamma_V$ we say that $y \in \Ray(\Gamma)$ is \emph{$\beta$-summable}.

 \begin{lemma}\label{08-01-18c} Let $y \in \Ray(\Gamma)$. The following are equivalent.
 \begin{itemize}
 \item $y$ is $\beta$-summable.
 \item $\sigma^k(y)$ is $\beta$-summable for some $k \in \mathbb N$.
 \item $\sigma^k(y)$ is $\beta$-summable for all $k \in \mathbb N$.
 \end{itemize}
 \end{lemma}
 \begin{proof} This follows from Theorem \ref{24-11-17} because $\mathbb B(y) = \mathbb B(\sigma^k(y))$ for all $k \in \mathbb N$.
 \end{proof}

\subsection{Reduced rays and meager digraphs} Define $\eta : \Ray (\Gamma) \to \left(\Gamma_V\right)^{\mathbb N}$ such that
$$
\eta(p)_k = s(p_k) \ .
$$
Let $x = (x_i)_{i=1}^{\infty}$ and $y = (y_i)_{i=1}^{\infty}$ be two elements of $\eta\left(\Ray(\Gamma)\right)$. Write\label{prece}
$$
x \preceq y
$$
when there is a sequence $1= i_1 < i_2<  \cdots$ in $\mathbb N$ such that $x_{i_j} =  y_j$ for all $j = 1,2,3, \cdots $. Note that $\preceq$ is a partial ordering in $\eta\left(\Ray(\Gamma)\right)$. A ray $p \in \Ray(\Gamma)$ will be called a \emph{reduced ray} when $\eta(p)$ is maximal for this partial order, and the set of reduced rays will be denoted by $\Rray(\Gamma)$.\label{Rray} A ray $y = (y_i)_{i=1}^{\infty} \in \Ray(\Gamma)$ is reduced when it does not contain a finite subpath of length $\geq 2$ for which there is an arrow with the same start and terminal vertex; i.e. when
$$
a \in \Gamma_{Ar}, \ s(a) = s(y_i), \ r(a) = s(y_j) \ \Rightarrow \ j -i \leq 1 \ .
$$ 
Note that $\Rray(\Gamma)$ is a closed subset of $P(\Gamma)$.

\begin{lemma}\label{05-01-18} Let $p \in \Ray(\Gamma)$ be a ray such that no $s(p_k)$ is an infinite emitter; i.e. $\# s^{-1}\left(s(p_k)\right) < \infty$ for all $k$. There is a reduced ray $q \in \Rray(\Gamma)$ such that $\eta(p) \preceq \eta(q)$.
\end{lemma}
\begin{proof} The set $A = \left\{ y \in \eta(\Ray(\Gamma)): \ \eta(p) \preceq y \right\}$ is not empty as it contains $\eta(p)$. By Zorn's lemma it suffices therefore to show that a totally ordered subset $\mathcal A$ of $A$ has an upper bound in $A$.
For this note first that $A$ is compact in $\left(\Gamma_V\right)^{\mathbb N}$ thanks to the assumptions on $p$. Consider $\mathcal A$ as a net in $\left(\Gamma_V\right)^{\mathbb N}$. Since $A$ is compact this net has a convergent subnet. I.e. there is a directed set $I$ and a monotone map $h : I \to \mathcal A$ such that $h(I)$ is cofinal in $\mathcal A$ and $\lim_{\alpha \in I} h(\alpha) = y \in A$. We prove that $z \preceq y$ for all $z \in \mathcal A$ which will finish the proof. Let $z \in \mathcal A$. It suffices to show that for each $j \in \mathbb N$ there are $i_1 < i_2 < i_3 < \cdots < i_j$ in $\mathbb N$ such that $i_1 = 1$ and $y_k = z_{i_k}, \ k = 1,2, \cdots, j$. To this end note that there is a $\alpha_0 \in I$ such that $h(\alpha)[1,j] = y[1,j]$ for all $\alpha \geq \alpha_0$ in $I$. Since $h(I)$ is cofinal in $\mathcal A$ there is also a $\alpha_1 \in I$ such that $z \preceq h(\alpha_1)$. Since $I$ is directed there is an $\alpha \in I$ such that $\alpha_0  \leq \alpha$ and $\alpha_1 \leq \alpha$. Then $z \preceq h(\alpha)$ which implies that there are $1=i_1< i_2< \cdots < i_j$ in $\mathbb N$ such that $z_{i_k} =  h(\alpha)_k, \ k =1,2,\cdots, j$. Since $h(\alpha)[1,j] = y[1,j]$ because $\alpha \geq \alpha_0$, this finishes the proof.   
\end{proof}

\begin{lemma}\label{05-01-18b} Let $p,q \in \Ray(\Gamma)$ such that $\eta(p) \preceq \eta(q)$. Then $\mathbb V_{\beta}(v,p) \geq \mathbb V_{\beta}(v,q)$ for all $v \in \Gamma_V$.
\end{lemma}
\begin{proof} Let $1= i_1 < i_2<  \cdots$ be a sequence in $\mathbb N$ such that $s\left(p_{i_j}\right) = s(q_j)$ for all $j = 1,2,3, \cdots $. Then 
$$
\mathbb W\left(p[1,i_j[\right) \leq \mathbb W\left(q[1,j[\right) \ ,
$$
and hence
$$
\frac{\sum_{n=0}^{\infty} A(\beta)^n_{v,s(q_j)}}{\mathbb W\left(q[1,j[\right)} \ \leq \ \frac{\sum_{n=0}^{\infty} A(\beta)^n_{v,s(p_{i_j})}}{\mathbb W\left(p[1,i_j[\right)}
$$
for all $v,j$. The conclusion follows.
\end{proof}

 Let $\Rray(\Gamma)/\text{Tail}$ denote the set of tail-equivalence classes of reduced rays in $\Gamma$. Since tail equivalent wandering paths define the same end there is a natural map
\begin{equation}\label{06-11-17}
\Rray(\Gamma)/\text{Tail} \ \to \ \mathcal E(\Gamma) \ .
\end{equation}

\begin{lemma}\label{07-01-2018} Assume that $\Gamma$ does not contain infinitely many infinite emitters. For every wandering path $p \in \Wan(\Gamma)$ there is a reduced ray $q \in \Rray(\Gamma)$ such that $p \in \mathbb B(q) \subseteq \mathcal E(p)$. 
\end{lemma}
\begin{proof} Let $R : \Wan(\Gamma) \to \Ray(\Gamma)$ be the retraction \eqref{21-02-18b}. Since there are at most finitely many infinite emitters by assumption there is a $k \in \mathbb N$ such that $\sigma^k(R(p))$ does not contain any infinite emitter and by Lemma \ref{05-01-18} there is therefore a reduced ray $q \in \Rray(\Gamma)$ such that $\eta\left( \sigma^k(R(p))\right) \preceq \eta(q)$. Then $q$ has the stated property.
\end{proof} 

It follows from Lemma \ref{07-01-2018} that the map \eqref{06-11-17} is surjective when $\Gamma$ does not contain infinitely many infinite emitters. It is generally far from injective, but there is a large class of digraphs where it \emph{is} injective, and we make the following definition.

\begin{defn} A countable directed graph $\Gamma$ will be called  \emph{meager} when the map \eqref{06-11-17} is a bijection.
\end{defn}

\begin{lemma}\label{08-01-18} Let $\Gamma$ be a meager digraph with at most finitely many infinite emitters. For every end $E \in \mathcal E(\Gamma)$ there is a reduced ray $p \in \Rray(\Gamma)$, unique up tail equivalence, such that $E = \mathbb B(p)$.
\end{lemma}
\begin{proof} By Lemma \ref{07-01-2018} there is a reduced ray $q$ such that $\mathbb B(q) \subseteq E$. Let $p \in E$. Using Lemma \ref{07-01-2018} again we get a reduced ray $q'$ such that $p \in \mathbb B(q') \subseteq E$. Since $\Gamma$ is meager the two reduced rays $q$ and $q'$ are tail equivalent which implies that $\mathbb B(q) = \mathbb B(q')$. It follows that $p \in \mathbb B(q)$ and hence that $E = \mathbb B(q)$. 
\end{proof}

\begin{remark}\label{rem09-01-18} The paper \cite{Th3} contains a study of digraphs $\Gamma$ for which the set $\Wan(\Gamma)/\text{Tail}$ is countable; called \emph{graphs with countably many exits}. The end space $\mathcal E(\Gamma)$ does not feature in \cite{Th3} but the proof of Lemma 9.3 in \cite{Th3} shows that for digraphs with countably many exits the map
$$
\Wan(\Gamma)/\text{Tail} \to \mathcal E(\Gamma)
$$
is a bijection. Combined with Lemma \ref{07-01-2018} this shows that also the map \eqref{06-11-17} is a bijection when $\Gamma$ has countably many exits and at most finitely many infinite emitters. Therefore the examples constructed in \cite{Th3} shows that a strongly connected meager digraph may have a very rich structure of conformal measures, already when the potential is the constant function $F =1$.
\end{remark}  
\begin{lemma}\label{16-11-17a} Let $E$ be an end in $\Gamma$. Then $E \cap \Ray(\Gamma)$ is a $G_{\delta}$ set in $P(\Gamma)$.
\end{lemma}
\begin{proof} Let $x \in E$. Choose a sequence $F_1 \subseteq F_2 \subseteq F_3 \subseteq \cdots$ of finite subsets in $\Gamma_V$ such that $\bigcup_n F_n = \Gamma_V$. For each $n \in \mathbb N$ let $A_n$ denote the set of finite paths $\mu = a_1a_2a_3 \cdots a_{|\mu|}$ in $\Gamma$ with the property that there are natural numbers $n <  i_1 < i_2 \leq|\mu|$ and $n <  j_1 < j_2 $ such that the following hold:
\begin{itemize}
\item[i)] The vertexes in $\{s(a_1)\} \cup \left\{ r(a_i): \ 1 \leq i \leq |\mu|\right\}$ are distinct, 
\item[ii)] there is a finite path $\nu$ in $\Gamma$ such that $\nu \cap F_n = \emptyset, \ s(\nu) = s(x_{j_1}), \ r(\nu) = s(a_{i_1})$, and
\item[iii)] there is a finite path $\nu'$ in $\Gamma$ such that $\nu' \cap F_n = \emptyset, \ s(\nu') = s(a_{i_2}), \ r(\nu') = s(x_{j_2})$.
\end{itemize}
Then $E \cap \Ray(\Gamma)= \bigcap_n \left( \bigcup_{\mu \in A_n} Z(\mu) \right)$; a $G_{\delta}$ set in $P(\Gamma)$.
\end{proof}

\begin{lemma}\label{08-01-18e} Let $\Gamma$ be a meager digraph with at most finitely many infinite emitters. There is a Borel map $\xi_R : \mathcal E(\Gamma) \to \Rray(\Gamma)$\label{xiR} such that $\mathbb B\left( \xi_R(E)\right) = E$ for all $E \in \mathcal E(\Gamma)$.
\end{lemma}
\begin{proof} We want to apply a result by Srivastava, stated as Theorem 5.9.2 in \cite{Sr}, to the partition
\begin{equation}\label{01-03-18}
\Rray(\Gamma) = \bigsqcup_{E \in \mathcal E(\Gamma)} E \cap \Rray(\Gamma) \ .
\end{equation}
For this note first of all that $\Rray(\Gamma)$ as a closed subset of the Polish space $P(\Gamma)$ is itself a Polish space. Furthermore, it follows from Lemma \ref{16-11-17a} that \eqref{01-03-18} is a partition of $\Rray(\Gamma)$ into $G_{\delta}$-sets; non-empty since $\Gamma$ is meager. To show that the partition 
is a Borel partition of $\Ray(\Gamma)$ as defined in \cite{Sr} we must show that
$$
[U] \overset{def}{=} \bigcup_{y \in U} \left\{ x \in \Rray(\Gamma) : \ x \ \sim \ y \right\}
$$
is a Borel subset of $\Rray(\Gamma)$ for every open set $U \subseteq \Rray(\Gamma)$. It suffices to establish this when $U = Z(\mu)\cap \Rray(\Gamma)$ for some $\mu \in P_f( \Gamma)$, in which case it follows from the observation that
$$
[Z(\mu)\cap \Rray(\Gamma)] = \bigcup_{n,m}\sigma^{-n}\left(\sigma^m(Z(\mu)\cap \Rray(\Gamma))\right) \cap \Rray(\Gamma) \ .
$$
We conclude now from Theorem 5.9.2 on page 213 of \cite{Sr} that there is a Borel subset $A \subseteq \Rray(\Gamma)$ such that $A \cap E$ contains exactly one element for each end $E \in \mathcal E(\Gamma)$. We can therefore define an injective map $\xi_R : \mathcal E(\Gamma) \to A \subseteq \Rray(\Gamma)$ such that
$$
E \cap A \ = \ \{\xi_R(E)\} \ .
$$
Then $\mathcal E\left( \xi_R(E)\right) = E$ for all $E \in \mathcal E(\Gamma)$. By Lemma \ref{08-01-18} there is a reduced ray $p$ such that $\mathbb B(p) = E = \mathcal E\left( \xi_R(E)\right)$. Then $p$ is tail-equivalent to $\xi_R(E)$ since $\Gamma$ is meager and hence $\mathbb B(\xi_R(E)) = \mathbb B(p) = E$. It remains to show that $\xi_R$ is a Borel map. For this we consider the diagram \eqref{20-11-17a} and note that $\chi'(A) = \chi\left(\mathcal E(A)\right) = \chi(\mathcal E(\Gamma))$. Note that $\chi'$ is Borel by Lemma \ref{aug13xa} and injective on $A$ by Lemma \ref{aug14(3)}. Since $X_{\Gamma}$ is a compact metric space and hence Polish, this implies that $\chi(\mathcal E(\Gamma)) = \chi'(A)$ is a Borel subset of $X_{\Gamma}$, cf. e.g. Theorem 8.3.7 on page 276 in \cite{Co}. It follows that there is a Borel map $\xi :\chi(\mathcal E(\Gamma)) \to A$ such that $\chi' \circ \xi(x) = x$ for all $x \in\chi(\mathcal E(\Gamma))$. By definition of the topology on $\mathcal E(\Gamma)$ the map $\chi$ is a homeomorphism onto its image which we now know is a Borel subset of $X_{\Gamma}$. Hence $\chi$ is an isomorphism of Borel spaces. This completes the proof because $\xi_R = \xi \circ \chi$.
\end{proof}

\begin{remark}\label{01-03-18a} In the course of the preceding proof it was established that $\mathcal E(\Gamma)$ with its Borel $\sigma$-algebra is isomorphic to a Borel subset of $X_{\Gamma}$ and hence is a standard Borel space. This is true in general; an application of Srivastava's theorem to the partition
$$
\Ray(\Gamma) = \bigsqcup_{E \in \mathcal E(\Gamma)} E \cap \Ray(\Gamma) \ 
$$
gives a Borel subset $A \subseteq \Ray(\Gamma)$ such that $\chi'$ is injective and surjective on $A$. In particular, $\chi(\mathcal E(\Gamma)) = \chi'(A)$ is a Borel subset of $X_{\Gamma}$.
\end{remark}

Set\label{RRaybeta}
$$
\Rray_{\beta}(\Gamma) = \left\{ y \in \Rray(\Gamma): \ \mathbb V_{\beta}(v_0,y) < \infty \right\} \ ;
$$
a Borel subset of $\Rray(\Gamma)$. It follows from Theorem \ref{24-11-17} that for every $y \in \Rray_{\beta}(\Gamma)$ there is unique $v_0$-normalized $e^{\beta F}$-conformal measure $m_y$ concentrated on $\mathbb B(y)$. Set\label{Ebetagamma}
$$
\mathcal E_{\beta}(\Gamma) = \xi_R^{-1}\left(\Rray_{\beta}(\Gamma)\right) \ ,
$$
where $\xi_R : \mathcal E(\Gamma) \to \Rray(\Gamma)$ is the Borel map from Lemma \ref{08-01-18e}. We set
$$\label{mE}
m_E = m_{\xi_R(E)}
$$
when $  E \in \mathcal E_{\beta}(\Gamma)$. Note that $m_E$ is $v_0$-normalized and concentrated on $E = \mathbb B\left(\xi_R(E)\right)$.

\begin{lemma}\label{09-11-17x} Let $\Gamma$ be a meager digraph with at most finitely many infinite emitters. Let $F : \Gamma_{Ar} \to \mathbb R$ be a potential such that \eqref{v0trans} holds for some $v_0 \in \Gamma_V$. 
\begin{itemize}
\item[a)] Let $E \in \mathcal E(\Gamma)$. There is an $e^{\beta F}$-conformal measure concentrated on $E$ if and only if $E \in \mathcal E_{\beta}(\Gamma)$.
\item[b)] For each end $E \in \mathcal E_{\beta}(\Gamma)$ the measure $m_E$ is the only $v_0$-normalized $e^{\beta F}$-conformal measure concentrated on $E$, and the map 
\begin{equation}\label{08-01themap}
\mathcal E_{\beta}(\Gamma) \ni E \mapsto \ m_E 
\end{equation}
is a bijection between $\mathcal E_{\beta}(\Gamma)$ and the set of extremal $v_0$-normalized $e^{\beta F}$-conformal measures on $P(\Gamma)$.
\item[c)] The map 
$$
\mathcal E_{\beta}(\Gamma) \ni E \ \mapsto m_E(B) 
$$
is a Borel function for all Borel sets $B \subseteq P( \Gamma)$.
\end{itemize}
\end{lemma}
\begin{proof} a) If $m$ is a normalized $e^{\beta F}$-conformal measure concentrated on $E\in \mathcal E(\Gamma)$, then Lemma \ref{08-01-18} and Theorem \ref{24-11-17} imply that $E = \mathbb B(y)$ for some $y \in \Rray_{\beta}(\Gamma)$. Note that $\xi_R(E)$ and $y$ are tail-equivalent since $\Gamma$ is meager. It follows then from Lemma \ref{08-01-18c} that $\xi_R(E) \in \Rray_{\beta}(\Gamma)$; i.e. $E \in \mathcal E_{\beta}(\Gamma)$.

b) It follows from  Theorem \ref{24-11-17} and Lemma \ref{08-01-18} that $m_E$ is the only normalized $e^{\beta F}$-conformal measure concentrated on $E$. The injectivity of the map \eqref{08-01themap} follows because $m_E$ is concentrated on $E$ and its range is contained in the set of extremal normalized $e^{\beta F}$-conformal measures by Theorem \ref{24-11-17}. To see that it hits them all, let $m$ be a $v_0$-normalized extremal $e^{\beta F}$-conformal measure. Then $m$ is concentrated on an end $E$ by Proposition \ref{aug23}. By Lemma \ref{08-01-18}, $E = \mathbb B(y)$ for some $y \in \Rray(\Gamma)$. It follows from Theorem \ref{24-11-17} that $y \in \Rray_{\beta}(\Gamma)$ and $m = m_y$. Since $\xi_R(E)$ and $y$ are both in $E \cap \Rray(\Gamma)$, it follows that $\xi_R(E)$ and $y$ are tail-equivalent and hence $m_y = m_{\xi_R(E)} = m_E$.

c) Let $v \in \Gamma_V$. When $y \in \Rray_{\beta}(\Gamma)$ and $\mu \in P_f(\Gamma)$ the formula
$$
m_y(Z(\mu)) = e^{-\beta F(\mu)}\mathbb V_{\beta}(r(\mu),y)\mathbb V_{\beta}(v_0,y)^{-1}
$$
shows that $y \mapsto m_y(Z(v) \cap Z(\mu))$ is a Borel function of $y \in \Rray_{\beta}(\Gamma)$. It follows then from Lemma \ref{03-03-18} that $y \mapsto m_y(B \cap Z(v))$ is Borel, and hence that $y \mapsto m_y(B) = \sum_{v \in \Gamma_V} m_y(Z(v) \cap B)$ is Borel for all Borel subsets $B \subseteq P(\Gamma)$. This completes the proof because $\xi_R$ is Borel by Lemma \ref{08-01-18e}.
\end{proof}

\begin{thm}\label{08-01-18d} Let $\Gamma$ be a meager digraph with at most finitely many infinite emitters. Let $F : \Gamma_{Ar} \to \mathbb R$ be a potential such that \eqref{v0trans} holds for some $v_0\in \Gamma_V$. There is an affine bijection between the $v_0$-normalized $e^{\beta F}$-conformal measures $m$ on $P(\Gamma)$ and the Borel probability measures $\nu$ on $\mathcal E_{\beta}(\Gamma)$ such that
\begin{equation}\label{08-01nymap}
m(B) = \int_{\mathcal E_{\beta}(\Gamma)} m_E(B) \ \mathrm{d}\nu(E) \ 
\end{equation}
for all Borel sets $B \subseteq P(\Gamma)$.
\end{thm}
\begin{proof} It follows from c) of Lemma \ref{09-11-17x} that \eqref{08-01nymap} defines a Borel measure on $P(\Gamma)$, and it is a $v_0$-normalized $e^{\beta F}$-conformal measure since each $m_E$ is. Conversely, consider a given normalized $e^{\beta F}$-conformal measure $m$ on $P(\Gamma)$. By Proposition \ref{02-12-17c} there is a Borel probability measure $m^{\mathcal E}$ on $\mathcal E(\Gamma)$ and a $(\mathcal E, m^{\mathcal E})$-disintegration 
$$
m = \int_{\mathcal E(\Gamma)} m'_E \ \mathrm{d}m^{\mathcal E}(E) 
$$
of $m$ such that for $m^{\mathcal E}$-almost all $E \in \mathcal E(\Gamma)$ the measure $m'_E$ is an $e^{\beta F}$-conformal measure concentrated on $E$. It follows from a) of Lemma \ref{09-11-17x} that $m^{\mathcal E}\left(\mathcal E_{\beta}(\Gamma)\right) = 1$, and from b) of Lemma \ref{09-11-17x} that for all $E \in \mathcal E_{\beta}(\Gamma)$ the equality $m'_E = m'_E\left(Z(v_0)\right) m_E$ holds. 
 Set $g(E) =  m'_E\left(Z(v_0)\right), E \in  \mathcal E_{\beta}(\Gamma)$. 
Then $g$ is a Borel function and the measure $\nu = gm^{\mathcal E}$ makes \eqref{08-01nymap} hold. Since
$$
\nu(\mathcal E_{\beta}(\Gamma)) = \int_{\mathcal E_{\beta}(\Gamma)} m_E(Z(v_0)) \ \mathrm{d} \nu(E) = m(Z(v_0)) = 1 \ ,
$$
we have in $\nu$ the desired Borel probability measure on $\mathcal E_{\beta}(\Gamma)$. To establish the injectivity of the map, consider a Borel set $B' \subseteq \mathcal E_{\beta}(\Gamma)$ and  apply \eqref{08-01nymap} with $B = \mathcal E^{-1}(B')$. Then
$$
m(B) = \int_{\mathcal E_{\beta}( \Gamma)} m_E \left( \mathcal E^{-1}(B')\right) \ \mathrm{d} \nu(E) = \nu(B') \ .
$$
Hence $m$ determines $\nu$.
\end{proof}

\subsection{Examples}\label{building} Perhaps the main virtue of the class of meager digraphs is that it comprises both the graphs with countably many exits studied in \cite{Th3}, as was pointed out in Remark \ref{rem09-01-18}, and also the digraphs obtained from trees. Before we consider the latter class we describe first a procedure by which a row-finite meager digraph can be modified keeping the resulting graph meager and with the same end space. 

\begin{example}\label{att} (Attaching finite digraphs.) Let $\Gamma$ be a row-finite digraph.
For each vertex $v \in \Gamma_V$ we choose a finite strongly connected digraph $H^v$ and vertex $u_v \in H^v_V$. Let $\Gamma'$ be the digraph whose vertex set is
$$
\Gamma'_V = \left\{ (v,x): \ v \in \Gamma_V, \ x \in H^v_V \right\} ,
$$
and whose arrows are
$$
\Gamma'_{Ar} = \Gamma_{Ar} \sqcup_{v \in \Gamma_V} H^v_{Ar}
$$
with range and source maps $r',s' : \Gamma'_{Ar} \to \Gamma'_V$ given by
$$
r'(a) = \left(r(a),u_{r(a)}\right), \ s(a) = \left(s(a),u_{s(a)}\right)
$$
when $a \in \Gamma_{Ar}$ and 
$$
r'(b) = (v,r(b)), \ s'(b) = (v,s(b))
$$
when $b \in H^v_{Ar}$. There is then an embedding $\Gamma \subseteq \Gamma'$ obtained by sending $v \in \Gamma_V$ to $(v,u_v) \in \Gamma'_V$ and $a \in \Gamma_{Ar}$ to itself. It is easy to see that $\Rray(\Gamma) = \Rray(\Gamma')$, $\mathcal E(\Gamma) = \mathcal E(\Gamma')$, and that $\Gamma'$ is meager when $\Gamma$ is.

\end{example}

\begin{example}\label{trees} (Digraphs from trees.) Let $T$ be countable tree, i.e. a countable connected undirected graph without loops. Turn $T$ into a directed graph by exchanging every edge in $T$ with two arrows going in opposite directions. That is, we exchange every edge $
\begin{xymatrix}{
\ar@{-}[r] & }
\end{xymatrix}
$
in $T$ by 
\begin{equation*}
\begin{xymatrix}{
 \ar@/^/[r]& \ar@/^/[l] }
\end{xymatrix}
\end{equation*}
The resulting digraph $\Gamma$ is strongly connected and meager. In Theorem \ref{08-01-18d} it is assumed that the graph has at most finitely many infinite emitters, but this is only to ensure the validity of Lemma \ref{07-01-2018}. For the graph $\Gamma$ the conclusion of that lemma is satisfied because every ray is reduced. Hence the conclusion in Theorem \ref{08-01-18d} is valid for $\Gamma$, also when $T$ and $\Gamma$ have infinitely many infinite emitters. This includes trees $T$ without infinite paths. When $T$ is such a tree there are no wandering paths, i.e. $\Wan(\Gamma) = \emptyset$, and by Lemma \ref{12-11-17x} there will be no $e^{\beta F}$-conformal measures at all in the transient case. So assume $T$ has infinite paths and that we are in the transient case, i.e. $F$ is  a potential and $\beta$ a real number such that the matrix $A(\beta)$ is transient. For any ray $p \in \Ray(\Gamma)$ one easily sees that
\begin{equation}\label{14-03-18c}
\mathbb V_{\beta}(s(p),p) = \sum_{n=0}^{\infty} A(\beta)^n_{s(p),s(p)} \ ,
\end{equation}
which implies that all ends $E \in \mathcal E(\Gamma)$ are $\beta$-summable, i.e. $\mathcal E_{\beta}(\Gamma) = \mathcal E(\Gamma)$ for all $\beta$ in the transient range. It follows that as a convex set, the set of normalized $e^{\beta F}$-conformal measures are always the same when $A(\beta)$ is transient; namely the set of Borel probability measures on the end space $\mathcal E(\Gamma)$.

This can also be deduced from the theory of random walks on trees via the method described in Section \ref{randomw} because the minimal Martin boundary is known for all transient random walks on trees. To apply the random walk theory all we need is to establish the existence of an $A(\beta)$-harmonic vector in the transient case, so that the Doob-type transform \eqref{doobs} can be used to obtain a stochastic matrix. In the general case the existence of an $A(\beta)$-harmonic vector follows from Theorem \ref{24-11-17}, but when $T$ is locally finite and $\Gamma$ is row-finite it follows from a fundamental result by Pruitt, \cite{Pr}. Using \eqref{doobs} it is then straightforward to obtain most, if not all, information regarding $A(\beta)$-harmonic vectors from the well-understood random walks on trees. See Theorem 9.22 in \cite{Wo2}. To have an almost complete picture for digraphs obtained from trees in this way what remains is to decide if $A(\beta)$ is transient or not; a problem which is well-known and hard already for stochastic matrices. See e.g. Chapter 9 in \cite{Wo2}.

\end{example}

\begin{example}\label{dihdral} (A Cayley graph for the infinite dihedral group.) Very few Cayley graphs are meager, but the Cayley graph of the infinite dihedral group which was considered in 7.2 of \cite{CT3} is. The infinite dihedral group is generated by two elements, $a$ and $b$, subject to the conditions that $bab = a^{-1}$ and $b^2 = 1$. The corresponding Cayley graph $\Gamma$ looks as follows.

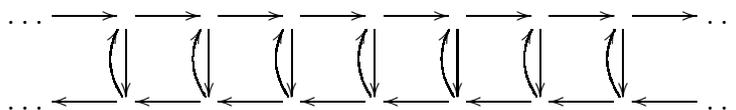
\centerline{
\begin{xymatrix}{ 
\hdots  \ar[r] &  \ar[d]\ar[r] &  \ar[d]
    \ar[r] &
     \ar[d] \ar[r] &   \ar[r] \ar[d] &  \ar[r] \ar[d]  &  \ar[d]
    \ar[r]  &  \ar[r] \ar[d] & \hdots \\
\hdots & \ar[l]  \ar@/^/[u] &   \ar[l]  \ar@/^/[u] &
\ar[l]  \ar@/^/[u] &  \ar[l]  \ar@/^/[u]  &  \ar[l] 
\ar@/^/[u] &  \ar[l]  \ar@/^/[u] &  \ar[l]  \ar@/^/[u] & \ar[l]\hdots
}\end{xymatrix}}

\bigskip

For the constant potential $F =1$ the $A(\beta)$-harmonic vectors were determined in \cite{CT3} by bare hands. For a general potential function on $\Gamma$, assume that $A(\beta)$ is transient. Up to tail equivalence the graph $\Gamma$ contains exactly two rays, they are both reduced and they represent distinct ends of which there are two. In particular, $\Gamma$ is a meager graph. Note that \eqref{14-03-18c} is valid for all rays $p$ in $\Gamma$, showing that both ends are $\beta$-summable whenever $A(\beta)$ is transient. Hence $\mathcal E_{\beta}(\Gamma) = \mathcal E(\Gamma)$ for all $\beta$ in the transient range, exactly as was found in \cite{CT3} for the constant potential.  
\end{example}

\section{Bratteli diagrams and end spaces}\label{bratdiag}

\subsection{The end space of a Bratteli diagram}

In the theory of $C^*$-algebras, and nowadays also in the theory of dynamical systems, there is a very important class of directed graphs called \emph{Bratteli diagrams} after Ola Bratteli who introduced them in \cite{Br1} for the study of AF-algebras; approximately finite-dimensional $C^*$-algebras. A Bratteli diagram is a row-finite directed graph $\Br$\label{Br} whose vertex set $\Br_V$ is partitioned into \emph{level sets},
$$
\Br_V = \sqcup_{n=0}^{\infty} \Br_{n} \ ,
$$
where the number of vertexes in the $n$'th level $\Br_{n}$ is finite, $\Br_{0}$ consists of a single vertex $v_0$ and the arrows emitted from $\Br_n$ end in $\Br_{n+1}$, i.e. $r\left(s^{-1}(\Br_n)\right) \subseteq \Br_{n+1}$ for all $n$.  Also, as is customary, we assume that $v_0$ is the only source in $\Br$ and that there are no sinks. The unital AF-algebra defined from a Bratteli diagram $\Br$ is a unital inductive limit of finite dimensional $C^*$-algebras, cf. \cite{Br1}, which we denote by $AF(\Br)$. It is isomorphic to the corner 
$$
P_{v_0}C^*(\Br)P_{v_0} \ ,
$$
when $C^*(\Br)$ is the graph $C^*$-algebra defined by $\Br$. In particular, $AF(\Br)$ and $C^*(\Br)$ are stably isomorphic since $P_{v_0}$ is a full projection in $C^*(\Br)$, cf. \cite{B}.

\begin{defn}\label{primBr} A Bratteli diagram $\Br$ is \emph{primitive} when the following holds: For any two vertexes $v,w \in \Br_V$ there are paths $\mu_v$ and $\mu_w$ in $\Br$ such that $s(\mu_v) = v, \ s(\mu_w) = w$ and $r(\mu_v) = r(\mu_w)$. 
\end{defn}

As shown by Bratteli in Corollary 3.9 of \cite{Br1}, a Bratteli diagram is primitive if and only if $AF(\Br)$ is a primitive (or prime) $C^*$-algebra.

\begin{defn}\label{simpleBr} A Bratteli diagram $\Br$ is \emph{simple} when the following holds: For every vertex $v \in \Br_V$ there is an $n \in \mathbb N$ such that for all $w \in \Br_{n}$ there is a path $\mu_{v,w}$ in $\Br$ with $s(\mu_{v,w}) = v$ and $r(\mu_{v,w}) = w$. 
\end{defn}

Thus a simple Bratteli diagram is also primitive, but the converse is not true. As shown by Bratteli in Corollary 3.5 of \cite{Br1}, a Bratteli diagram is simple if and only if $AF(\Br)$ is simple, which happens if and only of $C^*(\Br)$ is simple.

Let $E \in \mathcal E(\Br)$ be an end in the Bratteli diagram $\Br$. Set
\begin{equation}\label{30-09-17d}
I_E = \Br_V \backslash \left\{s(p) : \ p \in E \right\} \ .
\end{equation}

\begin{lemma}\label{27-09-17} Let $\Br$ be a Bratteli diagram. The map $E \to I_E$ from \eqref{30-09-17d} is a bijection between $\mathcal E(\Br)$ and the subsets $I \subseteq \Br_V$ with the following properties:
\begin{itemize}
\item[a)] $I \neq \Br_V$ (i.e. $I$ is a proper subset of $\Br_V$).
\item[b)] $s(a) \in I \Rightarrow r(a) \in I$ for all $a \in \Br_{Ar}$ (i.e. $I$ is hereditary).
\item[c)] $r\left(s^{-1}(v)\right) \subseteq I \Rightarrow v \in I$ (i.e. $I$ is saturated).
\item[d)] For all pairs $v,w \in \Br_V \backslash I$ there are paths $\mu_v$ and $\mu_w$ in $\Br$ such that $s(\mu_v) = v, \ s(\mu_w) = w$ and $r(\mu_v) = r(\mu_w) \in  \Br_V \backslash I$. 
\end{itemize}
\end{lemma}
\begin{proof} It is straightforward to check that $I_E$ has the four properties for any end $E$. Consider then a subset $I \subseteq \Br_V$ having the four properties. Then $v_0 \notin I$ since otherwise b) will contradict a). Similarly, $\Br_{n} \backslash I$ can not be empty for any $n$ since otherwise b) and c) would imply that $I= \Br_V$, contradicting a). Using d) successively we construct a sequence $k_1 < k_2 < \cdots $ in $\mathbb N$ and vertexes $w_n \in \Br_{k_n}\backslash I$ such that for all elements $v \in \Br_{k_n}\backslash I$ there is a path in $\Br$ from $v$ to $w_{n+1}$. There is then an infinite path $p \in P(\Br)$ emitted from $v_0$ and passing through all the $w_n$'s in order. As a consequence of its construction, and c), it follows that $p$ will have the property that all elements of $\Br_V \backslash I$ can reach $p$, while b) implies that no element of $I$ can. Set then $E_I = \mathcal E(p) \in \mathcal E(\Br)$, and note that $E_I$ only depends on $I$ and not on which path $p$ with the specified property we choose. We leave the reader to check that the two operations, $E \to I_E$ and $I \to E_I$, are each others inverses; i.e. $E_{I_E} = E  \ \text{and} \ I_{E_I} = I \ $.
\end{proof}

 It follows from the work of Bratteli, \cite{Br1}, that the subsets of $\Br_V$ with the properties a)-d) of Lemma \ref{27-09-17} are in bijective correspondence with the primitive ideals of $AF(\Br)$. In the following we will identify the two sets, using Bratteli's results as excuse.

\begin{prop}\label{baadop} The map $E \to I_E$ is a homeomorphims from the end space $\mathcal E(\Br)$ of $\Br$ onto the primitive ideal space of $AF(\Br)$ equipped with the weakest topology which makes the function 
$$
I \ \mapsto \ \left\| q + I \right\|
$$
continuous on the space of primitive ideals for every projection $q \in AF(\Br)$.
\end{prop}
\begin{proof} It follows from Lemma \ref{27-09-17}, together with Lemma 3.2 and Theorem 3.8 in \cite{Br1}, that the map is a bijection. The assertion concerning the topology requires a detour into operator algebra theory, and since we shall not need it we leave the proof to the interested reader.


\end{proof}

\begin{remark}\label{24-01-18} The topology of the primitive ideal space of $AF(\Br)$ considered in Proposition \ref{baadop} is weaker than the Jacobson (or hull-kernel) topology which the primitive ideal space is often equipped with. In general, the two topologies are different; the primitive ideal space of $AF(\Br)$ is not Hausdorff in the Jacobsen topology in general; see e.g. \cite{Br2}.  Since a positive element in an AF-algebra can be approximated in norm by positive elements with finite spectrum (a property known as 'real rank zero'), the topology described in Proposition \ref{baadop} is the so-called Fell-topology introduced in \cite{Fe}.

\end{remark}

 Let $p \in P(\Gamma)$. We denote in the following by $[p]$ the set of paths $q \in P(\Gamma)$ tail-equivalent to $p$, and by $\overline{[p]}$ the closure of $[p]$ in $P(\Gamma)$.

\begin{lemma}\label{01-10-17c} Let $\Br$ be a Bratteli diagram and $p,q \in P(\Br)$ two infinite paths in $\Br$. Then
\begin{itemize}
\item[i)]  $p \to q$ if and only if $\overline{[p]} \subseteq \overline{[q]}$, and
\item[ii)] $p \sim q$ if and only if $\overline{[p]} =\overline{[q]}$.
\end{itemize}
\end{lemma}
\begin{proof} Left to the reader.
\end{proof}

\begin{prop}\label{01-10-17b} Every Bratteli diagram $\Br$ contains a closed end, i.e. an end $E \in \mathcal E(\Br)$ such that $E$ is  closed as a subset of $P(\Br)$.
\end{prop}
\begin{proof} Let $M$ be the set of paths in $\Br$ which start at $v_0$ - the source in $\Br$. Note that $M$ is a compact metric space. Furthermore, when $p,q \in P(\Br)$,
\begin{equation}\label{29-11-17e}
\overline{[p]} \cap M = \overline{[q]} \cap M \ \Rightarrow \  \overline{[p]}  = \overline{[q]}  \ .
\end{equation} 
The collection $\mathcal C$ of closed non-empty subsets $C$ of $M$ that are invariant under tail-equivalence (i.e. $p \in C \Rightarrow [p] \cap M \subseteq C$), is partially ordered by inclusion. It follows then from Zorns lemma that there is a minimal element $C_0$ in $\mathcal C$. Consider a path $p \in C_0$. The minimality of $C_0$ implies that $C_0 = \overline{[p]}\cap M$. We claim that $\mathcal E(p)$ is closed in $P(\Br)$. For this it suffices to show that $\mathcal E(p) \cap M$ is closed in $M$, which will follow if we show that
\begin{equation}\label{29-11-17f}
\mathcal E(p) \cap M = C_0 \ . 
\end{equation}
Let $q \in \mathcal E(p) \cap M$. Since $q \sim p$, it follows from ii) of Lemma \ref{01-10-17c} that $q \in \overline{[p]}$ and hence that $q \in \overline{[p]}\cap M = C_0$. Conversely, if $q\in C_0$ it follows from the minimality of $C_0$ that $\overline{[q]} \cap M = C_0 = \overline{[p]}\cap M$, and then from \eqref{29-11-17e} that $\overline{[q]} = \overline{[p]}$. By ii) of Lemma \ref{01-10-17b} this implies that $q \in \mathcal E(p)$.  
\end{proof}

\subsection{On the end space of a row-finite digraph}\label{gammabrat} Let $\Gamma$ be an infinite row-finite directed graph with a vertex $v_0$ from where all other vertexes can be reached. When $D \subseteq \Gamma_V$ and $w \in D$, set\label{OmegaD}
$$
\Omega_D(w) = \left\{(y_i)_{i=1}^{\infty} \in \Wan(\Gamma) : \ s(y_1) = w, \ s(y_i) \notin D, \ i \geq 2 \right\}  \ ,
$$
and\label{partD}
$$
\partial D = \left\{ w \in D: \ \Omega_D(w) \neq \emptyset \right\} \ .
$$
As in Section \ref{7.2} we choose a sequence $\{v_0\}  \sqsubset  D_1 \sqsubset D_2 \sqsubset D_3 \sqsubset \cdots$ of finite subsets of $\Gamma_V$ such that $\bigcup_n D_n = \Gamma_V$. Set $D_0 = \{v_0\}$ and $D_{-1} = \emptyset$. Define a Bratteli diagram $\Br(\Gamma)$\label{BrGamma} in the following way. The $n$'th level set $\Br(\Gamma)_{n}$ of vertexes in $\Br(\Gamma)$ consists of the set 
$$
\partial D_n = \left\{ w \in D_n : \ \Omega_{D_n}(w) \neq \emptyset \right\}
$$
and there is an arrow from $v \in \partial D_n$ to $w \in \partial D_{n+1}$ when $L_{D_n}(v,w)  \neq \emptyset$. There is then at most one arrow from a vertex to another in $\Br(\Gamma)$, and we consider therefore paths in $\Br(\Gamma)$ as sequences of vertexes instead of sequences of arrows. We define a map
\begin{equation}\label{pi-map}
\pi : \Wan(\Gamma) \to P(\Br(\Gamma))
\end{equation}
such that 
$$
\pi(y) = \left(\pi(y)_i\right)_{i=1}^{\infty} \in \left(\Br(\Gamma)_V\right)^{\mathbb N}
$$ 
where the vertexes $\pi(y)_i \in \Br(\Gamma)_V$ are defined recursively in the following way. First $\pi(y)_1 = w \in \partial D_k = \Br(\Gamma)_k$ is determined by the following two conditions:
\begin{enumerate}
\item[i)] $s(y_i) \notin D_{k-1} \ \forall i$,
\item[ii)] $\exists j \in \mathbb N : \ s(y_j) = w , \ s(y_i) \notin D_k , \ \forall i > j$. 
\end{enumerate}
For $i > 1$ define $\pi(y)_i = w' \in \partial D_{k+i-1}$ such that $s(y_j) = w'$ for some $j$ and $s(y_l) \notin D_{k+i-1} \ \forall l > j$. There is then a sequence $i_1 < i_2 < i_3 < \cdots $ in $\mathbb N$ such that 
\begin{enumerate}
\item[a)]  $s(y_i) \notin D_{k-1} \ \forall i $,
\item[b)] $s\left(y_{i_j}\right) = \pi(y)_j \in  \partial D_{k + j -1}, \ j \geq 1$,
\item[c)] $s(y_i) \notin D_{k+j-1} \ \forall i >  i_j$.
\end{enumerate}

\begin{lemma}\label{04-04c}

\begin{itemize}
\item[i)] $\pi : \Wan(\Gamma) \to P(\Br(\Gamma))$ is surjective.
\item[ii)] Let $x, y \in \Wan(\Gamma)$. Then $x \to y$ in $\Wan(\Gamma)$ if and only if $\pi(x) \to \pi(y)$ in $P(\Br(\Gamma))$.
\item[iii)] Let $x, y \in \Wan(\Gamma)$. Then $x \sim y$ in $\Wan(\Gamma)$ if and only if $\pi(x) \sim \pi(y)$ in $P(\Br(\Gamma))$.
\end{itemize}
\end{lemma}
\begin{proof} i) : Consider a path $p = (v_i)_{i =1}^{\infty}$ in $\Br(\Gamma)$. Then $v_i \in  \partial D_{k+i-1}$ for some $k \in \mathbb N$ and there is an element $\mu_i \in L_{D_{k+i-1}}\left(v_i, v_{i+1}\right)$ for all $i$. The infinite concatenation
$$
y = \mu_1\mu_2\mu_3 \cdots 
$$
is an element $y \in \Wan(\Gamma)$ such that $\pi(y) = p$.

ii): Consider paths $x,y \in \Wan(\Gamma)$ such that $x \to y$ in $\Wan(\Gamma)$. There is a $k_0 \in \mathbb N$ and sequences $i_{k_0} < i_{k_0+1}< i_{k_0+2} < \cdots$ and $j_{k_0} < j_{k_0+1} < j_{k_0+2} < \cdots$ in $\mathbb N$ such that $s(x_{i_k}) = \pi(x)_k \in \partial D_k$ and $s(y_{j_k}) = \pi(y)_k \in \partial D_k$ for all $k \geq k_0$. For each $k \geq k_0$ there is a path $\mu \in P_f(\Gamma)$ such that $\mu \cap D_k = \emptyset$, $s(\mu) = s(x_{l'})$ and $r(\mu) = s(y_{j_l})$ for some $l' > i_k$ and $l > k$. The concatenation 
$$
z = x[1,l'[\mu y[j_l,\infty[
$$
is an element in $\Wan(\Gamma)$ and the path $\pi(z) \in P(\Br(\Gamma))$ will contain a finite path in $\Br(\Gamma)$ from $\pi(x)_k$ to $\pi(y)_l$. Since $k \geq k_0$ was arbitrary, it follows that $\pi(x) \to \pi(y)$ in $P(\Br(\Gamma))$. Conversely, assume that $\pi(x) \to \pi(y)$ in $P(\Br(\Gamma))$. For each $k \geq k_0$ there is then a path in $\Br(\Gamma)$ from $\pi(x)_k \in \partial D_k$ to $\pi(y)_l \in \partial D_l$ for some $l > k$. By definition of $\Br(\Gamma)$ this gives us a path $\mu \in P_f(\Gamma)$ such that $\mu \cap D_{k-1} = \emptyset, \ s(\mu) = s(x_{i_k})$ and $r(\mu) = s(y_{j_l})$. This shows that $x \to y$ in $\Wan(\Gamma)$. iii) follows from ii).

\end{proof}

It follows from Lemma \ref{04-04c} that $\pi$ induces a bijection $[\pi] : \mathcal E(\Gamma) \to \mathcal E(\Br(\Gamma))$ defined such that 
\begin{equation*}\label{17-04c}
[\pi](\mathcal E(p)) = \mathcal E(\pi(p))  = \pi\left(\mathcal E(p)\right)
\end{equation*}
when $p \in \Wan(\Gamma)$.

\begin{prop}\label{01-10-2017end} $[\pi] : \mathcal E(\Gamma) \to \mathcal E(\Br(\Gamma))$ is a homeomorphism.
\end{prop}
\begin{proof} It remains to prove that $[\pi]$ is open and continuous. Consider finite sets $I \subseteq F \neq \emptyset$. Choose $k$ such that $F \sqsubset D_k$. When $p \in U_{F;I}$, it follows from
Lemma \ref{nov3} that
$p \in U_{D_k; J} \subseteq U_{ F; I}$
when we take $J = [D_k]_p$. Note that $J \subseteq \partial D_k$. Since
$$
 [\pi]\left(\left\{\mathcal E(q): \ q \in U_{D_k;J}\right\}\right) = \left\{\mathcal E(q') : \ q' \in U_{\partial D_k; J} \subseteq P(\Br(\Gamma)) \right\}  
 $$ 
is open in $\mathcal E(\Br(\Gamma))$, it follows that $ [\pi]\left(\left\{\mathcal E(p): \ p \in U_{F;I}\right\}\right)$ is open in $\mathcal E(\Br(\Gamma))$; i.e. $[\pi]$ is open. Let then $I \subseteq F \neq \emptyset$ be finite subsets of $\Br(\Gamma)_V$. Choose $N \in \mathbb N$ such that
$$
F \subseteq \bigcup_{n=0}^{N} \partial D_n \ . 
$$ 
Let $p \in U_{F;I} \subseteq P(\Br(\Gamma))$ and set $J = [\partial D_N]_p$. Then $p \in U_{\partial D_N; J} \subseteq  U_{F;I}$ by Lemma \ref{nov3}. Note that
$\pi^{-1}\left( U_{\partial D_N; J}\right) =  U_{D_N; J}$,
which implies that
$$
[\pi]^{-1}\left( \left\{ \mathcal E(p) : \ p \in  U_{\partial D_N; J} \right\} \right) = \left\{\mathcal E(q) : \ q \in  U_{D_N; J} \right\} \  .
$$
It follows that $[\pi]^{-1}\left( \left\{ \mathcal E(p) : \ p \in  U_{F; I} \right\} \right)$ is open in $\mathcal E(\Gamma)$, and hence that $[\pi]$ is continuous.
\end{proof}

\subsection{The minimal ends}
 Let $E \in \mathcal E(\Gamma)$ be an end in $\Gamma$. Set
$$
E_{D_k} = \left\{ [D_k]_p : \ p \in E \right\} \subseteq \partial D_k \ .
$$
Let $\Br(E)$ be the Bratteli diagram with $E_{D_n}$ as $n$'th level set of vertexes and an edge from $v \in E_{D_n}$ to $w \in E_{D_{n+1}}$ when $L_{D_n}(v,w) \neq \emptyset$. $\Br(E)$ is a subgraph of $\Br(\Gamma)$.

\begin{lemma}\label{prim} $\Br(E)$ is a primitive Bratteli diagram.
\end{lemma}
\begin{proof} Let $v \in E_{D_n},\  w \in E_{D_m}$. Let $p \in E$. Since $v \overset{D_n}{\to} p$ and $w \overset{D_m}{\to} p$ there is a $k > \max\{n,m\}$ and $j \in \mathbb N$ such that $s(p_j) \in [D_k]_p$, $s(p_i) \notin D_k, \ i > j$, and such that $L_{D_n}(v,s(p_j)) \neq \emptyset$ and $L_{D_m}(w,s(p_j)) \neq \emptyset$. There are therefore also paths in $\Br(\Gamma)$ from $v$ to $s(p_j)$ and from $w$ to $s(p_j)$. These paths proceed in $\Br(E)$ and we have therefore verified the condition in Definition \ref{primBr}.
\end{proof}

\begin{lemma}\label{30-09-17g} The closure $\overline{\pi(E)}$ of $\pi(E)$ in $P(\Br(\Gamma))$ is $P\left(\Br(E)\right)$.
\end{lemma} 
\begin{proof} It is clear that $\overline{\pi(E)} \subseteq P(\Br(E))$ as the latter set is closed in $P(\Br(\Gamma))$. Let $q = (q_i)_{i=1}^{\infty} \in P\left(\Br(E)\right)$. Let $n \in \mathbb N$. It suffices to find a path $p'\in \pi(E)$ such that $q_i= p'_i, \ i \leq n$. For this let $k \in \mathbb N$ be such that $q_1 \in E_{D_k}$. Then $q_n \in [D_{k+n-1}]_p$ where $p \in E$. Furthermore, by definition of $\Br(\Gamma)$ there are paths $\mu_i \in L_{D_{k+i-1}}(q_i,q_{i+1}), i = 1,2,\cdots, n-1$, such that $s(\mu_i) = q_i$ and $r(\mu_i) = q_{i+1}$ for all $i$. Since $q_n \overset{D_{k+n-1}}{\to} p$, there is a $l \in \mathbb N$ and a finite path $\mu$ in $\Gamma$ such that $s(\mu) = q_n$, $r(\mu) = s(p_l)$ and such that only the initial vertex, which is $q_n$, in the path $\mu p_{[l,\infty)}$ is from $D_{k+n-1}$. The concatenation
$$
p' = \mu_1\mu_2 \cdots \mu_{n-1}\mu p_{[l,\infty)}
$$
is then an element of $P(\Gamma)$ which is tail-equivalent to $p$; in particular $p' \in E$. Since $\pi(p')_i = q_i$ for $i \leq n$, the proof is complete.
\end{proof}

When $p,q \in \Wan(\Gamma)$, write
$$
\mathcal E(p) \leq \mathcal E(q)
$$
when $p \to q$. This defines a partial ordering of $\mathcal E(\Gamma)$. An end $E \in \mathcal E(\Gamma)$ is \emph{minimal} when it is minimal for this partial order; i.e. when $p \to q \in E \  \Rightarrow \ \mathcal E(p) = E$.

\begin{prop}\label{10-09-17} Let $E \in \mathcal E(\Gamma)$ be an end in $\Gamma$.
 The following are equivalent:
\begin{enumerate}
\item[a)] $\pi(E)$ is closed in $P(\Br(\Gamma))$.
\item[b)] $\Br(E)$ is a simple Bratteli diagram. 
\item[c)] $E$ is minimal.
\end{enumerate}
When these conditions hold, $\pi(E) = P(\Br(E))$.
\end{prop}

\begin{proof} a) $\Rightarrow$ b): Let $v \in E_{D_k}$. We must show that there is an $l > k$ such that all elements of $E_{D_l}$ can be reached (in $\Br(E)$) from $v$. To do this by contradiction, assume that for all $l > k$ there is a vertex $w_l \in E_{D_l}$ which can not be reached in $\Br(E)$ from $v$. For each $l > k$ there \emph{is} a path in $\Br(E)$ from a vertex in $E_{D_k}$ to $w_l$ since $w_l \in E_{D_l}$. This path goes through elements in $ E_{D_k}, E_{D_{k+1}}, E_{D_{k+2}}, \cdots, E_{D_{l-1}}$ none of which can be reached from $v$. Since we can do this for all $l > k$ and since there are only a finite number of arrows between successive levels in $\Br(\Gamma)$ it follows that there is an infinite path $z = (z_i)_{i=0}^{\infty}  \in P(\Br(\Gamma))$ such that for all $i$ the vertex $z_i \in E_{D_{k+i}}$ can not be reached from $v$. Let $p \in E$. For all $n \in \mathbb N$ there is a path $z^n$ in $\Br(E)$ such that $z^n_i = z_i $ for all $i \leq n$ and such that $z^n_j = \pi(p)_j$ for all sufficiently large $j$. Then $z^n \sim  \pi(p)$ and it follows from Lemma \ref{04-04c} that $z^n \in \pi(E)$. Thus $z$ can be approximated arbitrarily well with elements from $\pi(E)$ and it follows, since we assume that $\pi(E)$ is closed, that $z \in \pi(E)$. By Lemma \ref{04-04c} this implies that $z \sim \pi(p)$ in $P(\Br(\Gamma))$, which gives the desired contradiction because $v$, as an element of $E_{D_k}$, is able to reach $\pi(p)_j$ and hence also $z_j$ for infinitely many $j$.

b) $\Rightarrow$ c) : Let $p,q \in \Wan(\Gamma)$ such that $\mathcal E(p) = E$ and $q \to p$. Assume for convenience that $s(p) = s(q) = v_0$. Since $q \to p$ we see that $\pi(q)_n \in [D_n]_p$ for all $n$; i.e. $\pi(q) \in P(\Br(E))$. Since we assume that $\Br(E)$ is a simple Bratteli diagram it follows that $\pi(p) \to \pi(q)$ in $P(\Br(E)) \subseteq P(\Br(\Gamma))$, and ii) of Lemma \ref{04-04c} implies that $p \to q$. Hence $p \sim q$. It follows that $E$ is minimal.

c) $\Rightarrow$ a): Let $q = (q_i)_{i=0}^{\infty} \in \Wan(\Gamma)$ such that $\pi(q) \in \overline{\pi(E)}$ and let $p \in  E$. For each $n$ there is an element $q^n = (q^n_i)_{i=0}^{\infty} \in E$ such that $\pi(q^n)_i = \pi(q)_i$ for all $i \leq n$. Since $\pi(q^n) \sim \pi(p)$ there is an $m > n$ and a path in $\Br(\Gamma)$ from $\pi(q)_n = \pi(q^n)_n$ to $\pi(p)_m$. Since this is true for all $n$, it follows that $\pi(q)  \to \pi( p)$. By ii) in Lemma \ref{04-04c}, it follows from this that $q\to p$; i.e. $\mathcal E(q) \leq \mathcal E(p) =E$ in $\mathcal E(\Gamma)$. It follows from the assumed minimality of $E$ that $q \sim p$, i.e. $\pi(q) \in \pi(E)$, and we conclude that $\pi(E)$ is closed in $P(\Br(\Gamma))$.

The last statement follows from a) and Lemma \ref{30-09-17g}.
\end{proof}

\begin{remark}\label{14-02-18} An end $E \in \mathcal E(\Gamma)$ which is closed in $P(\Gamma)$ is also minimal, but the converse is not true. In the graph $\Gamma_1$ in Example \ref{28-09-17}, for example, both ends are minimal but none of them is closed.
\end{remark}

The next corollary will not be needed in the following, but it is included to illuminate the relation between minimal and general ends. It shows also that Proposition \ref{01-10-17b} could be obtained from the well-known fact that a unital AF-algebra, as indeed all unital $C^*$-algebras, has a simple quotient. In any case Zorn's lemma and the axiom of choice is at core of it.

\begin{cor}\label{30-10-17} There is a bijection between the minimal ends in $\Gamma$ and the maximal ideals of $AF(\Br(\Gamma))$.
\end{cor}
\begin{proof} This follows from Propositions \ref{baadop}, \ref{01-10-2017end} and \ref{10-09-17} since 
$$
AF(\Br(\Gamma))/I_{\pi(E)} \ \simeq \ AF(\Br(E)) \ 
$$
for all ends $E$, cf. \cite{Br1}.
\end{proof}

\begin{prop}\label{11-09-17a} Let $E \in \mathcal E(\Gamma)$. There is a minimal end $E_{min} \in \mathcal E(\Gamma)$ such that $E_{min} \leq E$.
\end{prop}
\begin{proof} By Proposition \ref{01-10-17b} the Bratteli diagram $\Br(E)$ contains a closed end $C$. Note that if $p,q \in P\left(\Br(\Gamma)\right)$, $q \in P(\Br(E))$ and $p \to q$ in $P(\Br(\Gamma))$, the definition of $\Br(E)$ implies that $p \in P(\Br(E))$. It follows therefore that $C$ is also an end in $\Br(\Gamma)$, and by Proposition \ref{01-10-2017end} there is therefore an end $E_{min}$ in $\Gamma$ such that $\pi\left(E_{min}\right) = C$. By Proposition \ref{10-09-17} $E_{min}$ is a minimal end in $\Gamma$.  Let $q \in E_{min}$ and $p \in E$. Since $\pi(q) \in P(\Br(E))$ it follows from the definition of $\Br(E)$ that $\pi(q) \to \pi(p)$ in $P(\Br(\Gamma))$ and then from ii) in Lemma \ref{04-04c} that $q \to p$ in $P(\Gamma)$. This shows that $E_{min} \leq E$.
\end{proof}

By Proposition \ref{11-09-17a} all infinite row-finite digraphs with a vertex from where all vertexes can be reached contains a minimal end. For almost undirected digraphs much more is true:

\begin{lemma}\label{almost/min} Let $\Gamma$ be a strongly connected row-finite almost undirected digraph. It follows that all ends of $\Gamma$ are minimal.
\end{lemma}
\begin{proof} It suffices to show that $p \to q \Rightarrow q \to p$ for $p,q \in \Wan(\Gamma)$. Consider a finite subset $F \subseteq \Gamma_V$ and a natural number $K$. In the notation from the proof of Theorem \ref{28-09-17=}, choose $n \in \mathbb N$ such that $F \subseteq D_{n-N}$. Since $p \to q$ there are natural numbers $k,l > K$ such that $[s(p_k)]_n = [s(q_l)]_n$ in $\mathcal C\left(\Gamma^{(n)}\right)$. It follows then from Observation \ref{30-09-17obs} that there is a path $\mu \in P_f(\Gamma)$ such that $\mu \cap F = \emptyset$, $s(\mu) = s(q_l)$ and $r(\mu) =s(p_k)$. Since $F$ and $K$ are arbitrary, it follows that $q \to p$.
\end{proof}

\begin{remark}\label{01-10-17h} The minimal end $E_{min}$ in Proposition \ref{11-09-17a} may not be unique. From an $C^*$-algebra perspective this is not surprising since a primitive $AF$-algebra may very well have many simple quotients. As a specific example consider the digraph $\Gamma^0$ in Remark \ref{18-02-18}. The ray $p^0$ represents an end which dominates the two distinct minimal ends represented by the rays $p^{\pm}$. If we add return paths to $v_0$ as explained in Section \ref{adding} we obtain strongly connected examples to the same effect.
\end{remark}

Let $\mathcal E_{min}(\Gamma)$ denote the set of minimal ends in $\mathcal E(\Gamma)$; it is a Borel subset of $\mathcal E(\Gamma)$.

\subsection{Harmonic measures on $P(\Gamma)$ and $P(\Br(\Gamma))$}
 We retain the assumption that $\Gamma$ is an infinite row-finite directed graph with a vertex $v_0$ such that \eqref{v0trans} holds, and that $\{v_0\}  \sqsubset  D_1 \sqsubset D_2 \sqsubset D_3 \sqsubset \cdots$ is a sequence of finite subsets of $\Gamma_V$ such that $\bigcup_n D_n = \Gamma_V$. And we set $D_0 = \{v_0\}, \ D_{-1} = \emptyset$.
 
 Let $F :  \Gamma_{Ar} \to \mathbb R$ be a potential function and assume that \eqref{v0trans} holds.
When $D$ and $D'$ are finite subsets of $\Gamma_V$ with $D \sqsubset D'$, define\label{MDD}
$$
M(D;D') : \partial D \times \partial D' \to [0,\infty)
$$
such that 
$$
M(D;D')_{v,w} = \sum_{\mu \in L_D(v,w)}e^{-\beta F(\mu)} \ .
$$
In particular, for each $n \in \mathbb N$, we can define\label{Mn}
\begin{equation*}\label{19-12-17a}
M(n) = M(D_{n};D_{n+1})  \ .
\end{equation*}
Let $m$ be a $v_0$-normalized $e^{\beta F}$-conformal measure on $P(\Gamma)$. For $w \in \partial D_k$, set
\begin{equation}\label{dsb}
\psi^k_w = m\left(\Omega_{D_k}(w)\right)  \ .
\end{equation}
It follows from Lemma \ref{nov2} and Lemma \ref{12-11-17x} that
\begin{equation}\label{05-04a}
\sum_{w \in \partial D_{k+1}} M(k)_{v,w}\psi^{k+1}_w = \psi^k_v
\end{equation}
for all $v \in \partial D_k$ and all $k =0,1,2, \cdots $, and
$$
1= m(Z(v_0)) = \left( \sum_{n=0}^{\infty} A(\beta)^n_{v_0,v_0}\right)\psi^0_{v_0} \ .
$$  
\begin{defn}\label{31-10-2017} In the following we let $\Delta_{\beta F}$ denote the set of sequences
$$
\left(\psi^k\right)_{k=1}^{\infty} \in \prod_{k=1}^{\infty} [0,\infty)^{\partial D_k} 
$$
such that 
\begin{itemize}
\item $M(0)\psi^1 = \left(\sum_{k=0}^{\infty} A(\beta)^k_{v_0,v_0}\right)^{-1} $,  and 
\item $M(n)\psi^{n+1} = \psi^{n}$ for $n \geq 1$.
\end{itemize}
\end{defn}
Then every normalized $e^{\beta F}$-conformal measure on $P(\Gamma)$ gives rise to an element $\psi \in \Delta_{\beta F}$ given by \eqref{dsb}. We aim to show that this map from normalized $e^{\beta F}$-conformal measures to $\Delta_{\beta F}$ is a bijection.

For $k =0,1,2, \cdots$, set 
$$
L_k \ = \bigcup_{v \in \partial D_k,\ w\in \partial D_{k+1}} L_{D_k}(v,w) 
$$
and\label{omegak}
$$
\Omega_k \ = \bigcup_{w \in \partial D_k} \Omega_{D_k}(w)  \ .
$$
For each $n = 0,1,2, \cdots$, the sets of the form
\begin{equation}\label{AA1}
Z\left( \nu \mu_j \mu_{j+1}\mu_{j+2} \cdots \mu_{j+n}\right)\Omega_{j+n+1} \ ,
\end{equation}
for some $j \geq 0$, where $\nu \in P_f(\Gamma)$ is a path with $\nu \cap D_{j-1} = \emptyset$ and $\mu_{j+i} \in L_{j+i}, \ i =0,1,2, \cdots , n$, constitute a countable Borel partition $\mathcal P_n$ of $\Wan(\Gamma)$.

\begin{lemma}\label{Trump6}  Let $\mathcal A_n$ be the $\sigma$-algebra in $\Wan(\Gamma)$ generated by $\mathcal P_n$. Then $\mathcal A_n \subseteq \mathcal A_{n+1}$ and $\mathcal A = \bigcup_{n=0}^{\infty} \mathcal A_n$ is an algebra of sets generating the $\sigma$-algebra of Borel sets in $\Wan (\Gamma)$.
\end{lemma}
\begin{proof} The inclusion $\mathcal A_n \subseteq \mathcal A_{n+1}$ follows from the observation that
\begin{equation}\label{Trump7}
\begin{split}
&Z\left(\nu \mu_{j} \mu_{j+1}\cdots \mu_{j+n}\right)\Omega_{j+n+1} \\
& = \bigcup_{\mu\in L_{j+n+1}} Z\left(\nu \mu_{j}\mu_{j+1} \cdots \mu_{j+n} \mu \right)\Omega_{j+n+2} \ .
\end{split}
\end{equation}
In particular, $\mathcal A$ is an algebra of sets, clearly consisting of Borel sets in $\Wan(\Gamma)$. To show that $\mathcal A$ generates the Borel $\sigma$-algebra of $\Wan(\Gamma)$ it suffices to show that $Z(\nu') \cap \Wan(\Gamma)$ is in the $\sigma$-algebra generated by $\mathcal A$ for every finite path $\nu'$. Since
$$
Z(\nu') \cap \Wan(\Gamma) = \bigcup_{P \in \mathcal P_{|\nu'|}} Z(\nu') \cap P \ ,
$$
it suffices to observe that the intersection
$$
Z(\nu') \cap Z\left(\nu \mu_j \mu_{j+1} \cdots \mu_{j + |\nu'|}\right)\Omega_{j+|\nu'|+1}
$$
is either empty or equal to $Z\left(\nu \mu_j \mu_{j+1} \cdots \mu_{j + |\nu'|}\right)\Omega_{j+|\nu'|+1}$.
\end{proof}


\begin{lemma}\label{Trump8} For every $\psi =\left(\psi^k\right)_{k=1}^{\infty} \in \Delta_{\beta F}$ there is a unique $v_0$-normalized $e^{\beta F}$-conformal measure $m^{\psi}$ on $P(\Gamma)$\label{mpsi2} such that
\begin{equation}\label{17-04}
m^{\psi}\left(\Omega_{D_k}(w)
\right) = \psi^k_w
\end{equation}
for all $w \in \partial D_k$ and all $k = 1,2,3, \cdots $. 
\end{lemma}
\begin{proof} Assume first that $m^{\psi}$ exists. Note that $ m^{\psi}$ is concentrated on $\Wan(\Gamma)$ by Lemma \ref{12-11-17x}. It follows from \eqref{17-04} and Lemma \ref{nov2} that 
\begin{equation*}\label{AA2}
\begin{split}
&m^{\psi}\left( Z( \nu \mu_j\mu_{j+1} \cdots \mu_{j+n})\Omega_{j+n+1}\right) \\
& = \exp\left( - \beta \left( F(\nu) + \sum_{i=0}^n F\left(\mu_{j+i}\right)\right)\right)m^{\psi}\left( \Omega_{D_{j+n+1}}\left(r(\mu_{j+n}\right)\right) \\
&= \exp\left( - \beta \left( F(\nu) + \sum_{i=0}^n F\left(\mu_{j+i}\right)\right)\right) \psi^{j+n+1}_{r(\mu_{j+n})} \ .
\end{split}
\end{equation*}
The uniqueness of $m^{\psi}$ follows therefore from Lemma \ref{Trump6} by using Theorem 1-19 in \cite{KSK}. To prove existence of $m^{\psi}$ note that since $\mathcal A_n$ is the $\sigma$-algebra of the countable partition $\mathcal P_n$ we can define a measure $m_n$ on $\mathcal A_n$ such that
\begin{equation*}\label{Trump13}
\begin{split}
&m_n\left(Z\left(\nu  \mu_j \mu_{j+1} \cdots \mu_{j+n}\right)\Omega_{j+n+1}\right) = e^{-\beta \left(F(\nu) + \sum_{i=j}^{j+n} F(\mu_i)\right)}\psi^{j+n+1}_{r\left(\mu_{j+n}\right)} \  .
\end{split}
\end{equation*}
It is important here to observe that $\mathcal A_n$ is simply the sets that are unions of elements from the countable partition $\mathcal P_n$ and that $m_n$ therefore is countably additive on $\mathcal A_n$. It follows from (\ref{Trump7}) and \eqref{05-04a} that the $m_n$'s are compatible and hence extend to a finitely additive measure $m$ on the algebra $\mathcal A$ of Lemma \ref{Trump6}. For $w \in \Gamma_V$ set
$$
P(w) = \left\{\mu \in P_f(\Gamma): \ s(\mu) = v_0, \ r(\mu) = w \right\}.
$$ 
Then 
\begin{equation*}\label{Trump133}
\begin{split}
&m(Z(v_0) \cap \Wan(\Gamma)) = \sum_{w \in \partial D_1} \sum_{\mu \in P(w)} e^{-\beta F(\mu)} \psi^1_w \\
&= \left(\sum_{n=0}^{\infty} A(\beta)^n_{v_0,v_0}\right) \sum_{w \in \partial D_1} M(0)_{v_0,w}\psi^1_w = 1 \ .
\end{split}
\end{equation*}
Note that for any finite path $\mu$ in $\Gamma$ and any set $A \subseteq Z(r(\mu))$ such that $A \in \mathcal A$, the set
$Z(\mu) A$
is also in $\mathcal A$ and
\begin{equation}\label{11-09-17}
m(Z(\mu)A)= e^{-\beta F(\mu)} m(A) \ .
\end{equation}
Let $v \in \Gamma_V$. Since any vertex in $\Gamma_V$ can be reached from $v_0$, there is a finite path $\mu$ in $\Gamma_V$ such that $s(\mu) =v_0$ and $r(\mu) = v$ and
\begin{equation*}
\begin{split}
&m(Z(v) \cap \Wan(\Gamma)) = e^{\beta F(\mu)}m\left(Z(\mu) \cap \Wan(\Gamma)\right) \\
& \leq e^{\beta F(\mu)} m(Z(v_0) \cap \Wan(\Gamma)) =e^{\beta F(\mu)} < \infty  \ .
\end{split}
\end{equation*}
Hence $\Wan(\Gamma)$ is the union of the countable collection $Z(v) \cap \Wan(\Gamma), v \in \Gamma_V$, of sets from $\mathcal A$, all of finite $m$-measure. Fix a vertex $v \in \Gamma_V$. Then
$$
\mathcal A(v) = \left\{ A \cap Z(v) : \ A \in \mathcal A \right\}
$$
is a field of subsets of $Z(v)$ in the sense of \cite{KSK} and $m$ is an additive set function defined on $\mathcal A(v)$. To show that $m$ is completely additive on $\mathcal A(v)$ in the sense of \cite{KSK}, it suffices to show, by Corollary 1-17 in \cite{KSK}, that if $ A_1 \supseteq A_2 \supseteq A_3 \supseteq \cdots$ is a sequence from $\mathcal A(v)$ such that $\bigcap_n A_n = \emptyset $, then $\lim_{n \to \infty} m(A_n) = 0$.  This follows standard lines, cf. e.g. the proof of Theorem 1.12 in \cite{Wo1}. In the present setting it runs as follows. Assume for a contradiction that $\lim_{n \to \infty} m(A_n) \neq 0$. We construct a sequence $k(1) < k(2) < \cdots $ in $\mathbb N$ and for each $n$ an element $C_n \in \mathcal P_{k(n)}$ such that
\begin{enumerate}
\item[i)] $C_{n+1} \subseteq C_n$,
\item[ii)] $C_n \subseteq A_n$, and
\item[iii)] $\lim_{j \to \infty} m(A_j \cap C_n) > 0$
\end{enumerate}
for all $n$. Choose first the sequence $\{k(n)\}$ such that $A_n \in \mathcal A_{k(n)}$ and $k(n) < k(n+1)$ for all $n$. Since $m$ is a countably additive on $\mathcal A_{k(j)}$ and $m(A_1) < \infty$ we have that
$$
0 < \lim_{j \to \infty}  m(A_j) = \lim_{j \to \infty} \sum_{C \in \mathcal P_{k(1)}} m(A_j \cap C) = \sum_{C \in \mathcal P_{k(1)}} \lim_{j \to \infty}  m(A_j \cap C) \ .
$$
There is therefore a $C_1 \in \mathcal P_{k(1)}$ for which $\lim_{j \to \infty} m(A_j \cap C_1) > 0$. Since $A_1 \in \mathcal A_{k(1)}$ and since $m(A_1 \cap C_1) > 0$ it follows that $C_1 \subseteq A_1$. Then $C_1$ satisfies ii) and iii) for $n =1$. Repeating this argument with $\{A_n\}$ replaced by $\{A_n \cap C_1\}$ we obtain $C_2 \in \mathcal P_{k(2)}$ such that $C_2 \subseteq A_2 \cap C_1$ and i), ii) and iii) hold for $n =1$ while ii) and iii) hold for $n=2$. We continue this way by induction to obtain the sequence $\{C_n\}$. Write
$$
C_n = Z(\nu \mu_j \mu_{j+1} \cdots \mu_{j+k(n)})\Omega_{j+k(n)+1}  \ ,
$$
where $\nu \cap D_{j-1} = \emptyset, \ \mu_{j+i} \in L_{j+i}, \ i =0,1,2, \cdots, k(n)$. Since $C_{n+1} \subseteq C_n$ we get an element $z\in P(\Gamma)$ as the infinite concatenation
$$
z = \nu \mu_j \mu_{j+1} \mu_{j+2} \cdots .
$$
Then $z \in \bigcap_n C_n$, contradicting the assumption that $\bigcap_n A_n = \emptyset$. Hence $m$ is completely additive on $\mathcal A(v)$ and it follows therefore from Theorem 1-19 in \cite{KSK} and Lemma \ref{Trump6} that $m$ extends to a Borel measure $m^{v}$ on $Z(v) \cap \Wan(\Gamma)$. We combine these measures and define
$$
m^{\psi}(B) = \sum_{v \in \Gamma_V} m^v\left(B \cap Z(v) \cap \Wan(\Gamma)\right)  \ .
$$
Note that \eqref{17-04} holds by construction. To show that $m^{\psi}$ is $e^{\beta F}$-conformal consider the following two Borel measures on $Z(a)$ for some $a\in \Gamma_{Ar}$:
$$
B \mapsto m^{\psi}\left(\sigma\left(B\right)\right)
$$
and
$$
B \mapsto e^{\beta F(a)}m^{\psi}(B) \  .
$$ 
It follows from \eqref{11-09-17} that these measures agree when $A$ is a subset of $Z(a)$ such that $A \in \mathcal A(s(a))$. It follows from the uniqueness part of the statement in Theorem 1-19 of \cite{KSK} that they also agree on Borel subsets of $Z(a) \cap \Wan(\Gamma)$, and hence by definition of $m^{\psi}$ on Borel subsets of $Z(a)$.
\end{proof}

\begin{prop}\label{thm?} The map $\psi \mapsto m^{\psi}$ given by Lemma \ref{Trump8} is an affine homeomorphism from $\Delta_{\beta F}$ onto $M^{v_0}_{\beta F}(\Gamma)$.
\end{prop}
\begin{proof} It remains only to show that the  bijection $\psi \mapsto m^{\psi}$ is a homeomorphism. To this end note that when $\psi = \left(\psi^k\right)_{k=1}^{\infty} \in \Delta_{\beta F}$, we have the upper bound
\begin{equation*}
\begin{split} 
& \psi^k_w  \ \leq \  \left( \sum_{n=0}^{\infty} A(\beta)^n_{v_0, w}\right)^{-1} \ ,
\end{split}
\end{equation*}
for all $w \in \partial D_k$ and all $k$. This shows that $\Delta_{\beta F}$ is compact. The continuity of the map $\psi \mapsto m^{\psi}$ follows then from the identity 
\begin{equation*}\label{27-09-17a}
\begin{split}
& m^{\psi}(Z(v)) = \sum_{w \in \partial D_k} \sum_{n=0}^{\infty} A(\beta)^n_{v, w}  m^{\psi}\left(\Omega_{D_k}(w)\right)= \sum_{w \in \partial D_k} \sum_{n=0}^{\infty} A(\beta)^n_{v, w} \psi^k_w \ ,
\end{split}
\end{equation*}
valid for all $k$ and all $v \in D_k$ because $m^{\psi}$ is concentrated on $\Wan(\Gamma)$.

\end{proof}

Given an arrow $a \in \Br(\Gamma)$ with $s(a) = v\in \partial D_k, \ r(a) \in \partial D_{k+1}$, set
$$
F_{\beta}(a) = - \log \left(\sum_{\mu \in L_{D_k}(v,w)} e^{-\beta F(\mu)} \right) \ .
$$
This defines a $\beta$-dependent potential on $\Br(\Gamma)$.

\begin{thm}\label{05-04b}  Assume that $\Gamma$ is an infinite row-finite directed graph and that there is a vertex $v_0 \in \Gamma_V$ such that
$$
0 < \sum_{n=0}^{\infty} A(\beta)^n_{v_0,v} < \infty
$$
for all $v \in \Gamma_V$.  There are affine homeomorphisms between the following sets:
\begin{itemize}
\item[a)] The $v_0$-normalized $e^{\beta F}$-conformal measures on $P(\Gamma)$.
\item[b)] The set $\Delta_{\beta F}$, cf. Definition \ref{31-10-2017}.
\item[c)] The $v_0$-normalized $e^{F_{\beta}}$-conformal measures on $P(\Br(\Gamma))$.
\end{itemize}
\end{thm}
\begin{proof} The affine homeomorphism between the sets of a) and b) is given in Proposition \ref{thm?}, and the affine homeomorphism between the sets of b) and c) follows by applying Proposition \ref{nov1} to $\Br(\Gamma)$ with $\beta F$ replaced by $F_{\beta}$.  
\end{proof}

It follows that under the assumptions of Theorem \ref{05-04b} the problem of finding the $e^{\beta F}$-conformal measures on the path-space of $\Gamma$ can be translated to a similar problem involving the Bratteli diagram $\Br( \Gamma)$ instead. In terms of $C^*$-algebras it follows that there is a bijection between the set of rays of $\beta$-KMS-weights for the action $\alpha^F$ on $C^*(\Gamma)$ and the $1$-KMS states for the generalized gauge action on the AF-algebra $AF(\Br(\Gamma))$ given by the potential $F_{\beta}$; assuming that there are no sinks in $\Gamma$.

\begin{lemma}\label{22-12-17a} Let $m \mapsto m'$ be the affine homeomorphim of Theorem \ref{05-04b} from the $v_0$-normalized $e^{\beta F}$-conformal measures $m$ on $P(\Gamma)$ onto the $v_0$-normalized $e^{F_{\beta}}$-conformal measures $m'$ on $P(\Br(\Gamma))$, and let $E$ be an end in $\Gamma$. Then $m$ is concentrated on $E$ if and only if $m'$ is concentrated on $\pi(E)$.
\end{lemma}
\begin{proof} We need the observation that $\pi: \Wan(\Gamma) \to P(\Br(\Gamma))$ is a Borel map. This follows from Lemma \ref{03-03-18} by observing that $\pi^{-1}(Z(\mu))$ is a Borel subset of $\Wan(\Gamma)$ for every $\mu \in P_f(\Br(\Gamma))$. Furthermore, if we assume that $s(\mu) = v_0$ we find that
$m\left(Z(v_0) \cap \pi^{-1}(Z(\mu))\right) \\
= m'(Z(\mu))$.
It follows therefore from Lemma \ref{03-03-18} that
$$
m\left(Z(v_0) \cap \pi^{-1}(B)\right) =  m'\left(Z(v_0) \cap B\right)
$$
for all Borel subsets $B \subseteq P(\Br(\Gamma))$. In particular, since $E = \pi^{-1}(\pi(E))$ by iii) in Lemma \ref{04-04c}, we find that
\begin{equation*}
m\left(Z(v_0) \backslash E\right) =  m'\left( Z(v_0) \backslash \pi(E) \right) \ .
\end{equation*}
Since $m\left(\Wan(\Gamma) \backslash E\right) = 0$ if and only if $m\left(Z(v_0) \backslash  E\right) = 0$  and 
$$
m' \left(P(\Br(\Gamma)) \backslash \pi(E)\right)  =0
$$
if and only if 
$m'\left( Z(v_0) \backslash  \pi(E) \right) = 0$, we conclude that $m$ is concentrated on $E$ if and only if $m'$ is concentrated on $\pi(E)$.


\end{proof}

When $E \in \mathcal E(\Gamma)$ is an end we let $M^{v_0}_{\beta F}(E)$ be the set of $v_0$-normalized $e^{\beta F}$-conformal measures that are concentrated on $E$; i.e. \label{MbetaFv0}
$$
M^{v_0}_{\beta F} (E) = \left\{ m\in M^{v_0}_{\beta F} (\Gamma) : \ m(P(\Gamma)\backslash E) = 0\right\} \ ;
$$
a possibly empty face in $ M^{v_0}_{\beta F} (\Gamma)$. In view of Theorem \ref{05-04b} and Lemma \ref{22-12-17a} we have the following.

\begin{prop}\label{26-02-18} In the setting of Theorem \ref{05-04b} there is an affine homeomorphism between $M^{v_0}_{\beta F}(E)$ and $M^{v_0}_{F_{\beta}}(\pi(E))$ for all ends $E \in \mathcal E(\Gamma)$.
\end{prop}

Note that $\pi(E)$ is an end in $P(\Br(E))$ and that $\Br(E)$ is a primitive Bratteli diagram by Lemma \ref{prim}. Therefore, when Proposition \ref{26-02-18} is combined with Lemma \ref{02-12-17c} and Proposition \ref{01-10-2017end}, we see that the search for the $e^{\beta F}$-conformal measures on $\Gamma$ can be translated to the problem of finding the $e^{F_{\beta}}$-conformal measures on $\Br(E)$ for each $E \in \mathcal E(\Br(\Gamma))$.

For general ends $E \in \mathcal E(\Gamma)$ the end $\pi(E)$ will only be one out of possibly many ends in $P(\Br(E))$; and $M^{v_0}_{\beta F}(E)$ can easily be empty. Minimal ends which we study next behave much better with respect to both these things.

\begin{thm}\label{12-09-17} Assume that $\Gamma$ is a row-finite directed graph and that there is a vertex $v_0 \in \Gamma_V$ such that
$$
0 < \sum_{n=0}^{\infty} A(\beta)^n_{v_0,v} < \infty
$$
for all $v \in \Gamma_V$. Let $E$ be a minimal end in $\Gamma$. 

\begin{itemize}
\item[i)] The set of normalized $e^{\beta F}$-conformal measures concentrated on $E$ is a non-empty closed face $M^{v_0}_{\beta F}(E)$ in the simplex of all normalized $e^{\beta F}$-conformal measures.
\item[ii)] Set 
$$
\Delta_E = \left\{\psi  \in \Delta_{\beta F} : \ \psi^n_w = 0, \ w \in \partial D_n \backslash E_{D_n} \ \forall n \right\} \ .
$$
Then 
$$
M^{v_0}_{\beta F}(E) = \left\{m^{\psi}: \ \psi \in \Delta_E \right\} \ .
$$
\end{itemize}
\end{thm}
\begin{proof} It follows from Proposition \ref{10-09-17} that $\pi(E) = P(\Br(E))$. Since $\pi(p)_k \notin E_{D_k} \Rightarrow \pi(p)_{k+1} \notin E_{D_{k+1}}$, this implies that
$$
\Wan(\Gamma) \backslash E=  \bigcup_{n \in \mathbb N}  \bigcup_{j \in\mathbb N} \sigma^{-j}\left( \bigcup_{w \in \partial D_n \backslash E_{D_n}} \Omega_{D_n}(w)\right)  \ .
$$
It follows that a normalized $e^{\beta F}$-conformal measure $m$ on $P(\Gamma)$ is concentrated on $E$ iff $m\left(\Omega_{D_n}(w)\right) = 0$ for all $n$ and all $w \in \partial D_n \backslash E_{D_n}$, i.e. iff the element $\psi \in \Delta_{\beta F}$ for which $m^{\psi} = m$ has the property that $\psi^n_w = 0$ for all $n$ and all $w \in \partial D_n \backslash E_{D_n}$. These elements clearly form a closed convex face $F$ in $\Delta_{\beta F}$. It follows therefore from Proposition \ref{thm?} that the set of normalized $e^{\beta F}$-conformal measures concentrated on $E$ constitute a closed face in the simplex of all normalized $e^{\beta F}$-conformal measures.

It remains to show that this face is not empty. To this end we define vectors 
$$
\psi(n) \in \prod_{k=1}^{\infty} [0,\infty)^{\partial  D_k}
$$
in the following way. Let $e^{n} \in [0,\infty)^{\partial D_n}$ be the vector
$$ 
e^n_w = \begin{cases} 1, & \ \ w \in E_{D_n}  \\ 0, & \ w \notin E_{D_n} \ . \end{cases} 
$$
Then
$$
\lambda = M(0)M(1)M(2) \cdots M(n-1)e^n \in \mathbb R_+
$$
and we set 
$$
\psi(n)^n = \lambda^{-1} \left(\sum_{n=0}^{\infty} A(\beta)^n_{v_0,v_0}\right)^{-1} e^n \ .
$$ 
For $1 \leq j \leq n-1$ we set 
$$
\psi(n)^{j} = M(j)M(j+1) \cdots M(n-1)\psi(n)^n \in \mathbb R_+^{\partial D_j}
$$
and for $j > n$ we set $\psi(n)^j = 0$.  Then
$$
0 \leq \psi(n)^j_w \leq \left(\sum_{m=0}^{\infty} A(\beta)^m_{v_0,w}\right)^{-1}
$$
for all $n$, all $j \geq 1$ and all $w \in \partial D_j$. It follows from Tychonovs theorem that the sequence $\left\{\psi(n)\right\}_{n=1}^{\infty}$ has a convergent subsequence in $\prod_{k=1}^{\infty} [0,\infty)^{\partial D_k}$. The limit of such a subsequence is a vector $\psi \in \Delta_{\beta F}$ such that $\psi^n_w = 0$ when $w \in \partial D_n \backslash E_{D_n}$. 
\end{proof}

\begin{cor}\label{13-12-17} Assume that $\Gamma$ is a row-finite directed graph and that there is a vertex $v_0 \in \Gamma_V$ such that
$$
0 < \sum_{n=0}^{\infty} A(\beta)^n_{v_0,v} < \infty
$$
for all $v \in \Gamma_V$. For each minimal end $E \in \mathcal E(\Gamma)$ there is an extremal normalized $e^{\beta F}$-conformal measure $m$ concentrated on $E$.
\end{cor}
\begin{proof} This follows from the words 'non-empty' and 'closed' in the previous theorem, combined with the fact that a closed face in a compact convex set contains an extreme point.
\end{proof}

Thus the minimal ends of a row-finite graph play a role analogous to the role of the bare exits for graphs with countably many ends, cf. \cite{Th3}.

\begin{thm}\label{13-12-17a}  Assume that $\Gamma$ is a row-finite, strongly connected and almost undirected digraph. Assume that $A(\beta)$ is transient.
There is a surjective map $m \to E$ from the set of extremal $e^{\beta F}$-conformal measures $m$ to the end space $\mathcal E(\Gamma)$ of $\Gamma$ defined such that $m$ is concentrated on $E$. 
\end{thm}
\begin{proof} It follows from Lemma \ref{almost/min} and Corollary \ref{13-12-17} that the map \eqref{14-01-18map} is surjective.
\end{proof}

\begin{cor}\label{08-01-18h} Let $\Gamma$ be digraph which is meager, row-finite, strongly connected and almost undirected. Let $F : \Gamma_{Ar} \to \mathbb R$ be a potential such that $A(\beta)$ is transient. There is an affine bijection between the normalized $e^{\beta F}$-conformal measures $m$ on $P(\Gamma)$ and the Borel probability measures $\nu$ on $\mathcal E(\Gamma)$ such that
\begin{equation*}\label{08-01nymaph}
m(B) = \int_{\mathcal E(\Gamma)} m_E(B) \ \mathrm{d}\nu(E) \ 
\end{equation*}
for all Borel sets $B \subseteq P(\Gamma)$.
\end{cor}
\begin{proof} Since all ends are minimal by Lemma \ref{almost/min} they all support an $e^{\beta F}$-conformal measure by Corollary \ref{13-12-17}. Hence $\mathcal E_{\beta}(\Gamma) = \mathcal E(\Gamma)$ by a) of Lemma \ref{09-11-17x} and the statement follows therefore from Theorem \ref{08-01-18d}.
\end{proof}

\begin{remark}\label{27-02-18e} In \cite{PW1} Picardello and Woess construct a surjective map from the Martin boundary of a transient uniformly irreducible random walk on an undirected graph $\mathcal G$ to the end space of the graph. See Theorem 26.2 in \cite{Wo1}. If we define a directed graph $\Gamma$ as the digraph with the same set of vertexes as $\mathcal G$ and with an arrow from one vertex to the another when there is a non-zero probability to go from the first to the second, then we can apply Corollary \ref{13-12-17} to conclude that already the minimal Martin boundary surjects onto the end space of the graph. In fact, it follows that every end supports the $h$-process defined by some minimal harmonic  function $h$ for the random walk.
\end{remark}

\subsection{Examples}

\begin{example}\label{22-01-18}

As explained in the introduction, Section \ref{introsec}, it follows from Theorem \ref{13-12-17a} that the problem of finding the KMS-weights for generalized gauge actions on the $C^*$-algebras of row-finite strongly connected almost undirected graphs can be translated to the same problem for simple unital AF-algebras. Despite that AF-algebras are among the best understood $C^*$-algebras and despite that one of the first papers on KMS-states, by Powers and Sakai, \cite{PS}, was about AF-algebras, very little is known about the possible structure of KMS-states for one-parameter groups on AF-algebras, or even UHF-algebras for that matter. The paper \cite{Ki} by Kishimoto exhibits examples showing that locally representable actions on AF- and UHF-algebras can have a complicated structure of KMS-states. It is not clear if the actions considered by Kishimoto can be realized as generalized gauge actions, and we include therefore here the following example which shows that actions of this type also allows for phase transitions. It seems certain that we are scratching in the surface of a vast unexplored territory.   

 In the following Bratteli diagram 
\begin{equation}\label{14-11-17e}
\begin{xymatrix}{\Br & & \ar[ld]_-1 v_0 \ar[rd]^-1 & \\
& \ar[d]_-1  \ar[rrd]^(0.3){b_1} &  & \ar[d]^-1 \ar[lld]_(0.3){b_1} \\  
& \ar[d]_-1  \ar[rrd]^(0.3){b_2} &  & \ar[d]^-1 \ar[lld]_(0.3){b_2} \\  
& \ar[d]_-1  \ar[rrd]^(0.3){b_3} &  & \ar[d]^-1 \ar[lld]_(0.3){b_3} \\  
&\ar[d]_-1  \ar[rrd]^(0.3){b_4} &  & \ar[d]^-1 \ar[lld]_(0.3){b_4} \\  
&& & \\
&\vdots &\vdots  & \vdots  }
 \end{xymatrix}
 \end{equation}
the arrows carry labels that are \emph{not} multiplicities, but signify the value which the potential $F$ takes on them. The multiplicities of all arrows are one and hence $AF(\Br)$ is the Fermion- or CAR-algebra, also known as the UHF-algebra of type $2^{\infty}$. To fix the sequence $\{b_j\}$ let $\alpha > 0$ be a positive number and set
$$
b_j = \frac{\log j}{\alpha}  , \ j \geq 1 \ .
$$
If $y \in Z(v_0)$ is any of the two rays emitted from $v_0$ whose arrows all carry the label $1$, we find that
$$
\mathbb V_{\beta} (v_0,y) = \lim_{n \to \infty} \frac{\prod_{j=1}^n \left(e^{-\beta} + e^{-b_j \beta}\right) }{e^{-n \beta}} =  \prod_{j=1}^{\infty} \left(1 + e^{(1-b_j)\beta}\right) \ .
$$  
Hence these two rays are $\beta$-summable simultaneously and it happens if and only if $\sum_{j=1}^{\infty} e^{-b_j \beta} < \infty$. By the choice of $\beta_j$ we have that
$$
e^{-b_j\beta} = j^{-\frac{\beta}{\alpha}} \ .
$$
The two rays are therefore $\beta$-summable if and only if $\beta > \alpha$. It follows from Theorem \ref{24-11-17} that there are at least two extremal $e^{\beta F}$-conformal measures when $\beta > \alpha$, but from Theorem \ref{OKok} it follows that there can not be more than two since there are only two vertexes at each level in $\Br$; any ray different from the two we have considered will have to share infinitely many vertexes with at least one of them and hence the limits in \eqref{just} will be the same. When $\beta \leq \alpha$, where $\sum_{j=1}^{\infty} e^{-b_j \beta} = \infty$, it follows from Proposition 26.10 in \cite{Wo1} that the set $\Delta_{\beta F}$ from Definition \ref{31-10-2017} only contains one element and then from Proposition \ref{05-04b} that there is exactly one normalized $e^{\beta F}$-conformal measure on $P(\Br)$. We note that for the graph $\Br$ the map \eqref{14-01-18map} is injective for $\beta \leq \alpha$, but not when $\beta > \alpha$. To obtain a strongly connected graph with the property that \eqref{14-01-18map} is injective for some values of $\beta$ and not for others one can add return paths to $v_0$ as described in Section \ref{adding}.

\end{example}


\begin{example}\label{27-02-18d} (Pascal's triangle.) With the following example we show how to combine methods developed for random walks with the results of this paper in order to analyse the $e^{\beta F}$-conformal measures on the path space of a well-known Bratteli diagram which is not meager. With the the methods used in Section \ref{adding} we can modify Pascal's triangle to get a strongly connected graph with arbitrary positive Gurevich entropy, but the calculations and conclusions for the resulting graph would be almost identical to the following, except for the fact that the transient range would be reduced to a half-line.

Let $\Gamma$ be the digraph with $V = \mathbb N \times \mathbb N$ and a single arrow from $(n,m)$ to $(n',m')$ iff $(n',m') \in \left\{ (n+1,m), \ (n,m+1)\right\}$. Then $\Gamma$ looks as follows:

\bigskip

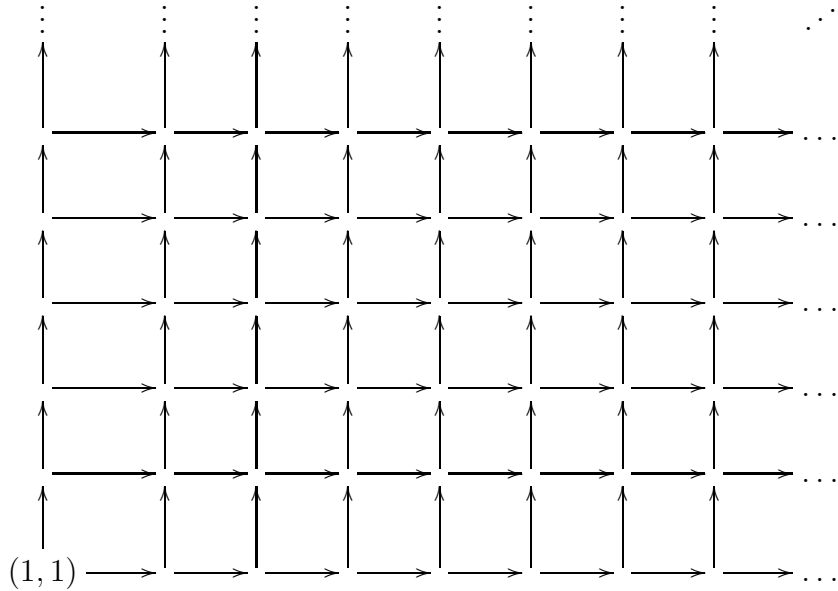
\begin{xymatrix}{
\vdots & \vdots &\vdots &\vdots &\vdots &\vdots &\vdots  &\vdots  &  \iddots \\
\ar[u] \ar[r]  &\ar[u] \ar[r]  &\ar[u] \ar[r]  &\ar[u] \ar[r]  &\ar[u] \ar[r]  &\ar[u] \ar[r]  &\ar[u]\ar[r]  & \ar[u] \ar[r] & \hdots \\ 
\ar[u] \ar[r]  &\ar[u] \ar[r]  &\ar[u] \ar[r]  &\ar[u] \ar[r]  &\ar[u] \ar[r]  &\ar[u] \ar[r]  &\ar[u] \ar[r]  & \ar[u] \ar[r] & \hdots\\
\ar[u] \ar[r]  &\ar[u] \ar[r]  &\ar[u] \ar[r]  &\ar[u] \ar[r]  &\ar[u] \ar[r]  &\ar[u] \ar[r]  &\ar[u] \ar[r]  & \ar[u] \ar[r] &  \hdots\\
\ar[u] \ar[r]  &\ar[u] \ar[r]  &\ar[u] \ar[r]  &\ar[u] \ar[r]  &\ar[u] \ar[r]  &\ar[u] \ar[r]  &\ar[u] \ar[r]  & \ar[u] \ar[r] & \hdots \\
\ar[u] \ar[r]  &\ar[u] \ar[r]  &\ar[u] \ar[r]  &\ar[u] \ar[r]  &\ar[u] \ar[r]  &\ar[u] \ar[r]  &\ar[u] \ar[r]  & \ar[u] \ar[r] &  \hdots\\
(1,1) \ar[u] \ar[r]  &\ar[u] \ar[r]  &\ar[u] \ar[r]  &\ar[u] \ar[r]  &\ar[u] \ar[r]  &\ar[u] \ar[r]  &\ar[u] \ar[r]  & \ar[u] \ar[r] &  \hdots\\
 }
\end{xymatrix}   

\bigskip

 We shall consider a potential $F : \Gamma_{Ar} \to \mathbb R$ which may not be constant. Let $(u'_i)_{i=1}^{\infty}$ and $(v'_i)_{i=1}^{\infty}$ be sequences of real numbers and set
$$
F\left( (k,n), (k+1,n)\right) = u'_k, \ \ F\left((k,n),(k,n+1)\right) = v'_n \ 
$$
for all $(k,n) \in \mathbb N \times \mathbb N$. Let $\beta \in \mathbb R$ and set
$$
u_k(\beta) = e^{-\beta u'_k} \ \text{and} \ v_n(\beta) = e^{-\beta v'_n} \ . 
$$ 

Despite that Pascal's triangle is not meager we can determine all $e^{\beta F}$-conformal measures on $P(\Gamma)$ by following the lead of \cite{Sa}. In fact, the following calculations are only marginally different from those performed by Sawyer in Section 6 of \cite{Sa} and we shall therefore be brief. When $(x,y), (n,m) \in \Gamma_V$ such that $x \leq n, \ y \leq m$, we find that
\begin{equation}\label{16-01-18}
\begin{split}
&\sum_{i=0}^{\infty} A(\beta)^i_{(x,y), (n,m)} = A(\beta)^{n+m-x-y}_{(x,y),(n,m)} \\
&= \begin{cases} 0 \ , & \ \text{when} \ (x,y) = (n,m), \\
\prod_{j=y}^{m-1} v_j(\beta) \ , \ & \ \text{when} \ x= n , \ y < m, \\
\prod_{i=x}^{n-1} u_i(\beta) \ , \ & \ \text{when} \ x < n, \ y = m ,\\
 {n+m-x-y \choose n-x} \prod_{i=x}^{n-1} u_i(\beta) \prod_{j=y}^{m-1}v_j(\beta) \ , & \ \text{when} \ x < n, \ y < m \ 
 \end{cases}
\end{split}
\end{equation}
when we use the convention that a product over the empty set is $1$. Let $p = \left(p_k\right)_{k=1}^{\infty}$ be an infinite path in $\Gamma$. Then $s(p_k) = (m_k,n_k)$ where $\{(m_k,n_k)\}$ is a sequence in $\Gamma_V$ such that $\lim_{k \to \infty} \max \{m_k,n_k\} = \infty$. Set $v_0 = (1,1)$. By using \eqref{16-01-18} we find that the limit 
$$
\lim_{ k \to \infty} K_{\beta}\left((x,y), (m_k,n_k)\right)
$$ exists for all $(x,y) \in \Gamma_V$ if and only if the limit
$$
\alpha = \lim_{k \to \infty} \frac{m_k}{m_k+n_k} 
$$
exists, in which case $\lim_{ k \to \infty} K_{\beta}\left((x,y), (m_k,n_k)\right) = \psi^{\alpha}_{(x,y)}$, where
\begin{equation}\label{14-01-18}
\psi^{\alpha}_{(x,y)} = \alpha^{x-1} (1-\alpha)^{y-1} \prod_{i=0}^{x-1} u_i(\beta)^{-1}  \prod_{j=0}^{y-1} v_j(\beta)^{-1} \ 
\end{equation}
when we set $u_0(\beta) = v_0(\beta) = 1$ and use the convention that $0^0 = 1$, $0^k = 0, \ k >0$. To show that the vector $\psi^{\alpha}$ is a minimal $A(\beta)$-harmonic vector, note that by Corollary \ref{04-12-17e} there is a universally measurable set $Y_{\beta} \subseteq \Ray(\Gamma)$ such that the limit 
$$
\varphi_{(x,y)} = \lim_{ k \to \infty} K_{\beta}\left((x,y), s(p_k)\right)
$$ exists for all $(x,y) \in \Gamma_V$ and defines a minimal $A(\beta)$-harmonic vector $\varphi$ when $p \in Y_{\beta}$, and that there is a Borel probability measure $\nu$ on $Y_{\beta}$ such that
$$
\psi^{\alpha}_{(x,y)} = \int_{Y^{\beta}} \lim_{n \to \infty} K_{\beta}((x,y), s(p_n)) \ \mathrm{d}\nu(p) 
$$
for all $(x,y) \in \Gamma_V$. It follows from what was shown above that there is for each $p \in Y_{\beta}$ a number $\alpha(p) \in [0,1]$ such that 
$$
 \lim_{n \to \infty} K_{\beta}((x,y), s(p_n)) = \psi^{\alpha(p)}_{(x,y)} 
 $$
 for all $(x,y) \in \Gamma_V$. By using the formula \eqref{14-01-18} we find that 
 $$
 \alpha^{x-1}(1-\alpha)^{y-1} = \int_{Y_{\beta}} \alpha(p)^{x-1}(1-\alpha(p))^{y-1} \mathrm{d}\nu(p) \ ,
 $$
for all $(x,y) \in \Gamma_V$. It follows from this that $\alpha(p) = \alpha$ for $\nu$-almost all $p \in Y_{\beta}$, cf. \cite{Sa}; in particular, there is a $p  \in Y_{\beta}$ such that
$$
\psi^{\alpha}_{(x,y)} = \lim_{ k \to \infty} K_{\beta}\left((x,y), s(p_k)\right) 
$$
for all $(x,y)$, i.e. $\psi^{\alpha}$ is minimal. Therefore 
$$
 \partial M^{v_0}_{\beta F}(\Gamma) = \left\{ m_{\alpha}: \ \alpha \in [0,1] \right\}  \ 
 $$
when we let $m_{\alpha}$ be the $e^{\beta F}$-conformal measure on $P(\Gamma)$ with the property that $m_{\alpha}(Z(x,y)) = \psi^{\alpha}_{(x,y)}$ for all $(x,y) \in \Gamma_V$. 

To see which ends the measure $m_{\alpha}$ is concentrated on we must identify the end space $\mathcal E(\Gamma)$. This can either be read out from the above illustration of the digraph or it can be found by combining Proposition \ref{baadop} above with Bratteli's description of the primitive ideal space of $AF(\Gamma)$ given in Section 5 of \cite{Br1}.
To describe $\mathcal E(\Gamma)$, let $t_k$ be the ray in $ \Ray(\Gamma)$ which passes through the vertexes $(k,i), i \in \mathbb N$, and $t_{-k}$ the ray which passes through the vertexes $(i,k), i \in \mathbb N$. Let $t_0$ be any ray in $\Ray(\Gamma)$ which passes through vertexes with no bound on the first or second coordinate, e.g. any path which sees the vertexes $(k,k)$ for all $k$. Then 
$$
\mathcal E(\Gamma) = \left\{ \mathcal E(t_k): \ k \in \mathbb Z \right\} ,
$$
and the map 
$$
\mathcal E(t_k)  \ \mapsto \begin{cases} \frac{1}{k},  & \ k \neq 0 \\ 0, & \ k = 0 \ ,
\end{cases}
$$
is a homeomorphism from $\mathcal E(\Gamma)$ onto the set $
\left\{ \pm \frac{1}{n} : n =1,2,3, \cdots \right\} \cup \{0\}$ equipped with the relative topology inherited from $\mathbb R$. It follows from Theorem \ref{OKok} that $m_{\alpha}$ is concentrated on $\mathcal E(t_0)$ when $0 < \alpha < 1$, and from Theorem \ref{24-11-17} that $m_0$ is concentrated on $\mathcal E(t_1)$ and $m_1$ on $\mathcal E(t_{-1})$. In particular, none of the ends $\mathcal E(t_k), k \notin \{-1,0,1\}$, support an extremal $e^{\beta F}$-conformal measure despite the fact that all $t_k$ are elements of the set $X_{\beta}$ of Corollary \ref{OKok2}. Note that only $\mathcal E(t_1)$ and $\mathcal E(t_{-1})$ are minimal ends.

\end{example} 

\section{Gluing together Bratteli diagrams}\label{gluing}

In this section we construct a class of strongly connected row-finite graphs $\Gamma$ with the property that the simplices $M^{v_0}_{\beta}(\Gamma)$ of normalized $e^{\beta}$-conformal measures are not affinely homeomorphic for different $\beta$. Parts of the underlying ideas for the construction is present in Section 7 of \cite{Th3} and to some extend what we do is just to replace exits by simple Bratteli diagrams.

\subsection{Extensions of harmonic vectors}

Let $\Gamma$ be a digraph with a vertex $v_0 \in \Gamma_V$ from where all vertexes can be reached and let $F : \Gamma_{Ar} \to \mathbb R$ be an arbitrary potential. When $H \subseteq \Gamma_V$ is a hereditary subset such that $v_0 \notin H$ we let $\Gamma^H$ be the digraph with $\Gamma^H_V = H$ and $\Gamma^H_{Ar} = \left\{ a \in \Gamma_{Ar}: \ s(a) \in H \right\}$. Let $\Gamma^{H^c}$ be the digraph where $\Gamma^{H^c}_V = \Gamma_V \backslash H$ and 
$$
\Gamma^{H^c}_{Ar} = \left\{a \in \Gamma_{Ar} : \ \left\{s(a), r(a)\right\} \subseteq \Gamma_V \backslash H \right\}  \ .
$$ 
Let
$$
B(\beta)_{v,w} = A(\beta)_{v,w}, \ v,w \in H \ ,
$$
and
$$
C(\beta)_{v,w} = A(\beta)_{v,w}, \ v,w \in \Gamma_V \backslash H \ .
$$

\begin{lemma}\label{20-01-18d} Set
$M = s^{-1}\left(\Gamma_V \backslash H\right) \cap r^{-1}(H)$. Let $\psi : H \to [0,\infty)$ be a $B(\beta)$-harmonic vector. 
\begin{itemize}
\item[i)] There is an $A(\beta)$-harmonic vector $\varphi$ such that $\psi = \varphi|_H$ if and only if
\begin{equation}\label{16-01-18a}
\sum_{a\in M} \sum_{n=0}^{\infty} C(\beta)^n_{v_0, s(a)} e^{-\beta F(a)} \psi_{r(a)}  < \infty \ .
\end{equation}
\item[ii)] When \eqref{16-01-18a} holds, the vector
\begin{equation}\label{25-01-18c}
\overline{\psi}_v = \begin{cases} \sum_{a \in M} \sum_{n=0}^{\infty} C(\beta)^n_{v,s(a)} e^{-\beta F(a)} \psi_{r(a)} , & \ v \notin H \\ \psi_v, \ &  \ v \in H \  \end{cases}
\end{equation}
is the unique $A(\beta)$-harmonic vector with the following two properties:
\begin{itemize}
\item[a)] $\overline{\psi}|_H = \psi$ and 
\item[b)] $m_{\overline{\psi}}$ is concentrated on
$\bigcup_{n=0}^{\infty} \sigma^{-n}\left(P(\Gamma^H)\right)$.
\end{itemize}
\end{itemize}
\end{lemma}
\begin{proof} i) Assume first that $\varphi$ exists. Then
\begin{equation*}
\begin{split}
&\sum_{a \in M} \sum_{n=0}^{\infty} C(\beta)^n_{v_0, s(a)} e^{-\beta F(a)} \psi_{r(a)} \\
& =  m_{\varphi}\left(Z(v_0) \cap  \bigcup_{n=0}^{\infty} \sigma^{-n} \left(P(\Gamma^H) \right)\right) \ \leq \ m_{\varphi}(Z(v_0)) \ < \ \infty \ .
\end{split}
\end{equation*}
Conversely, assume that \eqref{16-01-18a} holds. Since $v_0 \notin H$ and all vertexes can be reached from $v_0$, it follows that
\begin{equation*}
\sum_{a\in M} \sum_{n=0}^{\infty} C(\beta)^n_{v, s(a)} e^{-\beta F(a)} \psi_{r(a)}  < \infty \ 
\end{equation*}
for all $v \notin H$. Define $\overline{\psi} : \Gamma_V \to [0,\infty)$ by \eqref{25-01-18c}.
When $v \notin  s(M)\cup H$,
 \begin{equation*}
\begin{split}
&\sum_{u \in \Gamma_V} A(\beta)_{v,u} \overline{\psi}_u = \sum_{a \in M} \sum_{n=0}^{\infty} C(\beta)^{n+1}_{v,s(a)} e^{-\beta F(a)} \psi_{r(a)} \\ 
& = \sum_{a \in M} \sum_{n=0}^{\infty} C(\beta)^n_{v,s(a)} e^{-\beta F(a)} \psi_{r(a)} = \overline{\psi}_v \ .
\end{split}
\end{equation*}
When $v = s(a)$ for some $a \in M$ we find that
\begin{equation*}
\begin{split}
&\sum_{u \in \Gamma_V} A(\beta)_{v,u} \overline{\psi}_u \\
&  = \sum_{u \in \Gamma_V \backslash H} A(\beta)_{s(a),u} \sum_{b\in M} \sum_{n=0}^{\infty} C(\beta)^n_{u,s(b)}e^{-\beta F(b)} \psi_{r(b)} \ + \ \sum_{u \in H} A(\beta)_{s(a),u} \psi_u \\
& = \sum_{b \in M} \sum_{n=0}^{\infty} C(\beta)^{n+1}_{s(a),s(b)} e^{-\beta F(b)} \psi_{r(b)} \ \  + \sum_{\{b \in M: \ s(b) = s(a)\}} e^{-\beta F(b)} \psi_{r(b)}\\ 
& = \sum_{b \in M} \sum_{n=0}^{\infty} C(\beta)^n_{s(a),s(b)} e^{-\beta F(b)} \psi_{r(b)} = \overline{\psi}_{s(a)} \ .
\end{split}
\end{equation*}
Finally, note that 
$$
\sum_{u \in \Gamma_V} A(\beta)_{v,u} \overline{\psi}_u =  \sum_{u \in H} B(\beta)_{v,u} {\psi}_u  = \psi_v = \overline{\psi}_v
$$ 
when $v \in H$. 

ii) : We assume that \eqref{16-01-18a} holds and define a Borel measure $m'$ on $P(\Gamma)$ such that
$$
m'(B) = m_{\overline{\psi}}\left( B \cap \bigcup_{n=0}^{\infty} \sigma^{-n}\left(P(\Gamma^H)\right)\right) \ 
$$
when $B \subseteq P(\Gamma)$ is a Borel set. Note that $m'$ is $e^{\beta F}$-conformal because $\bigcup_{n=0}^{\infty} \sigma^{-n}\left(P(\Gamma^H)\right)$ is totally $\sigma$-invariant and that $m'(Z(v)) = m_{\overline{\psi}}\left(Z(v)\right) = \overline{\psi}_v = \psi_v$ when $v  \in H$. Since $m'$ is concentrated on   $\bigcup_{n=0}^{\infty} \sigma^{-n}\left(P(\Gamma^H)\right)$ it suffices now to show that $\mu = m_{\overline{\psi}}$ for any $e^{\beta F}$-conformal measure $\mu$ which is concentrated on $\bigcup_{n=0}^{\infty} \sigma^{-n}\left(P(\Gamma^H)\right)$ and has the property that $\mu(Z(v)) = \psi_v$ when $v \in H$. To this end note that when $w \in \Gamma_V \backslash H$ we have that
\begin{equation*}
\begin{split}
&\mu(Z(w)) =  \mu\left( Z(w) \cap \bigcup_{n=0}^{\infty} \sigma^{-n}\left(P(\Gamma^H)\right)\right)\\
& =\sum_{a\in M} \sum_{n=0}^{\infty} C(\beta)^n_{w, s(a)} e^{-\beta F(a)} \mu\left(Z(r(a))\right) \\
& = \sum_{a\in M} \sum_{n=0}^{\infty} C(\beta)^n_{w, s(a)} e^{-\beta F(a)} \psi_{r(a)}=  \overline{\psi}_w \ .
\end{split}
\end{equation*}
It follows that $\mu(Z(v)) = m_{\overline{\psi}}(Z(v))$ for all $v \in \Gamma_V$ and hence that $\mu =  m_{\overline{\psi}}$, as desired.

\end{proof}


\subsection{Connecting simple Bratteli diagrams} The input for the construction is 
\begin{itemize}
\item a sequence $\{S_k\}_{k=1}^{\infty}$ of metrizable Choquet simplices,
\item a positive real number $h > 0$, and 
\item a sequence $\{I_k\}_{k=1}^{\infty}$ of sub-intervals of the interval $[h,\infty)$.
\end{itemize}
As in \cite{Th3} an interval should here be understood in the broadest possible sense; they may be open, closed, half-open, consist of a single point or even be empty. 
By Lemma 6.4 in \cite{Th3} there are sequences 
\begin{equation}\label{08-02-18d}
\{d^k_n\}_{n=1}^{\infty}, \ \{b^k_n\}_{n=1}^{\infty}, \  \{a^k_n\}_{n=1}^{\infty}
\end{equation} 
of non-zero natural numbers such that $a^k_1 =1$, and such that
\begin{equation}\label{02-02sum1}
\sum_{n =1}^{\infty} \frac{d^k_n}{a^k_{n}} e^{n\beta } \ < \ \infty
\end{equation}
and
\begin{equation}\label{02-02sum2}
\sum_{n =1}^{\infty} \frac{b^k_n}{a^k_n} e^{-n\beta } \ < \ \infty
\end{equation}
both hold if and only if $\beta \in I_k$. By Theorem 3.10 in \cite{Bl} there is a simple unital AF-algebra $A_k$ with tracial state space $T(A_k)$ affinely homeomorphic to the simplex $S_k$. Let $\Br^k$ be a Bratteli diagram with levels $\Br^k_i, i = 0,1,2, \cdots$, such that $AF(\Br^k)$ is isomorphic to $A_k$. Let $u^k_0 \in \Br^k_0$ be the vertex in the top level $\Br^k_0$ of $\Br^k$. We may assume that $A_k$ is infinite dimensional, and by using the  procedure which was called 'telescoping' and described in detail in Definition 3.2 on page 68 in \cite{GPS}, we can arrange that the adjacency matrix $B(k)$ of $\Br^k$ satisfies the inequalities 
\begin{equation}\label{01-01-18g}
B(k)^{2n-2}_{u^k_0,v}  \geq 2 a^k_{n}  \ , 
\end{equation}
for all $v \in \Br^k_{2n-2}, \ n \geq 2$,
and
\begin{equation}\label{01-01-18gg}
B(k)^{n-1}_{u^k_0,v}  \geq 2 a^k_{n}  \ ,
\end{equation}
for all $v \in \Br^k_{n-1}, n \geq 2$. Let $[X]$ denote the integer part of a real number $X$, and set
$$
x_n(v) = \left[ \frac{B(k)^{2n-2}_{u^k_0,v}}{a^k_n} \right] ,  \ v \in \Br^k_{2n-2} \ ,
$$
and
$$
y_n(v) = \left[\frac{B(k)^{n-1}_{u^k_0,v}}{a^k_n}\right], \ v \in \Br^k_{n-1} \ ,
$$
when $n \geq 1$. Thanks to \eqref{01-01-18g} and \eqref{01-01-18gg} we have that
\begin{equation}\label{02-02-18a}
\frac{1}{2} \frac{B(k)^{2n-2}_{u^k_0,v}}{a^k_n} \ \leq x_n(v) \ \leq \frac{B(k)^{2n-2}_{u^k_0,v}}{a^k_n} 
\end{equation}
for all $v \in \Br^k_{2n-2}$
and
\begin{equation}\label{02-02-18b}
\frac{1}{2} \frac{B(k)^{n-1}_{u^k_0,v}}{a^k_n} \ \leq y_n(v) \ \leq \frac{B(k)^{n-1}_{u^k_0,v}}{a^k_n} 
\end{equation}
for all $v \in \Br^k_{n-1}$, $ n \geq 1$.

We are ready to describe a row-finite digraph $\Gamma^0$. For this we visualize the $k$'th Bratteli diagram $\Br^k$ as a row:

\begin{equation*}\label{08-02-18}
\begin{xymatrix}{
\Br^k_0 \ar@{=>}@[red][r]  &  
 \Br^k_1   \ar@{=>}@[red][r]  &  
  \Br^k_2   \ar@{=>}@[red][r]  & 
   \Br^k_3   \ar@{=>}@[red][r]  & 
   \Br^k_4  \ar@{=>}@[red][r]   &  
    \Br^k_5   \ar@{=>}@[red][r]  &   & 
    \hdots & }
 \end{xymatrix}
\end{equation*}
where $\Br^k_i , i = 0,1,2, \cdots$, are the levels in $\Br^k$ and the red arrows symbolize the web of arrows from one level to the next. We put these Bratteli diagrams into a larger diagram as follows:
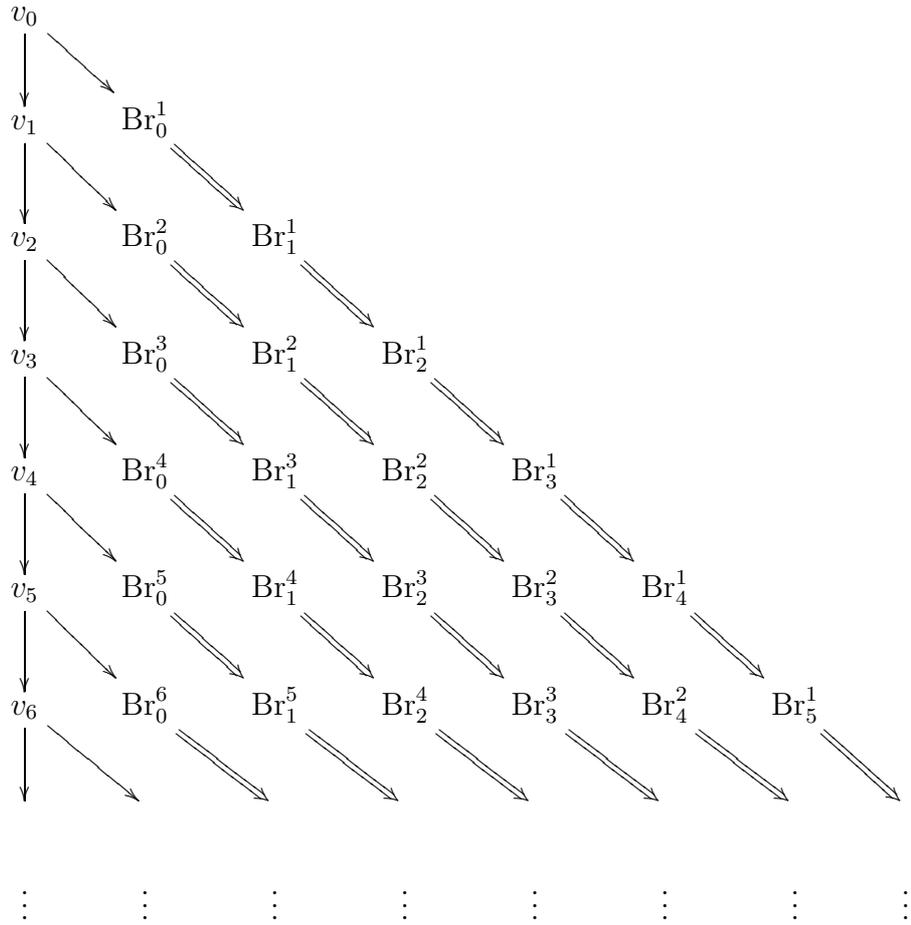
\begin{equation*}\label{31-01-18z}
\begin{xymatrix}{
 \ar[d] v_0 \ar[dr]  & \\
   \ar[d] v_1  \ar[dr]  & \Br^1_0 \ar@{=>}@[red][dr]  & \\ 
 \ar[d] v_2  \ar[dr]  &  \ar@{=>}@[red][dr] \Br^2_0    &  \Br^1_1 \ar@{=>}@[red][dr]  \\ 
 \ar[d] v_3 \ar[dr] & \Br^3_0 \ar@{=>}@[red][dr] & \Br^2_1 \ar@{=>}@[red][dr] &   \Br^1_2 \ar@{=>}@[red][dr] & \\
 \ar[d] v_4 \ar[dr]   & \ar@{=>}@[red][dr] \Br^4_0 & \ar@{=>}@[red][dr] \Br^3_1  &    \ar@{=>}@[red][dr] \Br^2_2&   \Br^1_3 \ar@{=>}@[red][dr]&\\
 \ar[d] v_5 \ar[dr]   & \ar@{=>}@[red][dr] \Br^5_0 & \ar@{=>}@[red][dr] \Br^4_1  &    \ar@{=>}@[red][dr] \Br^3_2&   \Br^2_3 \ar@{=>}@[red][dr]&\Br^1_4 \ar@{=>}@[red][dr] &\\ 
  \ar[d] v_6 \ar[dr]   & \ar@{=>}@[red][dr] \Br^6_0 & \ar@{=>}@[red][dr] \Br^5_1  &    \ar@{=>}@[red][dr] \Br^4_2&   \Br^3_3 \ar@{=>}@[red][dr]&\Br^2_4 \ar@{=>}@[red][dr] &\Br^1_5 \ar@{=>}@[red][dr]\\ 
  &&&&&&&&\\
 \vdots & \vdots &\vdots& \vdots & \vdots &\vdots &\vdots & \vdots &}
 \end{xymatrix}
\end{equation*}

\bigskip

We add then arrows from the $v_i$-vertexes to each of the Bratteli-subdiagrams $\Br^k$. The additional arrows with multiplicities going into $\Br^k$ come from the vertexes $v_i$ with $i \geq k$ as shown in the graph below.

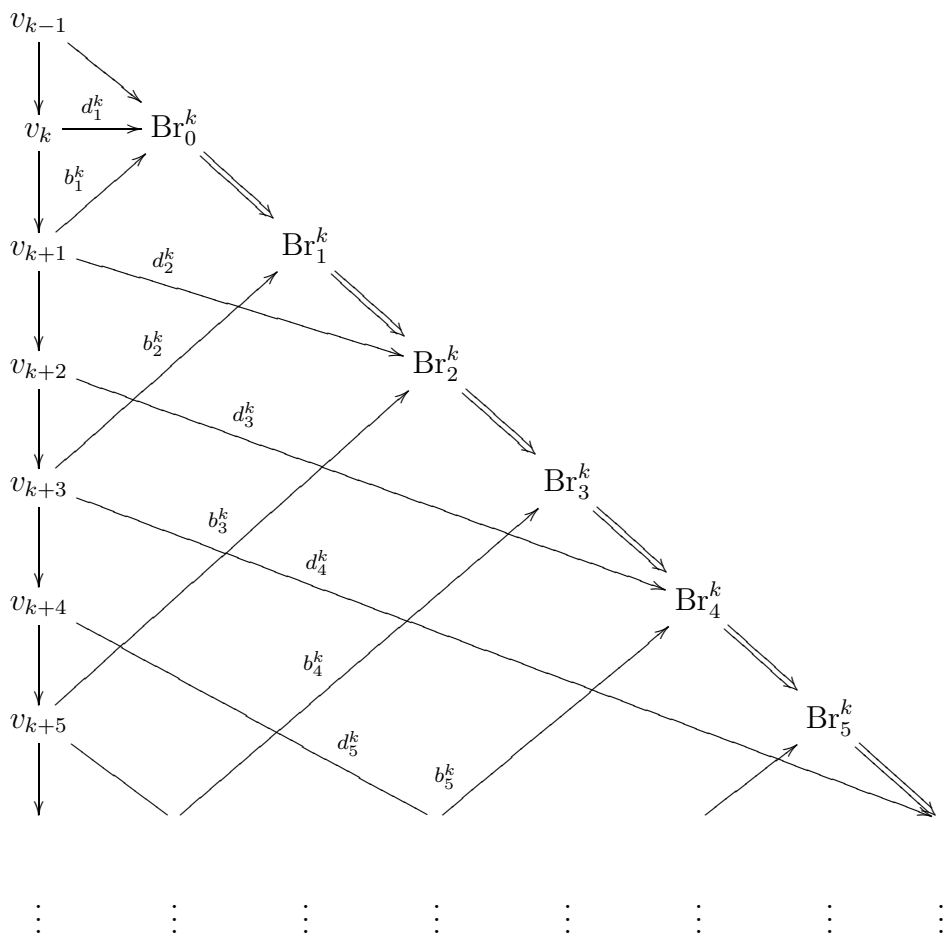
\begin{equation*}\label{01-02-18e}
\begin{xymatrix}{
 \ar[d] v_{k-1}   \ar[dr] & \\
   \ar[d] v_{k} \ar@[green][r]^(0.4){d^k_1}  & \Br^{k}_0 \ar@{=>}@[red][dr]  & \\ 
 \ar[d] v_{k+1} \ar@[blue][ru]^(0.4){b^k_1} \ar@[green][rrrd]^(0.3){d^k_2}    &     &  \Br^k_1 \ar@{=>}@[red][dr]  \\ 
 \ar[d] v_{k+2}   \ar@[green][rrrrrdd]^(0.3){d^k_3} & & &   \Br^k_2 \ar@{=>}@[red][dr] & \\
 \ar[d]  \ar@[blue][rruu]^(0.5){b^k_2} v_{k+3}  \ar@[green][rrrrrrrddd]^(0.3){d^k_4}    &  &   &   &   \Br^k_3 \ar@{=>}@[red][dr]& &\\
 \ar[d] v_{k+4} \ar@{-}@[green][rrrdd]^(0.75){d^k_5}   &  &   &   & &\Br^k_4 \ar@{=>}@[red][dr] & &\\ 
  \ar[d] v_{k+5}  \ar@[blue][rrruuu]^(0.5){b^k_3} \ar@{-}@[green][rd]  &  &  &  &   & &\Br^k_5 \ar@{=>}@[red][dr]&\\ 
  &    \ar@[blue][rrruuu]^(0.4){b^k_4}&&  \ar@[blue][rruu]^(0.1){b^k_5} &&  \ar@[blue][ru]&&&&\\
 \vdots & \vdots &\vdots& \vdots & \vdots &\vdots &\vdots & \vdots &}
 \end{xymatrix}
\end{equation*}  

The labels $d^k_i$ and $b^k_i$ of the arrows come from the sequences \eqref{08-02-18d} and should be interpreted in the following way: The green arrow from $v_{k+i}$ to $\Br^k_{2i}$ with label $d^k_{i+1}$ in the diagram is a collection of 
$$
\sum_{ v \in \Br^k_{2i}} d^k_{i+1} x_{i+1}(v)
$$
arrows with $d^k_{i+1} x_{i+1}(v)$ of them going from $v_{k+i}$ to $v \in \Br^k_{2i}$. Similarly, the blue arrow from $v_{k+2i+1}$ to $\Br^k_{i}$ with label $b^k_{i+1}$ in the diagram is a collection of 
$$
\sum_{ v \in \Br^k_{i}} b^k_{i+1} y_{i+1}(v)
$$
arrows with $b^k_{i+1} y_{i+1}(v)$ going from $v_{k+2i+1}$ to $v \in \Br^k_{i}$. When these arrows have been added for all $i = 0,1,2,3, \cdots$ and for all $k =1,2,3, \cdots$, the resulting graph $\Gamma^0$ will be row-finite and have the properties a), b) and c) required in Corollary \ref{attaching3}.

We consider the constant potential function $F = 1$ and seek to determine the sets $H^{v_0}_{\beta}(\Gamma^0)$ and $M^{v_0}_{\beta}\left(\Gamma^0\right)$. The first step for this will be to determine when an $e^{\beta}$-conformal measure on $P\left(\Br^k\right)$ extends to a $e^{\beta}$-conformal measure on $P(\Gamma^0)$, and for this we use Lemma \ref{20-01-18d}. Let therefore $m$ be a $e^{\beta}$-conformal measure on $P(\Br^k)$ normalized such that $m(Z(u^k_0)) = 1$. Set $\psi_v = m(Z(v)), \ v \in \Br^k_V$.
Let $C_k(\beta)$ be the matrix over $\Gamma^0_V \backslash \Br^k_V$ defined such that
$$
C_k(\beta)_{v,w} = e^{-\beta}A(\Gamma^0)_{v,w}, \ v,w \in \Gamma^0_V \backslash \Br^k_V \ ,
$$
and let 
$$
M_k = \left\{ a \in \Gamma^0_{Ar} : \ s(a) \notin \Br^k_V, \ r(a) \in \Br^k_V \ \right\} \ .
$$ 
For each $n = 0,1,2,3, \cdots$, let 
$\mathbb L_n$ be the set of finite paths $\mu = a_1a_2\cdots a_{|\mu|}$ in $\Gamma^0$ such that $s(\mu) = v_0, \ r(\mu) \in \Br^k_n$ and $s(a_i) \notin \Br^k_V, \ i =2,3,\cdots , |\mu|$.  Then 
$$
\sum_{a \in M_k} \sum_{n=0}^{\infty} C_k(\beta)^n_{v_0, s(a)}e^{-\beta} \psi_{r(a)} = \sum_{n=0}^{\infty} \sum_{\mu \in \mathbb L_n} e^{-\beta|\mu|} \psi_{r(\mu)}
$$
and it follows from i) of Lemma \ref{20-01-18d} that $\psi$ admits an extension to a $e^{-\beta}A(\Gamma^0)$-harmonic vector if and only if
$$
\sum_{n=0}^{\infty} \sum_{\mu \in \mathbb L_n} e^{-\beta|\mu|} \psi_{r(\mu)} < \infty \ .
$$  
For $v \in \Br^k_n$, set 
$$
\mathbb L_n(v) = \left\{ \mu \in \mathbb L_n: \ r(\mu) = v \right\} \ .
$$
Then, for $n \geq 2$,
\begin{equation}\label{24-01-18ax} 
\begin{split}
& \sum_{\mu \in \mathbb L_{2n-2}} e^{-\beta |\mu|} \psi_{r(\mu)} =  \sum_{v \in \Br^k_{2n-2}} \sum_{\mu \in \mathbb L_{2n-2}(v)} e^{-\beta |\mu|} \psi_{v} \\
&= \sum_{v \in \Br^k_{2n-2}} x_{n}(v)  d^k_{n} e^{- n \beta} e^{-k\beta} \psi_v  +   \sum_{v \in \Br^k_{2n-2}} y_{2n-1}(v)  b^k_{2n -1}   e^{-(4n-2)\beta }  e^{- k\beta}\psi_v  \ ,
\end{split}
\end{equation}
and similarly,
\begin{equation}\label{26-01-18kx} 
\begin{split}
& \sum_{\mu \in \mathbb L_{2n-1}} e^{-\beta |\mu|} \psi_{r(\mu)} =  \sum_{v \in \Br^k_{2n-1}} \sum_{\mu \in \mathbb L_{2n-1}(v)} e^{-\beta |\mu|} \psi_{v} \\
&= \sum_{v \in \Br^k_{2n-1}} y_{2n}(v)  b^k_{2n} e^{-4n \beta} e^{-k \beta} \psi_v  \ .
\end{split}
\end{equation}

Since 
\begin{equation}\label{02-02-18i}
\sum_{v \in \Br^k_j}B(k)^j_{u^k_0,v}\psi_v = e^{j\beta},
\end{equation}
we find from \eqref{02-02-18a}, \eqref{02-02-18b} and \eqref{24-01-18ax} that
\begin{equation*}
\begin{split}
& \frac{1}{2}\left( \frac{d^k_n}{a^k_n} e^{(n-2)\beta} + \frac{b^k_{2n-1}}{a^k_{2n-1}} e^{-2n\beta}\right) \ \leq \ e^{k\beta} \sum_{\mu \in \mathbb L_{2n-2}} e^{-\beta |\mu|}   \psi_{r(\mu)}\\
& \leq \ \frac{d^k_n}{a^k_n} e^{(n-2)\beta} + \frac{b^k_{2n-1}}{a^k_{2n-1}}e^{-2n\beta} \ ,
\end{split}
\end{equation*}
and from \eqref{02-02-18b} and \eqref{26-01-18kx} that
\begin{equation*}
\begin{split}
& \frac{1}{2}\frac{b^k_{2n}}{a^k_{2n}} e^{-(2n+1)\beta} \ \leq \ e^{k\beta} \sum_{\mu \in \mathbb L_{2n-1}} e^{-\beta |\mu|}   \psi_{r(\mu)}\\
& \leq \ \frac{b^k_{2n}}{a^k_{2n}} e^{-(2n+1)\beta}
\end{split}
\end{equation*}
for $n = 2,3, \cdots$.
It follows from these estimates and i) of Lemma \ref{20-01-18d}  that $\psi$ extends to a $e^{-\beta}A(\Gamma^0)$-harmonic vector $\overline{\psi} : \Gamma^0_V \to [0,\infty)$ if and only if both sums \eqref{02-02sum1} and \eqref{02-02sum2} are finite, i.e. if and only if $\beta \in I_k$. 

\begin{lemma}\label{08-02-18i} Let $\beta \in I_k$. For all $j = 0,1,2,3, \cdots $, the map 
$$
H^{u^k_0}_{\beta}(\Br^k) \ni \psi \mapsto \sum_{a \in M_k} \sum_{n=0}^{\infty} C_k(\beta)^n_{v_j, s(a)}e^{-\beta} \psi_{r(a)} 
$$ 
is continuous. 
\end{lemma}
\begin{proof} Set
$$
M_k(n) = \left\{ a \in M_k : \ r(a) \in \Br^k_n \right\} \ .
$$
Then
\begin{equation}\label{08-02-18n}
\begin{split}
& \sum_{a \in M_k} \sum_{i=0}^{\infty} C_k(\beta)^i_{v_j, s(a)}e^{-\beta}\psi_{r(a)}  = \sum_{n=0}^{\infty} \left(\sum_{a \in M_k(n)} \sum_{i=0}^{\infty} C_k(\beta)^i_{v_j, s(a)}e^{-\beta}\psi_{r(a)} \right) \ .
\end{split}
\end{equation}
For each $n$ the sum $\sum_{a \in M_k(n)} \sum_{i=0}^{\infty} C_k(\beta)^i_{v_j, s(a)}$ only has finitely many non-zero terms and hence 
$$
\psi \mapsto  \sum_{a \in M_k(n)} \sum_{i=0}^{\infty} C_k(\beta)^i_{v_j, s(a)}e^{-\beta} \psi_{r(a)} 
$$
is continuous for all $n$. Since
\begin{equation*}
\begin{split}
&C_k(\beta)^j_{v_0,v_j} \sum_{a \in M_k(n)} \sum_{i=0}^{\infty} C_k(\beta)^i_{v_j, s(a)}e^{-\beta}\psi_{r(a)}  \\
&\leq  \sum_{a \in M_k(n)} \sum_{i=0}^{\infty} C_k(\beta)^i_{v_0, s(a)}e^{-\beta}\psi_{r(a)}  = \sum_{\mu \in \mathbb L_n} e^{-\beta |\mu|} \psi_{r(\mu)} \ ,
\end{split}
\end{equation*}
we find that
$$
\sum_{a \in M_k(n)} \sum_{i=0}^{\infty} C_k(\beta)^i_{v_j, s(a)}e^{-\beta}\psi_{r(a)} \leq \left( C_k(\beta)^j_{v_0,v_j}\right)^{-1} \sum_{\mu \in \mathbb L_n} e^{-\beta |\mu|} \psi_{r(\mu)}
$$
for all $n$. Now, the estimates above show that
$$
 \sum_{\mu \in \mathbb L_{2n-2}} e^{-\beta |\mu|} \psi_{r(\mu)} \leq e^{-k\beta} \left(  \frac{d^k_n}{a^k_n} e^{(n-2)\beta} + \frac{b^k_{2n-1}}{a^k_{2n-1}}e^{-2n\beta} \right)
 $$
 and
 $$
\sum_{\mu \in \mathbb L_{2n-1}} e^{-\beta |\mu|} \psi_{r(\mu)} \leq e^{-k\beta} \frac{b^k_{2n}}{a^k_{2n}} e^{-(2n+1)\beta} ,
$$
for all $\psi \in  H^{u^k_0}_{\beta}(\Br^k)$ and all $n \geq 2$. Since $\beta \in I_k$ this implies that the sum \eqref{08-02-18n} converges uniformly in $\psi$. This completes the proof. 
\end{proof}

For the formulation of the next lemma note that since $\Br^k$ is  a simple Bratteli diagram, 
$$
E_k = \bigcup_{n=0}^{\infty} \sigma^{-n}\left( P\left( \Br^k\right)\right)
$$
is an end in $\Gamma^0$ for all $k = 1,2,3, \cdots $. There is one more end; the unique infinite path $p^0$ in $\Gamma^0$ which passes through all the $v_i$-vertexes represents an end $E_0$ disjoint from $\bigcup_{k=1}^{\infty}E_k$, and $E_0$ is the only minimal end in $\Gamma^0$; in fact, $E_0 \leq E_k$ for all $k$. 

In the following we let
\begin{equation*}\label{09-02-18}
H^{v_0}_{\beta}(E_k) = \left\{ \psi \in H^{v_0}_{\beta}(\Gamma^0) : \ m_{\psi} \in M^{v_0}_{\beta}(E_k) \right\} \ .
\end{equation*}
It follows from Lemma \ref{20-01-18d} that when $\beta \in I_k$ we can define a map
$$
\Lambda_k : H^{u^k_0}_{\beta}(\Br^k) \to H^{v_0}_{\beta}(\Gamma^0)
$$ 
such that
$$
\Lambda_k(\psi) = \frac{\overline{\psi}}{\overline{\psi}_{v_0}} \ .
$$

\begin{lemma}\label{05-02-18b} $\Lambda_k$ is continuous, injective,
\begin{equation}\label{05-02-18c}
\Lambda_k\left(H^{u^k_0}_{\beta}\left(\Br^k\right)\right) = H^{v_0}_{\beta}(E_k)
\end{equation}
 and
\begin{equation}\label{05-02-18d}
\Lambda_k\left( \partial H^{u^k_0}_{\beta}\left(\Br^k\right)\right) = \partial H^{v_0}_{\beta}(\Gamma^0) \cap  H^{v_0}_{\beta}(E_k) \ ,
 \end{equation}
 when $\beta \in I_k$.
 \end{lemma}
\begin{proof} The injectivity of $\Lambda_k$ is obvious and its continuity follows from Lemma \ref{08-02-18i}. The inclusion $\subseteq$ in \eqref{05-02-18c} follows from Lemma \ref{20-01-18d}. For the converse, let $\varphi \in H^{v_0}_{\beta}(E_k)$. Set 
$$
S(\varphi)_v = \frac{\varphi_v}{\varphi_{u^k_0}} \ , 
$$
when $v \in \Br^k_V$. Then $S(\varphi) \in H^{u^k_0}_{\beta}(\Br^k)$ and it follows from Lemma \ref{20-01-18d} that $\Lambda_k\left(S(\varphi)\right) = \varphi$. To establish \eqref{05-02-18d}, consider $\psi \in  \partial H^{u^k_0}_{\beta}\left(\Br^k\right)$ and $\varphi \in  H^{v_0}_{\beta}(\Gamma^0)$ such that $\varphi \leq t\Lambda_k(\psi)$ for some $t > 0$. This implies that $\varphi \in H^{v_0}_{\beta}(E_k)$ because it follows from \eqref{05-02-18c} that $\Lambda_k(\psi) \in H^{v_0}_{\beta}(E_k)$. Since 
$$
\varphi_v \ \leq \ \frac{t}{\overline{\psi}_{v_0}} \psi_v 
$$
for all $v \in \Br^k_V$ it follows that there is a $t' > 0$ such that $\varphi_v =t' \psi_v$ for all $v \in \Br^k_V$. It follows from Lemma \ref{20-01-18d} that $\varphi = t'\overline{\psi}$ and then that $\varphi = t' \overline{\psi}_{v_0} \Lambda_k(\psi)$. This shows that $ \Lambda_k(\psi) \in \partial H^{v_0}_{\beta}(\Gamma^0)$. Let then $\varphi \in \partial H^{v_0}_{\beta}(\Gamma^0) \cap  H^{v_0}_{\beta}(E_k)$. By \eqref{05-02-18c}, $\varphi = \Lambda_k(\psi)$ for some $\psi \in H^{u^k_0}_{\beta}(\Br^k)$. Assume that $\psi'  \in H^{u^k_0}_{\beta}(\Br^k)$ and $\psi' \leq t\psi$ for some $t > 0$. Then $\Lambda_k(\psi') \leq t' \Lambda_k( \psi) = t' \varphi$ for some $t' > 0$ which implies that $\Lambda_k(\psi') = t'' \varphi$ for some $t'' >0$ and hence 
$$
\psi' =\overline{\psi'}_{v_0}\Lambda_k(\psi')|_{\Br^k_V} = \overline{\psi'}_{v_0} t''\varphi|_{\Br^k_V} = \frac{\overline{\psi'}_{v_0}t''}{\overline{\psi}_{v_0}} \psi  \ .
$$ 
This shows that $\psi \in \partial H^{v_0}_{\beta}(\Br^k)$. 
\end{proof}

\begin{lemma}\label{29-01-18g} Let $\Br$ be a Bratteli diagram and $v_0 \in \Br_V$ its top vertex. Let $F$ be the constant potential function $F =1$. Then $M^{v_0}_{\beta}(\Br)$ is affinely homeomorphic to the tracial state space $T\left( AF(\Br)\right)$ of  $AF(\Br)$.
\end{lemma}
\begin{proof} Let $A$ be the adjacency matrix of $\Br$. It follows from Proposition \ref{thm?} that $M^{v_0}_{\beta}(\Br)$ is affinely homeomorphic to the closed subset of $\prod_{n=0}^{\infty} [0,\infty)^{\Br_n}$ consisting of the sequences $(\psi^n)_{n= 0}^{\infty}$ for which $\psi^0_{v_0} = 1$ and 
$$
\sum_{w \in \Br_{n+1}} A_{v,w}e^{-\beta}\psi^{n+1}_w = \psi^{n}_v
$$
for all $v\in \Br_n$ and all $n$. Via the map
$$
\left(\psi^n\right)_{n=0}^{\infty} \ \mapsto \ \left(e^{n \beta}\psi^n\right)_{n=0}^{\infty}
$$ 
this is affinely homeomorphic to the closed subset of $\prod_{n=1}^{\infty} [0,\infty)^{\Br_n}$ consisting of the sequences $(\varphi^n)_{n= 0}^{\infty}$ for which $\varphi^0_{v_0} = 1$ and 
$$
\sum_{w \in \Br_{n+1}} A_{v,w}\varphi^{n+1}_w =  \varphi^n_v
$$
for all $v\in \Br_n$ and all $n$. It is well-known that the latter set is a copy of $T(AF(\Br))$, but let's recall the reason. Let $M_k(\mathbb C)$ denote the $C^*$-algebra of complex $k \times k$ matrices. By definition, cf. \cite{Br1}, 
$$
AF(\Br) = \overline{\bigcup_n F_n},
$$
where $F_1 \subseteq F_2 \subseteq \cdots$ is an increasing sequence of finite-dimensional $C^*$-algebras such that
$$
F_n \ \simeq \  \oplus_{v \in \Br_n} M_{A^n_{v_0,v}}\left(\mathbb C\right)
$$ 
and the inclusion 
$$
\oplus_{v \in \Br_n} M_{A^n_{v_0,v}}\left(\mathbb C\right) \ \subseteq \ \oplus_{v \in \Br_{n+1}} M_{A^{n+1}_{v_0,v}}\left(\mathbb C\right)
$$
is a standard homomorphism whose multiplicity matrix is given by $A_{v,w}, \ v \in \Br_n, \ w \in \Br_{n+1}$. The tracial state of $AF(\Br)$ corresponding to a sequence $(\varphi^n)_{n=0}^{\infty}$ as above is the unique state
 whose restriction to 
$$
 M_{A^n_{v_0,v}}\left(\mathbb C\right) \subseteq F_n
 $$ 
 is $\varphi^n_v\Tr$ when we let $\Tr$ denote the canonical trace of a matrix algebra.
\end{proof}

Recall that a metrizable Choquet simplex $S$ is a Bauer simplex when the set $\partial S$ of its extreme points is closed in $S$. This implies that $S$ is affinely homeomorphic to the Borel probability measures $M(\partial S)$ equipped with the weak*-topology. Conversely, for any compact metric space $X$ the set $M(X)$ of Borel probability measures on $X$, equipped with the weak*-topology, is a Bauer simplex with $\partial M(X) = X$.

\begin{lemma}\label{19-02-18} Let $k \in \{1,2,3, \cdots \}$ and assume that $S_k$ is a Bauer simplex. Then $M^{v_0}_{\beta}(E_k) = \emptyset$ when $\beta \notin I_k$, and $M_{\beta}^{v_0}(E_k)$ is a closed face in $M^{v_0}_{\beta}(\Gamma^0)$ affinely homeomorphic to $S_k$ when $\beta \in I_k$.
\end{lemma}
\begin{proof} The $e^{-\beta} A(\Gamma^0)$-harmonic vector corresponding to an element of $M^{v_0}_{\beta}(E_k)$ is an extension of a $e^{-\beta}B(k)$-harmonic vector which we have shown can not exist unless $\beta \in I_k$. Hence $M^{v_0}_{\beta}(E_k) = \emptyset$ when $\beta \notin I_k$. Assume then that $\beta \in I_k$. It is obvious that $M^{v_0}_{\beta}(E_k)$ is a face in $M^{v_0}_{\beta}(\Gamma^0)$. Since $H^{u^k_0}_{\beta}(\Br^k)$ is compact it follows from the continuity of $\Lambda_k$ and \eqref{05-02-18c} that $H^{v_0}_{\beta}(E_k)$ is closed in $H^{v_0}_{\beta}(\Gamma^0)$. Hence $M^{v_0}_{\beta}(E_k)$ is a closed face in $M^{v_0}_{\beta}(\Gamma^0)$; in particular, $M^{v_0}_{\beta}(E_k)$ is a Choquet simplex. Note that $H^{u_0^k}_{\beta}\left(\Br^k\right)$ is affinely homeomorphic to $S_k$ by Lemma \ref{29-01-18g}. Since $M^{v_0}_{\beta}(E_k)$ is affinely homeomorphic to $H_{\beta}^{v_0}(E_k)$ it follows first from Lemma \ref{05-02-18b} that $\partial M^{v_0}_{\beta}(E_k)$ is closed in $M^{v_0}_{\beta}(E_k)$, which is therefore a Bauer simplex, and then that $\partial S_k$ is homeomorphic to $\partial M^{v_0}_{\beta}(E_k)$. Since a homeomorphism between the extreme boundaries of two Bauer simplices induce an affine homeomorphism between the simplices themselves, we conclude that $M^{v_0}_{\beta}(E_k)$ is affinely homeomorphic to $S_k$ when $\beta \in I_k$.

\end{proof}

Concerning the end $E_0$ note that the ray $p^0$ is $\beta$-summable for all $\beta \in \mathbb R$ and $E_0 = \mathbb B(p^0)$ consists of the elements of $P(\Gamma^0)$ that are tail-equivalent to $p^0$. It follows therefore from Theorem \ref{24-11-17} that $M^{v_0}_{\beta}(E_0) = \partial M^{v_0}_{\beta}(E_0)$ consists of one element for each $\beta$; the $e^{\beta}$-conformal measure $\nu_{\beta}$ such that
$\nu_{\beta}\left(Z(v_j)\right) = e^{j\beta}, \ j =0,1,2,3, \cdots$, and $\nu_{\beta}\left(Z(v)\right) = 0, \ v \in \bigcup_{k=1}^{\infty} \Br^k_V$.

\begin{lemma}\label{08-02-18j} Assume that each $S_k$ is a Bauer simplex, i.e. assume that there are compact metric spaces $X_k, k = 1,2,3, \cdots$, such that $S_k$ is affinely homeomorphic to the simplex $M(X_k)$ of Borel probability measures on $X_k$. For $\beta > 0$, set $I_{\beta} = \left\{ k \in \mathbb N : \ \beta \in I_k\right\}$. It follows that $M_{\beta}^{v_0}(\Gamma^0)$ is affinely homeomorphic to the Bauer simplex $M(X_{\beta})$ of Borel probability measures on the one-point compactification
$$
X_{\beta} = \{*\} \sqcup \sqcup_{k \in I_{\beta}} X_k
$$
of the topological disjoint union $\sqcup_{k \in I_{\beta}} X_k$.
\end{lemma}
\begin{proof} It follows from Proposition \ref{aug23} that
$$
\partial M^{v_0}_{\beta}(\Gamma^0) = \bigsqcup_{k=0}^{\infty} \partial M^{v_0}_{\beta}(E_k) \ ,
$$
as sets. By using Lemma \ref{19-02-18} and the identifications $\partial M^{v_0}_{\beta}(E_k) = X_k$ and $\nu_{\beta} = *$ this gives us a bijective map
$$
\Lambda : X_{\beta} = \{*\} \sqcup \sqcup_{k \in I_{\beta}} X_k \to \partial M^{v_0}_{\beta}(\Gamma^0) \ .
$$
It remains only to show that $\Lambda$ is continuous when $\partial M^{v_0}_{\beta}(\Gamma^0)$ is given the relative topology inherited from $M^{v_0}_{\beta}(\Gamma^0)$ and $X_{\beta}$ the topology of the one-point compactification of the topological disjoint union $\sqcup_{k \in I_{\beta}} X_k$. Assume $\lim_{k \to \infty} x_k = x$ in $X_{\beta}$. Unless $x = *$, all except finitely many $x_k$'s are in $X_j$ for some $j \in  I_{\beta}$ and then $\lim_{k \to \infty} \Lambda(x_k) = \Lambda(x)$ in $M^{v_0}_{\beta}(\Gamma^0)$ since the restriction of $\Lambda$ to $X_j$ is continuous. When $x = *$ the sequence $x_k$ must eventually leave $\partial M^{v_0}_{\beta}(E_j) = X_j$ for all $j \in I_{\beta}$ which implies that 
$$
\lim_{k \to \infty} \Lambda(x_k)(Z(v_j)) = e^{j \beta}
$$ 
and 
$$
\lim_{k \to \infty} \Lambda(x_k)(Z(v)) = 0, \ v \in \Br^j_V \ ,
$$ 
for all $j$. It follows that $\lim_{k \to \infty} \Lambda(x_k) = \nu_{\beta}$ when $x = *$.

\end{proof}

\begin{thm}\label{09-02-18a} Let $h > 0$ be a positive real number.

\begin{itemize}
\item Let $I_k, \ k  = 1,2,3, \cdots$, be a sequence of intervals in $]h,\infty[$.
\item Let $X_k, k = 1,2,3, \cdots$, be a sequence of compact metric spaces.
\end{itemize}
There is a strongly connected recurrent row-finite graph $\Gamma$ and a vertex $v_0 \in \Gamma_V$ such that
\begin{itemize}
\item $h(\Gamma) = h$,
\item $M^{v_0}_{h}(\Gamma)$ contains one element; a conservative $e^h$-conformal measure, and
\item for all $\beta > h$ the Choquet simplex $M^{v_0}_{\beta}(\Gamma)$ is affinely homeomorphic to the Bauer simplex $M(X_{\beta})$ of Borel probability measures on the one-point compactification $X_{\beta}$ of the topological disjoint union
$$
\sqcup_{k \in I_{\beta}} X_k \ ,
$$
where $I_{\beta} = \left\{k : \ \beta \in I_k \right\}$.
\end{itemize}
\end{thm}
\begin{proof} Let $\Gamma$ be graph obtained by applying Corollary \ref{attaching3} to the graph $\Gamma^0$ constructed above. Then $h(\Gamma) = h$ and the statement concerning the cases where $\beta > h$ follows from Proposition \ref{07-02-18k} and Lemma \ref{08-02-18j}. The case $\beta = h$ is covered by the results from \cite{Th3}, in particular Theorem 4.10 in \cite{Th3}. 
\end{proof}

If instead of $\Gamma$ we pick $\Gamma'$ from Corollary \ref{attaching3} we get the following. 

\begin{thm}\label{09-02-18b} Let $h > 0$ be a positive real number.

\begin{itemize}
\item Let $I_k, \ k  = 1,2,3, \cdots$, be a sequence of intervals in $[h,\infty[$.
\item Let $X_k, k = 1,2,3, \cdots$, be a sequence of compact metric spaces.
\end{itemize}
There is a strongly connected transient row-finite graph $\Gamma$ and a vertex $v_0 \in \Gamma_V$ such that
\begin{itemize}
\item $h(\Gamma) = h$, and
\item for all $\beta \geq h$ the Choquet simplex $M^{v_0}_{\beta}(\Gamma)$ is affinely homeomorphic to the Bauer simplex $M(X_{\beta})$ of Borel probability measures on the one-point compactification $X_{\beta}$ of the topological disjoint union
$$
\sqcup_{k \in I_{\beta}} X_k \ ,
$$
where $I_{\beta} = \left\{k : \ \beta \in I_k \right\}$.
\end{itemize}
\end{thm}

\begin{cor}\label{09-02-18c} Let $h > 0$ be a positive real number. There is a strongly connected recurrent row-finite graph $\Gamma$ with a vertex $v_0 \in \Gamma_V$ such that
\begin{itemize}
\item $h(\Gamma) = h$,
\item $M^{v_0}_{h}(\Gamma)$ consists of one point; a regular  conservative Borel measure on $P(\Gamma)$,
\item $M^{v_0}_{\beta}(\Gamma)$ is a Bauer simplex for all $\beta > h$, and
\item for all $\beta, \beta' > h$ the two simplices $M^{v_0}_{\beta}(\Gamma)$ and $M^{v_0}_{\beta'}(\Gamma)$ are only affinely homeomorphic if $\beta = \beta'$.
\end{itemize}
\end{cor}
\begin{proof} Let $I_k, k = 1,2,3, \cdots$, be a numbering of the bounded sub-intervals of $]h,\infty[$ with rational endpoints, and choose for each $k$ a compact connected space $X_k$ of covering dimension $k$, e.g. $X_k = [0,1]^k$. Apply Theorem \ref{09-02-18a} with this choice. For $\beta, \beta' > h$ the two spaces $X_{\beta}$ and $X_{\beta'}$ are only homeomorphic when the set of dimensions of their connected components are the same which only happens if $I_{\beta} = I_{\beta'}$. By the choice of intervals this is only true if $\beta = \beta'$.
\end{proof}

Similarly,

\begin{cor}\label{09-02-18d} Let $h > 0$ be a positive real number. There is a strongly connected transient row-finite graph $\Gamma$ with a vertex $v_0 \in \Gamma_V$ such that
\begin{itemize}
\item $h(\Gamma) = h$,
\item $M^{v_0}_{\beta}(\Gamma)$ is Bauer simplex for all $\beta \geq h$, and
\item for all $\beta, \beta' \geq h$ the two simplices $M^{v_0}_{\beta}(\Gamma)$ and $M^{v_0}_{\beta'}(\Gamma)$ are only affinely homeomorphic if $\beta = \beta'$.
\end{itemize}
\end{cor}

The construction of the examples leading to Corollary \ref{09-02-18c} and Corollary \ref{09-02-18d} can be varied quite a bit. For example, we used the covering dimensions of the connected components of the extreme boundary to distinguish the simplices corresponding to different $\beta$'s, but we might as well have used other characteristics of these components, like their fundamental groups, for example, while fixing the covering dimension to $2$, say.
\subsection{Re-formulations}
It seems appropriate to formulate some consequences of the last two corollaries for KMS-states and for countable state Markov chains. The following is now an immediate consequence of Corollary \ref{09-02-18c}.

\begin{thm}\label{09-02-18KMS} Let $h > 0$ be a positive real number. There is a strongly connected recurrent row-finite graph $\Gamma$ with a vertex $v_0 \in \Gamma_V$ such that the following holds for the restriction of the gauge action on $C^*(\Gamma)$ to the corner $P_{v_0}C^*(\Gamma)P_{v_0}$:
\begin{itemize}
\item There are no $\beta$-KMS states when $\beta < h$.
\item There is a unique $\beta$-KMS state for $\beta = h$.
\item For $\beta > h$ the simplex $S_{\beta}$ of $\beta$-KMS states is an infinite dimensional Bauer simplex.
\item For $\beta, \beta' \geq h$, the simplices $S_{\beta}$ and $S_{\beta'}$ are only affinely homeomorphic when $\beta = \beta'$.
\end{itemize}
\end{thm}

A similar reformulation of Corollary \ref{09-02-18d} establish the following.

 \begin{thm}\label{06-02-18d} Let $h > 0$ be a positive real number. There is a strongly connected transient row-finite graph $\Gamma$ and a vertex $v_0 \in \Gamma_V$ such that the following holds for the restriction of the gauge action on $C^*(\Gamma)$ to the corner $P_{v_0}C^*(\Gamma)P_{v_0}$:
\begin{itemize}
\item There are no $\beta$-KMS states when $\beta < h$.
\item For $\beta \geq h$ the simplex $S_{\beta}$ of $\beta$-KMS states is an infinite dimensional Bauer simplex.
\item For $\beta, \beta' \geq  h$, the simplices $S_{\beta}$ and $S_{\beta'}$ are only affinely homeomorphic when $\beta = \beta'$.
\end{itemize}
 \end{thm}

To describe some of the information regarding countable state Markov chains which can be obtained from the constructions above we adopt the notation and terminology from the book by Wolfgang Woess, \cite{Wo1}, which seems to be a standard reference.

\begin{thm}\label{09-02-18Markov} Let $0<\rho \leq 1$. There is an irreducible countable state Markov chain with finite range $(X,P)$ and a state $v_0 \in X$ such that 
\begin{itemize}
\item the spectral radius $\rho(P)$ is $\rho$,
\item $(X,P)$ is $\rho$-recurrent,
\item there is exactly one $\rho$-harmonic vector $\psi \in \mathcal H(P,\rho)$ with $\psi_{v_0} =1$, 
\item for $ t >  \rho$ the simplex 
\begin{equation}\label{26-02-18s}
H_t = \left\{\psi \in \mathcal H(P,t): \ \psi_{v_0} =1 \right\}
\end{equation}
is a Bauer simplex and
\item for $t,t' \geq \rho$ the simplices $H_t$ and $H_{t'}$ are only affinely homeomorphic when $t = t'$.
\end{itemize}
\end{thm}
\begin{proof} Let $A$ be the adjacency matrix of the graph $\Gamma$ in Corollary \ref{09-02-18c}. There is a $\beta \geq h$ such that $e^{h-\beta} = \rho$. Let $\psi$ is a $e^{-\beta }A$-harmonic vector and set
$$
P_{v,w} = e^{-\beta} A_{v,w}\psi_v^{-1}\psi_w \ .
$$
Then $P$ is stochastic and the transition matrix for a Markov chain with $X = \Gamma_V$ as state space. We leave the reader to check that $(X,P)$ has the asserted properties.
\end{proof}

If we in the above proof choose the graph from Corollary \ref{09-02-18d} instead we obtain the following

\begin{thm}\label{09-02-18Markov(trans)} Let $0< \rho \leq 1$. There is an irreducible countable state Markov chain with finite range $(X,P)$ and a state $v_0 \in X$ such that 
\begin{itemize}
\item the spectral radius $\rho(P)$ is $\rho$,
\item $(X,P)$ is $\rho$-transient,
\item for $ t \geq  \rho$ the simplex $H_t$ from \eqref{26-02-18s} is a Bauer simplex and
\item for $t,t' \geq  \rho$ the simplices $H_t$ and $H_{t'}$ are affinely homeomorphic only when $t = t'$.
\end{itemize}
\end{thm}

The properties of the examples described in Theorem \ref{09-02-18Markov} and Theorem \ref{09-02-18Markov(trans)} are closely related to the question about stability of the Martin boundary as formulated by Picardello and Woess in \cite{PW2}, \cite{PW3}. The main difference is that the stability issue of Picardello and Woess is about the full Martin boundary while the examples here concern the minimal Martin boundary. The most recent work concerning the stability of the Martin boundary of Markov chains seems to be \cite{I-R}, which also does not consider the minimal Martin boundary. 


\section{Appendix}\label{App1}

We justify in this section the statements concerning existence or absence of almost $A(\beta)$-harmonic vectors made in Section \ref{recap} and supply the arguments for Proposition \ref{17-03-18d}.

\begin{prop}\label{A1} Assume that $C^*(\Gamma)$ is simple and that $A(\beta)$ is recurrent, i.e. \eqref{27-02-18a} holds for some $v \in \Gamma_V$. There is an almost $A(\beta)$-harmonic vector if and only if $\limsup_n \left(A(\beta)^n_{v,v}\right)^{\frac{1}{n}} = 1$. When it exists it is $A(\beta)$-harmonic and unique up to multiplication by scalars. 
\end{prop}
\begin{proof}  Let $NW_{\Gamma}$ denote the set of non-wandering vertexes in $\Gamma$. The recurrence condition implies that $NW_{\Gamma} \neq \emptyset$. The stated conclusions follow then from Proposition 4.9 and Theorem 4.14 in \cite{Th3}.
\end{proof}


\begin{prop}\label{13-11-17} Assume that $\Gamma$ is cofinal and row-finite without sinks and that \eqref{transient0} holds. There is an $e^{\beta F}$-conformal measure on $P(\Gamma)$ if and only if $\Wan(\Gamma) \neq \emptyset$.
\end{prop}  
\begin{proof} If $\Wan(\Gamma)$ is empty it follows from Lemma \ref{12-11-17x} that there are no $e^{\beta F}$-conformal measure on $P(\Gamma)$. Assume $\Wan(\Gamma)$ is not empty. Observe first that it follows from Proposition 4.3 in \cite{Th3} that in a cofinal digraph with $NW_{\Gamma} \neq \emptyset$ all infinite paths end up in the subgraph whose vertexes are the elements of $NW_{\Gamma}$. Therefore, if $NW_{\Gamma}$ is a finite set, no infinite path can be wandering. Since we assume that $\Wan(\Gamma) \neq \emptyset$ it follows that $NW_{\Gamma}$ is either empty or an infinite set. When $NW_{\Gamma} = \emptyset$ the existence of an $A(\beta)$-harmonic vector follows from Theorem 4.8 in \cite{Th3}. When $NW_{\Gamma}$ is infinite the existence follows by combining Proposition \ref{nov1} above with the Corollary in \cite{Pr} and Lemma 2.3 in \cite{Th2}.
\end{proof}

 We consider cofinal digraphs of type A,B,C,D and E as explained by the diagram following Theorem \ref{intro2} in Section \ref{recap}, and we assume that $A(\beta)$ is transient in the sense that \eqref{transient0} holds.

\subsection{Type A} Since there are no infinite emitters and no sink all the KMS measures considered in Section 3 of \cite{Th3} are harmonic, and hence Theorem 2.7 and Theorem 3.8 of \cite{Th3} imply that there are no proper almost $A(\beta)$-harmonic vectors. The existence of an $A(\beta)$-harmonic vector follows from Proposition \ref{13-11-17} and Proposition \ref{nov1}.

\subsection{Type B} It follows as for type A that there are no proper almost $A(\beta)$-harmonic vectors. That there also are no $A(\beta)$-harmonic vectors follows from Proposition \ref{nov1} and Proposition \ref{13-11-17}.

\subsection{Type C} Since there are no sinks it follows from Corollary 3.5 of \cite{Th3} that there is a bijective correspondence between the set of infinite emitters in $\Gamma$ and the set of extremal boundary $\beta$-KMS measures on the space $\Omega_{\Gamma}$; see \cite{Th3} for the definitions. By Theorem 2.7 and Theorem 3.8 in \cite{Th3} the latter set is in bijective correspondence with the extremal rays of proper almost $A(\beta)$-harmonic vectors. The graphs constructed in Section 8.1 in \cite{Th3} are examples of graphs of type C for which there are $A(\beta)$-harmonic vectors, also in the transient case, and Example \ref{05-12-17} below gives an example of a strongly connected digraph of type C for which there are no $A(\beta)$-harmonic vectors when $A(\beta)$ is transient.

\subsection{Type D} For the same reason as for type C there is a bijective correspondence between extremal rays of proper almost $A(\beta)$-harmonic vectors and the infinite emitters in $\Gamma_V$. As for graphs of type B it follows from Proposition \ref{nov1} and Lemma \ref{12-11-17x} of the present paper that there are no $A(\beta)$-harmonic vectors. 
 
\subsection{Type E} By 3) in Corollary 4.2 of \cite{Th3} there are no wandering paths in this case. The absence of $A(\beta)$-harmonic vectors follows therefore again from Proposition \ref{nov1} and Lemma \ref{12-11-17x}. That there is an essentially unique proper almost $A(\beta)$-harmonic vector follows from a) of Theorem 4.8 in \cite{Th3}.

 \begin{example}\label{05-12-17}

Consider the following graph $\Gamma^0$ and equip it with the constant potential $F =1$.

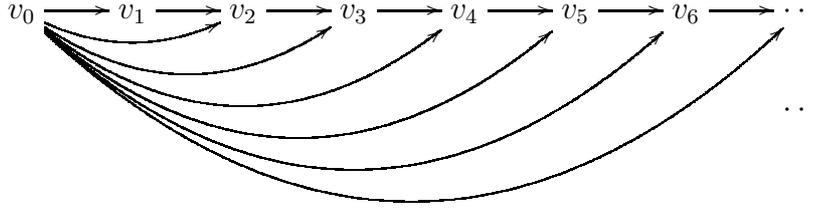
\begin{xymatrix}{
v_0  \ar@/_1pc/[rr]  \ar@/_2pc/[rrr] \ar@/_3pc/[rrrr]  \ar@/_4pc/[rrrrr]    \ar@/_5pc/[rrrrrr]   \ar@/_6pc/[rrrrrrr]  \ar[r] & v_1\ar[r]  & v_2 \ar[r]  & v_3 \ar[r]  & v_4 \ar[r]  & v_5 \ar[r]  & v_6 \ar[r] & \cdots \\
&&&&&&&\cdots& }
\end{xymatrix}   

\bigskip

\bigskip

\bigskip

\bigskip

Let $A$ be the adjacency matrix of $\Gamma^0$. An $e^{-\beta}A$-harmonic vector $\psi$ must satisfy the following conditions:
\begin{itemize}
\item $\psi_{v_k} = e^{(k-1)\beta} \psi_{v_1}, \ k \geq 2$, and
\item $\psi_{v_0} = e^{-\beta}\psi_{v_1} + \sum_{k=2}^{\infty} e^{-\beta}\psi_{v_k} $.
\end{itemize}
No non-zero non-negative vector $\psi$ meets these conditions when $\beta > 0$, i.e. $H^{v_0}_{\beta}(\Gamma^0) = \emptyset$ when $\beta > 0$. Since $\Gamma^0$ is cofinal and $e^{-\beta}A$ is transient for all $\beta \in \mathbb R$ this gives the example required above concerning graphs of type C. A strongly connected example can be obtained by adding return paths to $v_0$. Specifically, $\Gamma^0$ satisfies the three conditions in Corollary \ref{25-06-18d} and we can therefore add return paths to $v_0$ to obtain a strongly connected graph $\Gamma$ for which the Gurevich entropy $h(\Gamma)$ can be any positive number $h > 0$. Let $B$ be the adjacency matrix of $\Gamma$. For $\beta > h(\Gamma)$ the matrix $e^{-\beta}B$ is transient, but it follows from Proposition \ref{07-02-18k} that $H^{v_0}_{\beta}(\Gamma) = \emptyset$. However, there will be a unique ray of proper almost $e^{-\beta}B$-harmonic vectors, and hence a unique ray of $\beta$-KMS weights for the gauge action on $C^*(\Gamma)$ for all $\beta \geq h(\Gamma)$; all resulting from the presence of the infinite emitter $v_0$. 
\end{example}

\subsection{The set of inverse temperatures} In this section we prove Proposition \ref{17-03-18d} from Section \ref{recap}. The proposition consists of four items where the fourth and last item follows directly from Theorem 4.8 in \cite{Th3}. It remains therefore only to consider the first three cases in which the set of non-wandering vertexes $NW_{\Gamma}$ is non-empty. By Proposition 4.9 in \cite{Th3} we may therefore assume that $\Gamma$ is strongly connected.

Let $\Gamma$ be a strongly connected digraph with $C^*(\Gamma)$ simple and $F : \Gamma_{Ar} \to \mathbb R$ a potential. For all $\beta \in \mathbb R$, set
$$
\rho(A(\beta)) = \limsup_n \left( A(\beta)^n_{v,v}\right)^{\frac{1}{n}} ,
$$
for some $v \in \Gamma_V$. Since $\Gamma$ is strongly connected this number in $[0,\infty]$ is independent of the vertex $v$, and
$$
\mathbb P(-\beta F) = \log \rho(A(\beta))
$$
is the pressure function which was used in \cite{Th3}.
\begin{lemma}\label{postiveF}  Let $\Gamma$ be a strongly connected digraph with $C^*(\Gamma)$ simple. Assume that either
\begin{itemize}
\item[a)] there is a loop $\mu$ in $\Gamma$ with $F(\mu) = 0$ or
\item[b)] there are loops $\mu_1$ and $\mu_2$ in $\Gamma$ such that $F(\mu_1) < 0 < F(\mu_2)$.
\end{itemize}
Then $\rho(A(\beta)) > 1$ for all $\beta \in \mathbb R$, and there are no $\beta$-KMS weights for $\alpha^F$.

\end{lemma}
\begin{proof} Assume a). Set $v = s(\mu)$. Since $C^*(\Gamma)$ is simple all loops have an exit by Theorem \ref{Sz}. Since $\Gamma$ is strongly connected this implies that there is a path $\nu
$ such that $|\nu| = m|\mu|$ for some $m \in \mathbb N$, $s(\nu) =
r(\nu) = s(\mu) = v$ and $\nu$ is not the concatenation of $m$ copies of
$\mu$. Since $F(\mu) =0$ it follows that
$$
A(\beta)^{nm|\mu|}_{v,v}  \geq
\left(A(\beta)^{m|\mu|}_{v,v}\right)^n  \geq (e^{-\beta m F(\mu)} + e^{-\beta F(\nu)})^n
=( 1+ e^{-\beta  F(\nu)})^n
$$
for all $n \in \mathbb N$, showing that
\begin{equation*}\label{30-04-18}
\rho\left( A(\beta)\right)  \geq  \left(1 + e^{-\beta
    F(\nu)}\right)^{\frac{1}{m|\mu|}} > 1
\end{equation*}
for all $\beta \in \mathbb R$.

Assume b). Let $v \in \Gamma_V$. The assumptions imply that there are loops $\nu_1,\nu_2$ in $\Gamma$ such that $s(\nu_1)= s(\nu_2) =v$ and  $F(\nu_1) < 0
  < F(\nu_2)$. Then
$$
A(\beta)^{n|\nu_1||\nu_2|}_{v,v} \geq \max \left\{
  e^{-\beta n |\nu_2| F(\nu_1)}, e^{-\beta n |\nu_1| F(\nu_2)} \right\}
$$
for all $n \in \mathbb N$, proving that
$$
\rho\left( A(\beta)\right)  \geq  \max \left\{
  e^{-\beta \frac{F(\nu_1)}{|\nu_1|}}, e^{-\beta
    \frac{F(\nu_2)}{|\nu_2|}} \right\} > 1
$$ 
for all $\beta \neq 0$. When $\beta = 0$ we can use the arguments from a) to deduce that for any loop $\mu$ in $\Gamma$ there is an $m \in \mathbb N$ such that
\begin{equation}\label{17-03-18b}
\rho(A(0)) \ \geq \ 2^{\frac{1}{m|\mu|}} \ > \ 1 \ .
\end{equation}
Since $\rho(A(\beta)) > 1$ for all $\beta \in \mathbb R$, it follows from Lemma 4.13 in \cite{Th3} that there are no $\beta$-KMS weights for $\alpha^F$. 
\end{proof}


\begin{lemma}\label{a1}  Let $\Gamma$ be a strongly connected digraph. Then $0 < \rho\left(A(\beta)\right) \leq \infty$ for all
  $\beta \in \mathbb R$ and the function
$$
\mathbb R \ni \beta \mapsto \rho\left(A(\beta)\right)
$$
is lower semi-continuous.
\end{lemma}
\begin{proof} Let $\mu$ be a loop in $\Gamma$ and set $v = s(\mu)$. Then 
\begin{equation*}
 A(\beta)^{|\mu|n}_{v,v} \geq e^{-\beta n F(\mu)}
\end{equation*}
for all $n$, and hence
$$
\rho(A(\beta)) \geq \left( e^{-\beta F(\mu)}\right)^{\frac{1}{|\mu|}} \ >  \ 0 \  .
$$
It follows from Lemma 3.10
in \cite{GV} that we can choose an increasing sequence $H_1 \subseteq
H_2 \subseteq H_3 \subseteq \cdots $ of strongly connected finite
subgraphs of $\Gamma$ such that
$$
\lim_{n \to \infty} \rho\left( A(\beta)|_{H_n}\right) = \sup_n
\rho\left( A(\beta)|_{H_n}\right) =
\rho\left(A(\beta)\right)
$$
for all $\beta \in \mathbb R$. Since $  \beta \mapsto \rho\left(
  A(\beta)|_{H_n}\right)$ is continuous for each $n$ by Lemma 4.1 in
\cite{CT1}, it follows that $\rho\left(A(\beta)\right)$ is lower semi-continuous. 

\end{proof}

\begin{prop}\label{x1} Let $\Gamma$ be a strongly connected digraph with $C^*(\Gamma)$ simple and let $F : \Gamma_{Ar} \to \mathbb R$ be a potential. Assume that $\Gamma$ is infinite. There is a $\beta$-KMS weight for $\alpha^F$ if
  and only if $\rho(A(\beta)) \leq 1$.
\end{prop}
\begin{proof} When $\Gamma_V$ is infinite the assertion follows from Proposition 4.19 in \cite{Th3}. Assume $\Gamma_V$ is finite. There are then infinite emitters since we assume that $\Gamma$ is infinite. If there exists a $\beta$-KMS weight it follows from Lemma 4.13 in \cite{Th3} that $\rho(A(\beta)) \leq 1$. Conversely, if $\rho(A(\beta)) \leq 1$ and $A(\beta)$ is recurrent it follows that $\rho(A(\beta)) = 1$ and the existence of a $\beta$-KMS weight is guaranteed by Theorem 4.14 in \cite{Th3}. When $\rho(A(\beta)) \leq 1$ and the matrix $A(\beta)$ is transient the existence of a $\beta$-KMS weight follows from Proposition 4.16 in \cite{Th3}.
\end{proof}

\begin{lemma}\label{positiveF2} Let $\Gamma$ be a strongly connected infinite digraph such that $C^*(\Gamma)$ is simple, and let $F : \Gamma_{Ar} \to \mathbb R$ be a potential. If there exists a KMS weight for
  $\alpha^F$ the set 
  $$\beta(F) = \left\{ \beta \in \mathbb R: \ \text{There is a $\beta$-KMS weight for $\alpha^F$} \right\}
  $$
 is either an interval of the form $[\beta_0, \infty)$ for some
$\beta_0 > 0$ or an interval of the form $(-\infty, \beta_0]$ for some
$\beta_0 < 0$. The first case occurs when there is a loop $\mu$ in $\Gamma$ such that $F(\mu) > 0$, and the second when there is a loop $\mu$ in $\Gamma$ such that $F(\mu) < 0$. 
\end{lemma}
\begin{proof} Assume that there exists a $\beta$-KMS weight for some $\beta$. By Proposition \ref{x1} and Lemma \ref{postiveF} this implies that neither a) nor b) from
  Lemma \ref{postiveF} holds. Assume that $F(\mu) > 0$ for all loops $\mu$ in $\Gamma$. Fix a vertex $v \in \Gamma_V$ and let $\mathcal L_n$ be the set of loops of length $n$ based at $v$. If $\beta < \beta'$ we find that
$$
A(\beta)^n_{v,v} = \sum_{\mu \in \mathcal L_n} e^{-\beta F(\mu)} \  \geq \ \sum_{\mu \in \mathcal L_n} e^{-\beta' F(\mu)} = A(\beta')^n_{v,v} \ .
$$
Hence $\beta \mapsto \rho(A(\beta))$ is non-increasing. By assumption and Proposition \ref{x1} the set
$$
\left\{ \beta \in \mathbb R : \rho(A(\beta)) \leq 1 \right\}
$$
is not empty. Since
$\rho(A(0)) > 1$ by \eqref{17-03-18b} it follows that 
$$
\beta_0 = \inf \left\{ \beta \in \mathbb R : \rho(A(\beta)) \leq 1 \right\}  \ 
$$
is not negative. In fact, since $\rho(A(0)) > 1$ and $\beta \mapsto \rho(A(\beta))$ is lower semi-continuous by Lemma \ref{a1}, $\beta_0 > 0$ and $\rho(A(\beta_0)) \leq 1$. Since $\beta \mapsto \rho\left(A(\beta)\right)$ is non-increasing, $\rho(A(\beta)) \leq 1$ for all $\beta \geq \beta_0$ which by Proposition \ref{x1} implies that $\beta(F)$ is the set $[\beta_0,\infty)$. The case when $F(\mu) < 0$ for all loops is handled the same way, leading to the conclusion that there is then a $\beta_0 < 0$ such that $\beta(F)$ is the interval $(-\infty, \beta_0]$.
\end{proof}

Since we can assume that $\Gamma$ is strongly connected the first and third item of Proposition \ref{17-03-18d} now follow from Lemma \ref{positiveF2}. Finally, the second item of Proposition \ref{17-03-18d} where we may assume that $\Gamma$ is both finite and strongly connected follows from \cite{CT1}; in particular, from Lemma 4.2 and Theorem 4.10 in \cite{CT1}.

It remains to justify the remark made at the end of Section \ref{recap};  'that when the set $\beta(F)$ of Proposition \ref{17-03-18d} is an infinite interval, the recurrent case occurs only when $\beta$ is equal to $\beta_0$, the endpoint of the interval and sometimes not even then.'

\begin{prop}\label{23-07-18b} In the setting of Lemma \ref{positiveF2}, assume $\beta \in \beta(F)$ and that $\sum_{n=0}^{\infty} A(\beta)^n_{v,v} = \infty$ for some $v \in \Gamma_V$. Then $\beta = \beta_0$.
\end{prop}
\begin{proof} By Lemma \ref{postiveF} we have either that $F(\mu) > 0$ for all loops $\mu$ or $F(\mu) < 0$ for all loops $\mu$. Assume that $F(\mu) > 0$ for all loops and for a contradiction that $\beta > \beta_0$. For each $n \in \mathbb N$ and $\beta' \in \{\beta, \beta_0\}$, set
$$
l^n(\beta')_{v,v} = \sum_{\nu} e^{-\beta' F(\nu)},
$$
where we sum over all loops $\nu$ in $\Gamma$ of length $|\nu| = n$ with $s(\nu) = r(\nu) = v$ and the property that $v$ only occurs in the start of $\nu$ and the end of $\nu$. Then
\begin{equation}\label{23-07-18d}
\sum_{n=0}^{\infty} A(\beta')^n_{v,v} = \sum_{k=0}^{\infty} \left( \sum_{j=1}^{\infty} l^j(\beta')_{v,v}\right)^k
\end{equation}
by Lemma 4.11 in \cite{Th3}. Furthermore, since $\beta' \in \beta(F)$ it follows from Theorem 2.7 in \cite{Th3} that there is a non-zero vector $\psi \in [0,\infty)^{\Gamma_V}$ such that $\sum_{w \in \Gamma_V} A(\beta')_{u,w}\psi_w \leq \psi_u$ for all $u,w \in \Gamma_V$. Note that $\psi_v > 0$ since $\Gamma$ is strongly connected. By a result of Vere-Jones, restated in Lemma 3.6 in \cite{Th2}, it follows that 
\begin{equation}\label{23-07-18e} 
\sum_{j=1}^{\infty} l^j(\beta')_{v,v} \leq 1 \ .
\end{equation}
In particular, $\sum_{j=1}^{\infty} l^j(\beta_0)_{v,v} \leq 1$ which implies that $\sum_{j=1}^{\infty} l^j(\beta)_{v,v} <1$ since $e^{-\beta F(\mu)} < e^{-\beta_0 F(\mu)}$ for every loop $\mu$. Then \eqref{23-07-18d} implies that $\sum_{n=0}^{\infty} A(\beta)^n_{v,v}< \infty$, contradicting the assumption. - The argument which handles the case when $F(\mu) < 0$ for all loops is completely analogous.   
\end{proof}

Note that in the setting of Lemma \ref{positiveF2} there will be a recurrent $\beta_0$-KMS weight if and only if 
$$
\sum_{j=1}^{\infty} l^j(\beta_0)_{v,v} = 1,
$$
which very often is not the case.

\section{Appendix}\label{appB} In this appendix we obtain an integral representation of $e^{\beta F}$-conformal measures and $A(\beta)$-harmonic vectors similar to the Poisson-Martin integral representation from the theory of countable state Markov chains. The key tool is a selection theorem of Burgess, cf. \cite{Bu1} and \cite{Bu2}.

Let $\Gamma$ be a countable digraph with a vertex $v_0$ such that \eqref{v0trans} holds. Recall that $X_{\beta}$ is the set of infinite paths $p \in P(\Gamma)$ with the property that $\psi_v = \lim_{k \to \infty} K_{\beta}(v,s(p_k))$ exists for all $v \in \Gamma_V$ and the resulting $v_0$-normalized $A(\beta)$-harmonic vector $\psi$ is extremal. Set
$$
C_{\beta} = X_{\beta} \cap \Ray(\Gamma) \cap Z(v_0) \ ;
$$ 
a Borel subset of $Z(v_0)$.  Define $K : C_{\beta}  \to \prod_{v \in \Gamma_V} \left[0,b_v\right]$ such that
$$
K(x)_v = \lim_{k \to \infty} K_{\beta}(v,s(x_k)) \ .
$$
Then $K$ is a Borel map and it defines an equivalence relation $x \ \sim_K \ y$ on  $C_{\beta}$ such that $x \ \sim_K \ y$ iff $K(x) = K(y)$. 

\begin{lemma}\label{A1} There is a topology $\tau$ on $C_{\beta}$ which is finer than the relative topology inherited from $P(\Gamma)$ such that 
\begin{itemize}
\item $C_{\beta}$ is a Polish space in the $\tau$-topology,
\item the Borel $\sigma$-algebra generated by $\tau$ is the same as the Borel $\sigma$-algebra inherited from $P(\Gamma)$ and
\item $K : C_{\beta} \to \prod_{v \in \Gamma_V} \left[0,b_v\right]$  is continuous with respect to the $\tau$-topology.
\end{itemize}
\end{lemma}
\begin{proof} This follows from standard results on the Borel structure of Polish spaces, cf. eg. Theorem 3.2.4 and Corollary 3.2.6 in \cite{Sr}. 
\end{proof}

Thanks to Lemma \ref{A1} we can apply a result of Burgess, \cite{Bu1}, \cite{Bu2}, stated in the corollary to Proposition I of \cite{Bu2}. Recall that an \emph{analytic subset} of a Polish space is the image of a Polish space under a continuous map.

\begin{prop}\label{A2} Consider $C_{\beta}$ as a Polish space in the topology from Lemma \ref{A1}. There is a map $T: C_{\beta} \to C_{\beta}$ which is measurable with respect to the $\sigma$-algebra generated by the analytic subsets such that
\begin{itemize}
\item $K(T(p)) = K(p)$ for all $p \in C_{\beta}$,
\item $K(p) = K(q) \ \Leftrightarrow \ T(p) = T(q)$ for all $p,q \in C_{\beta}$ and
\item $T(C_{\beta})$ is a co-analytic subset of $C_{\beta}$, i.e. $C_{\beta} \backslash T(C_{\beta})$ is analytic.
\end{itemize}
\end{prop} 

The \emph{universally measurable} subsets of $P(\Gamma)$ are the subsets of $P(\Gamma)$ that are $\mu$-measurable for every $\sigma$-finite Borel measure $\mu$ on $P(\Gamma)$, cf. page 280 in \cite{Co}. Thus $A \subseteq P(\Gamma)$ is universally measurable iff the following holds: For every $\sigma$-finite Borel measure $\mu$ on $P(\Gamma)$ there are Borel sets $B_1,B_2$ in $P(\Gamma)$ such that $B_1 \subseteq A \subseteq B_2$ and $\mu(B_2 \backslash B_1) = 0$. By Corollary 8.4.3 in \cite{Co} every analytic subset of a Polish space is universally measurable. Note also that the set of universally measurable sets constitute a $\sigma$-algebra containing the Borel sets. Therefore Proposition \ref{A2} has the following

\begin{cor}\label{A3} There is a map $T: C_{\beta} \to C_{\beta}$ which is measurable with respect to the $\sigma$-algebra of universally measurable sets in $C_{\beta}$ such that
\begin{itemize}
\item $K(T(p)) = K(p)$ for all $p \in C_{\beta}$,
\item $K(p) = K(q) \ \Leftrightarrow \ T(p) = T(q)$ for all $p,q \in C_{\beta}$ and
\item $T(C_{\beta})$ is universally measurable.
\end{itemize}
\end{cor}

For every element $y \in C_{\beta}$ we let $m_y$ be the unique normalized $e^{\beta F}$-conformal measure on $P(\Gamma)$ such that $m_y(Z(v)) = \lim_{k \to \infty} K_{\beta}(v,s(y_k))$ for all $v \in \Gamma_V$. 

\begin{lemma}\label{04-12-17g} $m_y$ is an extremal $v_0$-normalized $e^{\beta F}$-conformal measure concentrated on
\begin{equation}\label{04-12-17h}
\left\{p \in \Wan(\Gamma) : \ \lim_{k \to \infty} K_{\beta}(v,s(p_k)) = \lim_{k \to \infty} K_{\beta}(v,s(y_k)) \ \text{for all} \ v \in \Gamma_V \right\} \ .
\end{equation}
\end{lemma}
\begin{proof} By definition of $C_{\beta}$ the vector 
$$
\psi_v = \lim_{k \to \infty} K_{\beta}(v,s(y_k)), \ \ v \in \Gamma_V,
$$
is in $\partial H^{v_0}_{\beta F}(\Gamma)$. Since the bijection of Proposition \ref{nov1} is affine, this implies that $m_y \in \partial M^{v_0}_{\beta F}(\Gamma)$. It follows then from Theorem \ref{OKok} and Lemma \ref{12-11-17x} that $m_y$ is concentrated on \eqref{04-12-17h}.
\end{proof}

\begin{lemma}\label{12-02-18} $\partial M^{v_0}_{\beta F}(\Gamma) = \left\{ m_y: \ y \in C_{\beta} \right\}$.
\end{lemma}
\begin{proof} By Lemma \ref{04-12-17g} it remains only to prove the inclusion $\subseteq$. Let $m \in \partial M^{v_0}_{\beta F}(\Gamma)$. It follows from Theorem \ref{OKok} and Lemma \ref{12-11-17x} that there is a $y' \in Z(v_0) \cap \Wan(\Gamma)$ such that $m(Z(v)) = \lim_{k \to \infty} K_{\beta}(v,s(y'_k))$ for all $v \in \Gamma_V$. Set $y = R(y')$ where $R : \Wan(\Gamma) \to \Ray(\Gamma)$ is the retraction \eqref{21-02-18b}. Then $y \in C_{\beta}$ and it follows from Proposition \ref{nov1} that $m=m_y$.
\end{proof}

Set 
$$
Y_{\beta} = T(C_{\beta}) \ ,
$$
where $T : C_{\beta} \to C_{\beta}$ is the universally measurable map from Corollary \ref{A3}. It follows from the third item of Corollary \ref{A3} that $Y_{\beta}$ is universally measurable, and from the first item of  Corollary \ref{A3} and Proposition \ref{nov1} that $m_y = m_{T(y)}$ for all $y \in C_{\beta}$. By Lemma \ref{12-02-18} we can therefore conclude that
$$
\partial M^{v_0}_{\beta F}(\Gamma) = \left\{ m_y: \ y \in Y_{\beta} \right\} \ .
$$ 
If $y,y' \in Y_{\beta}$ are such that $m_y = m_{y'}$ it follows from Lemma \ref{04-12-17g} that $K(y) = K(y')$ and then from the first two items of Corollary \ref{A3} that $y = y'$. Thus the map $y \mapsto m_y$ is injective on $Y_{\beta}$ and hence is a bijection from $Y_{\beta}$ onto $\partial M^{v_0}_{\beta F}(\Gamma)$.

Let $\nu$ be a Borel probability measure on $Y_{\beta}$. That is, $\nu$ is a Borel probability measure on $P(\Gamma)$ such that there is a Borel subset $B_{\nu} \subseteq Y_{\beta}$ with $\nu(B_{\nu}) =1$. An application of Lemma \ref{03-03-18} shows that the map $C_{\beta} \ni y \mapsto m_y(B)$ is a Borel function for all Borel subsets $B \subseteq P(\Gamma)$, and the integral
$$
m(B) = \int_{Y_{\beta}} m_y(B) \ d\nu(y) 
$$
is therefore defined. The resulting measure
\begin{equation}\label{A5}
m = \int_{Y_{\beta}} m_y \ \mathrm{d}\nu(y)
\end{equation}
is a $v_0$-normalized $e^{\beta}$-conformal measure since each $m_y$ is.

\begin{lemma}\label{09-11-17a}
The retraction $R: \Wan(\Gamma) \to \Ray(\Gamma)$ from  \eqref{21-02-18b} is a Borel map.
\end{lemma}
\begin{proof} 
Let $\mu = e_1e_2\cdots e_m$ be a finite path in $\Gamma$. It suffices to show that $R^{-1}\left(Z(\mu) \cap \Ray(\Gamma)\right)$ is a Borel set in $\Wan(\Gamma)$. Let $M$ be the vertexes occurring in $\mu$. If they are not distinct, $Z(\mu) \cap \Ray(\Gamma) = \emptyset$ and we are done. So assume that the vertexes in $M$ are distinct. Let 
$$
A = \left\{ p \in Z(r(\mu)) \cap \Wan(\Gamma) : \ s(p_i) \notin M \ \forall i \right\} \ ,
$$
which is the intersection of $\Wan(\Gamma)$ with a closed subset of $P(\Gamma)$ and hence a Borel set. Note that
\begin{equation*}
\begin{split}
&R^{-1}\left(Z(\mu) \cap \Ray(\Gamma)\right) \\
&=  \bigcup_{\nu_1,\nu_2,\cdots, \nu_{m}} Z\left( \nu_1e_1\nu_2e_2 \cdots e_{n-1}\nu_{m}e_m\right) A \ ,
\end{split}
\end{equation*}
where the $\nu_i$'s are finite paths in $\Gamma$ such that $s(\nu_i) = r(\nu_i) = s(e_i)$ and $\nu_i$ does not contain any of the vertexes $s(e_j), \ j < i$. This exhibits $R^{-1}\left(Z(\mu) \cap \Ray(\Gamma)\right)$ as a Borel set.
\end{proof}

\begin{thm}\label{03-12-17}  The map $\nu \mapsto m$ defined by \eqref{A5} is an affine bijection between the set of Borel probability measures $\nu$ on $Y_{\beta}$ and the set $M^{v_0}_{\beta F}(\Gamma)$ of normalized $e^{\beta F}$-conformal measures $m$ on $P(\Gamma)$.
 \end{thm}
\begin{proof} To show that the map is surjective, let $m \in M^{v_0}_{\beta F}(\Gamma)$. Set
$$
B_{\beta}  = Z(v_0) \cap \Wan(\Gamma) \cap X_{\beta}  \ ,
$$
which is a Borel set such that $R\left(B_{\beta}\right) = C_{\beta}$, where $R$ is the retraction \eqref{21-02-18b}. Furthermore, $m(B_{\beta}) = m(Z(v_0)) = 1$ by Corollary \ref{OKok2} and Lemma \ref{12-11-17x}. $R$ is a Borel map by Lemma \ref{09-11-17a} and hence the composition $W = T \circ R$ is a universally measurable map
$$
W  : \ B_{\beta} \to Y_{\beta} \ .
$$
When we also let $m$ denote its own completion we obtain a Borel measure $ \nu$ on $P(\Gamma)$ defined such that
$$
\nu (B) =  m \circ W^{-1}\left( B \cap Y_{\beta}\right) \ .
$$
Since $Y_{\beta}$ is universally measurable and $m\left(W^{-1}(Y_{\beta})\right) = m(B_{\beta}) = 1$, we see that $\nu$ is a Borel probability measure on $Y_{\beta}$. Using the first item in Corollary \ref{A3} as well as Corollary \ref{OKok2} we find that
\begin{equation*}
\begin{split}
&\int_{Y_{\beta}} m_y(Z(v)) \ \mathrm{d}\nu(y) = \int_{B_{\beta}} m_{W(p)}(Z(v)) \ \mathrm{d} m(p) \\
&=  \int_{B_{\beta}} \lim_{k \to \infty} K_{\beta}(v,s(W(p)_k))  \ \mathrm{d} m(p) \\
&=  \int_{Z(v_0)} \lim_{k \to \infty} K_{\beta}(v,s(p_k)) \ \mathrm{d} m(p) =  m(Z(v))
\end{split}
\end{equation*}
for all $v \in \Gamma_V$; i.e. $m =  \int_{Y_{\beta}} m_y \ \mathrm{d}\nu(y)$ by Proposition \ref{nov1}.

To show that the map is injective, assume that \eqref{A5} holds. Then 
$$
m(U) = \int_{Y_{\beta}} m_y(U) \ \mathrm{d} \nu(y)
$$
for all universally measurable subsets $U \subseteq P(\Gamma)$. Indeed, there are Borel sets $B_1 \subseteq  U \subseteq B_2$ such that $m(B_2 \backslash B_1) = 0$ and hence
$$
0 = \int_{Y_{\beta}} m_y(B_2 \backslash B_1) \ \mathrm{d} \nu(y),
$$
which implies that $m_y(U) = m_y(B_2)$ for all $\nu$-almost all $y$. 
In particular, it follows that for all universally measurable subsets $U \subseteq Y_{\beta}$,
\begin{equation}\label{25-02-18c}
m\left(W^{-1}\left(U\right)\right) = \int_{Y_{\beta}} m_y(W^{-1}(U)) \ \mathrm{d} \nu(y) \ .
\end{equation}
Consider an element $y \in Y_{\beta}$. Then $y = T(y')$ for some $y' \in C_{\beta}$ and by using the first two items in Corollary \ref{A3} we find that
\begin{equation*}
\begin{split}
&T^{-1}(y) = \left\{z \in C_{\beta} : \ T(z) = T(y') \right\} = \left\{ z \in C_{\beta} : K(z) = K(y') \right\} \\
&= \left\{ z \in C_{\beta} : K(z) = K(y) \right\} \\
&= \left\{ p \in C_{\beta}: \ \lim_{k \to \infty} K_{\beta}(v,s(p_k)) = \lim_{k \to \infty} K_{\beta}(v,s(y_k)) \ \forall v \in \Gamma_V \right\} \ ,
\end{split}
\end{equation*}
and hence
\begin{equation*}
\begin{split}
&W^{-1}(y) = \\
&\left\{ p \in Z(v_0) \cap \Wan(\Gamma): \ \lim_{k \to \infty} K_{\beta}(v,s(p_k)) = \lim_{k \to \infty} K_{\beta}(v,s(y_k)) \ \forall v \in \Gamma_V \right\} \ .
\end{split}
\end{equation*}
Therefore, if $y \notin U$, the set $W^{-1}(U)$ will be disjoint from the set
$$
\left\{ p \in X_{\beta}: \ \lim_{k \to \infty} K_{\beta}(v,s(p_k)) = \lim_{k \to \infty} K_{\beta}(v,s(y_k)) \ \forall v \in \Gamma_V \right\} \ 
$$
where $m_y$ is concentrated by Lemma \ref{04-12-17g}. Hence $m_y(W^{-1}(U)) = 0$ when $y \notin U$. When $y \in U$ it follows for the same reasons that
$$
m_y(W^{-1}(U)) = m_y(W^{-1}(y)) = m_y(Z(v_0)) = 1  \ .
$$
Inserted into \eqref{25-02-18c} it follows that
$$
m\left(W^{-1}\left(U\right)\right) = \nu(U) \ ,
$$
showing that $m$ determines $\nu$.

\end{proof}

In view of Proposition \ref{nov1} we have the following

\begin{cor}\label{04-12-17e} Let $\psi \in H^{v_0}_{\beta F}(\Gamma)$. There is a unique Borel probability measure $\nu$ on $Y_{\beta}$ such that 
\begin{equation}\label{04-12-17i}
\psi_v = \int_{Y_{\beta}} \lim_{k \to \infty} K_{\beta}(v, s(y_k)) \ \mathrm{d}\nu(y)
\end{equation}
for all $v \in \Gamma_V$. Conversely, for every Borel probability measure $\nu$ on $Y_{\beta}$ the equation \eqref{04-12-17i} defines an element $\psi \in H^{v_0}_{\beta F}(\Gamma)$. 
\end{cor}

\begin{remark}\label{04-12-17k} By taking $A(\beta)$ to be a stochastic matrix, Corollary \ref{04-12-17e} gives an integral representation of the harmonic functions of the associated countable state Markov chain, very similar to the Martin representation, cf. Theorem 4.1 in \cite{Sa}. With the approach presented here it is not necessary to introduce the Martin boundary. The subset $Y_{\beta}$ of $P(\Gamma)$ plays the role of the minimal Martin boundary, but no substitute of the Martin boundary itself is needed for the integral representation.  
\end{remark}

The fact that the set of $e^{\beta F}$-conformal measures on $P(\Gamma)$ in the transient case can be identified with the bounded Borel measures on a set of rays should not lead one to think that $e^{\beta F}$-conformal measures are concentrated on the set of rays. In fact, the $e^{\beta F}$-conformal measures very often annihilate the set of rays. The following gives a simple example of this.

\begin{example}\label{11-04-18} Consider the following digraph $\Gamma^0$:

\bigskip

\bigskip
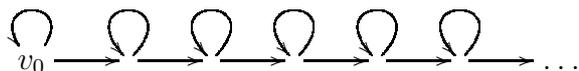
\begin{xymatrix}{\ar[r] v_0 \ar@(ur,ul) & \ar[r]  \ar@(ur,ul) & \ar[r]  \ar@(ur,ul)& \ar[r]  \ar@(ur,ul) & \ar[r]  \ar@(ur,ul) & \ar@(ur,ul) \ar[r] &  \hdots }
\end{xymatrix}
\smallskip Let $h > 0$. By adding return paths to $v_0$ as described in Section \ref{adding} we obtain a strongly connected recurrent digraph $\Gamma$ with Gurevich entropy $h(\Gamma) = h$. Then $\Gamma$ is meager with one end represented by the unique ray $r$ emitted from $v_0$. Note that $r$ is $\beta$-summable for all $\beta > h$ and that there is a unique $v_0$-normalized $e^{\beta}$-conformal measure $m$ on $P(\Gamma)$ when $\beta > h$ by Theorem \ref{08-01-18d}. The exit defined by $r$ is not $\beta$-summable in the sense of \cite{Th3}, and since all rays in $\Gamma$ are tail-equivalent to $r$ it follows therefore from Proposition 5.6 in \cite{Th3} that $m\left(\Ray(\Gamma)\right) = 0$. 

\end{example}

\section{Appendix}\label{OlaGeorge}

In this appendix we consider the tensor product of the examples from Theorem \ref{09-02-18KMS} with the examples constructed by Bratteli, Elliott and Herman in \cite{BEH}. This will lead to

\begin{thm}\label{05-08-18a} Let $F$ be a subset of positive real numbers which is closed as a subset of $\mathbb R $. There is a simple unital $C^*$-algebra with a continuous one-parameter group of automorphisms such that the Choquet simplex $S_{\beta}$ of $\beta$-KMS states is non-empty if and only if $\beta \in F$, and for $\beta,\beta' \in F$ the simplices $S_{\beta}$ and $S_{\beta'}$ are affinely homeomorphic only when $\beta = \beta'$.
\end{thm}
  
The input from \cite{BEH} is the following, cf. Theorem 3.2 in \cite{BEH}.

\begin{thm}\label{BEH} (Bratteli, Elliott and Herman) Let $L$ be a closed subset of $\mathbb R$. There is a simple unital $C^*$-algebra $B$ with a continuous one-parameter group $\gamma$ of automorphisms of $B$ such that there is a $\beta$-KMS state for $\gamma$ if and only if $\beta \in L$, and for each $\beta \in L$ the $\beta$-KMS state is unique.
\end{thm}

The algebra $B$ in Theorem \ref{BEH} is a corner in the crossed product of an AF-algebra by a single automorphism, constructed via the classification of AF-algebras which had just been completed around the time \cite{BEH} was written. The one-parameter group $\gamma$ is the restriction to the corner of the dual action on the crossed product.

In order to prove Theorem \ref{05-08-18a} we make first some elementary observations about KMS-states for tensor product actions. Let $A$ and $B$ be unital $C^*$-algebras, $\alpha$ a continuous one-parameter group of automorphisms of $A$ and $\gamma$ a continuous one-parameter group of automorphisms of $B$. On the minimal (or spatial) tensor product $A \otimes B$ we consider the tensor product action $\alpha \otimes \gamma$ defined such that
$$
\left(\alpha \otimes \gamma\right)_t(a \otimes b) = \alpha_t(a) \otimes \gamma_t(b) 
$$
when $t \in \mathbb R, \ a \in A$ and $b \in B$. When $\eta$ is a $\beta$-KMS for $\alpha$ and $\rho$ a $\beta$-KMS state for $\gamma$, the tensor product state $\eta \otimes \rho$ on $A \otimes B$, defined such that 
$$
\left(\eta \otimes \rho\right)(a \otimes b) = \eta(a)\rho(b) \ ,
$$
when $a \in A$ and $b \in B$, is easily seen to be a $\beta$-KMS state for $\alpha \otimes \gamma$. Conversely, when $\omega$ is $\beta$-KMS state of $\alpha \otimes \gamma$ the restrictions of $\omega$ to $A \subseteq A \otimes B$ and $B \subseteq A \otimes B$ are $\beta$-KMS states of $\alpha$ and $\gamma$, respectively. From these two observations it follows that the inverse temperatures of $(A\otimes B, \alpha \otimes \gamma)$, by which we mean the set of real numbers $\beta$ for which there is a $\beta$-KMS state for $\alpha \otimes \gamma$, is the intersection of the inverse temperatures for $(A,\alpha)$ and the inverse temperatures for $(B,\gamma)$. In the following proof these general observations will be combined with special features of the gauge action on graph $C^*$-algebras and the actions considered in \cite{BEH}.

\bigskip

\emph{Proof of Theorem \ref{05-08-18a}:} Let $ r$ be positive irrational number and set
$$
L = r F = \left\{ r s  : \ s \in F \right\} \ .
$$  
Let $(B,\gamma)$ be the $C^*$-algebra and one-parameter group arising from Theorem \ref{BEH} for this choice of $L$ and set
$$
\gamma'_t = \gamma_{rt} \ .
$$ 
Choose $h > 0$ such that $h < s$ for all $s \in F$ and let $\Gamma$ be the strongly connected row-finite graph from Theorem \ref{09-02-18KMS} corresponding to this $h$. Let $\alpha$ be the restriction to $P_{v_0}C^*(\Gamma)P_{v_0}$ of the gauge action on $C^*(\Gamma)$. We will show that the $C^*$-algebra $(P_{v_0}C^*(\Gamma)P_{v_0}) \otimes B$ and the one-parameter group $\alpha \otimes \gamma'$ have the properties specified in Theorem \ref{05-08-18a}. First observe that since $ P_{v_0}C^*(\Gamma)P_{v_0}$ and $B$ are simple it follows from a result of Takesaki, \cite{Ta}, that $(P_{v_0}C^*(\Gamma)P_{v_0}) \otimes B$ is simple. Next note that a state on $B$ is a $\beta$-KMS state for $\gamma'$ iff it is a $r\beta$-KMS state for $\gamma$, and that consequently $F$ is the set of inverse temperatures for $\gamma'$. By Theorem \ref{09-02-18KMS} the set of inverse temperatures for $\alpha$ is $[h,\infty)$; a set which contains $F$. It follows that $F$ is also the set of inverse temperatures for $\alpha \otimes \gamma'$. For $\beta \in F$ let $S_{\beta}$ be the simplex of $\beta$-KMS states for $\alpha$ and $\rho$ the unique $\beta$-KMS state for $\gamma'$. By Theorem \ref{09-02-18KMS} $S_{\beta}$ is not affinely homeomorphic to $S_{\beta'}$ when $\beta \neq \beta'$. It suffices therefore now to prove the following

\begin{obs}\label{obs} Let $\beta \in F$. The map 
$S_{\beta} \ni \omega_{\beta} \ \mapsto \ \omega_{\beta} \otimes \rho$ is an affine homeomorphism from $S_{\beta}$ onto the simplex of $\beta$-KMS states for $\alpha \otimes \gamma'$.
\end{obs}

Of course, only the surjectivity of the map in Observation \ref{obs} requires proof. For this we set $\mathcal P = \left\{ \mu \in P_f(\Gamma): \ s(\mu) =v_0\right\}$; the set of finite paths in $\Gamma$ emitted from $v_0$. Then $\{S_{\mu}S_{\nu}^* : \ \mu,\nu \in \mathcal P\}$ spans a dense $*$-subalgebra of $P_{v_0}C^*(\Gamma)P_{v_0}$. When $\mu, \nu \in \mathcal P$ we find that
\begin{equation*}
\begin{split}
&R^{-1}\int_0^R\alpha_t\left(S_{\mu}S_{\nu}^*\right) \ \mathrm{d} t \\ 
& = \begin{cases} S_{\mu}S_{\nu}^* \ , \ & \ \text{when} \ |\mu| = |\nu| \\ 
\\ \frac{1}{i R(|\mu| - |\nu|)} \left( e^{iR (|\mu| - |\nu|)} -1\right)S_{\mu}S_{\nu}^* \ , \ & \ \text{when} \ |\mu| \neq |\nu| \  \end{cases}
\end{split}
\end{equation*}
for all $R > 0$. It follows that the limit
$$
Q(a) = \lim_{R \to \infty} R^{-1}\int_0^R\alpha_t\left(a\right) \ \mathrm{d} t 
$$
exists for all $a \in P_{v_0}C^*(\Gamma)P_{v_0}$, defining a conditional expectation 
$$
Q : P_{v_0}C^*(\Gamma)P_{v_0} \to \left(P_{v_0}C^*(\Gamma)P_{v_0}\right)^{\alpha}
$$
onto the fixed-point $C^*$-algebra $\left(P_{v_0}C^*(\Gamma)P_{v_0}\right)^{\alpha}$ of $\alpha$ with the property that $Q(S_{\mu}S_{\nu}^*) = 0$ when $|\mu| \neq |\nu|$. 
Let $E_k, \ k \in \mathbb Z$, be the eigenspaces for $\gamma$, i.e.
$$
E_k = \left\{b \in B : \ \gamma_t(b) = e^{ ik t} b \ \ \forall t \in \mathbb R \right\} \ .
$$
Then, as was already pointed out just before the statement of Theorem 3.2 in \cite{BEH}, the $C^*$-algebra $B$ is the closed linear span of the eigenspaces $E_k, \ k \in \mathbb Z$.
Now let $\omega$ be a $\beta$-KMS state for $\alpha \otimes \gamma'$. Then $\omega \circ (\alpha \otimes \gamma')_t = \omega$ for all $t$ by Proposition 5.3.3 in \cite{BR}. When $b_k \in E_k$, $\mu,\nu \in \mathcal P$ and $|\mu| \neq |\nu|$, it follows that $|\mu| - |\nu| + kr \neq 0$ since $r$ is irrational and hence also that
\begin{equation*}
\begin{split} 
&\omega \left( (S_{\mu}S_{\nu}^*) \otimes b_k\right) \\
&= \frac{1}{R} \int_0^R \omega \left( \alpha_t(S_{\mu}S_{\nu}^*) \otimes \gamma'_t(b_k)\right) \ \mathrm{d} t \\
\\
& =  \frac{\omega \left((S_{\mu}S_{\nu}^*) \otimes b_k\right)}{i (|\mu| - |\nu| + k r)R} \left( e^{i (|\mu| - |\nu| + k r)R} - 1\right) \ 
\end{split}
\end{equation*}
for all $R > 0$. Letting $R \to \infty$ it follows that
$$
\omega \left( (S_{\mu}S_{\nu}^*) \otimes b_k\right) = 0 = \omega \left( (Q\left(S_{\mu}S_{\nu}^*\right)) \otimes b_k\right) \ .
$$
Thus
$$
\omega \left( (S_{\mu}S_{\nu}^*) \otimes b_k\right) = \omega \left( Q\left(S_{\mu}S_{\nu}^*\right) \otimes b_k\right)  
$$
for all $\mu,\nu \in \mathcal P$ since the identity is trivially true when $|\mu| = |\nu|$. As $\mu,\nu \in \mathcal P, \  k \in \mathbb Z$ and $b_k \in E_k$ were all arbitrary, it follows by linearity and continuity that $\omega$ factorises through $Q \otimes \id_B$, i.e. $$
\omega = \omega \circ \left(Q\otimes \id_B\right) \ .
$$
When $d \in \left(P_{v_0}C^*(\Gamma)P_{v_0}\right)^{\alpha}$ is a positive element the functional $\omega_d : B \to \mathbb C $ defined such that
$$
\omega_d(b) = \omega(d \otimes b)
$$
is a non-negative multiple of a $\beta$-KMS state for $\gamma'$ and it must therefore be a non-negative multiple of $\rho$, i.e.
$$
\omega(d \otimes b) = s(d)\rho(b) \ \ \forall b \in B \ ,
$$
 for some $s(d) \geq 0$. It follows that $s$ extends to a state $\omega'$ on $\left(P_{v_0}C^*(\Gamma)P_{v_0}\right)^{\alpha}$ such that $\omega (d\otimes b)  = \omega'(d)\rho (b)$ for all $d \in \left(P_{v_0}C^*(\Gamma)P_{v_0}\right)^{\alpha}$ and $b \in B$. Set $\omega_{\beta} = \omega' \circ Q$ and note that $\omega = \omega_{\beta} \otimes \rho$. Since $\omega$ is a $\beta$-KMS state for $\alpha \otimes \gamma'$ it follows that $\omega_{\beta}$ is a $\beta$-KMS state for $\alpha$, completing the proof of Observation \ref{obs} and hence also of Theorem \ref{05-08-18a}.

\newpage




	\section{ List of notation}

\begin{center}
\begin{tabular}{c c c c c c c c c c c c c }
Symbol&&Page&&&&&&&&Symbol & & Page\\
$\Gamma_{Ar}$  &&\pageref{GammaAr}&&&&&&&&$L_F(v,w)$  &  & \pageref{LFvw}\\
 $\Gamma_V$&& \pageref{GammaV}&&&&&&&& $[F]_p$   &  & \pageref{[F]p}\\
 $\Gamma_{Ar}$ &&  \pageref{GammaAr}&&&&&&&& $U_{F;I}$  & &  \pageref{UFI}\\
 $P(\Gamma)$  & & \pageref{PGamma}&&&&&&&& $\mathcal E(p)$  &  & \pageref{Ep}\\
$P_f(\Gamma)$  & & \pageref{PfGamma}&&&&&&&& $X_F$   &  &  \pageref{XF}\\
 $C^*(\Gamma)$  & & \pageref{C*g}&&&&&&&&$X_{\Gamma}$  &  &  \pageref{Xgamma}\\
 $\alpha^F$  & & \pageref{alphaF}&&&&&&&& $\mathcal F$ &   &  \pageref{matcF}\\
$A(\beta)$ & &\pageref{Abeta}&&&&&&&&$\mathcal C( \Gamma)$ &  &  \pageref{CGamma}\\
 $A(\Gamma)$ & &\pageref{AGamma}&&&&&&&& $\Int $  &  & \pageref{int}\\
$\beta(F)$&&\pageref{betaF}&&&&&&&&$\asymp$  &  &  \pageref{asymp}\\
$NW_{\Gamma}$&&\pageref{NWGamma}&&&&&&&& $\mathcal C^{\infty}(\Gamma)$ &  & \pageref{CinfG}\\
 $\Wan(\Gamma)$ & &\pageref{WanG}&&&&&&&&$\mathbb W(\mu)$  & &  \pageref{Wmu}\\
 $Z(\mu)$ & &\pageref{Zmu}&&&&&&&& $L_k(y)$&&\pageref{Lky}\\
    $M_{\beta F}(\Gamma)$ & &\pageref{MbetaF}&&&&&&&&$y[i,j[$ & &  \pageref{yij} \\
  $H_{\beta F}(\Gamma)$ & &\pageref{HbetaF}&&&&&&&& $B_k(y)$  &  &  \pageref{Bky}\\
  $m_{\psi}$ & & \pageref{mpsi} &&&&&&&& $\mathbb V_{\beta}(v,y)$ &  & \pageref{Vbetavy}  \\
  $b_v$ & & \pageref{bv} &&&&&&&& $\mathbb B_k(y)$  & & \pageref{BBky} \\
  $\Delta$ & & \pageref{delta}&&&&&&&&$\mathbb B(y)$  &  &  \pageref{BB(y)}\\
 $\partial \Delta$ &  & \pageref{partdelta}&&&&&&&&$\Rray(\Gamma)$ &  & \pageref{Rray} \\
$H^{v_0}_{\beta F}(\Gamma)$ & & \pageref{Hv0F}&&&&&&&&$\preceq$&&\pageref{prece}\\
 $\partial H^{v_0}_{\beta F}(\Gamma)$ & & \pageref{partHv0F} &&&&&&&&$\xi_R$&& \pageref{xiR}\\
  $M^{v_0}_{\beta F}(\Gamma)$ & & \pageref{Mv0F}&&&&&&&&$\Rray_{\beta}(\Gamma)$&&\pageref{RRaybeta}\\
  $\partial M^{v_0}_{\beta F}(\Gamma)$ & &\pageref{partMv0F} &&&&&&&& $\mathcal E_{\beta}(\Gamma)$&&\pageref{Ebetagamma}\\ 
$K_{\beta}$ & &\pageref{Kbeta}&&&&&&&&$m_E$&&\pageref{mE}\\
 $\partial K_{\beta}$ & &\pageref{partKbeta}&&&&&&&&$\Omega_D(w)$&&\pageref{OmegaD}\\
 $X_{\beta}$  & &  \pageref{Xbeta}&&&&&&&&$\partial D$&&\pageref{partD}\\
 $\mathcal B$ &  & \pageref{Bset}&&&&&&&&$\Br(\Gamma)$&&\pageref{BrGamma}\\
 $\Ray(\Gamma)$ &  & \pageref{Ray}&&&&&&&&$M(n)$&&\pageref{Mn}\\
$R$&& \pageref{R}&&&&&&&&$M(D;D')$&&\pageref{MDD}\\
$R^{\Gamma}_{\beta F}$ & & \pageref{RGammaF}&&&&&&&& $\Delta_{\beta F}$&&\pageref{31-10-2017}\\
$h(\Gamma)$  &  & \pageref{hGamma}&&&&&&&&$M^{v_0}_{\beta F}(E)$&&\pageref{MbetaFv0}\\
$l^n_{v,w}(\Gamma)$&&\pageref{lnvwgamma}&&&&&&&&&&\\
$p \to q$ &  & \pageref{ptoq}&&&&&&&&&&\\
 $\mathcal E(\Gamma)$ &  &  \pageref{Egamma}&&&&&&&&&&\\
$p \sim q$  &  & \pageref{psimq}&&&&&&&&&&\\

  \end{tabular}

\end{center}

\bigskip

\bigskip

\bigskip
\end{document}